\documentclass[12pt]{article}

\usepackage[colorlinks=true,allcolors=blue,bookmarksopen,bookmarksdepth=3,linktocpage=true]{hyperref}

\usepackage{amsmath,amssymb,amsthm,eucal,tikz,tikz-cd}

\usepackage[margin=1in]{geometry}

\usepackage[export]{adjustbox}

\usepackage{enumitem}
\setenumerate{label=(\roman*),topsep=1pt,itemsep=1pt,partopsep=0pt,parsep=0pt}

\newtheorem{theorem}{Theorem}[section]
\newtheorem{lemma}[theorem]{Lemma}
\newtheorem{corollary}[theorem]{Corollary}
\newtheorem{proposition}[theorem]{Proposition}
\newtheorem{conjecture}[theorem]{Conjecture}

\theoremstyle{definition}
\newtheorem{definition}[theorem]{Definition}
\newtheorem{proposition-definition}[theorem]{Proposition-Definition}
\newtheorem{lemma-definition}[theorem]{Lemma-Definition}
\newtheorem{construction}[theorem]{Construction}

\theoremstyle{remark}
\newtheorem{remark}[theorem]{Remark}
\newtheorem{example}[theorem]{Example}
\newtheorem{warning}[theorem]{Warning}

\numberwithin{equation}{section}

\newcommand{\crit}{\mathrm{crit}}
\newcommand{\subcrit}{\mathrm{subcrit}}
\renewcommand{\H}{\mathcal H}
\renewcommand{\Re}{\operatorname{Re}}
\renewcommand{\Im}{\operatorname{Im}}
\newcommand{\cyl}{\mathrm{cyl}}
\newcommand{\1}{\mathbf 1}
\newcommand{\J}{\mathcal J}
\newcommand{\A}{\mathcal A}
\newcommand{\B}{\mathcal B}
\newcommand{\C}{\mathcal C}
\newcommand{\E}{\mathcal E}
\newcommand{\ainf}{A_\infty}
\newcommand{\SSS}{\mathcal S}

\newcommand{\sL}{\mathcal L}

\newcommand{\ZZ}{\mathbb Z}
\newcommand{\CC}{\mathbb C}
\newcommand{\RR}{\mathbb R}
\newcommand{\PP}{\mathbb P}
\newcommand{\R}{\mathbf R}
\newcommand{\Q}{\mathcal Q}
\newcommand{\M}{\mathcal M}
\newcommand{\N}{\mathcal N}
\newcommand{\T}{\mathcal T}
\newcommand{\G}{\mathcal G}
\newcommand{\Y}{\mathcal Y}
\newcommand{\rr}{\mathfrak r}
\newcommand{\f}{\mathfrak f}
\newcommand{\g}{\mathfrak g}
\newcommand{\h}{\mathfrak h}
\newcommand{\cc}{\mathfrak c}

\newcommand{\Rbar}{\overline{\mathcal R}}
\newcommand{\Sbar}{\overline{\mathcal S}}

\newcommand{\OO}{\mathcal O}
\newcommand{\D}{\mathcal D}
\newcommand{\W}{\mathcal W}
\newcommand{\std}{\mathrm{std}}

\newcommand{\Nbd}{\operatorname{\mathcal{N}bd}}
\newcommand{\id}{\operatorname{id}}

\newcommand{\Ob}{\operatorname{Ob}}
\newcommand{\Tw}{\operatorname{Tw}}
\newcommand{\F}{\mathcal F}
\newcommand{\End}{\operatorname{End}}
\newcommand{\Pro}{\operatorname{Pro}}
\newcommand{\cone}{\operatorname{cone}}
\newcommand{\cones}{\operatorname{cones}}

\newcommand{\Set}{\operatorname{Set}}

\newcommand{\codim}{\operatorname{codim}}
\newcommand{\supp}{\operatorname{supp}}
\newcommand{\pre}{\operatorname{pre}}
\newcommand{\fiber}{[\operatorname{fiber}]}
\newcommand{\cocore}{[\operatorname{cocore}]}
\newcommand{\tildetimes}{\mathbin{\tilde\times}}
\newcommand{\op}{\mathrm{op}}

\newcommand{\inn}{\mathrm{in}}
\newcommand{\out}{\mathrm{out}}
\newcommand{\Fun}{\operatorname{Fun}}
\newcommand{\Hom}{\operatorname{Hom}}

\newcommand{\hocolim}{\operatornamewithlimits{hocolim}}
\newcommand{\charfol}{\mathrm{char.fol.}}
\newcommand{\const}{\mathrm{const}}
\newcommand{\pr}{\mathrm{prod}}
\newcommand{\sm}{\mathrm{sm}}
\newcommand{\Perf}{\operatorname{Perf}}
\newcommand{\Coh}{\operatorname{Coh}}
\newcommand{\Mod}{\operatorname{Mod}}
\newcommand{\Groth}{\operatornamewithlimits{Groth}}

\newcommand{\FS}{\operatorname{FS}}
\newcommand{\diss}{\mathrm{diss}}

\begin{document}

\title{Sectorial descent for wrapped Fukaya categories}

\author{
Sheel Ganatra,
John Pardon,
and
Vivek Shende
}

\date{August 27, 2023}

\maketitle

\begin{abstract}
We develop a set of tools for doing computations in and of (partially) wrapped Fukaya categories.
In particular, we prove (1) a descent (cosheaf) property for the wrapped Fukaya category with respect to so-called Weinstein sectorial coverings and (2) that the partially wrapped Fukaya category of a Weinstein manifold with respect to a mostly Legendrian stop is generated by the cocores of the critical handles and the linking disks to the stop.
We also prove (3) a `stop removal equals localization' result, and (4) that the Fukaya--Seidel category of a Lefschetz fibration with Liouville fiber is generated by the Lefschetz thimbles.
These results are derived from three main ingredients, also of independent use: (5) a K\"unneth formula (6) an exact triangle in the Fukaya category associated to wrapping a Lagrangian through a Legendrian stop at infinity and (7) a geometric criterion for when a pushforward functor between wrapped Fukaya categories of Liouville sectors is fully faithful.
\end{abstract}

\setcounter{tocdepth}{3}

\newpage
\tableofcontents
\newpage

\setcounter{section}{-1}

\section{Introduction}\label{intro}

The goal of this paper is to develop a set of computational tools for wrapped Fukaya categories.
The capstone result, called \emph{(Weinstein) sectorial descent}, gives a \v Cech-type decomposition of the wrapped Fukaya category $\W(X)$ of a Weinstein manifold (or sector) $X$ from a \emph{Weinstein sectorial cover} $X=X_1\cup\cdots\cup X_n$.  
In such a cover, all multiple intersections $X_{i_1}\cap\cdots\cap X_{i_k}$ are Weinstein sectors, and 
our previous work \cite{gpssectorsoc} yields a map 
\begin{equation}\label{descentpreintro}
\hocolim_{\varnothing\ne I\subseteq\{1,\ldots,n\}}\W\biggl(\bigcap_{i\in I}X_i\biggr)\rightarrow\W(X).
\end{equation}
On the left hand side, $\W(\bigcap_{i \in I} X_i)$ is defined in terms of the sector $\bigcap_{i\in I}X_i$ alone; 
in particular, its holomorphic disks and wrappings do not explore the larger space $X$. 

Sectorial descent (Theorem \ref{weinsteindescent}) is the assertion that \eqref{descentpreintro} 
induces an quasi-equivalence on triangulated envelopes, e.g.\ as implemented by categories of twisted complexes. 
This  is a variation on Kontsevich's conjecture \cite{kontsevichnotes} that the wrapped Fukaya
category of a Weinstein manifold $X$ is the category of global sections of a natural cosheaf of categories on any core of $X$.
Such a core is in general highly singular and depends on a choice of symplectic primitive (Liouville form). Sectorial descent, 
by contrast, has the virtue of neither requiring a discussion of singular spaces nor depending on a choice of primitive.
The relationship to Kontsevich's statement is that, roughly speaking, for any given choice of primitive, 
a cover of the core should lift to a Weinstein sectorial cover of the total space.
For example, a ribbon graph as the core of a punctured surface determines a covering by
 `$A_{n-1}$ Liouville sectors' ($D^2$ minus $n$ boundary punctures) corresponding to the vertices of the graph (with $n$ being the degree of a given vertex).
Our formulation 
also refines Kontsevich's original conjecture by giving a Floer-theoretic interpretation of the local categories: they are partially wrapped Fukaya categories.

The main tool in the proof of sectorial descent, and also our central object of study in this paper, is the \emph{partially wrapped Fukaya category}.
This is a category $\W(X,\f)$, depending not only on a Liouville manifold $X$ but also a closed subset $\f\subseteq\partial_\infty X$ known as the \emph{stop}.
Its objects are (possibly non-compact) exact Lagrangian submanifolds of $X$ disjoint from $\f$, with morphisms given by Floer cohomology after wrapping Lagrangians in the complement of $\f$.  The category is defined for arbitrary such $(X, \f)$, but our results are generally sharpest when $X$ is Weinstein and
$\f$ is \emph{mostly Legendrian} in the sense of admitting a closed-open decomposition $\f=\f^\subcrit\cup\f^\crit$ where $\f^\crit$ is Legendrian and the dimension of $\f^\subcrit$ is strictly smaller (see Definition \ref{mostlyLegdef}).  In this setting, 
we establish a package of new structural results of independent interest concerning partially wrapped Fukaya categories, including a K\"{u}nneth formula, an exact triangle, generation results, and geometric criteria for functors between such Fukaya categories to be fully faithful embeddings or to be localizations.
All these results are key ingredients in the proof of sectorial descent. 

The enormous variety of possible stops $\f$ gives great expressive power to the partially wrapped Fukaya category compared with the (fully) wrapped Fukaya category.
Let us give a simple illustration. 
The stopped Liouville manifold $(\CC,\{\pm\infty\})$ allows for a `stabilization' operation in the partially wrapped context, namely there is
a fully faithful embedding $\W(X)\hookrightarrow\W(X\times(\CC,\{\pm\infty\}))$.
By contrast, in the fully wrapped context there is no such embedding: $\W(X\times\CC) = 0$;
multiplying by base or fiber in $T^*S^1$ gives a map $\W(X)\to\W(X\times T^*S^1)$ which is faithful, but still not fully faithful.

The idea that Fukaya categories of non-compact Lagrangians with some sort of wrapping should exist and be of interest goes back to early work on mirror symmetry.
Fukaya categories of Landau--Ginzburg models $f: X \to \CC$, which were introduced by Kontsevich \cite{kontsevich1998lectures}
and developed by Seidel \cite{seidelbook}, are 
partially wrapped Fukaya categories with stop $\f = f^{-1}(-\infty)$.
The wrapped Fukaya category defined by Abouzaid--Seidel \cite{abouzaidseidel} is the case $\f=\varnothing$.
The general framework of partially wrapped Fukaya categories was introduced by Auroux \cite{aurouxborderedfloericm,aurouxborderedfloergokova}, and precise definitions in the case that the stop is a \emph{Liouville hypersurface} $F\subseteq\partial_\infty X$ are given in \cite{sylvanthesis,gpssectorsoc}.
In addition, the study of Legendrian contact homology for Legendrian $\Lambda$, going back to ideas of Chekanov \cite{chekanov} and Eliashberg \cite{eliashberg-icm}, can be understood as the study of a partially wrapped Fukaya category with $\f = \Lambda$ \cite{sylvantalk, ekholm-lekili}.  

The above situations can largely be subsumed into the `mostly Legendrian' setup, because (as we show) 
$\W(X,F) = \W(X, \operatorname{core}(F))$; when $F$ is Weinstein, its core is mostly Legendrian. 
However, mostly Legendrian stops arise naturally in other ways as well (for example as conormals to stratifications), 
and it is highly desirable to have a theory which does not require the stop to be the core of a Liouville hypersurface.

The key property of mostly Legendrian stops $\f$ is 
that a generic isotopy of Legendrians inside $\partial_\infty X$ will intersect $\f$ at a discrete set of times by passing transversally through $\f^\crit$.
The effect of such a crossing can be quantified using  
the \emph{Lagrangian linking disks} $D_p$ at the smooth Legendrian points $p\in\f^\crit$.
Namely, if 
$L^w\subseteq X$ is obtained from $L$ by passing $\partial_\infty L$ positively transversally through a smooth Legendrian point $p\in\f$, 
then there is an exact triangle $L^w\to L\to D_p$ in $\W(X,\f)$ (Theorem \ref{wrapcone}).
This \emph{wrapping exact triangle} is fundamental to our work in this paper.

As a first application of stabilization and the wrapping exact triangle, we offer a suprisingly simple proof that the wrapped Fukaya category of a Weinstein manifold 
is generated by its cocores (Theorem \ref{generation}); another proof of this result is due independently to Chantraine--Dimitroglou Rizell--Ghiggini--Golovko \cite{cdrgggeneration}.

The wrapping exact triangle also allows one to relate the partially wrapped Fukaya categories with respect to different stops, 
which can greatly simplify computations of such categories.  Indeed, one reason such computations are hard is
that direct computation of $\W(X, \f)$ involves understanding explicitly the long time Reeb flow on $\partial_\infty X\setminus\f$; 
for most $\f$, this is essentially intractable.  
However, enlarging the stop can simplify the Reeb flow.
That is, one can often find $\g \supseteq \f$ for which the category $\W(X,\g)$ is easy to calculate or 
has convenient categorical properties such as properness, etc.
The two categories are then related by the \emph{stop removal formula} (Theorem 
\ref{stopremoval}): for stops $\f\subseteq\g$ with $\g\setminus\f$ mostly Legendrian, 
the pushforward functor $\W(X,\g)\to\W(X,\f)$ is the quotient by the Lagrangian linking disks of $(\g\setminus\f)^\crit$ (aka meridians).
This formula reduces the study of $\W(X,\f)$ to the study of $\W(X,\g)$ together with the linking disks of $(\g\setminus\f)^\crit$.
It is a direct consequence of the wrapping exact triangle; its precursors include \cite{abouzaidseidelunpublished, sylvanthesis}.

Of particular importance is the use of stop enlargement to create fully faithful functors.
Given an inclusion of Liouville sectors $X\hookrightarrow Y$, we may add a stop to $Y$ along the `exit set' of $\partial\partial_\infty X$ to obtain a Liouville sector $Y_{|X}$ and a chain of inclusions $X\hookrightarrow Y_{|X} \hookrightarrow Y$.
Now the induced functor $\W(X)\to\W( Y_{|X})$ is fully faithful, simply because Lagrangians in $X$, when wrapped in $ Y_{|X}$, fall immediately into the stop without exploring the rest of $ Y_{|X} \setminus X$.
As a special case: if $F$ is a Liouville hypersurface inside $\partial_\infty Y$, then there are fully faithful embeddings $\W(F)\hookrightarrow\W(F\times T^*[0,1])\hookrightarrow\W(Y,F\sqcup F^+)$, where $F^+$ denotes a small positive pushoff of $F$.
Even a complicated $F$ may sit in a simple $Y$ (e.g.\ a cotangent bundle).

In the factorization $X\hookrightarrow Y_{|X} \hookrightarrow Y$, the same reasoning shows moreover that $\W(Y_{|X})$ contains both $\W(X)$ and $\W(Y\setminus X)$ as semi-orthogonal full subcategories.
Adding appropriate Weinstein hypotheses so as to know generation by cocores, this is moreover a semi-orthogonal decomposition $\W(Y_{|X}) = \langle \W(Y\setminus X), \W(X) \rangle $.

Applying this construction in the setting of a sectorial cover $X = Y \cup Z$ produces semi-orthogonal decompositions 
\begin{align}
\W(Z_{|Y \cap Z}) & =   \langle \W(Y \cap Z), \W(Z \setminus Y \cap Z) \rangle \\
\W(Y_{|Y\cap Z}) & =  \langle \W(Y \cap Z), \W(Y \setminus Y \cap Z) \rangle \\
\W(X_{|Y \cap Z}) & =  \langle \W(Y \cap Z),  \W(Y \setminus Y \cap Z) \oplus  \W(Z \setminus Y \cap Z) \rangle 
\end{align}
from which it follows formally that
\begin{equation}
\hocolim \biggl( \W(Y \cap Z)  \rightrightarrows \W(Z_{|Y \cap Z}) \oplus \W(Y_{|Y\cap Z}) \biggr) \xrightarrow{\sim} \W(X_{|Y \cap Z}).
\end{equation}
Appealing to stop removal and the fact that homotopy colimits commute with localizations, we deduce the corresponding formula for the covering $X = Y \cup Z$.
Arguing by induction establishes sectorial descent.

\section{Statement of results}\label{statementresults}

\subsection{Partially wrapped Fukaya categories}

Let us begin by fixing our notation and terminology for partially wrapped Fukaya categories (see \S\ref{partiallywrappedsection} for a full treatment); our setup is rather more general than that considered before \cite{aurouxborderedfloericm,aurouxborderedfloergokova,sylvanthesis,katzarkovkerr,gpssectorsoc}.

Recall that a \emph{Liouville manifold} is an exact symplectic manifold $(X,\lambda)$ which is `cylindrical at infinity', meaning that the complement of a compact set is identified (necessarily uniquely) with the positive half of the symplectization of a contact manifold, denoted $\partial_\infty X$.
The adjective \emph{cylindrical} applied to an object living on an Liouville manifold means `invariant under the Liouville flow'; by default, `cylindrical' means `cylindrical at infinity', unless otherwise specified.

For a Liouville manifold $X$ and any closed subset $\f\subseteq\partial_\infty X$, henceforth referred to as a \emph{stop}, we denote by $\W(X,\f)$ the \emph{partially wrapped Fukaya category of $X$ stopped at $\f$}: this is the Fukaya category whose objects are exact cylindrical (at infinity)
Lagrangians inside $X$ disjoint at infinity from $\f$ and whose morphisms are calculated by wrapping Lagrangians in the complement of $\f$.
For an inclusion of stops $\f\supseteq\f'$, there is a canonical pushforward functor 
\begin{equation}\label{partialpushforward}
\W(X,\f)\to\W(X,\f').
\end{equation}
We will refer to a pair $(X,\f)$ as a \emph{stopped Liouville manifold}.
More generally, we can take $X$ to be a Liouville sector in the following sense:

\begin{definition}[{Liouville sector \cite[Definition 2.4]{gpssectorsoc}}]\label{liouvillesectordefn}
A \emph{Liouville sector} is an exact symplectic manifold-with-boundary $(X,\lambda)$ which is cylindrical at infinity and for which there exists a function $I:\partial X\to\RR$ which is linear at infinity and whose Hamiltonian vector field $X_I$ is outward pointing along $\partial X$.
\end{definition}

\begin{remark}[Coordinates near the boundary of a Liouville sector]\label{sectorbdrycoords}
The existence of the ``defining function'' $I$ above is equivalent to the existence of coordinates of the form $X=F\times\CC_{\Re\geq 0}$ (or, equivalently, $X=F\times T^\ast\RR_{\geq 0}$) over a cylindrical neighborhood of the boundary, where $F$ is a Liouville manifold called the \emph{symplectic boundary} of $X$ (see \cite[\S 2.6]{gpssectorsoc}).
The associated projection $\pi:\Nbd^Z\partial X\to\CC_{\Re\geq 0}$ (where $\Nbd^Z$ means `a cylindrical neighborhood of')
prevents holomorphic curves in $X$ from approaching $\partial X$ (for almost complex structures which make $\pi$ holomorphic, of which there is a plentiful supply).
\end{remark}

A closely related notion is that of \emph{Liouville pair} 
$(\bar X,F)$ \cite{avdek,avdekjsg,eliashbergweinsteinrevisited}, consisting of a Liouville manifold  $\bar X$  and   a \emph{Liouville hypersurface} $F_0\subseteq\partial_\infty\bar X$, 
i.e.\ a hypersurface-with-boundary $F_0$ 
along with a contact form $\alpha$ on $\partial_\infty\bar X$ for which $(F_0,\alpha|_{F_0})$ is a Liouville domain whose completion is $F$.
By removing from $\bar X$ a standard neighborhood of $F_0$, we obtain a Liouville sector $X=\bar X\setminus\Nbd F_0$ \cite[Definition 2.14]{gpssectorsoc}.
Going in the other direction, $\bar X$ may be obtained from $X$ by gluing $F\times\CC_{\Re\leq\varepsilon}$ onto $X$ via the coordinates $X=F\times\CC_{\Re\geq 0}$ from Remark \ref{sectorbdrycoords}.
Up to contractible choice, this gives a 
correspondence between Liouville sectors $X$ and Liouville pairs $(\bar X,F)$ (see \cite[Lemma 2.32]{gpssectorsoc}).  We refer to
$\bar X$ as the \emph{convex completion} (or \emph{convexification}) of $X$ and to $F$ as the \emph{symplectic boundary} of $X$.

Given a stopped Liouville sector $(X,\f)$ (meaning $\f$ is a closed subset of $(\partial_\infty X)^\circ:=\partial_\infty X\setminus\partial\partial_\infty X$), 
there is a partially wrapped Fukaya category $\W(X,\f)$ as above, where we compute morphisms by wrapping inside $(\partial_\infty X)^\circ\setminus\f$.
For an inclusion of stopped Liouville sectors $(X,\f)\hookrightarrow(X',\f')$ (meaning $X\hookrightarrow X'$ and $\f'\cap(\partial_\infty X)^\circ\subseteq\f$), there is a pushforward functor $\W(X,\f)\to\W(X',\f')$ (the case $\f = \f' = \varnothing$ was described in \cite[\S 3]{gpssectorsoc}).

For a Liouville pair $(\bar X,F)$ and $X=\bar X\setminus\Nbd F_0$ the associated Liouville sector, the natural functor
\begin{equation}\label{sectorstopequivalence}
\W(X)\xrightarrow\sim\W(\bar X,F_0)
\end{equation}
is a quasi-equivalence (see Corollary \ref{horizsmallstop}).

The \emph{core} $\cc_F=:\f$ of $F$ refers to the set of points which do not escape to the boundary under the Liouville flow.
The core is a closed subset, in general highly singular.
In the other direction, we call $F_0$ a \emph{ribbon} for $\f$.
The natural functor
\begin{equation}\label{stopcoreequivalence}
\W(\bar X,F_0)\xrightarrow\sim\W(\bar X,\f)
\end{equation}
is also a quasi-equivalence (see Corollary \ref{horizsmallstop}).
It is frequently convenient to have at hand both descriptions $\W(X)$ and $\W(\bar X,\f)$ of the same Fukaya category.

\begin{remark}[Existence and choice of ribbons]\label{ribbon}
It is likely a subtle question whether a given closed subset $\f\subseteq\partial_\infty X$ admits a ribbon, and whether such a ribbon is unique.
A choice of ribbon is needed to define the corresponding Liouville sector, which may be desirable in that
Fukaya categories of sectors admit a greater range of pushforward functors. 
(Note however that
the pair $(X,\partial_\infty X\setminus\f)$ is an \emph{open Liouville sector} in the sense of \cite[Remark 2.8]{gpssectorsoc} given only the 
existence of a ribbon for $\f$.)
Of course, the partially wrapped Fukaya category $\W(X,\f)$ is defined without any assumptions at all about ribbons for $\f$.
\end{remark}

A deformation of Liouville sectors $X$ or of Liouville pairs $(\bar X,F)$ induces an equivalence on partially wrapped Fukaya categories.
Note, however, that during a deformation of Liouville pairs, the core $\f$ of $F$ may change rather drastically, and hence it is natural to ask the general question of which sorts of deformations of a stop $\f$ induce equivalences on partially wrapped Fukaya categories.
Among other results in this direction, we show (in \S\ref{sec: ccc}) that constancy of the complement of $\f$ as a contact manifold is sufficient:

\begin{theorem}\label{winvariancestrong}
Let $X$ be a Liouville sector, and let $\f_{[0,1]}\subseteq(\partial_\infty X)^\circ\times[0,1]$ be a closed subset.
If the projection of the complement of $\f_{[0,1]}$ to $[0,1]$ is trivial as a family of contact manifolds, then there is a natural quasi-equivalence
\begin{equation}
\W(X,\f_0)=\W(X,\f_1)
\end{equation}
where $\f_t$ denotes the fiber of $\f_{[0,1]}$ over $t\in[0,1]$.
\end{theorem}

\subsection{K\"unneth embedding}

For stopped Liouville manifolds $(X,\f)$ and $(Y,\g)$, we denote by $(X,\f)\times(Y,\g)$ the product $X\times Y$ equipped with the product stop
\begin{equation}
(\f\times\cc_Y)\cup(\f\times\g\times\RR)\cup(\cc_X\times\g),
\end{equation}
where $\cc_X$ and $\cc_Y$ denote the cores of $X$ and $Y$, respectively.
To make sense of this definition, note that if we let $\cc_{X,\f}:=\cc_X\cup(\f\times\RR)\subseteq X$ denote the ``relative core'' (where $\f\times\RR\subseteq X$ indicates the locus of points which limit to $\f$ under the positive Liouville flow), then we have $\cc_{(X,\f)\times(Y,\g)}=\cc_{X,\f}\times\cc_{Y,\g}$.

Floer theory is generally well-behaved under taking products, but in the wrapped setting, difficulties arise because 
the product of cylindrical structures is not generally cylindrical.  This is because `cylindrical' is a condition at infinity,
and the infinite end of a product is not just the product of the infinite ends.  
Nevertheless, some results have been established in the wrapped setting: 
Oancea \cite{oanceakunneth} proved a K\"unneth formula for symplectic cohomology of Liouville manifolds (see also Groman \cite{groman}).
The analogous problem for wrapped Fukaya categories was first studied by Gao \cite{gaokunneth,gaokunneth2}, who constructed a K\"unneth functor after enlarging the target category to include products of cylindrical Lagrangians (which are themselves not generally cylindrical).

In \S\ref{sec: kunneth}, we construct a K\"unneth functor for partially wrapped Fukaya categories.
We employ a cylindrization procedure for products of Lagrangians (see \S\ref{sec: products}) which allows us to avoid enlarging the target category.

\begin{theorem}[K\"unneth embedding]\label{kunneth}
For Liouville sectors $X$ and $Y$, there is a fully faithful $\ainf$-bilinear-functor%
\footnote{The notation $\A \otimes \B \to \C$ indicates an $\ainf$-bilinear-functor from $(\A,\B)$ to $\C$ in the sense of \cite{lyubashenkomultilinear}; we attach no meaning to the `tensor product of $\ainf$-categories' notation $\A \otimes \B$ on its own.  Given a suitable definition of a tensor product $\ainf$-category $\A\otimes\B$, these notions should be quasi-equivalent.}
\begin{equation}
\W(X)\otimes\W(Y)\hookrightarrow\W(X\times Y),
\end{equation}
and for stopped Liouville manifolds $(X,\f)$ and $(Y,\g)$, there is a fully faithful functor
\begin{equation}
\W(X,\f)\otimes\W(Y,\g)\hookrightarrow\W((X,\f)\times(Y,\g)).
\end{equation}
Both these functors send $(L\subseteq X,K\subseteq Y)$ to a canonical cylindrical perturbation $L\tildetimes K\subseteq X\times Y$ of the product $L\times K\subseteq X\times Y$.  If $\lambda_X|_L\equiv 0$ and $\lambda_Y|_K\equiv 0$, then $L\times K$ is already cylindrical and no perturbation is necessary.
\end{theorem}

The K\"unneth embedding immediately gives rise to ``stabilization functors''
\begin{align}
\label{kunnethstabsector}\W(X)&\hookrightarrow\W(X\times T^\ast[0,1]),\\
\label{kunnethstabstopped}\W(X,\f)&\hookrightarrow\W(X\times\CC,(\cc_X\times\{\pm\infty\})\cup(\f\times\RR)),
\end{align}
(the former for Liouville sectors sending $L \mapsto L \tildetimes\fiber$, the latter for stopped Liouville manifolds sending $L \mapsto L \tildetimes i\RR$), which are of particular interest and use.

The K\"unneth embedding gives rise to a notion of ``representability'' for bimodules: one can ask whether a given $(\W(X)\otimes\W(Y))$-bimodule is the pullback of a Yoneda module on $\W(X\times Y)$.
We consider specifically the product $(X^-,\f)\times(X,\f)$ for $(X,\f)$ a stopped Liouville manifold.
Since $\W(X^-,\f)=\W(X,\f)^\op$ (Lemma \ref{oppositesw}), an object of $\W((X^-,\f)\times(X,\f))$ pulls back under K\"unneth to a $\W(X,\f)$-bimodule.
The diagonal $\Delta\subseteq X^-\times X$ is not an object of $\W((X^-,\f)\times(X,\f))$, as it runs into the stop at infinity.
In favorable cases (for example, when $\f$ has a ribbon, see Lemma \ref{diagonalpushoff}), there exist \emph{positive/negative pushoffs} $\Delta^\pm$ of $\Delta$ which are disjoint from the product stop (positively/negatively pushing off $\Delta$ from the product stop is equivalent to positively/negative displacing the stop $\f$ from itself; see \S\ref{diagonalsection} for further discussion).

\begin{proposition}[Pulling back the diagonal]\label{kunnethdiagonal}
Let $(X,\f)$ be a stopped Liouville manifold.
The Yoneda module $\Hom(\Delta^-,-)$ of any negative pushoff $\Delta^-$ of the diagonal pulls back under the K\"unneth embedding of Theorem \ref{kunneth} to the diagonal bimodule of $\W(X,\f)$.
\end{proposition}

This result comes with a small caveat: it requires the Floer data used to define the K\"unneth functor in Theorem \ref{kunneth} to be chosen carefully (we expect, but do not quite prove, that the K\"unneth functors we define are independent of the auxiliary choices going into their definition).

\subsection{Mostly Legendrian stops}\label{singularisotropicstopsection}

In this paper, we are particularly interested in stops which are mostly Legendrian in the sense of the following working definition:

\begin{definition}[Mostly Legendrian]\label{mostlyLegdef}
A closed subset $\f$ of a contact manifold $Y^{2n-1}$ is called \emph{mostly Legendrian} iff it admits a decomposition $\f=\f^\subcrit\cup\f^\crit$ for which $\f^\subcrit$ is closed and is contained in the smooth image of a second countable manifold of dimension $<n-1$ (i.e.\ strictly less than Legendrian), and $\f^\crit\subseteq Y\setminus\f^\subcrit$ is a Legendrian submanifold.
If $\f$ is mostly Legendrian, then there is a canonical largest choice of $\f^\crit\subseteq\f$ as the smooth Legendrian locus of $\f$ and smallest choice of $\f^\subcrit=\f\setminus\f^\crit$ as its complement.
The notion of a \emph{mostly Lagrangian} closed subset of a symplectic manifold is defined analogously.
\end{definition}

The key property of this definition is that a generic positive Legendrian isotopy will intersect a given mostly Legendrian $\f$ only by passing through $\f^\crit$ transversally (Lemma \ref{genericpositiveisotopy}).
In most applications, the relevant mostly Legendrian stops admit some sort of reasonable \emph{finite} (or at least locally finite) stratification by \emph{disjoint} locally closed \emph{isotropic} submanifolds.
However, the definition allows for rather more general phenomena (see Example \ref{mostlyLegexa}).

Weinstein hypersurfaces (i.e.\ Liouville hypersurfaces which are Weinstein) are an important source of mostly Legendrian stops.
Recall that a \emph{Weinstein manifold} is a Liouville manifold $W$ for which the Liouville vector field $Z$ is gradient-like with respect to a proper Morse function $\phi:W\to\RR_{\geq 0}$ (see \cite{cieliebakeliashberg}).
The zeroes of the Liouville vector field on a Weinstein manifold have index $\leq\frac 12\dim W$; those for which this inequality is strict are called \emph{subcritical}, and those of index $=\frac 12\dim W$ are called \emph{critical}.
We extend this terminology to isotropic submanifolds of symplectic and contact manifolds: critical isotropics are those which are Lagrangian/Legendrian, and those of smaller dimension are called subcritical.
The core $\cc_W$ of a Weinstein manifold $W$ is the union of the cores of the handles (i.e.\ the stable manifolds of the zeroes of $Z$).
If the cocores of the critical handles are properly embedded (this is a generic condition), 
then writing $\cc_W=\cc_W^\subcrit\cup\cc_W^\crit$ as the union of the cores of the subcritical and critical handles, 
respectively, we see that $\cc_W^\subcrit$ is closed, and hence this decompositions exhibits $\cc_W$ as mostly Lagrangian.
Additionally, as $Z$ is tangent to the cores of the handles, the Liouville form vanishes identically on them.
Thus the core $\f$ of a Weinstein hypersurface $F_0\subseteq\partial_\infty\bar X$ (with properly embedded critical cocores) is mostly Legendrian.

\begin{remark}\label{ribbonhard}
Even for the mostly Legendrian stops which arise naturally in practice, the question of whether they admit Weinstein ribbons (possibly after small deformation) does not have an obvious answer.
It is hence important, from a practical standpoint, to have a theory which applies in the generality of mostly Legendrian stops, without any assumptions about the existence of ribbons.
\end{remark}

\begin{example}\label{mostlyLegexa}
The union $\f$ of a Legendrian of dimension $>0$ and a sequence of points limiting to a point on the Legendrian is mostly Legendrian---one takes $\f^\crit$ to be the Legendrian minus the limit point.
A cantor set contained in a submanifold of dimension $<n-1$ is also mostly Legendrian.
On the other hand, if $\f\subseteq Y$ is a Legendrian and $\g\subseteq Y$ is $\f$ union a Legendrian in $Y\setminus\f$ accumulating at all points of $\f$, then $\f\subseteq Y$ and $\g\setminus\f\subseteq Y\setminus\f$ are both mostly Legendrian, but $\g\subseteq Y$ is not.
Note that the maps covering $\f^\subcrit$ need not be disjoint: the union of two submanifolds of dimension $<n-1$ interesecting along a cantor set is mostly Legendrian.
Finally, note that there is no constraint on behavior approaching the boundary: the collection of conormals to inverses of integers $\{\frac 1n\}_{n\geq 2}$ is mostly Legendrian inside $(\partial_\infty T^*[0,1])^\circ$, although its closure inside $(\partial_\infty T^*[-1,1])^\circ$ is not.
\end{example}

\subsection{Wrapping exact triangle}

Underlying almost all of our results is an exact triangle (constructed in \S\S\ref{lagrdisksection}--\ref{wet}) describing the effect in the Fukaya category of ``wrapping through a stop''.
To state it, recall that given a local Legendrian submanifold $\Lambda\subseteq\partial_\infty X$ near a point $p\in\Lambda$, there is a \emph{Lagrangian linking disk} $D_p\subseteq X$ whose boundary at infinity is a Legendrian unknot linking $\Lambda$ at $p$ (also known as `meridians', these are defined in detail in \S\ref{linkingdisksection}).
These linking disks may be thought of as `cocores at infinity'; precisely, if a stop $\f$ is the core of a Weinstein hypersurface $F_0$, then the canonical embedding $F\times T^*[0,1]\hookrightarrow(X,\f)$ sends a cocore in $F$ times a fiber of $T^*[0,1]$ to the linking disk at the corresponding smooth Legendrian point of $\f$.

\begin{theorem}[Wrapping exact triangle]\label{wrapcone}
Let $(X,\f)$ be a stopped Liouville sector, and let $p\in\f$ be a point near which $\f$ is a Legendrian submanifold.
If $L\subseteq X$ is an exact Lagrangian submanifold and $L^w\subseteq X$ is obtained from $L$ by passing $\partial_\infty L$ through $\f$ transversally at $p$ in the positive direction, then there is an exact triangle
\begin{equation}
L^w\to L\to D_p\to
\end{equation}
in $\W(X,\f)$, where $D_p\subseteq X$ denotes the Lagrangian disk linking $\f$ at $p$ and the map $L^w\to L$ is the continuation map.
\end{theorem}

Many previous authors have also found exact triangles in the Fukaya category associated to other geometric operations, e.g.\ Seidel's exact triangle of a Dehn twist \cite{seideldehntwist}, the exact triangle associated to Polterovich surgery (which was studied on Floer cohomology in \cite{fooochapter10}), and Biran--Cornea's work on 
exact triangles in the Fukaya category of $M$ associated to Lagrangian cobordims in $M \times \CC$ \cite{birancorneacobordism1}.

We will deduce the wrapping exact triangle from a more general `surgery at infinity' exact triangle, plus a geometric argument relating surgery
with a linking disk to wrapping past a stop.  The surgery at infinity applies to two 
Lagrangians $L,K\subseteq X$ together with a contact Darboux chart of a specific form (see Figure \ref{figurehandlebeforeafter}).
It attaches a \emph{relatively non-exact embedded Lagrangian $1$-handle} (defined in \S\ref{onehandlesection}) to $L\sqcup K$ at infinity to produce a Lagrangian $L\#_\gamma K$.
The notation $\gamma$ refers to an obvious `short' Reeb chord from $\partial_\infty L$ to $\partial_\infty K$ in the Darboux chart of the surgery; we call $\gamma$ the `center' of the surgery.
Topologically, $L\#_\gamma K$ is simply the boundary connect sum of $L$ and $K$ along $\gamma$.
Figure \ref{figurehandlecontact} gives a picture in dimension two.

\begin{figure}[ht]
\centering
\includegraphics[max width=.95\textwidth]{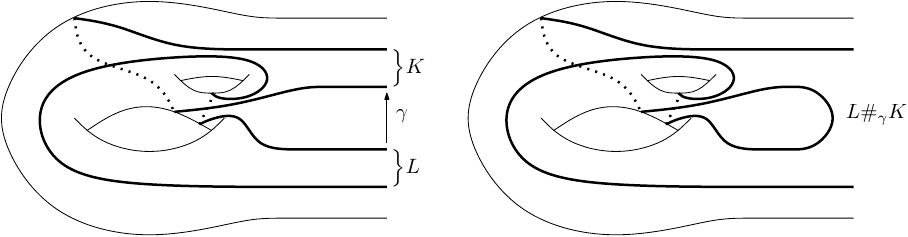}
\caption{Left: Two Lagrangians $L$ and $K$ together with a Reeb chord $\gamma$ from $\partial_\infty L$ to $\partial_\infty K$.  Right: The result $L\#_\gamma K$ of attaching a relatively non-exact embedded Lagrangian $1$-handle to $L\sqcup K$ with center $\gamma$.}\label{figurehandlecontact}
\end{figure}

\begin{figure}[ht]
\centering
\includegraphics[max width=.95\textwidth]{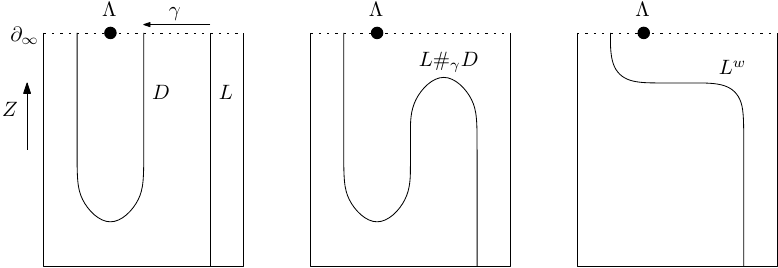}
\caption{Picture proof of Proposition \ref{wrapishandle} in dimension two.  This picture provides at least a moral proof in all dimensions by taking product with $(\CC^{n-1},\RR^{n-1})$.}\label{figureonehandlecancel}
\end{figure}

The relationship between surgery and wrapping in dimension two can be seen in Figure \ref{figureonehandlecancel}.  We show
in \S\ref{wrapishandlesection} the higher dimensional analogue:

\begin{proposition}\label{wrapishandle}
Let $X$ be a Liouville sector, $L\subseteq X$ a cylindrical Lagrangian, and $\Lambda\subseteq\partial_\infty X$ a Legendrian.
Let $L\leadsto L^w$ be a positive wrapping which passes through $\Lambda$ exactly once, transversely, at $p\in\Lambda$.
Then $L^w$ is isotopic to the result $L\#_\gamma D_p$ of attaching a relatively non-exact embedded Lagrangian $1$-handle to $L$ and the linking disk $D_p$.
\end{proposition}

Finally, the surgery triangle:

\begin{proposition}[Surgery exact triangle]\label{handlecone}
Let $(X, \f)$ be a stopped Liouville sector.
Let $L,K\subseteq X$ be disjoint exact Lagrangians (disjoint from $\f$ at infinity), and let $L\#_\gamma K$ be the result of attaching (in the complement of $\f$) a relatively non-exact embedded Lagrangian $1$-handle to $L$ and $K$ with center $\gamma$.
There is an exact triangle in $\W(X,\f)$:
\begin{equation}
L \xrightarrow\gamma K \to L \#_\gamma K \to.
\end{equation}
\end{proposition}

The existence of some triangle with the above terms follows from an action filtration argument, which we make in rather 
greater generality (see Proposition \ref{cobordismtwisted}), plus some further arguments to identify the first map as $\gamma$. 
The wrapping exact triangle (Theorem \ref{wrapcone}) is established by combining Propositions \ref{wrapishandle} and \ref{handlecone},
with an additional argument to identify the map $L^w \to L$ as the continuation map.

\subsection{Generation by cocores and linking disks}

An important problem in Floer theory is to find objects which generate the Fukaya category.
In \S\ref{generationsection}, we show how to prove a number of different generation results using the wrapping exact triangle.

\begin{theorem}[Generation by cocores and linking disks]\label{generation}
Let $(X,\f)$ be a stopped Weinstein manifold with $\f$ mostly Legendrian.
Suppose that the cocores of the critical handles of $X$ are properly embedded and disjoint from $\f$ at infinity.
Then $\W(X,\f)$ is generated by the cocores of the critical handles and the linking disks of $\f^\crit$.
\end{theorem}

The linking disk $D_p$ at $p\in\f^\crit$ varies continuously in $p$, and thus 
its isomorphism class depends only on the connected component of $p$ in $\f^\crit$.
Thus we need only take one such linking disk for each component of $\f^\crit$ in Theorem \ref{generation}.

The case of Theorem \ref{generation} in which $\f$ admits a Weinstein ribbon (compare Remark \ref{ribbonhard}) was recently proven independently by Chantraine--Dimitroglou Rizell--Ghiggini--Golovko \cite{cdrgggeneration}.%
\footnote{The results of \cite{cdrgggeneration} are phrased in the equivalent (by \eqref{sectorstopequivalence} and \eqref{stopcoreequivalence}) language of wrapped Fukaya categories of \emph{Weinstein sectors}, which are Liouville sectors $X$ whose convexification $\bar X$ and symplectic boundary $F$ are both (up to deformation) Weinstein.
They show that $\W(X)$ is generated by the cocores of $\bar X$ together with the stabilized cocores of $F$, which by our discussion in \S\ref{cocorediskkunneth} coincide with the linking disks to the critical part of the core $\f = \cc_F$ at infinity.}
This generation result for Weinstein manifolds has long been expected (see Bourgeois--Ekholm--Eliashberg \cite{bourgeoisekholmeliashberg} and the discussion beneath \cite[Theorem 1.1]{abouzaidcriterion}), however it evaded proof for some time.

Let us sketch the proof of Theorem \ref{generation}.  Begin with the K\"unneth embedding
\begin{align}
\W(X,\f)&\hookrightarrow\W(X\times\CC,(\cc_X\times\{\pm\infty\})\cup(\f\times\RR)),\\
L&\mapsto L\tildetimes i\RR,
\end{align}
and the geometric observation that the images of the cocores and linking disks under this functor are precisely the linking disks of the product stop $(\cc_X\times\{\pm\infty\})\cup(\f\times\RR)$.
It therefore suffices to show that $L\tildetimes i\RR$ is generated by the linking disks of this product stop.
By the wrapping exact triangle Theorem \ref{wrapcone}, this will be the case as long as $L\tildetimes i\RR$ can be isotoped through the product stop to a zero object.
For such an isotopy, we can simply take (a generic perturbation of) the (cylindrized) product of $L$ with an isotopy of $i\RR$ inside $(\CC,\{\pm\infty\})$ which passes one end through $+\infty\in\partial_\infty\CC$ to obtain a zero object of $\W(\CC,\{\pm\infty\})$.

In fact, this argument can be made to work under weaker hypotheses, giving the following result:

\begin{theorem}[Generation by generalized cocores]\label{generationpractice}
Let $(X,\f)$ be a stopped Liouville manifold whose relative core $\cc_{X,\f}:=\cc_X\cup(\f\times\RR)$ is mostly Lagrangian.
For every component of $\cc_{X,\f}^\crit$, fix a `generalized cocore': an exact cylindrical Lagrangian $L\subseteq X$ intersecting $\cc_{X,\f}$ exactly once, transversally, somewhere in the given component of $\cc_{X,\f}^\crit$.
Then $\W(X,\f)$ is generated by these generalized cocores.
\end{theorem}

Observe that Theorem \ref{generationpractice} is indeed a generalization of Theorem \ref{generation}: the hypotheses of Theorem \ref{generation} imply that the relative core $\cc_{X,\f}$ is mostly Lagrangian, and we may take the cocores and the linking disks as the generalized cocores.
Note that the existence of generalized cocores is a hypothesis of Theorem \ref{generationpractice}, not a conclusion.%
\footnote{We do not know how to construct (in general) a generalized cocore associated to a smooth Lagrangian point $p\in\cc_{X,\f}$, or even the corresponding object of $\W(X,\f)$.
However, the corresponding ``once stabilized'' object in $\W((X,\f)\times(\CC_{\Re\geq 0},\infty))$ is easy to define: it is simply the Lagrangian linking disk at $p\times\infty$.}
A stopped Liouville manifold $(X,\f)$ satisfying the hypotheses of Theorem \ref{generationpractice} will be called \emph{weakly Weinstein}.
Note that if $X$ is Weinstein and $\f$ is mostly Legendrian, then $(X,\f)$ is weakly Weinstein after small perturbation.

Theorem \ref{generationpractice} is useful since many naturally arising (stopped) Liouville manifolds are weakly Weinstein but not Weinstein.  For example, this is true for a cotangent bundle with a mostly Legendrian stop, or an affine algebraic variety equipped with a non-Morse plurisubharmonic function.
While these usually can be perturbed to become Weinstein, it is often more convenient to deal directly with the given Liouville vector field, e.g.\ because it is more explicit or symmetric.

\begin{example}[$\W(T^\ast Q)$ is generated by fibers]\label{cotangentgeneration}
Let $Q$ be a compact manifold-with-boundary.
The cotangent bundle $T^*Q$ is a Liouville sector, whose associated stopped Liouville manifold may be described as $(T^\ast Q^\circ,Z_{T^\ast Q^\circ}+\pi^\ast V)$ where $V$ is a complete vector field on $Q^\circ$ supported near the boundary and given in collar coordinates $(-\infty,0]\times\partial Q$ by $\varphi(t)\partial_t$ for $\varphi$ supported near zero ($\pi$ denotes the tautological lift from vector fields on $Q^\circ$ to Hamiltonian vector fields on $T^\ast Q^\circ$).
The relative core is simply $Q^\circ$ itself, and any cotangent fiber is a generalized cocore.
Thus Theorem \ref{generationpractice} implies that $\fiber\in\W(T^\ast Q)$ (or rather one fiber over each connected component of $Q$) generates.  In the case $Q$ has no boundary, this result is due originally to Abouzaid \cite{abouzaidcotangent}.
\end{example}

Another generation result which follows from the wrapping exact triangle is the following:

\begin{theorem}\label{halfplanegeneration}
Let $F$ be a Liouville manifold, and let $\Lambda\subseteq\partial_\infty(F\times\CC_{\Re\geq 0})^\circ$ be compact and mostly Legendrian.
Then the linking disks to $\Lambda^\crit$ generate $\W(F\times\CC_{\Re\geq 0},\Lambda)$.
\end{theorem}

Note that we \emph{do not} assume that $F$ is Weinstein in Theorem \ref{halfplanegeneration}.
By combining Theorems \ref{generation} and \ref{halfplanegeneration}, we also derive the following generation result for Fukaya--Seidel categories%
\footnote{From our perspective, the Fukaya--Seidel category is by definition the partially wrapped Fukaya category of a particular sort of stopped Liouville manifold, however existing definitions \cite{seidelbook,maydanskiyseidel,girouxpardon,seidellefschetzvi} are technically somewhat different.  In \S\ref{fscomparison}, we give a comparison between these definitions in one particular simplified setting.}
of Lefschetz fibrations%
\footnote{The references \cite{seidelbook,maydanskiyseidel,girouxpardon,seidellefschetzvi} also detail various technically differing notions of Lefschetz fibrations.  For us, a Lefschetz fibration is by definition the result of attaching critical Weinstein handles to a product $F\times\CC_{\Re\geq 0}$ or equivalently $F\times(\CC,\{-\infty\})$ ($F$ a Liouville manifold termed the `fiber') along the ordered Legendrian lifts of exact Lagrangian spheres $V_1,\ldots,V_k\subseteq F$ (termed the `vanishing cycles').  The cocores of these Weinstein handles are then called the `Lefschetz thimbles'.  For a comparison of this definition with other models, see \cite[\S 6.2]{girouxpardon}.}
(which should be compared to earlier results of Seidel \cite{seidelbook} and Biran--Cornea \cite{birancornealefschetz}; see also Corollary \ref{lefschetzgenerationII} below):

\begin{corollary}[Generation by Lefschetz thimbles]\label{lefschetzgeneration}
Let $\pi:\bar X\to\CC$ be a Lefschetz fibration with Weinstein fiber $F$ with core $\f$.
The Fukaya--Seidel category $\W(\bar X,\f\times\{-\infty\})$ is generated by the Lefschetz thimbles.
\end{corollary}

The above generation results allow us to deduce in many cases that the K\"unneth embedding is a \emph{pre-triangulated equivalence} 
(meaning fully faithful and image generates; compare with `quasi-equivalence' = fully faithful and essentially surjective, 
and `Morita equivalence' = fully faithful and image split-generates).
In particular, since being weakly Weinstein is closed under taking products, and products of generalized cocores are generalized cocores, we have:

\begin{corollary}[Surjectivity of K\"unneth]\label{weinsteinkunneth}
If $(X,\f)$ and $(Y,\g)$ are weakly Weinstein, then the K\"unneth embedding $\W(X,\f)\otimes\W(Y,\g)\xrightarrow\sim\W((X,\f)\times(Y,\g))$ is a pre-triangulated equivalence.
\end{corollary}

Corollary \ref{weinsteinkunneth} provides a potential path towards exhibiting Liouville manifolds which are not Weinstein: if $\W(X\times Y)$ is not generated by 
(cylindrizations of) 
product Lagrangians, then at least one of $X$ or $Y$ cannot be deformed to be Weinstein, or even weakly Weinstein.
In particular, if $\W(X\times T^\ast[0,1])$ is not generated by $L\tildetimes\fiber$ for $L\subseteq X$, then $X$ is not Weinstein.

Since the diagonal bimodule is representable by Proposition \ref{kunnethdiagonal}, we can apply Corollary \ref{weinsteinkunneth} to it to obtain the following categorical result (compare \cite{ganatrawrapcy}):

\begin{corollary}[Categorical smoothness]\label{homologicallysmooth}
If $(X,\f)$ is weakly Weinstein and $\f$ admits a ribbon, then $\W(X,\f)$ is smooth.
\end{corollary}

In fact, the hypothesis that $\f$ admit a ribbon is unnecessarily strong: all we actually need is its consequence that $\f$ can be positively displaced from itself (see Lemma \ref{diagonalpushoff}).

\subsection{Stop removal}

The real power of the partially wrapped Fukaya category comes from the ability to relate the partially wrapped Fukaya categories associated to different stops.
In general, for  $X$ a Liouville manifold (or sector) with two stops $\f\subseteq\g\subseteq(\partial_\infty X)^\circ$, there is a `stop removal' functor
$\W(X, \g) \to \W(X, \f)$.
Essentially by definition,  
this morphism factors through the localization of $\W(X, \g)$ along continuation
maps for positive isotopies through $\g \setminus \f$; one can show (Lemma \ref{lem: general stop removal}) the induced functor from the localization is fully faithful if any Lagrangian $L \subset X$ disjoint from $\g$ admits a cofinal sequence of wrappings in $(X,\f)$ whose endpoints are disjoint from the larger stop $\g$ (as holds in the case $\g \setminus \f$ is mostly Legendrian by general position arguments). 
It is not clear how one would compute this localization in general, but (as we explain in \S\ref{wet}) when $\g \setminus \f$ is mostly Legendrian, 
it follows easily from general position arguments and the wrapping exact triangle that:

\begin{theorem}[Stop removal]\label{stopremoval}
Let $X$ be a Liouville manifold (or sector) with two stops $\f\subseteq\g\subseteq(\partial_\infty X)^\circ$, such that $\g\setminus\f\subseteq(\partial_\infty X)^\circ\setminus\f$ is mostly Legendrian.
Then pushforward induces a quasi-equivalence
\begin{equation}\label{stopremovalqe}
\W(X,\g)/\D\xrightarrow\sim\W(X,\f),
\end{equation}
where $\D$ denotes the collection of linking disks of $(\g\setminus\f)^\crit$.
\end{theorem}

This statement generalizes (but the proof does not use) two prior results: Abouzaid--Seidel \cite{abouzaidseidelunpublished} proved
that for Lefschetz fibrations, the wrapped Fukaya category of the total space is a localization of the Fukaya--Seidel category, and Sylvan \cite{sylvanthesis} showed (under certain hypotheses) that the partially wrapped Fukaya category of a Liouville pair localizes to give the (fully) wrapped Fukaya category.
Note that both these results concern removing an entire connected component of a stop,  whereas in Theorem \ref{stopremoval} it is not required that $\f$ be a connected component of $\g$.
This gives a significant added flexibility which is important in applications, being in particular essential in the proof of Theorem \ref{weinsteindescent} and in \cite{gpswrappedconstructible}.
Also note that Theorem \ref{stopremoval} does not require the existence of a ribbon (compare Remark \ref{ribbonhard}).

\begin{example}
For a Weinstein manifold $X$, consider again the pair $(X \times \CC_{\Re\geq 0}, \cc_X \times \infty)$ (compare with the proof of Theorem \ref{generation}).
Applying Theorem \ref{stopremoval} gives $\W(X \times \CC_{\Re\geq 0}, \cc_X \times \infty) / \D = \W(X \times \CC_{\Re\geq 0}) = 0$.
Combined with the  K\"unneth stabilization functor, we learn that the cocores \emph{split}-generate $\W(X)$.
Getting \emph{generation} as in Theorem \ref{generation} requires the more controlled argument via the wrapping exact triangle given above.
\end{example}

\begin{example}[Calculation of $\W(T^\ast S^1)$]\label{cylinderstop}
Let $X = T^* S^1$ and let $\Lambda$ be the co-$0$-sphere over some fixed point $x \in S^1$.
Let $L$ be the cotangent fiber over a different point, and let $L(1)$ be $L$ `wrapped once around'.
The stop prevents any nontrivial wrapping at infinity, and so 
one has $HW^\bullet(L, L) = \ZZ = HW^\bullet(L(1), L(1))$, and also $HW^\bullet(L(1), L)_{T^\ast S^1,\Lambda} = 0$ (they are disjoint after suitable wrapping) and $HW^\bullet(L, L(1))_{T^\ast S^1,\Lambda} = \ZZ^{\oplus 2}$, 
generated by one trajectory at infinity in each component of $\partial_\infty T^*S^1$.
That is, $(L, L(1))$ is an `exceptional collection'.

The objects $L, L(1)\in\W(T^\ast S^1,\Lambda)$ generate: $L$ is the cocore to
the zero section, and (e.g.\ by Theorem \ref{wrapcone}) the cones on the generating
morphisms of $HW^\bullet(L, L(1))_{T^\ast S^1,\Lambda} $ are the linking disks to the stop. 

From these facts one deduces mirror symmetry $\W(T^\ast S^1,\Lambda) \cong \Perf(\bullet \rightrightarrows \bullet) \cong\Coh(\PP^1)$,
by matching generating exceptional collections.  Note that this induces: 
\begin{align}
L \to L(1) \to D_1 \to &\qquad \iff \qquad \OO \to \OO(1) \to \OO_0 \xrightarrow{[1]}\\
L \to L(1) \to D_2 \to &\qquad \iff \qquad \OO \to \OO(1) \to \OO_\infty \xrightarrow{[1]}
\end{align}
where $D_1,D_2\in\W(T^\ast S^1,\Lambda)$ are the linking arcs around $\Lambda$.

As is well known, $HW^\bullet(L, L)_{T^\ast S^1} = \ZZ[t, t^{-1}]$;
however, a direct computation requires some infinite process or picture.
One can instead argue by stop and removal.
We already saw that the linking disks (arcs in this case) are sent under the above isomorphism to the skyscraper sheaves at $0,\infty\in\PP^1$.
Thus by stop removal we have
\begin{equation}
\W(T^\ast S^1) = \W(T^\ast S^1,\Lambda) / \D  \cong\Coh(\PP^1)/\langle \OO, \OO_\infty \rangle = 
\Coh (\PP^1 \setminus \{0, \infty\}) = \Perf \ZZ[t, t^{-1}].
\end{equation}
Following $L$ (or $L(1)$) along this sequence of equivalences shows that it goes to the object $\ZZ[t, t^{-1}] \in \Perf \ZZ[t, t^{-1}]$, and hence we can conclude $HW^\bullet(L, L)_{T^\ast S^1} \cong \ZZ[t, t^{-1}]$.
(The reader may be more familiar with this example in its incarnation as the Lefschetz fibration / Landau--Ginzburg model $z + \frac{1}{z}: \CC^\times \to \CC$.)

Similar reasoning about surfaces in general appears previously in Lekili--Polishchuk \cite{lekilipolishchuk}.
\end{example}

\begin{example}[Fukaya--Seidel categories]
Let $\pi:\bar X\to\CC$ be a Liouville Landau--Ginzburg model 
with Weinstein fiber $F$ with core $\f$.
Theorem \ref{stopremoval} implies that the functor
\begin{equation}
\W(\bar X,\f\times\{-\infty\})\to\W(\bar X)
\end{equation}
(from the Fukaya--Seidel category of $\pi$ to the wrapped Fukaya category of the total space)
is precisely localization at the linking disks of $\f\times\{-\infty\}$.
By Theorem \ref{generation}, the full subcategory they span is also the essential image of the functor $\W(F)\to\W(\bar X,\f\times\{-\infty\})$ obtained by composing the K\"unneth stabilization functor \eqref{kunnethstabsector} with pushforward under the canonical embedding $F\times T^\ast[0,1]\to(\bar X,\f\times\{-\infty\})$ near the stop.
This is comparable to, but not quite the same as, 
the localization presentation of the wrapped Fukaya category obtained in Abouzaid--Seidel's work in the case of Lefschetz fibrations \cite{abouzaidseidelunpublished}.
For a concrete example of such a situation in mirror symmetry, see Keating \cite{keatingcusp}.

If the critical locus of $\pi$ is compact, then $F$ is the page of an open book decomposition of $\partial_\infty\bar X$, and hence the boundary at infinity of the associated Liouville sector is, up to deformation, $F_0\times[0,1]$ by \cite[Lemma 2.18]{gpssectorsoc}.  By \cite[Lemma 3.44]{gpssectorsoc}, this implies that 
$\W(\bar X,\f\times\{-\infty\})$ is \emph{proper} in the sense of non-commutative geometry (i.e.\ has finite-dimensional morphism spaces).
Such a fibration $\pi$ thus provides a geometrically motivated categorical compactification 
of $\W(\bar X)$:
the Fukaya--Seidel category $\W(\bar X,\f\times\{-\infty\})$ is proper and localizes to give $\W(\bar X)$.
The flexibility of adding stops at will allows for more general, and partial, compactifications.
\end{example}

Another stop removal result somewhat orthogonal to Theorem \ref{stopremoval} is the fact that removing a contact stop is fully faithful (which follows immediately from wrapping using a Reeb vector field which is tangent to the stop being removed):

\begin{proposition}[Removing a contact stop]\label{removecontact}
Let $X$ be a Liouville manifold (or sector) with two stops $\f\subseteq\g\subseteq(\partial_\infty X)^\circ$, such that $\g\setminus\f\subseteq(\partial_\infty X)^\circ\setminus\f$ is a contact submanifold.
Then pushforward $\W(X,\g)\to\W(X,\f)$ is fully faithful.
\end{proposition}

Let us call a Liouville manifold $X$ \emph{inessential} iff, possibly after deforming to a different Liouville form, the core $\cc_X\times 0$ inside the contactization $(X\times\RR,\lambda+dt)$ is contained inside a closed contact submanifold.
For example, we shall see in Lemma \ref{stabinessential} that $W\times\CC$ is inessential for any Liouville manifold $W$.
Now Proposition \ref{removecontact} implies that removing an inessential hypersurface is also fully faithful:

\begin{proposition}[Removing an inessential hypersurface]\label{removeinessential}
Let $(X,\f)$ be a stopped Liouville sector and $H_0\subseteq(\partial_\infty X)^\circ\setminus\f$ an inessential Liouville hypersurface.
Then pushforward $\W(X,\f\sqcup H_0)\to\W(X,\f)$ is fully faithful.
\end{proposition}

We derive from this the following two corollaries, the latter of which is a strengthening of Corollary \ref{lefschetzgeneration}.

\begin{corollary}\label{inessentialvanish}
If $X$ is inessential then $\W(X)=0$.
\end{corollary}

\begin{corollary}[Generation by Lefschetz thimbles]\label{lefschetzgenerationII}
Let $\pi:\bar X\to\CC$ be a Lefschetz fibration with Liouville fiber $F$ with core $\f$.
The Fukaya--Seidel category $\W(\bar X,\f\times\{-\infty\})$ is generated by the Lefschetz thimbles.
\end{corollary}

\subsection{Stopped inclusions}

We introduce a geometric notion of `forward/backward stopped inclusions' of stopped Liouville sectors in \S\ref{stoppedinclusionssection}, which provides another useful source of fully faithful pushforward functors.

Since wrapped Floer cohomology is computed by wrapping only one factor, to ensure that a pushforward functor $\W(X,\f)\to\W(X',\f')$ is fully faithful, it is enough to ensure that when any Lagrangian $L \subseteq X$ disjoint from $\f$ is wrapped inside $X'$ stopped at $\f'$, 
then once it leaves $X$ it never returns (at least for a cofinal collection of wrappings).
We formulate geometric conditions on inclusions of Liouville sectors, called being \emph{(tautologically) forward stopped} (see Definitions \ref{tautforwardstoppeddef} and \ref{forwardstoppeddef}), which guarantee this property of wrapping and hence that the associated pushforward functor is fully faithful (Corollary \ref{stoppedff}).
This notion naturally depends only on the contact geometry of the boundary at infinity.

The notion of forward/backward stopped inclusions is a crucial ingredient in the proof of the homotopy pushout/colimit formulas below.
It is also an important ingredient (see \S\ref{viterborestrictionsection}) in Sylvan's proposal \cite{sylvantalk} to define (and generalize) Abouzaid--Seidel's Viterbo restriction functor \cite{abouzaidseidel} in terms of stop removal functors.

\subsection{Homotopy pushout formula}\label{gluingintrosec}

In \S\ref{pushoutsec}  we use stop removal and forward stopped inclusions to prove certain homotopy pushout formulae for wrapped Fukaya categories.
One such result (also observed by Sylvan \cite{sylvantalk}) is:

\begin{theorem}[Homotopy pushout formula]\label{mvthm}
Let $X=X_1\cup X_2$ be a Liouville sector written as the union of two Liouville sectors $X_1$ and $X_2$ meeting along a hypersurface $X_1\cap X_2$ inside $X$ disjoint from $\partial X$.
Writing a neighborhood of this hypersurface as $F\times T^\ast[0,1]$, suppose in addition that $F$ is Weinstein (up to deformation).
Let $\rr\subseteq(\partial X)^\circ$ be a stop disjoint from $\partial_\infty(X_1\cap X_2)$, and let $\rr_i:=\rr\cap(\partial_\infty X_i)^\circ$.
Then the induced diagram of $\ainf$-categories
\begin{equation}\label{htpypushout}
\begin{tikzcd}
\W(F)\ar{r}\ar{d}&\W(X_1,\rr_1)\ar{d}\\
\W(X_2,\rr_2)\ar{r}&\W(X,\rr)
\end{tikzcd}
\end{equation}
induces a fully faithful functor
\begin{equation}\label{htpycoliminmvthm}
\hocolim\bigl(\W(X_2,\rr_2)\leftarrow\W(F)\to\W(X_1,\rr_1)\bigr)\hookrightarrow\W(X,\rr)
\end{equation}
(so we say \eqref{htpypushout} is an \emph{almost homotopy pushout}).
If, in addition, $X$ is a Liouville manifold, (the convexifications of) $X_i$ are Weinstein (up to deformation), and the $\rr_i$ are mostly Legendrian, then \eqref{htpycoliminmvthm} is a pre-triangulated equivalence (so we say \eqref{htpypushout} is a \emph{homotopy pushout}).
\end{theorem}

The real content here is the full faithfulness, as the final statement about generation is simply an application of Theorem \ref{generation} (when $F$ and the convexifications of $X_i$ are all Weinstein up to deformation, it follows that $X$ is as well).
The proof of Theorem \ref{mvthm} consists of adding and removing a stop via Theorem \ref{stopremoval}, similar in spirit to Example \ref{cylinderstop}.
We sketched already the basic idea of the proof at the end of \S\ref{intro}.
The origin of the hypothesis that $F$ is Weinstein is the use of Theorems \ref{generation} and \ref{stopremoval} in the proof.

Theorem \ref{mvthm} provides some understanding of the effect of Weinstein handle attachment on the wrapped Fukaya category.
Indeed, Weinstein handle attachment is a special case  of the gluing operation in Theorem \ref{mvthm} (compare \cite[\S 3.1]{eliashbergweinsteinrevisited}).
We conclude: 

\begin{corollary}[Effect of Weinstein handle attachment]\label{handleeffect}
Let $(X,\f)$ be a stopped Liouville sector of dimension $2n$ obtained from a stopped Liouville sector $(X^\inn,\f^\inn)$ by attaching a Weinstein $k$-handle along an isotropic sphere $\Lambda^{k-1}\subseteq(\partial_\infty X^\inn)^\circ\setminus\f^\inn$.  Completing inside $X$ gives a map 
$\W(X^\inn,\f^\inn\cup\Lambda) \to \W(X,\f)$. 

In the subcritical case $k<n$, we have 
\begin{equation}
\W(X^\inn,\f^\inn) \xleftarrow{\sim} \W(X^\inn,\f^\inn\cup\Lambda) \hookrightarrow\W(X,\f)
\end{equation}

In the critical case $k=n>1$, there is an almost homotopy pushout
\begin{equation}\label{handlepushoutstatement}
\begin{tikzcd}
C_{-\bullet}(\Omega S^{n-1})\ar{r}\ar{d}&\W(X^\inn,\f^\inn\cup\Lambda)\ar{d}\\
\ZZ\ar{r}&\W(X,\f),
\end{tikzcd}
\end{equation}
implying, in particular, a quasi-isomorphism of $\ainf$-algebras
\begin{equation}\label{handlepushoutalgebras}
CW^\bullet(\cocore,\cocore)_{X,\f}=CW^\bullet(D,D)_{X^\inn,\f^\inn\cup\Lambda}\otimes_{C_{-\bullet}(\Omega S^{n-1})}\ZZ,
\end{equation}
where $D$ denotes the linking disk of $\Lambda$ and $\cocore\subseteq X$ denotes the cocore of the added handle.
For $k=n=1$, replace $C_{-\bullet}(\Omega S^{n-1})$ above with $\ZZ\sqcup\ZZ$, namely the disjoint union of two copies of the $\ainf$-category $\ZZ$ (a single object with endomorphism algebra $\ZZ$).
\end{corollary}

The fact that subcritical handle attachment does not change the wrapped Fukaya category was a folklore 
result (the closest results in the literature concern the closed string analogue symplectic cohomology \cite{cieliebakhandle}; see also \cite{ekholmng}).
The partially wrapped Floer cochains of the linking disk appearing in the case of critical handle attachment are conjectured 
\cite{sylvantalk,ekholm-lekili} to agree with the Legendrian contact homology of the attaching sphere; establishing this conjecture 
would identify our statement above with the surgery formula of Bourgeois--Ekholm--Eliashberg \cite[Remark 5.9]{bourgeoisekholmeliashberg}).

Let us also give some examples concerning surfaces: 

\begin{example}[Partially wrapped Fukaya categories of surfaces]
Let $\Sigma$ be any $2$-dimensional Liouville sector.
That is, $\Sigma$ is a surface-with-boundary with no compact components and no circle boundary components.
Choose any collection of arcs going from non-compact ends to non-compact ends dividing $\Sigma$ into `$A_2$' Liouville sectors (a disk minus three boundary punctures).
The wrapped Fukaya category of the $A_2$ sector is given by
\begin{equation}
\Tw\W(\CC,\{e^{2\pi ik/3}\}_{k=0,1,2})=\Perf(\bullet\to\bullet)
\end{equation}
(for example, this follows from Theorem \ref{wrapcone}).
The $A_2$ sectors overlap over $A_1$ sectors $T^\ast[0,1]$ with wrapped Fukaya category
\begin{equation}
\Tw\W(T^\ast[0,1])=\Tw\W(\CC,\{\pm\infty\})=\Perf(\bullet).
\end{equation}
Iterated applications of Theorem \ref{mvthm} thus yield a description of $\W(\Sigma)$ as a homotopy colimit of copies of these categories $\Perf(\bullet\to\bullet)$ and $\Perf(\bullet)$.

In anticipation this fact, said colimit was historically known as the `topological Fukaya category' 
of the surface and has been studied algebraically \cite{Dyckerhoff-Kapranov, Pascaleff-Sibilla}.
\end{example}

\begin{example}[Partially wrapped Fukaya categories of fibrations over surfaces]
Continuing the preceding example, suppose $\pi:X \to \Sigma$ is an exact symplectic (non-singular) fibration with Weinstein fiber $F$.
The same decomposition of $\Sigma$ pulled back to $X$ yields a description of $\W(X)$ as a homotopy colimit of copies of $\W(F)$ and $\W(F)\otimes\Perf(\bullet\to\bullet)$.

Finally, consider the case where $\pi$ is allowed to have singularities (e.g.\ a Lefschetz fibration).
We now choose arcs dividing $\Sigma$ into $A_2$ sectors containing no critical values and half-planes $\CC_{\Re\geq 0}$ each containing a single critical value.
This again yields a homotopy colimit presentation of $\W(X)$.
Note that in the case of a Lefschetz fibration, the pieces $X_\alpha$ resulting as inverse images of half-planes containing a single critical value satisfy $\W(X_\alpha)=\Perf(\bullet)$ by Corollary \ref{lefschetzgeneration} and observing that wrapping the thimble creates no new self intersections.
\end{example}

\subsection{Homotopy colimit formula}

In order to generalize the homotopy pushout formula (Theorem \ref{mvthm}) to $n$-element covers of 
Liouville sectors by Liouville sectors, we need to first identify the correct class of such covers: 

\begin{definition}[Sectorial covering]\label{admintrodef}
Let $X$ be a Liouville manifold-with-boundary (meaning it is exact and cylindrical at infinity).
A collection of cylindrical hypersurfaces $H_1,\ldots,H_n\subseteq X$ is called \emph{sectorial} iff their characteristic foliations are $\omega$-orthogonal over their intersections and there exist functions $I_i:\Nbd^ZH_i\to\RR$ (linear near infinity) satisfying $dI_i|_{\charfol(H_i)}\ne 0$ (equivalently, $X_{I_i}\notin TH_i$), $dI_i|_{\charfol(H_j)}=0$ (equivalently, $X_{I_i}\in TH_j$) for $i\ne j$, and $\{I_i,I_j\}=0$.
A covering of a Liouville manifold (or sector) $X$ by finitely many Liouville sectors $X_1,\ldots,X_n\subseteq X$ is called sectorial iff the collection of their boundaries $\partial X,\partial X_1,\ldots,\partial X_n$ is sectorial (here we understand $\partial X_i$ as the part of the boundary of $X_i$ not coming from $\partial X$, i.e.\ the boundary in the point set topological sense).
\end{definition}

Any sectorial collection of hypersurfaces $H_1,\ldots,H_n\subseteq X$ is necessarily mutually transverse, 
and all multiple intersections $H_{i_1}\cap\cdots\cap H_{i_k}$ are coisotropic (see \S\ref{admsubsec}).
Also note that a Liouville manifold-with-boundary 
$X$ is a Liouville sector iff $\partial X$ is sectorial.

\begin{example}\label{cotangentsectorial}
If $Q$ is a compact manifold-with-boundary and $Q_1,\ldots,Q_n\subseteq Q$ are codimension zero submanifolds-with-boundary whose boundaries are, together with $\partial Q$, mutually transverse, then $T^*Q_1,\ldots,T^*Q_n\subseteq T^*Q$ is sectorial.
\end{example}

Given a sectorial collection of hypersurfaces $H_1,\ldots,H_n\subseteq X$, there is an induced stratification of $X$ by strata $\bigcap_{i\in I}H_i\setminus\bigcup_{i\notin I}H_i$ for $I\subseteq\{1,\ldots,n\}$.
Each stratum is coisotropic and the symplectic reduction of its closure (not necessarily embedded in $X$) is a Liouville sector (with corners).
We say a sectorial covering is Weinstein iff (the convexifications of) each of these Liouville sectors is Weinstein 
(up to deformation, i.e.\ the Liouville form can be deformed by $df$ so that the associated Liouville flow is gradient-like).
This condition implies that both the convexification of $X$ and its symplectic boundary are Weinstein (see Lemma \ref{lem:morecovers}).

\begin{example}\label{cotangentstrataweinstein}
Continuing Example \ref{cotangentsectorial}, we note that the symplectic reductions of the strata of the covering $T^*Q_1,\ldots,T^*Q_n\subseteq T^*Q$ are simply the cotangent bundles of the strata of the covering $Q_1,\ldots,Q_n\subseteq Q$.
In particular, they are all Weinstein (after deformation).
\end{example}

In \S\ref{descent-sec} we establish:

\begin{theorem}[Descent for Weinstein sectorial coverings]\label{weinsteindescent}
For any Weinstein sectorial covering $X_1,\ldots,X_n$ of a Liouville sector $X$, the induced functor
\begin{equation}
\hocolim_{\varnothing\ne I\subseteq\{1,\ldots,n\}}\W\biggl(\bigcap_{i\in I}X_i\biggr)\xrightarrow\sim\W(X)
\end{equation}
is a pre-triangulated equivalence.
More generally, the same holds for any mostly Legendrian stop $\rr\subseteq(\partial_\infty X)^\circ$ which is disjoint from every $\partial X_i$ and from a neighborhood of $\partial X$. 
\end{theorem}

Descent statements such as Theorem \ref{weinsteindescent} formally reduce to the case of pushouts.
However, the relevant pushout is not Theorem \ref{mvthm}, but rather its generalization
to the case where we allow $F$ to be a Liouville sector, as opposed to a Liouville manifold.
The new work in establishing this generalization is the further study (in \S\ref{subsec:stopa2}) of the geometry of the relevant stops at 
infinity in order to show that the relevant inclusions of Liouville sectors are forward stopped.

\begin{example}[Calculation of $\W(T^*Q)$]
Continuing Examples \ref{cotangentsectorial}--\ref{cotangentstrataweinstein}, we may apply Theorem \ref{weinsteindescent} to deduce that $\W(T^*Q)=\hocolim_{\varnothing\ne I\subseteq\{1,\ldots,n\}}\W(T^*\bigcap_{i\in I}Q_i)$.
Taking a cover for which the multiple intersections $Q_1\cap\cdots\cap Q_n$ are (after smoothing corners) all balls, we may deduce that $CW^*(T^*_qQ,T^*_qQ)$ is quasi-isomorphic to $C_{-*}(\Omega_qQ)$ (with appropriate twisting) as $\ainf$-algebras, as proven by Abbondandolo--Schwarz \cite{abbondandoloschwarz} and Abouzaid \cite{abouzaidtwisted} in the case $Q$ has no boundary.%
\footnote{More precisely, this isomorphism can be derived from Theorem \ref{weinsteindescent} and the following somewhat technical, yet purely topological, folklore argument.  What Theorem \ref{weinsteindescent} says directly is that $\W(T^*Q)$ is the global sections of a cosheaf of $\ainf$-categories on $Q$ (namely $U\mapsto\W(T^*U)$), which in view of the quasi-equivalence $\W(T^*[\mathrm{ball}])=\Perf\ZZ$ generated by the cotangent fiber, is identified with $\underline{\Perf\ZZ}_{Q,\xi}$, the constant cosheaf on $Q$ with co-stalk $\Perf\ZZ$, with a twist $\xi$ corresponding to the algebro-topological data needed to promote the cotangent fibers to objects of the Fukaya category (the twist is discussed in detail in \cite{gpswrappedconstructible}).  This identifies $CW^*(T^*_qQ,T^*_qQ)$ with the endomorphism algebra of the object $\ZZ_q\in\underline{\Perf\ZZ}_{Q,\xi}(Q)$ arising from pushing forward the free rank one object $\ZZ\in\Perf\ZZ$ in the co-stalk at $q$ to the $\ainf$-category of global sections.  The endomorphism algebra of $\ZZ_q$ may in turn be identified with $C_{-*}(\Omega_qQ)$ (with a twist derived from $\xi$) as follows.  Let $Q$ be a manifold (or any other sufficiently nice topological space), and consider the precosheaf on $Q$ given by `$U\mapsto U$', valued in the $\infty$-category of spaces $\mathcal S$ (i.e.\ CW-complexes and mapping spaces between them).  This is a cosheaf by the `Lurie--Seifert--Van Kampen' Theorem \cite[A.3]{lurieha}.  Now if we view $X\in\mathcal S$ as an $\infty$-groupoid, the automorphism group of a point $x\in X$ is the based loop space $\Omega_xX$.  Finally, note that the inclusion $\mathcal S\hookrightarrow\operatorname{Cat}_\infty$ (of $\infty$-groupoids into $\infty$-categories) and the `singular chains on morphism spaces' functor from $\operatorname{Cat}_\infty$ to $\ainf$-categories are both cocontinuous.  Now twist.}
In the simple case $Q=S^1$, this argument reduces essentially to Example \ref{cylinderstop}.
\end{example}

\subsection{Cobordism attachment and twisted complexes}\label{cobattachtw}

To conclude the introduction, we explain the action filtration argument underlying the proof of the surgery exact triangle (Proposition \ref{handlecone}).

Let us first recall a consequence of the Viterbo restriction functor constructed by Abouzaid--Seidel \cite{abouzaidseidel}, which corresponds to the situation of a trivial action filtration.
Let $X$ be a Liouville manifold, let $L\subseteq X$ be an exact cylindrical Lagrangian whose primitive $f_L:L\to\RR$ ($df_L=\lambda|_L$) vanishes at infinity, and let $C\subseteq S\partial_\infty X$ be an exact symplectic cobordism in the symplectization of $\partial_\infty X$ whose primitive vanishes at minus infinity.
Denoting by $\#^CL$ the result of attaching $C$ to $L$ at infinity, there is then a quasi-isomorphism in the wrapped Fukaya category
\begin{equation}
\#^CL=L.
\end{equation}
To see this, consider the Viterbo restriction functor.
Given a Liouville subdomain $X_0\subseteq X$ whose completion is $X$, this is a functor to $\W(X)$ from the full subcategory of $\W(X)$ spanned by Lagrangians $K\subseteq X$ whose primitives $f_K$ vanish identically near $\partial X_0$ (this implies that $K$ is cylindrical near $\partial X_0$).
On objects, this functor is simply ``intersect with $X_0$ and complete'', and hence for suitable choice of $X_0$, the objects $\#^CL$ and $L$ are both sent to $L$ under this functor.
On the other hand, since the inclusion $X_0\subseteq X$ is ``trivial'' (the completion of $X_0$ is $X$) this restriction functor is (canonically quasi-isomorphic to) the identity functor.
Note that the vanishing of the primitive of $L$ near infinity and of the primitive of $C$ near minus infinity was crucial for this argument.
In \S\ref{caatc}, we prove the following result which extends the above picture to the ``relatively non-exact'' setting:

\begin{proposition}[Cobordism attachment and twisted complexes]\label{cobordismtwisted}
Let $(X,\f)$ be a stopped Liouville sector, and let $L_1,\ldots,L_n\subseteq X$ be disjoint exact Lagrangians (disjoint from $\f$ at infinity) whose primitives vanish at infinity.
Let $C\subseteq S\partial_\infty X$ be an exact Lagrangian cobordism (disjoint from $\RR\times\f$) with negative end $\partial_\infty L_1\sqcup\cdots\sqcup\partial_\infty L_n$, such that the primitive $f_C:C\to\RR$ of $\lambda|_C$ satisfies
\begin{equation}
f_C|_{\partial_\infty L_1}<\cdots<f_C|_{\partial_\infty L_n},
\end{equation}
regarding $\partial_\infty L_i$ as the negative ends of $C$ (note that these restrictions $f_C|_{\partial_\infty L_i}$ are simply real numbers).

Suppose in addition that the image of $C$ under the projection $S\partial_\infty X\to\partial_\infty X$ is ``thin'' in the sense that for every Lagrangian $K\subseteq X$ disjoint at infinity from $\f$, there exists a positive wrapping $K\leadsto K^w$ (away from $\f$) such that $\partial_\infty K^w$ is disjoint from $C$.

Then there is a quasi-isomorphism
\begin{equation}\label{twistedisomorphism}
\#_i^CL_i=[L_1\to\cdots\to L_n]
\end{equation}
in $\Tw\W(X,\f)$, where $\#_i^CL_i$ denotes the result of attaching the cobordism $C$ to $L_1\sqcup\cdots\sqcup L_n$ at infinity, and $[L_1\to\cdots\to L_n]$ denotes a twisted complex $(\bigoplus_{i=1}^nL_i,\sum_{i<j}D_{ij}^C)$ depending on $C$.
\end{proposition}

The thinness hypothesis on $C$ is convenient in the proof, but we expect that the result remains true without it (for example, no such hypothesis is required in Abouzaid--Seidel \cite{abouzaidseidel}).
A sufficient condition that $C$ be thin is that its projection to $\partial_\infty X$ be contained in a small neighborhood of a Weinstein hypersurface (see Proposition \ref{lem:pushoffcore} for a more general statement).

The proof of Proposition \ref{cobordismtwisted} proceeds by testing $\#_i^CL_i$ against arbitrary Lagrangians $A$ and considering the limit as the cobordism $C$ is pushed to infinity.
For any Lagrangian $A\subseteq X$ disjoint at infinity from $C$, there is a natural isomorphism of abelian groups
\begin{equation}
CF^\bullet(A,\#_i^CL_i)=CF^\bullet(A,L_1)\oplus\cdots\oplus CF^\bullet(A,L_n).
\end{equation}
In the limit as the cobordism is pushed to infinity, the actions of the intersections with $L_i$ become much larger than the actions of the intersections with $L_j$ for $i<j$.
The differential on $CF^\bullet(A,\#_i^CL_i)$ is thus lower triangular with respect to the above direct sum decomposition.
Moreover, the diagonal components of this differential coincide with the differentials on $CF^\bullet(A,L_i)$, since by action and monotonicity arguments, such disks cannot travel far enough to see the difference between $L_i$ and $\#_i^CL_i$.
This produces an isomorphism of complexes
\begin{equation}\label{complexiso}
CF^\bullet(A,\#_i^CL_i)=[CF^\bullet(A,L_1)\to\cdots\to CF^\bullet(A,L_n)],
\end{equation}
where the right hand side denotes a twisted complex of $CF^\bullet(A,L_i)$ with unspecified maps $CF^\bullet(A,L_i)\to CF^\bullet(A,L_j)$ for $i<j$.
Similar reasoning shows that \eqref{complexiso} is in fact an isomorphism of modules (i.e.\ is compatible with $\ainf$-multiplication on the left).
Hence the Yoneda lemma provides the desired quasi-isomorphism \eqref{twistedisomorphism} in the Fukaya category.
(Some algebraic complications arise from the fact that, by pushing the cobordism towards infinity, we can only guarantee \eqref{complexiso} for finitely many Lagrangians $A$ at a time, however these can be dealt with.)

Proposition \ref{cobordismtwisted} produces the surgery exact triangle (Proposition \ref{handlecone}), except we need a further argument to identify the morphism $L\to K$ with $\gamma$.
We consider testing the exact triangle against $A=L^w$, a positive wrapping of $L$ which wraps through the surgery locus (thus creating an intersection point with $K$) but not farther.
The cycle in $HW^\bullet(L,K)$ we are looking for is thus represented by the image of the continuation map under the map
\begin{equation}
HF^\bullet(L^w,L)\to HF^\bullet(L^w,K)
\end{equation}
forming the differential on the right side of \eqref{complexiso} with $A=L^w$.
Since $HF^\bullet(L^w,K)=\ZZ$ is generated freely by (the intersection point corresponding to) the short Reeb chord $\gamma$, this proves the desired statement up to an unknown integer factor.
If this integer factor were divisible by a prime $p$, then there would be a quasi-isomorphism $L\#_\gamma K=L\oplus K$ over $\ZZ/p$.
We can preclude the existence of such a quasi-isomorphism (in a further stopped Fukaya category) by testing both sides against a suitably chosen Lagrangian disk linking both $L$ and $K$ but unlinked with $L\#_\gamma K$.

\subsection{Remarks on organization and dependencies} \label{dependencies}

The dependency partial order of the results in this article admits many possible linearizations.
Here in \S\ref{statementresults}, we have chosen to order the results roughly in order of increasing complexity of 
statement.  By contrast, in the remainder of the text, we linearize in order of increasing complexity of proof. 
We also try to introduce as few as possible geometric constructions 
before a given categorical result.  

Other than in the definition of the partially wrapped Fukaya category (in \S\S\ref{wrapdefsec}--\ref{noncyl}), Floer-theoretic considerations (i.e.\ discussions of almost complex structures and holomorphic disks) occur only two places in this entire article: in \S\ref{caatc} to prove the cobordism attachment result Proposition \ref{cobordismtwisted} and in \S\ref{sec: kunneth} to prove the K\"unneth Theorem \ref{kunneth}.  In each instance, we are constructing objects 
by showing that some holomorphic curve construction defines a module, which we eventually show is representable.  
Outside these sections, the reader will not find any mention of almost complex structures or holomorphic curves.

In fact, many results and constructions in this article are entirely `geometric' in nature, i.e.\ do not involve Floer theory or Fukaya categories in their formulation or proofs.
Isolating such results may be especially of interest with respect
to  non-Floer-theoretic functors from Weinstein manifolds to categories such as \cite{Nadler-Tanaka, Nadler-Shende}.
In fact many sections contain only such results: \S \ref{wrappingcatsec}, \S \ref{lagrdisksection}, \S \ref{sec: products},   
\S \ref{cocorediskkunneth}, \S \ref{convexreview}--\ref{forwardstoppedsec}, \S \ref{admsubsec}.  

\subsection{Acknowledgements}

We thank Tobias Ekholm, Oleg Lazarev, Thomas Massoni, Patrick Massot, and Zack Sylvan for helpful discussions.

S.G.\ was partially supported by NSF grant DMS--1907635, and
would also like to thank the Mathematical Sciences Research Institute for its hospitality during a visit in Spring 2018 (supported by NSF grant DMS--1440140) during which some of this work was completed. This research was conducted during the period J.P.\ served as a Clay Research Fellow and was partially supported by a Packard Fellowship and by the National Science Foundation under the Alan T.\ Waterman Award, Grant No.\ 1747553.
V.S.\ was supported by a Villum Investigator grant, a Danish National Research Foundation chair, a Novo Nordisk start package, NSF CAREER grant DMS--1654545, 
and the Simons--CRM scholar-in-residence program.

\section{Foundations of wrapped Fukaya categories}\label{partiallywrappedsection}

In this section, we review the definition of the partially wrapped Fukaya category which we will use in this paper.

\subsection{Wrapping categories}\label{wrappingcatsec}

We begin with a general discussion of wrapping.

Given a co-oriented contact manifold $Y$, we consider the category whose objects are compact Legendrian submanifolds of $Y$ and whose morphisms are positive Legendrian isotopies modulo deformation rel endpoints.
(Recall that a Legendrian isotopy $\Lambda_t$ is said to be \emph{positive} iff for some, equivalently any, positive contact form $\alpha$, we have $\alpha(\partial_t \Lambda_t) > 0$ for all $t$; compare \cite[Definition 3.22]{gpssectorsoc}.)
Denote this category by $\operatorname{Leg}_Y^+$, and call it the \emph{(positive) wrapping category} of $Y$.
The (positive) wrapping category of a given compact Legendrian $\Lambda\subseteq Y$ is the slice category $(\Lambda\leadsto-)^+_Y:=(\operatorname{Leg}_Y^+)_{\Lambda/}$.

\begin{remark}\label{wrappingdifferentbutsame}
Our previous work \cite[\S 3.4]{gpssectorsoc} defines the wrapping category not as the slice category $(\operatorname{Leg}_Y^+)_{\Lambda/}$ but rather as the comma category $(\operatorname{Leg}_Y^+)_{f(\Lambda)/f(\cdot)}$ where $f:\operatorname{Leg}_Y^+\to\operatorname{Leg}_Y$ denotes the forgetful functor and $\operatorname{Leg}_Y$ is defined as $\operatorname{Leg}_Y^+$ but without the positivity condition on isotopies.
This difference is purely cosmetic.
The forgetful functor
\begin{equation}\label{wrappingforget}
(\operatorname{Leg}_Y^+)_{\Lambda/}\to(\operatorname{Leg}_Y^+)_{f(\Lambda)/f(\cdot)}
\end{equation}
from the wrapping category considered here to that considered in \cite[\S 3.4]{gpssectorsoc} exhibits the former as the slice category $((\operatorname{Leg}_Y^+)_{f(\Lambda)/f(\cdot)})_{\id_{f(\Lambda)}/}$ of the latter.
It follows that filteredness of the wrapping category in the prior sense \cite[Lemma 3.27]{gpssectorsoc} implies the wrapping category in the present sense is also filtered (though it is more natural to just prove it directly) and that the functor \eqref{wrappingforget} is cofinal.
Hence as far as taking direct limits over wrapping categories is concerned (which was the their only purpose in \cite{gpssectorsoc}), there is no difference between the two sorts of wrapping categories.
\end{remark}

Similarly, there is a wrapping category of exact cylindrical Lagrangians inside a stopped Liouville manifold-with-boundary $(X,\f)$, denoted $\operatorname{Lag}_{X,\f}^+$ (positivity of an isotopy simply means positivity at infinity), and we write $(L\leadsto-)^+_{X,\f}:=(\operatorname{Lag}_{X,\f}^+)_{L/}$.
When doing Floer theory, one usually cares not about Lagrangian submanifolds, but rather Lagrangian submanifolds equipped with auxiliary topological data, used to define gradings/orientations; the relevant wrapping category is then defined in terms of Lagrangians equipped with such auxiliary data, and isotopies thereof.

We will often use the following criterion for cofinality in wrapping categories.

\begin{lemma}\label{cofinalitycriterion}
Let $\{\Lambda_t\}_{t\geq 0}\subseteq Y$ be an isotopy of compact Legendrians inside a (not necessarily compact) contact manifold $Y$.
If there exists a contact form $\alpha$ on $Y$ such that
\begin{equation}
\int_0^\infty\min_{\Lambda_t}\alpha\biggl(\frac{\partial\Lambda_t}{\partial t}\biggr)\,dt=\infty,
\end{equation}
then $\{\Lambda_t\}_{t\geq 0}$ is a cofinal wrapping of $\Lambda_0$.
In particular, if $\Lambda_t$ escapes to infinity as $t\to\infty$ (i.e.\ is eventually disjoint from any given compact subset of $Y$), then it is a cofinal wrapping of $\Lambda_0$.

The same statement holds for Lagrangian wrapping categories, replacing $Y$ with the relevant boundary at infinity where wrapping takes place.
\end{lemma}

\begin{proof}
This is \cite[Lemma 3.29 and Remark 3.31]{gpssectorsoc}, which applies in the present context by Remark \ref{wrappingdifferentbutsame}.
\end{proof}

There is a functor $\partial_\infty:\operatorname{Lag}_{X,\f}^+\to\operatorname{Leg}_{(\partial_\infty X)^\circ\setminus\f}^+$ which induces a functor on slice categories $\partial_\infty:(L\leadsto-)_{X,\f}^+\to(\partial_\infty L\to-)_{(\partial_\infty X)^\circ\setminus\f}^+$.
In fact, there is a pair of functors
\begin{equation}\label{laglegcofinal}
\begin{tikzcd}[column sep = large]
(L\leadsto-)_{X,\f}^+\ar[yshift=0.5ex]{r}{\partial_\infty}&\ar[yshift=-0.5ex]{l}{\text{drag at $\infty$}}(\partial_\infty L\to-)_{(\partial_\infty X)^\circ\setminus\f}^+,
\end{tikzcd}
\end{equation}
and the composition $\partial_\infty\circ(\text{drag at $\infty$})$ is naturally isomorphic to the identity functor.
Using the cofinality criterion Lemma \ref{cofinalitycriterion}, it follows immediately that both of these functors are cofinal.
This formalizes the idea (which is also clear from Lemma \ref{cofinalitycriterion}) that wrapping is an operation which happens entirely at contact infinity.

For later purposes, it will be important to know that a given Lagrangian $L\subseteq X$ admits cofinal wrappings which are disjoint from certain sufficiently small subsets of $\partial_\infty X$.  We show this by a general position argument; similar arguments can be found
in \cite[Proposition 5.2]{courtemassot}.
We will just state results for wrapping Legendrians---this implies the corresponding statement for Lagrangians in view of the cofinal functors \eqref{laglegcofinal}.

\begin{lemma} \label{lem:pushoffcore}
Let $Y^{2n-1}$ be a contact manifold, and let $\f\subseteq Y$ be a closed subset.
Let $\Lambda\subseteq Y$ be a compact Legendrian.
\begin{enumerate}
\item\label{pushoffcofinal}If $\f$ is contained in the smooth 
image of a second countable manifold of dimension $\leq n-1$, then $\Lambda$ admits cofinal wrappings $\Lambda\leadsto\Lambda^w$ with $\Lambda^w$ disjoint from $\f$.
\item\label{pushoffribbon}If, in addition, $\f$ is the core of a Liouville hypersurface inside $Y$, then $\Lambda$ admits cofinal wrappings $\Lambda\leadsto\Lambda^w$ with 
each $\Lambda^w$ disjoint from a neighborhood of $\f$.
\end{enumerate}
\end{lemma}

\begin{proof}
Statement \ref{pushoffcofinal} follows from a general position argument.
Indeed, we first claim that for any compact manifold-with-boundary $f:N^k\to Y^{2n-1}$ of dimension $k\leq n-1$, the locus of Legendrians $\Lambda\subseteq Y$ which are disjoint from the image of $N$ is open and dense (in all Legendrians).
To see this, consider the maps
\begin{equation}
\{\Lambda\subseteq Y\}\leftarrow\{p\in\Lambda\subseteq Y\}\times N\xrightarrow{(p,f)}Y\times Y.
\end{equation}
The right map is transverse to the diagonal, and hence the inverse image of the diagonal is a smooth codimension $2n-1$ submanifold of the middle space.
The left map has $(k+n-1)$-dimensional fibers, so the projection of something of codimension $2n-1>k+n-1$ is nowhere dense by Sard--Smale.
This shows that the locus of Legendrians disjoint from the image of $N$ is dense, and openness is obvious.
By the Baire category theorem, the locus of Legendrians disjoint from any countable collection of such $N$ is also dense.
Now simply note that for any positive Legendrian isotopy $\Lambda\leadsto\Lambda^w$, every sufficiently small perturbation $\Lambda^{w\prime}$ of $\Lambda^w$ also has a positive Legendrian isotopy $\Lambda\leadsto\Lambda^{w\prime}$.
(For similar arguments, see \cite[Proposition 5.2]{courtemassot}.)

For statement \ref{pushoffribbon}, consider local coordinates on $Y$ given by $([-1,1]\times F,dz+\lambda)$, where $F$ is the Liouville hypersurface with core $\f$.
What we should show is that for any Legendrian $\Lambda$ possibly intersecting the neighborhood $[-1,1]\times F$, it can be pushed out by a positive isotopy.
In coordinates $([-1,1]\times F,dz+\lambda)$, we have a positive contact vector field $V_\varphi:=\varphi(z)\partial_z+\varphi'(z)Z_\lambda$ for any smooth function $\varphi:[-1,1]\to\RR_{\geq 0}$.
Consider specifically the case that $\varphi(z)=\varphi(-z)$, $z\varphi'(z)\leq 0$, $\supp\varphi=[-\frac 23,\frac 23]$, and $\varphi|_{[-\frac 13,\frac 13]}\equiv 1$.
Now the inverse image of $[-\frac 13,\frac 13]\times F$ under the time $t$ flow of $V_\varphi$ is eventually contained in any neighborhood of $\{-\frac 23\}\times\f$ as $t\to\infty$ (to see this, note that $V_\varphi$ is proportional to $\partial_z+(\log\varphi)'(z)Z_\lambda$, so as the $z$ coordinate decreases towards $-\frac 23$ from above, we flow for infinite time by $-Z_\lambda$).
It follows that if $\Lambda$ is disjoint from $\{-\frac 23\}\times\f$, then flowing under (an arbitrary positive extension of) $V_\varphi$ for sufficiently large time $t$ produces the desired positive isotopy $\Lambda\leadsto\Lambda^w$.
To conclude, simply note that an arbitrary $\Lambda$ can be first perturbed in the positive direction to become disjoint from $\{-\frac 23\}\times\f$ by the first part of the Lemma.
\end{proof}

\begin{lemma}\label{genericpositiveisotopy}
Let $Y$ be a contact manifold, and let $\f=\f^\subcrit\cup\f^\crit\subseteq Y$ be mostly Legendrian.
For compact Legendrians $\Lambda_1,\Lambda_2\subseteq Y$ disjoint from $\f$, consider the space of positive Legendrian isotopies $\Lambda_1\leadsto\Lambda_2$.
The subspace of isotopies which 
\begin{enumerate}
    \item remain disjoint from $\f^\subcrit$ and 
    \item intersect $\f^\crit$ only finitely many times, each time passing through transversally at a single point, 
\end{enumerate}
is open and dense.
\end{lemma}

\begin{proof}
Consider first the locus of positive isotopies which remain disjoint from $\f^\subcrit$ but have no constraint with respect to $\f^\crit$.
We claim that this locus is open and dense inside the space of all positive isotopies.
This follows from an argument identical to that used to prove the first part of Lemma \ref{lem:pushoffcore}.

It now suffices to show that the locus of positive isotopies disjoint from $\f^\subcrit$ and only intersecting $\f^\crit$ by passing through transversally finitely many times is open and dense in the space of positive isotopies disjoint from $\f^\subcrit$.
To see this, we consider the maps
\begin{align}
&\left\{\begin{matrix}\Lambda_1\leadsto\Lambda_2\text{ inside }Y\setminus\f^\subcrit\hfill\\(p,v)\text{ a point and a tangent direction}\hfill\\\hphantom{(p,v)}\text{ in the total space of the isotopy}\hfill\end{matrix}\right\}\xrightarrow{(p,v)}(TY\setminus Y)/\RR_{>0}.\nonumber\\
&\qquad\downarrow\\
&\{\Lambda_1\leadsto\Lambda_2\subseteq Y\setminus\f^\subcrit\}\nonumber
\end{align}
The horizontal map is again a submersion, and hence the inverse image of $T\f^\crit$ is a smooth codimension $(2(2n-1)-1)-(2(n-1)-1)=2n$ submanifold of the middle space.
The vertical map has fibers of dimension $2n-1$, so the projection of something of codimension $2n$ is nowhere dense by Sard--Smale.
This shows that the locus of positive Legendrian isotopies which pass through $\f^\crit$ transversally is dense, and openness is obvious.
\end{proof}

\subsection{\texorpdfstring{$\ainf$}{A-infty}-categories}\label{ainfdefns}

We work throughout with small (meaning objects and morphisms form sets) $\ainf$-categories, modules, and bimodules, over a commutative ring $R$, with cofibrancy assumptions as in \cite[\S 3.1]{gpssectorsoc} (these are vacuous over a field), graded by an abelian group $G$ with specified maps $\ZZ\to G\to\ZZ/2$ (`degree shift' and `parity') composing to the usual map $\ZZ\to\ZZ/2$.
The reasoning in this paper is essentially agnostic about the choice of grading and coefficient ring (which we will often just write as $\ZZ$ for clarity).

In general discussion of $\ainf$-categories, we assume only cohomological unitality.
However, all of the $\ainf$-categories, modules, and bimodules that we construct are in fact strictly unital.

As much of the $\ainf$-category literature assumes field coefficients, we record in \S\ref{ainftyappendix} proofs in the context of commutative ring coefficients of some facts we will require.
On the other hand, while it is well-known that quasi-equivalences of $\ainf$-categories are invertible up to natural quasi-isomorphism over a field (see \cite[Theorem 2.9 and Remark 2.11]{seidelbook}), we \emph{do not} address the question of whether this holds more generally under our cofibrancy assumptions.
As a result, \emph{all statements about $\ainf$-categories (in particular, Fukaya categories) should be interpreted as ``up to inverting quasi-equivalences''} (for example, a functor of $\ainf$-categories may be a formal composition of genuine functors and formal inverses of quasi-equivalences).

We will make frequent use of the quotient and localization operations on $\ainf$-categories from \cite{lyubashenkoovsienko,lyubashenkomanzyuk}, for which we use \cite[\S 3.1.3]{gpssectorsoc} as a reference.
Given an $\ainf$-category $\C$, we may consider its quotient $\C/\A$ by any set $\A$ of objects of $\Tw\C$ 
(meaning, $\A$ is a set, and there is a map, not necessarily injective, from $\A$ to the set of objects of $\Tw\C$; sometimes, but not always, $\A$ is simply a subset of the set of objects of $\C$).%
\footnote{Here $\Tw\C$ denotes `twisted complexes' in the sense of \cite[(3l)]{seidelbook}.  An object of $\Tw\C$ is a pair $(X,\delta)$, where $X$ is a finite formal direct sum $X_1[a_1]\oplus\cdots\oplus X_n[a_n]$ of shifts of objects $X_i\in\C$, and $\delta\in\Hom(X,X)$ has degree $1$, is strictly upper triangular, and satisfies $\sum_{k\geq 1}\mu^k(\delta,\ldots,\delta)=0$.
If $\C$ is small, then evidently so is $\Tw\C$.}
The localization $\C[W^{-1}]$ at a class of morphisms $W$ in $H^0\C$ is the quotient by all cones $[X\xrightarrow aY]$ where $a\in\C(X,Y)$ is a cycle representing an element of $W$ (since $\C$ is small, this is indeed a \emph{set} of objects of $\Tw\C$).
The quotient by $\A$ depends (up to quasi-equivalence) only on the set of isomorphism classes split-generated by $\A$ \cite[Corollary 3.14]{gpssectorsoc}.
The quotient operation is functorial, in the sense that for a functor $F:\C\to\D$ and a set $\A$ of objects of $\Tw\C$, there is an induced quotient functor $\C/\A\to\D/F(\A)$.
In this context, we have the following result:

\begin{lemma}\label{quotientfullfaithful}
Let $F:\C\to\D$ be a functor, and let $\A$ be a set of objects of $\C$ (or $\Tw\C$).
If $F$ is fully faithful, then so is the resulting quotient functor $\C/\A\to\D/F(\A)$.
\end{lemma}

\begin{proof}
The action of the quotient functor on morphisms is given by
\begin{multline}
\smash{\bigoplus_{\begin{smallmatrix}p\geq 0\\Y_1,\ldots,Y_p\in\A\end{smallmatrix}}}\C(X,Y_1)\otimes\cdots\otimes\C(Y_p,Z)\\\to\bigoplus_{\begin{smallmatrix}p\geq 0\\Y_1,\ldots,Y_p\in\A\end{smallmatrix}}\D(F(X),F(Y_1))\otimes\cdots\otimes\D(F(Y_p),F(Z)).
\end{multline}
The induced map on associated gradeds of the domain and codomain above (filtered by $p$)
is a quasi-isomorphism since $F$ is fully faithful, and hence the map itself is also quasi-isomorphism.
\end{proof}

\subsection{Abstract wrapped Floer setups}\label{abstractfloersec}

We divide the definition of the partially wrapped Fukaya category into two parts, one algebraic and one geometric (this is a reworking of \cite[\S 3]{gpssectorsoc}).
This subsection explains the algebraic part.
We introduce an abstract categorical structure called an `abstract (wrapped) Floer setup' which axiomatizes certain basic properties of holomorphic curve counts.
We then show that any such structure gives rise to an $\ainf$-category, called its wrapped Fukaya category.
The next subsection explains the geometric part; it shows that holomorphic curve counts actually satisfy the axioms of an abstract wrapped Floer setup.

The basic reason for most of the complication in this subsection is the fact on the one hand, an $\ainf$-category has a single well-defined operation for every tuple $(L_0,\ldots,L_k)$, while on the other, counting pseudo-holomorphic curves yields operations only for those tuples $(L_0,\ldots,L_k)$ which are mutually transverse, and these operations moreover depend on a choice of Floer data.
One of our main tasks here is to upgrade a structure of the latter sort to one of the former (we note that work of Efimov \cite{efimovks} addresses a closely related question).
Our other main task is to `localize at continuation maps' by `wrapping Lagrangians'.
In fact, we perform these two tasks simultaneously.

We first introduce a notion of `$\ainf$-pre-category'
which formalizes the notion of having $\ainf$ operations only for certain tuples $(L_0,\ldots,L_k)$, which moreover depend on extra data (in practice, Floer data) chosen for the given tuple.
Recall the semisimplex category $\Delta_+$ of finite totally ordered
sets and strictly order preserving inclusions.   A semisimplicial set is a functor $\Delta_+^\op \to \Set$; given 
a semisimplicial set $\Sigma$ we write $\Sigma_n$ for the image of the ordered set $\{0<1<\cdots<n\}$.

\begin{definition}
An $\ainf$-pre-category (contrast with the definitions in \cite{abouzaidhighergenus,efimovks}) consists of the following data:
\begin{enumerate}
\item A semisimplicial set $X$ (the `objects' are the vertices $X_0$).
\item For every pair of objects $x,y\in X_0$ for which there exists an edge $x\xrightarrow ey$, a graded module $\Hom(x,y)$.
\item For every simplex $\sigma\in X_k$ of dimension $k\geq 1$, a map
\begin{equation*}
\mu^k_\sigma:\Hom(v_0,v_1)\otimes\cdots\otimes\Hom(v_{k-1},v_k)\to\Hom(v_0,v_k)[2-k]
\end{equation*}
where $v_0,\ldots,v_k\in X_0$ denote the vertices of $\sigma$.
These maps $\mu^k_\sigma$ (together with those associated to each face of $\sigma$) should satisfy the $\ainf$ relations.
In particular, $(\Hom(x,y),\mu^1_e)$ is a cochain complex for every edge $e$ from $x$ to $y$, and we require that each such complex be cofibrant in the sense fixed in \S\ref{ainfdefns}.
\end{enumerate}
\end{definition}

\begin{remark}\label{ainfprealtdef}
The data of an $\ainf$-pre-category can also be presented as follows, which is more in line with how we will think of it below:
\begin{enumerate}
\item A set $\sL$ of `objects' (this is the set of vertices $X_0$).
\item For each tuple $(L_0,\ldots,L_k)$ ($k\geq 1$) of objects, a set $D(L_0,\ldots,L_k)$ (this is the set of $k$-simplices $X_k$ spanning vertices $(L_0,\ldots,L_k)$).
\item Forgetful maps $D(L_0,\ldots,L_k)\to D(L_{i_0},\ldots,L_{i_\ell})$ for subsequences $0\leq i_0<\cdots<i_\ell\leq k$ with $\ell\geq 1$, required to be compatible with composition (these are the `face maps' $X_k\to X_\ell$).
\item For each pair $(L_0,L_1)$ with $D(L_0,L_1)\ne\varnothing$, a graded module $\Hom(L_0,L_1)$.
\item For each element $\delta\in D(L_0,\ldots,L_k)$, a map
\begin{equation*}
\mu^k_\delta:\Hom(L_0,L_1)\otimes\cdots\otimes\Hom(L_{k-1},L_k)\to\Hom(L_0,L_k)[2-k]
\end{equation*}
satisfying the $\ainf$ relations with respect to the forgetulful maps on the sets $D$.
We require every cochain complex $(\Hom(L_0,L_1),\mu^1_\delta)$ to be cofibrant.
\end{enumerate}
\end{remark}

From an $\ainf$-pre-category, one can form certain directed $\ainf$-categories via the following construction, which will be particularly relevant for us later.
Recall that the \emph{nerve} of a poset $P$ is the semisimplicial set whose $n$-simplices are the totally ordered tuples $p_0>\cdots>p_n$ in $P$ (beware that this is the `opposite' of the standard convention).

\begin{definition}\label{def: Op}
Let $X$ be an $\ainf$-pre-category and let $P$ be a poset.
A map $\eta:P\to X$ (meaning a map from the nerve of $P$ to the semisimplicial set $X$) determines a strictly unital $\ainf$-category $\OO_P$ as follows.
An object of $\OO_P$ is an element $p\in P$, and the morphisms are given by
\begin{equation}
\OO_P(p,q):=\begin{cases}\Hom(\eta(p,q))&p>q,\\\ZZ \langle\1_p \rangle &p=q,\\0,&\text{else.}\end{cases}
\end{equation}
The $\ainf$ operation for a tuple $p_0>\cdots>p_k$ is the operation in $X$ associated to the image of the simplex $(p_0,\ldots,p_k)$ under $\eta$, namely $\mu^k_{\eta(p_0,\ldots,p_k)}$.
The remaining operations (namely those involving one or more $\1_p$) are specified uniquely by requiring that $\1_p$ be a strict unit (that is, they all vanish, except in the case of $k=2$ for triples $p>q=q$ and $p=p>q$, for which they are the `identity', up to a sign depending on one's sign convention, compare \cite{seidelsubalgebras, seidelbook}).
In particular, the $\ainf$-category $\OO_P$ is strictly unital.
Note that the morphism complexes in $\OO_P$ are indeed cofibrant, as required in our definition of $\ainf$-category.
\end{definition}

In practice, we will consider $\ainf$-pre-categories coming from counting holomorphic disks with Lagrangian boundary conditions.
An element of $X_k$ will be a tuple of mutually transverse Lagrangians $(L_0,\ldots,L_k)$ together with a choice of `Floer data' for counting holomorphic disks with boundary conditions $(L_0,\ldots,L_k)$ (read counterclockwise).
The upper left picture in Figure \ref{figureaxiomcounts} is intended to indicate such disks.

The construction of the Fukaya category we employ also requires counting the holomorphic disks illustrated in the other five pictures in Figure \ref{figureaxiomcounts}.
We organize the relevant sets of Floer data and associated curve counts into a structure which we call an `abstract Floer setup'.

\begin{figure}[ht]
\centering
\includegraphics[max width=.95\textwidth]{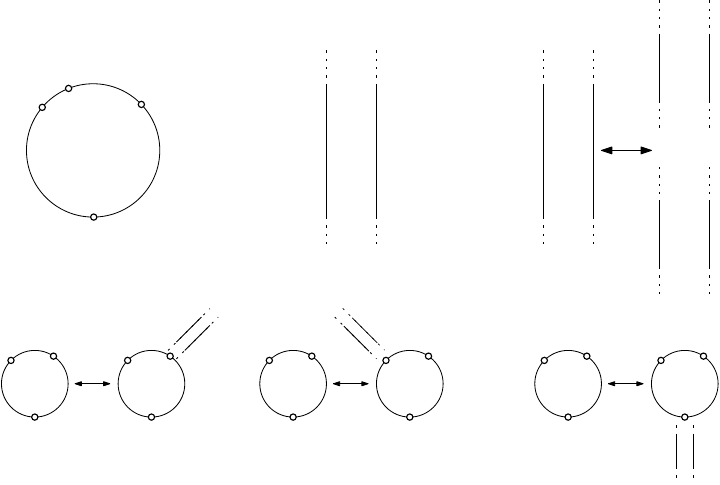}
\caption{An abstract Floer setup keeps track of Floer data for counting holomorphic maps with Lagrangian boundary conditions from the domains illustrated here.  These are, precisely: holomorphic disks with boundary conditions $(L_0,\ldots,L_k)$ (upper left), strips $\RR\times[0,1]$ with boundary conditions $(L_0,L_1)$ (upper middle), a one-parameter family of strips $\RR\times[0,1]$ over an interval which at one of the endpoints breaks into two strips $\RR\times[0,1]\sqcup\RR\times[0,1]$ (upper right), and families of three-pointed disks over intervals which at one of the endpoints break off a strip $\RR\times[0,1]$ at one of the three boundary marked points (lower row).  Floer data for each configuration, respectively, is specified by an element of $D$, $D'$, $D''$ (top row) or $D'''_i$ for $i=0,1,2$ (bottom row).}\label{figureaxiomcounts}
\end{figure}

\begin{definition}
An \emph{abstract Floer setup} $\SSS$ consists of the following data.
\begin{enumerate}
\item\label{afslagra}(Lagrangians) A set $\sL$.
\item(Composability) Subsets $\sL_k\subseteq\sL^{k+1}$ (whose elements are called `composable tuples') for integers $k\geq 1$ which are closed under passing to subsequences, in the sense that if $(L_0,\ldots,L_k)\in\sL_k$ then $(L_{i_0},\ldots,L_{i_\ell})\in\sL_\ell$ for every $0\leq i_0<\cdots<i_\ell\leq k$.
\item(Floer cochain modules) Graded modules $CF^\bullet(L,K)$ for all $(L,K)\in\sL_1$.
\item\label{afsfloerdata}(Floer data) Sets $D(L_0,\ldots,L_k)$ (whose elements are called `Floer data') for all composable tuples $(L_0,\ldots,L_k)$, together with restriction maps $D(L_0,\ldots,L_k)\to D(L_{i_0},\ldots,L_{i_\ell})$ for every subsequence $0\leq i_0<\cdots<i_\ell\leq k$ with $\ell\geq 1$.
The restriction maps should be compatible with composition; that is, $D$ is a contravariant functor to the category of sets from the category whose objects are composable tuples and whose morphisms are inclusions of subsequences.
\item\label{afsfloerdatacontr}(Contractibility of Floer data) We require that the map from $D(L_0,\ldots,L_k)$ to the limit of $D$ applied to all proper subsequences of $(L_0,\ldots,L_k)$ should be surjective (in the degenerate case $k=1$, in which there are no proper subsequences, this colimit is a single point, so the assertion becomes that $D(L_0,L_1)$ is nonempty); this encodes the fact that `Floer data can be constructed by induction'.
\item\label{afscounts}(Holomorphic disk counts) A map
\begin{equation*}
\mu^k_\delta:CF^\bullet(L_0,L_1)\otimes\cdots\otimes CF^\bullet(L_{k-1},L_k)\to CF^\bullet(L_0,L_k)[2-k]
\end{equation*}
for every element $\delta\in D(L_0,\ldots,L_k)$.
These should satisfy the relations induced by all `codimension one splittings' of disks (known as the $\ainf$ relations), with respect to the forgetful maps on $D$.
In particular, $(CF^\bullet(L,K),\mu^1_\delta)$ is a cochain complex, and we require each such complex to be cofibrant in the sense fixed in \S\ref{ainfdefns}.
\newcounter{savecounting}
\setcounter{savecounting}{\value{enumi}}
\end{enumerate}
Axioms \ref{afslagra}--\ref{afsfloerdata} and \ref{afscounts} above are equivalent to the axioms of Remark \ref{ainfprealtdef}; they are thus the same as specifying an $\ainf$-pre-category, which we denote by $\F_\SSS^{\pre}$.
Axiom \ref{afsfloerdatacontr} says roughly speaking that the `topology' of the space of operations is trivial, up to knowing which tuples are composable.
Data satsifying these axioms will be produced by counting holomorphic disks as in the upper left corner of Figure \ref{figureaxiomcounts}.
We also record counts of holomorphic disks in the other five pictures in Figure \ref{figureaxiomcounts}, which take the following form.
\begin{enumerate}
\setcounter{enumi}{\value{savecounting}}
\item(Floer data) Sets $D'(L,K)$, $D''(L,K)$, and $D'''_i(L_0,L_1,L_2)$ for $i=0,1,2$, and surjective maps
\begin{align*}
D'(L,K)&\to D(L,K)\times D(L,K),\\
D''(L,K)&\to D'(L,K)\times_{D(L,K)\times D(L,K)}(D'(L,K)\times_{D(L,K)}D'(L,K)),
\end{align*}
and from $D'''_i(L_0,L_1,L_2)$ for $i=0,1,2$, respectively, to
\begin{align*}
D(L_0,L_1,L_2)\times_{D(L_0,L_1)\times D(L_1,L_2)\times D(L_0,L_2)}(D(L_0,L_1,L_2)\times_{D(L_0,L_1)}D'(L_0,L_1)),\\
D(L_0,L_1,L_2)\times_{D(L_0,L_1)\times D(L_1,L_2)\times D(L_0,L_2)}(D(L_0,L_1,L_2)\times_{D(L_1,L_2)}D'(L_1,L_2)),\\
D(L_0,L_1,L_2)\times_{D(L_0,L_1)\times D(L_1,L_2)\times D(L_0,L_2)}(D(L_0,L_1,L_2)\times_{D(L_0,L_2)}D'(L_0,L_2)).
\end{align*}
\item(Holomorphic disk counts)
\begin{align*}
\delta' \in D'(L,K)  \quad &   \alpha_{\delta'}: CF^\bullet(L,K) \to CF^\bullet(L,K)  \\ 
\delta'' \in D''(L,K) \quad & \beta_{\delta''}: CF^\bullet(L,K)\to CF^\bullet(L,K)[-1] \\ 
\delta''' \in D'''_i(L_0,L_1,L_2) \quad & \gamma_{\delta'''}: CF^\bullet(L_0,L_1)\otimes CF^\bullet(L_1,L_2)\to CF^\bullet(L_0,L_2)[-1]
\end{align*}
satisfying the identities coming from codimension one degenerations of these disks, namely the following.
The map $\alpha_{\delta'}$ is a chain map, where the domain and codomain are equipped with the differentials $\mu^1_\delta$ associated to the image of $\delta'$ under the map $D'(L,K)\to D(L,K)\times D(L,K)$.
The map $\beta_{\delta''}$ is a chain homotopy between an $\alpha$ map and the composition of two $\alpha$ maps associated to the image of $\delta''$ under the forgetful map out of $D''(L,K)$.
The map $\gamma_{\delta'''}$ is a chain homotopy between a $\mu^2$ map and the composition of an $\alpha$ and a $\mu^2$ map (where $\alpha$ is composed depends on $i$).
\item A map $f: D(L,K)\to D'(L,K)$ whose composition with the forgetful map $D'(L,K)\to D(L,K)^2$ is the diagonal and 
such that every map $\alpha_{f(\delta)}: CF^\bullet(L,K)\to CF^\bullet(L,K)$ is the identity.
\end{enumerate}
\end{definition}

An abstract Floer setup $\SSS$ determines, quite directly, a pre-category $H^\bullet\F_\SSS^{\pre}$ enriched in the category of graded modules, as we are about to explain.
Note that a general $\ainf$-pre-category $\C$ \emph{does not} determine a cohomology pre-category $H^\bullet\C$, since nothing forces canonical isomorphisms between $H^\bullet(\Hom(x,y),\mu^1_e)$ for different edges $e:x\to y$.
The purpose of the extra sets of Floer data $D'$, $D''$, $D'''_i$, and their associated operations in an abstract Floer setup is precisely to fix such isomorphisms and to ensure they are compatible with multiplication $\mu^2$ up to homotopy.

We use $h$ to indicate passing from a complex to the associated object in the homotopy category of complexes.
An object of this category will be called cofibrant iff it can be represented by a cofibrant complex.
A morphism of cofibrant objects in the homotopy category is an isomorphism iff it is a quasi-isomorphism \cite[Lemma 3.6]{gpssectorsoc}.
Instead of defining $H^\bullet\F_\SSS^{\pre}$, we will do a bit better and define $h\F_\SSS^{\pre}$.

\begin{definition}
Let $\SSS$ be an abstract Floer setup.
Its \emph{Donaldson--Fukaya pre-category $h\F_\SSS^{\pre}$} consists of the following data:
\begin{enumerate}
\item The set of objects $\sL$ (from $\SSS$).
\item\label{precatcomposable}The composable tuples $(L_0,\ldots,L_k)$ in $\sL_k\subseteq\sL^{k+1}$ (from $\SSS$).
\item For $(L_0,L_1)\in\sL_1$, a cofibrant object $hF^\bullet(L_0,L_1)$ in the homotopy category of complexes.
\item For $(L_0,L_1,L_2)\in\sL_2$, a `composition' map $hF^\bullet(L_0,L_1)\otimes hF^\bullet(L_1,L_2)\to hF^\bullet(L_0,L_2)$.
\item For $(L_0,L_1,L_2,L_3)\in\sL_3$, the two maps $hF^\bullet(L_0,L_1)\otimes hF^\bullet(L_1,L_2)\otimes hF^\bullet(L_2,L_3)\to hF^\bullet(L_0,L_3)$ obtained by composing in either order agree.
\item For $(L_0,\ldots,L_k)\in\sL_k$ for $k\geq 4$, no additional data is specified, although we retain (as mentioned already in \ref{precatcomposable} above) the data from $\SSS$ of which such tuples are to be regarded as `composable'.
\end{enumerate}
These are constructed as follows from the data of $\SSS$.
An element of $D(L,K)$ determines a differential $\mu^1$ on $CF^\bullet(L,K)$.
An element of $D'(L,K)$ determines a chain map between $CF^\bullet(L,K)$ equipped with any two such differentials.
These chain maps compose up to chain homotopy by virtue of the elements of $D''(L,K)$.
Furthermore, the map on $CF^\bullet(L,K)$ determined by an element in the image of $D(L,K)\to D'(L,K)$ is the identity.
This defines a cofibrant object $hF^\bullet(L,K)$ in the homotopy category of complexes (well defined up to unique isomorphism).
In particular, its cohomology groups $HF^\bullet(L,K)$ are well defined.
The composition maps for composable $(L_0,L_1,L_2)$ are the operations $\mu^2$ associated to elements of $D(L_0,L_1,L_2)$.
These composition maps are well defined in the homotopy category by virtue of the homotopies encoded by elements of $D'''_i(L_0,L_1,L_2)$.
They are associative for composable $(L_0,L_1,L_2,L_3)$ using the $\mu^3$ associated to an element of $D(L_0,L_1,L_2,L_3)$.  
\end{definition}

The \emph{wrapped} Fukaya category $\W_\SSS$ will depend on an abstract Floer setup $\SSS$ together with extra data.
The most important piece of extra data is a class of morphisms $C$ called `continuation maps'.
Morally speaking, the wrapped Fukaya category is simply the localization of $\F_\SSS^{\pre}$ at $C$; in fact, Conjecture \ref{wrappedlocalizationconjecture} predicts that there is a precise sense in which this is true.
The construction of the wrapped Fukaya category will, however, depend \emph{a priori} on more data than just the continuation maps.%
\footnote{It is, of course, well known that localizations of categories typically become tractable only once additional assumptions are imposed (for example that the class of morphisms being localized at is a multiplicative system) or extra data is chosen (for example a model structure).}
Although Conjecture \ref{wrappedlocalizationconjecture} would imply that the final result of this construction depends only on the continuation maps, it would not imply that the extra data is totally extraneous either, since this extra data not only allows to define the wrapped Fukaya category but also to understand its morphisms in terms of geometrically wrapping Lagrangian submanifolds.
We encode the abstract data used to define the wrapped Fukaya category into the following axiomatic setup:

\begin{definition}\label{abstractwrappedfloersetupdef}
An \emph{abstract wrapped Floer setup} is an abstract Floer setup $\SSS$ along with the following:
\begin{enumerate}
\item\label{envelope}A category $H^\bullet\F_\SSS$ with objects $\sL$ (the Lagrangians of $\SSS$) enriched over graded modules, which coincides with $H^\bullet\F_\SSS^{\pre}$ for composable tuples (we call $H^\bullet\F_\SSS$ an `envelope' for $H^\bullet\F_\SSS^{\pre}$; morphisms in $H^\bullet\F_\SSS$ are denoted $HF^\bullet$).
\item\label{continuation}A set $C$ of morphisms in $H^0\F_\SSS$ (called `continuation maps') closed under composition.
\item\label{wrappingcategory}For all $L\in\sL$, a filtered category $R_L$ of countable cofinality (called the `wrapping category' of $L$) together with a functor $R_L\to((H^0\F_\SSS)_{/_CL})^\op$ (where $(H^0\F_\SSS)_{/_CL}$ is the `continuation slice category' of $L$, i.e.\ the full subcategory of the slice category $(H^0\F_\SSS)_{/L}$ spanned by those morphisms $L^w\to L$ which are continuation maps).
We define
\begin{equation*}
HW^\bullet(L,K)=\varinjlim_{L^w\in R_L}HF^\bullet(L^w,K),
\end{equation*}
and we require that the map $HW^\bullet(L,K)\to HW^\bullet(L,K')$ given by multiplying on the right by any continuation map $K\to K'$ be an isomorphism (this is the `right locality' property).
\item\label{perturbation}We require the following `factorization property': for all morphisms $A\leadsto B$ in $R_L$ and all finite sets $\mathbf K=\{(K^j_1,\ldots,K^j_{a_j})\}_j$ of composable tuples, there must exist a factorization $A\leadsto H\leadsto B$ in $R_L$ of the given morphism $A\leadsto B$ such that each tuple $(H,K^j_1,\ldots,K^j_{a_j})$ is composable (where here we mean the image of $H$ in $\sL$).
Moreover, we require that there always be uncountably many possibilities for the image of $H$ in $\sL$.
\end{enumerate}
\end{definition}

\begin{remark}
The existence for every $L$ of a filtered category $R_L$ satisfying the right locality property is equivalent to the assertion that the set of continuation maps $C$ is a `right multiplicative system' (for basic properties of right multiplicative systems see \cite[Chapter I]{gabrielzisman}).
In this case, the opposite continuation slice category $((H^0\F_\SSS)_{/_CL})^\op$ is itself filtered and satisfies the right locality property, and the functor $R_L\to((H^0\F_\SSS)_{/_CL})^\op$ is cofinal (making the choice of $R_L$ essentially irrelevant).
We make no logical appeal to any of these facts, rather we simply note that they suggest the existence of a somewhat weaker axiomatic framework which is sufficient for the construction of the wrapped Fukaya which follows below.
This is also evidenced by the fact that the wrapped Fukaya category depends only on the \emph{existence} of wrapping categories (see Lemma \ref{HWlocalizationproperty}, Remark \ref{WSprereq}, and the discussion surrounding \eqref{setupmorphism}).
Somewhat orthogonally, we also expect the choice of envelope to be unnecessary (see Conjecture \ref{wrappedlocalizationconjecture}), although the technical simplifications resulting from having a category (as opposed to the pre-category) should not be underestimated.
\end{remark}

\begin{definition}
The \emph{wrapped Donaldson--Fukaya category} $H^\bullet\W_\SSS$ of an abstract wrapped Floer setup is the full subcategory of $\Pro H^\bullet\F_\SSS$ spanned by $R_L$ for $L\in\sL$.
In other words, the objects of $H^\bullet\W_\SSS$ are the Lagrangians $\sL$ of $\SSS$, and morphisms are given by
\begin{equation}
\Hom(L,K)=\varprojlim_{K^w\in R_K}\varinjlim_{L^w\in R_L}HF^\bullet(L^w,K^w),
\end{equation}
with the evident notion of composition.
In view of the right locality property of wrapping, the inverse system over $K^w\in R_K$ is in fact constant, so $\Hom(L,K)$ is nothing other than $HW^\bullet(L,K)$.
\end{definition}

It is essentially tautological that $H^\bullet\W_\SSS$ is the localization of $H^\bullet\F_\SSS$ at the continuation maps, in the following precise sense:

\begin{lemma}\label{HWlocalizationproperty}
For any category $\C$, the map
\begin{equation}
\Fun(H^\bullet\W_\SSS,\C)\to\Fun(H^\bullet\F_\SSS,\C)
\end{equation}
is fully faithful, and its essential image consists of those functors $H^\bullet\F_\SSS\to\C$ which send continuation maps to isomorphisms.
\end{lemma}

\begin{proof}
For $F,G:H^\bullet\W_\SSS\to\C$, consider the map $\Hom(F,G)\to\Hom(F|_{H^\bullet\F_\SSS},G|_{H^\bullet\F_\SSS})$.
This map is an inclusion of subsets of $\prod_L\Hom(F(L),G(L))$, so is certainly injective.
It is also surjective since every morphism in $H^\bullet\W_\SSS$ is a composition of morphisms in $H^\bullet\F_\SSS$.

Finally, let $F:H^\bullet\F_\SSS\to\C$ be any functor sending continuation maps to isomorphisms.
To extend it to $H^\bullet\W_\SSS$, we consider the map
\begin{equation}
\Hom(L,K)=\varprojlim_{K^w\in R_K}\varinjlim_{L^w\in R_L}HF^\bullet(L^w,K^w)\xrightarrow{F}\varprojlim_{K^w\in R_K}\varinjlim_{L^w\in R_L}\Hom_\C(F(L^w),F(K^w)),
\end{equation}
and we note that the right hand side is identified with $\Hom_\C(F(L),F(K))$ since $F$ sends continuation maps to isomorphisms.
This is compatible with composition by inspection.
\end{proof}

Our final task is to upgrade $H^\bullet\W_\SSS$ to an $\ainf$-category $\W_\SSS$ called the \emph{wrapped Fukaya category} of the abstract wrapped Floer setup $\SSS$.

A \emph{decorated poset} for an abstract wrapped Floer setup $\SSS$ is a poset $P$ together with a map $\eta:P\to\F_\SSS^{\pre}$ as in Definition \ref{def: Op}.
Such a map $\eta$ is thus the assignment to each $p\in P$ of a Lagrangian $L_p\in\sL$ such that for each totally ordered subset $p_r>\ldots>p_0\in P$ the resulting tuple of Lagrangians $(L_{p_r},\ldots,L_{p_0})$ is composable, along with compatible choices of `Floer data', i.e.\ elements of $D(L_{p_r},\ldots,L_{p_0})$, for all such totally ordered subsets of $P$.
We denote the resulting $\ainf$-category by $\OO_P$ (Definition \ref{def: Op}).

Given a decorated poset $P$, we define
\begin{equation}
\W_P:=\OO_P[I^{-1}]
\end{equation}
to be the localization of $\OO_P$ at the set $I$ of morphisms in $H^0\OO_P$ which are (sent to) isomorphisms in $H^\bullet\W_\SSS$ (so in particular $I$ contains all continuation maps in $H^0\OO_P$).
For general decorated posets $P$, the category $\W_P$ has little significance.
However for certain carefully chosen $P$, it will be a full subcategory of $\W_\SSS$.

\begin{definition}
Let $P$ be a decorated poset.
A \emph{$P$-wrapping sequence} is a sequence $p_0<p_1<\cdots\in P$ which is cofinal in $P$ along with elements of $I$ in $H^0\OO_P(p_{i+1},p_i)=HF^0(L_{p_{i+1}},L_{p_i})$ such that the natural map
\begin{equation}\label{pwrappingdef}
\varinjlim_iHF^\bullet(L_{p_i},K)\to\varinjlim_iHW^\bullet(L_{p_i},K)=HW^\bullet(L_{p_0},K)
\end{equation}
is an isomorphism for every $K$.
If every $p\in P$ belongs to a $P$-wrapping sequence, we say that $P$ is \emph{sufficiently wrapped}.
Note that if $L_p\cong L_{p'}$ in $H^\bullet\W_\SSS$, then $p$ belongs to a $P$-wrapping sequence iff $p'$ does.
\end{definition}

\begin{lemma}\label{pwrappingworks}
For any $P$-wrapping sequence $p_0<p_1<\cdots\in P$, the natural maps
\begin{equation}\label{pwrappingworkseqn}
\varinjlim_iH^\bullet\OO_P(p_i,q)\to\varinjlim_iH^\bullet\W_P(p_i,q)
\end{equation}
are both isomorphisms.
More generally, for any left $\OO_P$-module $\M$, the natural maps
\begin{equation}\label{pwrappingworkseqnmodule}
\varinjlim_iH^\bullet\M(p_i)\to\varinjlim_iH^\bullet{}_{I^{-1}}\M(p_i)
\end{equation}
are both isomorphisms.
\end{lemma}

\begin{proof}
The direct limit $\varinjlim_iH^\bullet\OO_P(p_i,-)=HW^\bullet(L_{p_0},L_-)$ sends morphisms in $I$ to isomorphisms, which implies that \eqref{pwrappingworkseqn} is an isomorphism by \cite[Lemma 3.16]{gpssectorsoc}.
We may reduce \eqref{pwrappingworkseqnmodule} to \eqref{pwrappingworkseqn} as follows.
The natural map $\OO\otimes_\OO\M\to\M$ is a quasi-isomorphism \cite[Lemma 3.7]{gpssectorsoc}, so it suffices to consider modules of the form $\OO\otimes_\OO\M$.
Such modules have a `length filtration' (coming from the tensor product) whose subquotients take the form $\OO(-,q_0)\otimes\cdots\otimes\OO(q_{r-1},q_r)\otimes\M(q_r)$.
It thus suffices to prove the desired result for such modules.
Now each tensor factor in ${}\otimes\OO(q_0,q_1)\otimes\cdots\otimes\OO(q_{r-1},q_r)\otimes\M(q_r)$ is cofibrant, and tensoring with cofibrant complexes preserves acyclicity \cite[Definition 3.2(v)]{gpssectorsoc}.
We are thus reduced to the case of the modules $\OO(-,q_0)$ which is just \eqref{pwrappingworkseqn} (in these reduction steps, we are using the fact that the map \eqref{pwrappingworkseqnmodule} can be realized on the chain level by taking mapping telescopes).
\end{proof}

\begin{proposition}\label{suffwrapequiv}
If $P$ is sufficiently wrapped, then there is a canonical full faithful inclusion $H^\bullet\W_P\subseteq H^\bullet\W_\SSS$.
This inclusion is compatible with inclusions $P\subseteq P'$.
\end{proposition}

\begin{proof}
Let $p\in P$.
A choice of $P$-wrapping sequence $p=p_0<p_1<\cdots$ determines a chain of isomorphisms
\begin{equation}
HW^\bullet(L_p,L_q)\xleftarrow\sim\varinjlim_iH^\bullet\OO_P(p_i,q)\to\varinjlim_iH^\bullet\W_P(p_i,q)\xleftarrow\sim H^\bullet\W_P(p,q)
\end{equation}
in view of Lemma \ref{pwrappingworks} and the definition of a $P$-wrapping sequence.
Next we consider the canonicity of this isomorphism.

Consider an inclusion $P\subseteq P'$.
Let $p\in P$, and choose a $P$-wrapping sequence $p=p_0<p_1<\cdots\in P$.
Now each $p_i$, regarded as an element of $P'$, has a $P'$-wrapping sequence $p_i=p_{i0}<p_{i1}<\cdots$.
These choices are summarized in the following diagram of solid arrows
\begin{equation}
\begin{tikzcd}[row sep=small]
&\vdots\ar[d]&\vdots\ar[d]&\vdots\ar[d]\\
\cdots\ar[r,dotted]&p_{22}\ar[d]\ar[r,dotted]&p_{12}\ar[d]\ar[r,dotted]&p_{02}\ar[d]\\
\cdots\ar[r,dotted]&p_{21}\ar[d]\ar[r,dotted]&p_{11}\ar[d]\ar[r,dotted]&p_{01}\ar[d]\\
\cdots\ar[r]&p_{20}\ar[d,equal]\ar[r]&p_{10}\ar[d,equal]\ar[r]&p_{00}\ar[d,equal]\\
\cdots\ar[r]&p_2\ar[r]&p_1\ar[r]&p_0\ar[r,equal]&p
\end{tikzcd}
\end{equation}
Going by induction on $i$, we note that since the sequence $p_i=p_{i0}<p_{i1}<\cdots$ is cofinal in $P'$, we can delete some of its members and re-index so as to ensure that $p_{ij}>p_{i-1,j}$.
Moreover, in view of \eqref{pwrappingdef}, we can do this deleting and re-indexing (by induction on $j$) so as to ensure that each composition $p_{ij}\to p_{i,j-1}\to p_{i-1,j-1}$ factors as $p_{ij}\to p_{i-1,j}\to p_{i-1,j-1}$.
This fixes dotted arrows in the above diagram, making everything commute.
The solid arrows are by definition in $I$, and this implies the new dotted arrows are as well.
Now the following diagram obviously commutes:
\begin{equation}
\begin{tikzcd}[column sep=tiny,row sep=small]
{}&\ar[dl,swap,"\sim"]\varinjlim_jH^\bullet\OO_{P'}(p_{0j},q)\ar[r]\ar[d]&\varinjlim_jH^\bullet\W_{P'}(p_{0j},q)\ar[r,equal]\ar[d]&H^\bullet\W_{P'}(p,q)\ar[d,equal]\\
HW^\bullet(L_p,L_q)&\ar[l,swap,"\sim"]\varinjlim_{i,j}H^\bullet\OO_{P'}(p_{ij},q)\ar[r]&\varinjlim_{i,j}H^\bullet\W_{P'}(p_{ij},q)\ar[r,equal]&H^\bullet\W_{P'}(p,q)\\
{}&\ar[ul,"\sim"]\varinjlim_iH^\bullet\OO_P(p_i,q)\ar[r]\ar[u]&\varinjlim_iH^\bullet\W_P(p_i,q)\ar[r,equal]\ar[u]&H^\bullet\W_P(p,q)\ar[u]
\end{tikzcd}
\end{equation}
We have thus shown that the identifications $H^\bullet\W_P(p,q)=HW^\bullet(L_p,L_q)=H^\bullet\W_{P'}(p,q)$ are compatible with the map $H^\bullet\W_P(p,q)\to H^\bullet\W_{P'}(p,q)$.

Finally, let us conclude.
The first paragraph of this proof constructs an identification $H^\bullet\W_P(p,q)=HW^\bullet(L_p,L_q)$ depending on a choice of $P$-wrapping sequence for $p$.
The next paragraph shows that for $P\subseteq P'$, the identifications induced by any $P$-wrapping sequence and any $P'$-wrapping sequence of $p\in P$ coincide.
In particular, when $P=P'$, the identification $H^\bullet\W_P(p,q)=HW^\bullet(L_p,L_q)$ is independent of choice of $P$-wrapping sequence.
\end{proof}

We now turn to the construction of sufficiently wrapped decorated posets.
An inclusion of posets $P\subseteq P'$ will be called \emph{downward closed} iff $P\ni p\geq q\in P'$ implies $q\in P$.
A poset $P$ is called \emph{cofinite} iff $P^{\leq p}$ is finite for all $p\in P$.

\begin{lemma}\label{pcanwrap}
Every countable cofinite decorated poset $P$ admits a downward closed inclusion into a sufficiently wrapped countable cofinite decorated poset $P'$.
\end{lemma}

\begin{proof}
Choose arbitrarily a sequence of finite downward closed subsets $S_0\subseteq S_1\subseteq\cdots\subseteq P$ whose union is $P$.
Our new poset will be $P'=P\sqcup\ZZ_{\geq 0}$ as a set, equipped with the following ordering.
The ordering on $P$ will be the given one, and the ordering on $\ZZ_{\geq 0}$ the usual one.
The only additional order relations will be that $i\in\ZZ_{\geq 0}$ is larger than everything in $S_i$.

Our task is to assign Lagrangians to the elements of $\ZZ_{\geq 0}$ so that every finite chain in $P'$ is composable and so that $P'$ is sufficiently wrapped.
(The `Floer data' part of the decorations has no bearing on whether $P'$ is sufficiently wrapped, and can be chosen by induction on the skeleta of the nerve of $P'$ after fixing the Lagrangian labels.)

Color the elements of $\ZZ_{\geq 0}$ with the elements of $P$, so that every color appears infinitely often.
That is, choose a map $q:\ZZ_{\geq 0}\to P$ all of whose fibers are infinite.
For each $p\in P$, choose a cofinal sequence $L_p^0\leadsto L_p^1\leadsto\cdots$ in $R_{L_p}$.
Denote by $\cdots>i_{p,1}>i_{p,0}\in\ZZ_{\geq 0}$ the collection of indices $i$ with color $q(i)=p$.
We assign to $i_{p,k}$ the Lagrangian $H$ arising from a choice of factorization $L_p^k\leadsto H\leadsto L_p^{k+1}$ in $R_{L_p}$.
We choose these factorizations by induction on $i\in\ZZ_{\geq 0}$, using the factorization property Definition \ref{abstractwrappedfloersetupdef}\ref{perturbation} to ensure that all totally ordered subsets of $P'$ are composable.
Now each sequence $\cdots>i_{p,1}>i_{p,0}\in\ZZ_{\geq 0}\subseteq P'$ is a $P'$-wrapping sequence since it can lifted to $R_{L_p}$ and interleaved with $L_p^0\leadsto L_p^1\leadsto\cdots$, ensuring this lift to $R_{L_p}$ is cofinal.
\end{proof}

\begin{remark}\label{pcanwrapcountable}
In practice, the factorization axiom Definition \ref{abstractwrappedfloersetupdef}\ref{perturbation} often also holds for countable $\mathbf K$.
In this case, an alternative form of Lemma \ref{pcanwrap} also holds: every countable decorated poset admits a downward closed inclusion into a sufficiently wrapped countable decorated poset (i.e.\ delete the word `cofinite' in both the hypothesis and the conclusion).
The proof is the same, except that we no longer require $S_i$ to be finite (and so we could, for example, take $S_i=P$ for all $i$).
\end{remark}

Lemma \ref{pcanwrap} implies the existence of a countable cofinite sufficiently wrapped $P$ containing any countable set of Lagrangians we like.
Moreover, it also implies that given any two countable cofinite sufficiently wrapped posets $P$ and $P'$, their disjoint union can be included into a third countable cofinite sufficiently wrapped $P''$.
Thus any two $\W_P$ (for countable cofinite sufficiently wrapped $P$) embed fully faithfully into a third, in a way which respects the embeddings of their cohomology categories into $H^\bullet\W_\SSS$.
Thus when $H^\bullet\W_\SSS$ has countably many isomorphism classes, we can define $\W_\SSS$ as $\W_P$ for `any' countable cofinite sufficiently wrapped $P$ for which $H^\bullet\W_P\to H^\bullet\W_\SSS$ is essentially surjective (knowing that all such are `canonically' equivalent).
Under the conditions of Remark \ref{pcanwrapcountable}, this holds more generally for all sufficiently wrapped countable decorated posets $P$ (not necessarily cofinite).

We can in fact do slightly better: there is a canonical decorated poset $P$ with a canonical equivalence $H^\bullet\W_P=H^\bullet\W_\SSS$, defined as follows.
Let us say that a decorated poset $P$ has \emph{no duplicates} when $P^{\leq a}$ and $P^{\leq b}$ are non-isomorphic (as decorated posets) whenever $a\ne b$.
Given any two cofinite decorated posets with no duplicates $P$ and $P'$, there is at most one downward closed inclusion $P\to P'$.
There is now a universal cofinite decorated poset with no duplicates $P_{\mathrm{univ}}(\SSS)$ into which all others admit a unique downward closed embedding (compare \cite[Lemma 3.42]{gpssectorsoc}; the elements of $P_{\mathrm{univ}}(\SSS)$ are the isomorphism classes of cofinite decorated posets with no duplicates with a maximum element; clearly these form a \emph{set}, and each one has trivial automorphism group).

\begin{definition}\label{Sgeneralwrapped}
The wrapped Fukaya category of an abstract wrapped Floer setup is $\W_\SSS:=\W_{P_{\mathrm{univ}}(\SSS)}$.
\end{definition}

\begin{proposition}[{compare \cite[Proposition 3.43]{gpssectorsoc}}]
The cohomology category $H^\bullet\W_\SSS$ of the wrapped Fukaya category $\W_\SSS$ from Definition \ref{Sgeneralwrapped} is canonically equivalent to the wrapped Donaldson--Fukaya category $H^\bullet\W_\SSS$ defined previously.
\end{proposition}

\begin{proof}
In Lemma \ref{pcanwrap}, if $P$ has no duplicates, then there exists a $P'$ without duplicates.
Indeed, the Lagrangians labelling the additional elements $\ZZ_{\geq 0}$ of $P'$ are obtained from the factorization property Definition \ref{abstractwrappedfloersetupdef}\ref{perturbation}, which guarantees uncountably many possibilities, so we can inductively label $i\in\ZZ_{\geq 0}$ with a Lagrangian which is distinct from all Lagrangians in $P\sqcup\{0,\ldots,i-1\}$.
It follows that $P_{\mathrm{univ}}(\SSS)$ is the filtered union of its countable cofinite sufficiently wrapped subposets.
Now apply Proposition \ref{suffwrapequiv} to each of them, and note that the construction $P\mapsto\W_P$ commutes with direct limits, as does taking cohomology.
\end{proof}

\begin{remark}\label{WSprereq}
The definition of $\W_\SSS$ involved only the $\ainf$-pre-category $\F_\SSS^{\pre}$, the envelope $H^\bullet\F_\SSS^{\pre}\subseteq H^\bullet\F_\SSS$, and the functor $H^\bullet\F_\SSS\to H^\bullet\W_\SSS$ (and the \emph{existence} of an abstract wrapped Floer setup giving rise to these three pieces of data).
\end{remark}

\begin{conjecture}\label{wrappedlocalizationconjecture}
For suitable definition of $\Fun(\F_\SSS^{\pre},\C)$, the restriction functor
\begin{equation}
\Fun(\W_\SSS,\C)\to\Fun(\F_\SSS^{\pre},\C)
\end{equation}
is fully faithful, with essential image consisting precisely of those functors which send continuation maps to isomorphisms, for any $\ainf$-category $\C$.
\end{conjecture}

We now turn to the functoriality of the $\ainf$-category $\W_\SSS$.
Suppose $\SSS$ and $\SSS'$ are abstract wrapped Floer setups, and fix:
\begin{equation}\label{setupmorphism}
\begin{tikzcd}[row sep=small,column sep=tiny]
\F_\SSS^{\pre}\ar[d,hook]&H^\bullet\F_\SSS^{\pre}\ar[r,hook]\ar[d,hook]&H^\bullet\F_\SSS\ar[r]\ar[d,hook]&H^\bullet\W_\SSS\ar[d]\\
\F_{\SSS'}^{\pre}&H^\bullet\F_{\SSS'}^{\pre}\ar[r,hook]&H^\bullet\F_{\SSS'}\ar[r]&H^\bullet\W_{\SSS'}
\end{tikzcd}
\end{equation}
meaning that $\F_\SSS^{\pre}\subseteq\F_{\SSS'}^{\pre}$ is an inclusion of semisimplicial sets (in particular, specializing on vertex sets to an inclusion of sets of Lagrangians $\sL\subseteq\sL'$) covered by identifications of $\Hom$ modules compatible with the $\ainf$ operations, inducing on cohomology a functor $H^\bullet\F_\SSS^{\pre}\to H^\bullet\F_{\SSS'}^{\pre}$ which extends to a chosen fully faithful functor on envelopes $H^\bullet\F_\SSS\hookrightarrow H^\bullet\F_{\SSS'}$, which in turn extends to a functor $H^\bullet\W_\SSS\to H^\bullet\W_{\SSS'}$ (which by Lemma \ref{HWlocalizationproperty} is unique up to unique isomorphism if it exists).
We shall call such a diagram \eqref{setupmorphism} a \emph{morphism} of abstract wrapped Floer setups.

Given a morphism of abstract wrapped Floer setups $\SSS\to\SSS'$, a decorated poset $P$ for $\SSS$ is also a decorated poset for $\SSS'$, say denoted $P'$ to distinguish it from $P$.
We have $\OO_P=\OO_{P'}$, and any morphism in $H^0\OO_P$ which becomes an isomorphism in $H^\bullet\W_\SSS$ also becomes one in $H^\bullet\W_{\SSS'}$.
Thus there is an induced functor on localizations $\W_P\to\W_{P'}$.
Taking $P$ to be sufficiently wrapped countable cofinite, we have $\W_P=\W_\SSS$, and including $P'$ into a sufficiently wrapped countable cofinite decorated poset for $\SSS'$, we obtain a functor $\W_P'\to\W_{\SSS'}$.
Combining these defines an $\ainf$-functor%
\footnote{In fact, it a `naive inclusion functor', meaning that it is the inclusion of a subset of the objects and subcomplexes of their morphisms spaces, with all higher operations, namely $F^k$ for $k\geq 2$, vanishing.}
$\W_\SSS\to\W_{\SSS'}$ lifting the cohomology level functor $H^\bullet\W_\SSS\to H^\bullet\W_{\SSS'}$.

In fact, we can do a bit better and define a canonical functor $\W_\SSS\to\W_{\SSS'}$ as that induced by taking $P=P_{\mathrm{univ}}(\SSS)$ above and including $P'$ into $P_{\mathrm{univ}}(\SSS')$.
This construction is evidently compatible with composition: given morphisms $\SSS\to\SSS'\to\SSS''$, the induced functors $\W_\SSS\to\W_{\SSS'}\to\W_{\SSS''}$ compose to the functor induced by the composition of diagrams \eqref{setupmorphism} (note that the definition of localization in \cite[Definition 3.17]{gpssectorsoc} is strictly functorial).
In other words, we have defined a strict functor from the category of abstract wrapped Floer setups with morphisms \eqref{setupmorphism} to the category of small $\ainf$-categories and $\ainf$-functors.

\begin{remark}[Opposites]\label{opposites}
There is an evident notion of the \emph{opposite} of an abstract Floer setup $\SSS$, namely $\SSS^\op$ has the same set of Lagrangians, and associates to a tuple $(L_0,\ldots,L_n)$ the data that $\SSS$ associates to its reverse $(L_n,\ldots,L_0)$.
We cannot define the opposite of an abstract wrapped Floer setup: while it certainly makes sense to take the opposite of the underlying abstract Floer setup, the envelope, and the continuation maps, the definition of wrapping categories is manifestly asymmetric.
Define an \emph{abstract bi-wrapped Floer setup} to be an abstract wrapped Floer setup together with `op-wrapping categories' satisfying the opposite of Definition \ref{abstractwrappedfloersetupdef}\ref{wrappingcategory}--\ref{perturbation}; there is thus an evident notion of the opposite of an abstract bi-wrapped Floer setup.

For an abstract bi-wrapped Floer setup $\SSS$, we claim that there is a natural equivalence $\W_{\SSS^\op}=\W_\SSS^\op$ (this would follow immediately from our unproven Conjecture \ref{wrappedlocalizationconjecture}).
Let us call a decorated poset $P$ \emph{sufficiently bi-wrapped} when both $P$ and $P^\op$ are sufficiently wrapped.
Given a sufficiently bi-wrapped decorated poset $P$, we have $\W_\SSS^\op=\W_P^\op=\W_{P^\op}=\W_{\SSS^\op}$, so it suffices to show that sufficiently bi-wrapped decorated posets exist.
We can construct a sufficiently bi-wrapped decorated poset using the proof of Lemma \ref{pcanwrap}: take the poset to be $\ZZ$ with the standard order relation, color with isomorphism classes in $H^\bullet\W_\SSS$ so that the set of integers of any given color is bounded neither below nor above, and fill towards $+\infty$ with cofinal sequences in wrapping categories, and towards $-\infty$ with cofinal sequences in op-wrapping categories.
Under the assumptions of Remark \ref{pcanwrapcountable}, we also have an analogue of the statement of Lemma \ref{pcanwrap}: every countable decorated poset admits an inclusion (not downward closed) into a sufficiently bi-wrapped countable decorated poset (take $P'=\ZZ_{<0}\sqcup P\sqcup\ZZ_{>0}$ with the standard order relation on $\ZZ_{<0}\sqcup\ZZ_{>0}$ and $-i<S_i<i$, choose wrapping and op-wrapping sequences to fill $\ZZ_{>0}$ and $\ZZ_{<0}$, and perturb by induction on $|i|$).
\end{remark}

\subsection{Partially wrapped Fukaya categories}\label{wrapdefsec}

We now review the definition of (partially) wrapped Fukaya categories, reworking \cite[\S 3]{gpssectorsoc} and generalizing it to the partially wrapped setting.
The previous subsection introduced an axiomatic structure called an `abstract wrapped Floer setup' from which an $\ainf$-category called the wrapped Fukaya category is defined.
The goal of this subsection is to construct an abstract wrapped Floer setup from a stopped Liouville sector by counting pseudo-holomorphic curves.
We thus specify precisely what sort of Floer data to use, we prove transversality and compactness of moduli spaces, and we define continuation maps and wrapping categories.

We begin by explaining the construction of the abstract wrapped Floer setup $\SSS(X,\f)$ for a single stopped Liouville sector $(X,\f)$, where $X$ is equipped with a fixed choice of projection $\pi_X:\Nbd^Z\partial X\to\CC_{\Re\geq 0}$ as in \cite[Definition 2.26]{gpssectorsoc}.

\begin{definition}\label{geomtoabstract}
Let $X$ be a Liouville sector.
An abstract Floer setup $\SSS(X)$ is defined as follows.

An element of the set $\sL$ will be an exact cylindrical Lagrangian equipped with grading/orientation data \cite[\S 3.2]{gpssectorsoc}.
A tuple $(L_0,\ldots,L_k)$ will be called composable iff it is mutually transverse (equivalently, $L_i\pitchfork L_j$ for $i\ne j$ and all triple intersections are empty).
The graded module $CF^\bullet(L,K)$, for composable (i.e.\ transverse) pairs $(L,K)$, is free on the set of intersection points $L\cap K$ (generated by certain orientation lines, whose definition involves the grading/orientation data on $L$ and $K$, see \cite[\S 3.2]{gpssectorsoc}).

We now discuss Floer data and holomorphic curves.
Floer data for holomorphic disks with boundary conditions $(L_0,\ldots,L_k)$ is defined as follows.
Let $\Sbar_{k,1}$ denote the universal curve over the moduli space of $\Rbar_{k,1}$ stable disks with $k+1$ boundary marked points ($k$ `inputs' and $1$ `output'; by convention $\Sbar_{1,1}=[0,1]$).
We consider strip-like coordinates
\begin{align}
\label{coordsI}\xi_{L_0,\ldots,L_k;j}^+:[0,\infty)\times[0,1]\times\Rbar_{k,1}&\to\Sbar_{k,1}\quad j=1,\ldots,k\\
\label{coordsII}\xi_{L_0,\ldots,L_k}^-:(-\infty,0]\times[0,1]\times\Rbar_{k,1}&\to\Sbar_{k,1}
\end{align}
and families of cylindrical almost complex structures
\begin{equation}\label{Jfamilyforainfty}
J_{L_0,\ldots,L_k}:\Sbar_{k,1}\to\J(X)
\end{equation}
which make the projection $\pi_X:\Nbd^Z\partial X\to\CC_{\Re\geq 0}$ holomorphic (we note that cylindricity for families means that every point of $\Sbar_{k,1}$ has a neighborhood such that there is a compact $K\subseteq X$ outside which $J$ in that neighborhood is cylindrical).
In fact, we consider these not just for $(L_0,\ldots,L_k)$ but for all its subsequences $0\leq i_0<\cdots<i_\ell\leq k$.
These data \eqref{coordsI}--\eqref{Jfamilyforainfty} must then be compatible with gluing in the sense of \cite[\S 3.2]{gpssectorsoc}.
Elements of $D(L_0,\ldots,L_k)$ are those Floer data for which these moduli spaces are cut out transversally.
Such Floer data may be constructed by induction on $(L_0,\ldots,L_k)$ by a standard argument \cite[Lemma 3.18]{gpssectorsoc}.

The operations $\mu^k$ are defined by counting holomorphic disks with respect to chosen Floer data.
For this to make sense, we must know that the moduli spaces are compact.
The energy of a given pseudo-holomorphic disk is determined by its boundary conditions and puncture asymptotics by Stokes' theorem (the energy identity \cite[(3.38)]{gpssectorsoc}).
Compactness of the moduli spaces then follows by a monotonicity argument \cite[Proposition 3.19]{gpssectorsoc}.
To see that the complex $(CF^\bullet(L,K),\mu^1)$ is cofibrant, we note that filtering it by action expresses it as a finite iterated extension of free modules with zero differential.

The remaining Floer data, $D'$, $D''$, and $D'''_i$, is defined analogously.
This defines the abstract Floer setup $\SSS(X)$.

For a stop $\f\subseteq(\partial_\infty X)^\circ$, we let $\SSS(X,\f)$ be the restriction of $\SSS(X)$ to Lagrangians which are disjoint from $\f$ at infinity.
\end{definition}

\begin{remark}\label{striplikebottombiholospecial}
It is sometimes necessary (compare \cite[\S 5.2]{gpssectorsoc}) to restrict consideration to strip-like coordinates for which \eqref{coordsII} extends (necessarily uniquely) to a fiberwise biholomorphism
\begin{equation}
\xi_{L_0,\ldots,L_k}^-:\RR\times[0,1]\times\Rbar_{k,1}\to\Sbar_{k,1}
\end{equation}
(or rather, on each fiber it should be a biholomorphism onto the irreducible component containing the negative puncture).
Such strip-like coordinates can also be constructed by induction, so considering only such Floer data is still an abstract Floer setup.
\end{remark}

To upgrade this abstract Floer setup to an abstract wrapped Floer setup, the first step is to establish the following isotopy invariance structure of Floer cohomology.

\begin{lemma-definition}\label{isotopyinvariancedef}
Given exact cylindrical isotopies $L_t$ and $K_t$ such that $L_t$ and $K_t$ are disjoint at infinity for every $t\in[0,1]$ and both $hF^\bullet(L_0,K_0)$ and $hF^\bullet(L_1,K_1)$ are defined (meaning $L_0\pitchfork K_0$ and $L_1\pitchfork K_1$), there is an induced isomorphism $hF^\bullet(L_0,K_0)=hF^\bullet(L_1,K_1)$, compatible with concatenation of isotopies.
These isomorphisms are also compatible with Floer composition, in the sense that if $L_t$, $K_t$, $M_t$ are isotopies disjoint at infinity, then the following commutes
\begin{equation}
\begin{tikzcd}
hF^\bullet(L_0,K_0)\otimes hF^\bullet(K_0,M_0)\ar[r]\ar[d,equal]&hF^\bullet(L_0,M_0)\ar[d,equal]\\
hF^\bullet(L_1,K_1)\otimes hF^\bullet(K_1,M_1)\ar[r]&hF^\bullet(L_1,M_1)
\end{tikzcd}
\end{equation}
provided all objects appearing in it are defined.
\end{lemma-definition}

\begin{proof}
In the case the isotopies are compactly supported, the desired isomorphism is induced by counting holomorphic strips with moving Lagrangian boundary conditions, generalizing the counts from Floer data $D'(L,K)$.
Compactness of these moduli spaces follows from the same argument since the boundary conditions are fixed at infinity (the energy bound follows from the energy identity for moving Lagrangian boundary conditions \cite[(3.42)]{gpssectorsoc}).
For general isotopies $L_t$ and $K_t$ (disjoint at infinity), we choose an exact cylindrical symplectic isotopy $\Phi_t:X\to X$ ($\Phi_0$ is the identity, and $\Phi_t$ is the identity near $\partial X$ for sectors) with $\Phi_tL_t$ and $\Phi_tK_t$ fixed at infinity (the set of such isotopies is contractible), and we compose $hF^\bullet(L_0,K_0)=hF^\bullet(\Phi_1L_1,\Phi_1K_1)=hF^\bullet(L_1,K_1)$.
Compatibility with concatenation is immediate.
Compatibility with Floer composition holds by considering appropriate moduli spaces of disks with moving Lagrangian boundary conditions as in $D'''_i(L,K,M)$.
\end{proof}

Isotopy invariance allows us to define $hF^\bullet(L,K)$ for \emph{all} pairs $L$ and $K$ by perturbation.
Namely, we set $hF^\bullet(L,K):=hF^\bullet(L^+,K)$ for $L^+$ a small perturbation of $L$ which is positive at infinity and transverse to $K$.
Any two such perturbations $L^+$ are related by a small isotopy which is disjoint from $K$ at infinity and is unique up to contractible choice (note that this is true even when $L$ and $K$ are not disjoint at infinity).
Thus Lemma-Definition \ref{isotopyinvariancedef} provides canonical isomorphisms between $hF^\bullet(L^+,K)$ for different choices of $L^+$, thus making $hF^\bullet(L,K)$ well-defined.
Floer multiplication is defined on these $hF^\bullet$ objects for all triples $(L_0,L_1,L_2)$ and is associative for all quadruples $(L_0,L_1,L_2,L_3)$.
Passing to cohomology defines an `envelope' $H^\bullet\F_\SSS$ for $H^\bullet\F_\SSS^{\pre}$ in the sense of Definition \ref{abstractwrappedfloersetupdef}\ref{envelope}.

\begin{lemma-definition}\label{continuationdef}
Suppose every algebra $HF^\bullet(L,L)$ is unital and every module $HF^\bullet(L,K)$ is unital over both $HF^\bullet(L,L)$ and $HF^\bullet(K,K)$.
Then to each isotopy $L_t$ positive at infinity, there is an associated `continuation map' $c(L_t)\in HF^0(L_1,L_0)$, with the following properties:
\begin{enumerate}
\item The continuation map associated to a concatenation of isotopies is the composition of continuation maps associated to the isotopies.
\item If $L_t$ is disjoint at infinity from $K$ for every $t$, then multiplication by the continuation map
\begin{align}
HF^\bullet(K,L_1)&\to HF^\bullet(K,L_0)\\
HF^\bullet(L_0,K)&\to HF^\bullet(L_1,K)
\end{align}
agrees with the isotopy invariance isomorphisms.
\end{enumerate}
\end{lemma-definition}

\begin{proof}
The continuation map $c(L_t)\in HF^0(L_1,L_0)$ is defined by composing the units in $HF^0(L_t,L_t)$ in a very fine subdivision of the isotopy (note that the unit in $HF^\bullet(L,L)$ is necessarily homogeneous of degree zero; proof: units are unique, and the degree zero part of any unit is also a unit).
Unitality means this is well defined and compatible with composition.
Module unitality gives the second property.
\end{proof}

We emphasize that unitality of the algebras $HF^\bullet(L,L)$ and the modules $HF^\bullet(L,K)$ is a \emph{property}.
This property holds for Liouville sectors by an argument involving moving Lagrangian boundary conditions which move positively at infinity \cite[Proposition 3.23]{gpssectorsoc}.
Thus Lemma-Definition \ref{continuationdef} fixes continuation maps in the sense of Definition \ref{abstractwrappedfloersetupdef}\ref{continuation}.

We fix wrapping categories in the sense of Definition \ref{abstractwrappedfloersetupdef}\ref{wrappingcategory} to be those defined in \S\ref{wrappingcatsec}, namely $R_L=(L\leadsto-)_{X,\f}^+$.
The factorization property Definition \ref{abstractwrappedfloersetupdef}\ref{perturbation} is an immediate consequence of general position.
Let us now verify that wrapped Floer cohomology $HW^\bullet(L,K):=\varinjlim_{L^w\in R_L}HF^\bullet(L^w,K)$ satisfies the right locality property.

\begin{lemma}\label{wraprightlocal}
Multiplying by a continuation map $K\to K'$ is an isomorphism $HW^\bullet(L,K)\to HW^\bullet(L,K')$.
\end{lemma}

\begin{proof}
When the isotopy $K\leadsto K'$ is sufficiently small, Lemma \ref{lem:pushoffcore} guarantees that there exist cofinal wrappings $L^w$ of $L$ which are disjoint from the sweepout of $K\leadsto K'$, and for such $L^w$, multiplication with the continuation map is an isomorphism $HF^\bullet(L^w,K')\to HF^\bullet(L^w,K)$ by definition.
Now break an arbitrary isotopy $K\leadsto K'$ into such sufficiently small isotopies.
\end{proof}

This completes the definition of the abstract wrapped Floer setup $\SSS(X,\f)$.

\begin{remark}\label{isotopyisomorphism}
Any isotopy of exact cylindrical Lagrangians (disjoint from $\f$ at infinity) induces an isomorphism in $H^\bullet\W(X,\f)$.
Indeed, any isotopy which is positive at infinity or negative at infinity has an associated continuation map, which is an isomorphism in $H^\bullet\W(X,\f)$, and an arbitrary isotopy admits a `zig-zag' perturbation which is a concatenation of isotopies which are positive at infinity or negative at infinity.
\end{remark}

\begin{definition}
$\W(X,\f):=\W_{\SSS(X,f)}$ is the wrapped Fukaya category (Definition \ref{Sgeneralwrapped}) associated to the abstract wrapped Floer setup $\SSS(X,\f)$.
\end{definition}

\begin{remark}
This definition of the wrapped Fukaya category applies in somewhat more generality than is stated above.
In particular, we can take $X$ to be a symplectic manifold with $\omega|_{\pi_2(X)}=0$ having a positive symplectization end, and we can take elements of $\sL$ to be pairs $(L,J)$ where $L$ is a cylindrical Lagrangian and $J$ is a cylindrical almost complex structure for which there exist no $J$-holomorphic disks $(D^2,\partial D^2)\to(X,L)$.
The almost complex structures \eqref{Jfamilyforainfty} are then required to coincide with $J$ on the boundary component colored by $(L,J)\in\sL$.
\end{remark}

We now turn to functoriality with respect to inclusions of stopped Liouville sectors (which requires a slight modification to the abstract wrapped Floer setup $\SSS(X,\f)$ defined above).
By `inclusion of stopped Liouville sectors' $(X,\f)\hookrightarrow(X',\f')$, we mean that $\f'\cap(\partial_\infty X)^\circ\subseteq\f$ and that projections $\pi_X$ and $\pi_{X'}$ are fixed so that either $X\cap\partial X'=\varnothing$ or $X=X'$ and $\pi_X=\pi_{X'}$ (compare \cite[Convention 3.1]{gpssectorsoc}).

We redefine $\SSS(X,\f)$ as follows to make it strictly functorial under inclusions of stopped Liouville sectors (i.e.\ $(X,\f)\to(X',\f')$ induces $\SSS(X,\f)\to\SSS(X',\f')$ in the sense of \eqref{setupmorphism}).
We declare a Lagrangian for $(X,\f)$ to be a pair consisting of a Liouville subsector $X_0\subseteq X$ together with a Lagrangian in $X_0$ disjoint from $\f$ at infinity.
A tuple $((X_0,L_0),\ldots,(X_k,L_k))$ is called composable when $L_0,\ldots,L_k$ are mutually transverse.
The morphism complex $CF^\bullet((X_0,L_0),(X_1,L_1))$ vanishes unless $X_0\supseteq X_1$ is an inclusion of Liouville sectors (in the above sense), in which case it is $CF^\bullet(L_0,L_1)$.
The family of almost complex structures \eqref{Jfamilyforainfty} associated to such a composable tuple has target $\J(X_0)$ (\emph{not} $\J(X)$).
In the inductive construction of almost complex structures, at the inductive step for a tuple $((X_0,L_0),\ldots,(X_k,L_k))$, we note that the almost complex structures chosen for proper subtuples necessarily patch together to define a family of almost complex structures on a subset of $X_0\times\Sbar_{k,1}$, and that this patched family makes the projection $\pi_{X_0}$ holomorphic near $\partial X_0$ \emph{over its domain of definition} (note that this would fail if the family of almost complex structures \eqref{Jfamilyforainfty} had target $\J(X)$ instead of $\J(X_0)$).

The cohomology category $H^\bullet\F_\SSS^{\pre}$ assigns to a composable (i.e.\ transverse) pair $((X_0,L_0),(X_1,L_1))$ the group $HF^\bullet(L_0,L_1)$ if $X_0\supseteq X_1$ and zero otherwise.
We define the envelope $H^\bullet\F_\SSS$ to make the same assignment for all (not necessarily transverse) pairs $((X_0,L_0),(X_1,L_1))$, where $HF^\bullet$ is defined as above by positively perturbing the first argument.
Continuation maps in $\Hom((X_0,L_0),(X_1,L_1))$ are those induced by positive isotopies $L_1\leadsto L_0$ inside $X_0$.
We define the wrapping category of a given pair $(X_0,L_0)$ to be the wrapping category of $L_0$ in $(X,\f)$, paired with the subsector $X\subseteq X$ (that is, $X_0$ is ignored).
Right locality is simply Lemma \ref{wraprightlocal}, and the factorization property is also immediate.
It is also evident that $H^\bullet\W(X,\f)$ from this abstract Floer setup is the same as defined previously.
As for $\W(X,\f)$ itself, simply note that any decorated poset for the previously defined abstract wrapped Floer setup is one for the presently defined one, by pairing everything with the subsector $X\subseteq X$; since $H^\bullet\W(X,\f)$ is the same, this preserves being sufficiently wrapped.
Finally, we should note that $\SSS(X,\f)$ is strictly functorial in inclusions $(X,\f)\hookrightarrow(X',\f')$, giving the desired strict functor $(X,\f)\mapsto\W(X,\f)$.

\begin{remark}\label{openliouvillesectors}
The above construction of functorial wrapped Fukaya categories applies immediately to the more general context (which we will not need in this paper) of \emph{stopped open Liouville sectors} $(X,\partial_\infty X,\f)$ (where $(X,\partial_\infty X)$ is an open Liouville sector in the sense of \cite[Remark 2.8]{gpssectorsoc} and $\f\subseteq\partial_\infty X$ is a closed subset).
\end{remark}

\subsection{Dissipative Floer data}\label{noncyl}

In the previous subsection \S\ref{wrapdefsec}, we imposed a cylindricity assumption on Floer-theoretic objects (symplectic manifolds, Lagrangian submanifolds, and almost complex structures).
This setup suffices for the majority of this paper.
In this subsection, we generalize the abstract Floer setups constructed in \S\ref{wrapdefsec} by replacing the cylindricity condition with the weaker condition of \emph{dissipativity}, a notion due to Groman \cite{groman}.
This more general setup will be used during the proof of the K\"unneth embedding (Theorem \ref{kunneth}) and in proving invariance of the partially wrapped Fukaya category under contact isotopies of the complement of the stop (Theorem \ref{winvariancestrong}).
Note that while we show dissipative Floer data form abstract Floer setups, we do not show they form abstract \emph{wrapped} Floer setups (that is, we do not address the existence of appropriate wrapping sequences for dissipative Lagrangians).
We do, however, construct abstract wrapped Floer setups from cylindrical Lagrangians and dissipative almost complex structures (Lemma \ref{usedissipativeJ}).

Dissipativity is a property of symplectic manifolds, Lagrangian submanifolds, and almost complex structures which ensures that the proof of compactness based on monotonicity \cite[Proposition 3.19]{gpssectorsoc} goes through.
Its definition makes no reference to any cylindrical or other structure at infinity.
A crucial fact is that the family of all dissipative almost complex structures is either empty or contractible in the relevant sense.

We will consider only symplectic manifolds without boundary (so the discussion here applies to Liouville manifolds, but not to other Liouville sectors).

\begin{definition}[Dissipative symplectic manifolds and almost complex structures]\label{Xdissipative}
Let $(X,\omega)$ be a symplectic manifold with a family of compatible almost complex structures $J$ parameterized by a Riemann surface $S$.

For $K\subseteq U\subseteq X$ ($K$ compact, $U$ open and pre-compact), we define the quantity
\begin{equation}
\hbar(\omega,J,K,U)
\end{equation}
as follows.
We consider connected properly embedded pseudo-holomorphic curves in $S\times(U\setminus K)$ (not necessarily sections of the projection to $S$) which approach both $K$ and $\partial U$.
We define $\hbar(\omega,J,K,U)$ to be the minimum of $1$ and the infimal energy of such a pseudo-holomorphic curve.
Of course, this energy depends not only on $\omega$ but also on a choice of positive symplectic form on $S$; this choice is omitted from the notation since it does not affect the final result, provided that it is fixed.

The family $J$ will be called dissipative at $p\in S$ iff there exists a set of disjoint `shells' $U_i\setminus K_i$ ($K_0\subseteq U_0\subseteq K_1\subseteq\cdots\subseteq X$ with $K_i$ compact, $U_i$ open, and $X=\bigcup_iU_i$) such that
\begin{equation}
\sum_i\hbar(\omega,J|_{N_\varepsilon(p)},K_i,U_i)=\infty
\end{equation}
for some cofinal collection of neighborhoods $N_\varepsilon(p)\subseteq S$ of $p$.
Dissipativity evidently depends only on the germ of $J$ near infinity on $X$, and is independent of the choice of positive symplectic form on $S$ near $p$.

The symplectic manifold $(X,\omega)$ is called dissipative iff it admits an almost complex structure which, as a constant family over any Riemann surface, is dissipative.%
\footnote{One could imagine a weaker definition allowing domain dependent almost complex structures, but we won't need it.}
\end{definition}

\begin{lemma}
Any family of cylindrical almost complex structures on a Liouville manifold is dissipative.
More generally, so is any product of families of cylindrical almost complex structures on a finite product of Liouville manifolds.
\end{lemma}

\begin{proof}
We briefly summarize the proof from \cite{gpssectorsoc}.
A Liouville manifold with a cylindrical almost complex structure has uniformly bounded geometry \cite[Definition 2.42 and Lemma 2.43]{gpssectorsoc}, and bounds on geometry pass to finite products; the family version is \cite[Lemma 2.44]{gpssectorsoc}.
We inductively take $K_i$ to be a non-empty compact set containing $U_{i-1}$, and we take $U_i$ to be the open $1$-neighborhood of $K_i$.
Then by monotonicity \cite[Proposition 4.3.1]{sikorav}, we have a lower bound $\hbar(\omega,J|_{D^2_\varepsilon},K_i,U_i)\geq c\min(\varepsilon^2,1)$, where the constant $c>0$ is independent of $i$ because of uniformly bounded geometry.
\end{proof}

\begin{definition}[Dissipative Lagrangians and almost complex structures]\label{Ldissipative}
Let $(X,\omega)$ be a symplectic manifold with Lagrangian submanifold $L$ and a family of compatible almost complex structures $J$ parameterized by a Riemann surface with boundary $S$.
More generally, we could also label the components of $\partial S$ with Lagrangian submanifolds.

For $K\subseteq U\subseteq X$ ($K$ compact, $U$ open and pre-compact), we define the quantity
\begin{equation}
\hbar(\omega,J,K,U)
\end{equation}
as follows.
We consider connected properly embedded pseudo-holomorphic curves in $S\times(U\setminus K)$ with boundary along $\partial S\times L$ which approach both $K$ and $\partial U$.
We define $\hbar(\omega,J,K,U)$ to be the minimum of $1$ and the infimal energy of such a pseudo-holomorphic curve.

The family $J$ will be called dissipative at $p\in S$ iff there exists a set of disjoint `shells' $U_i\setminus K_i$ ($K_0\subseteq U_0\subseteq K_1\subseteq\cdots\subseteq X$ with $K_i$ compact, $U_i$ open, and $X=\bigcup_iU_i$) such that
\begin{equation}
\sum_i\hbar(\omega,J|_{N_\varepsilon(p)},K_i,U_i)=\infty
\end{equation}
for some cofinal collection of neighborhoods $N_\varepsilon(p)\subseteq S$ of $p$.
Dissipativity evidently depends only on the germ of $J$ and $L$ near infinity.

The pair $(X,L)$ is called dissipative iff it admits an almost complex structure which, as a constant family over any Riemann surface with boundary, is dissipative.
\end{definition}

\begin{lemma}
Any family of cylindrical almost complex structures on a Liouville manifold is dissipative for cylindrical Lagrangians (and the same for finite products).
\end{lemma}

\begin{proof}
This follows from monotonicity \cite[Proposition 4.7.2]{sikorav} and uniformly bounded geometry \cite[Definition 2.42 and Lemma 2.44]{gpssectorsoc} as above.
\end{proof}

\begin{definition}[Dissipative Lagrangian pairs and almost complex structures]\label{LLdissipative}
Let $X$ be symplectic and $L,L'\subseteq X$ Lagrangian.
A family of compatible almost complex structures $J$ parameterized by $[0,1]$ is called dissipative for $(L,L')$ iff the resulting $\RR$-invariant family on $\RR\times[0,1]$, with boundary components labelled by $L$ and $L'$, respectively, is dissipative, and moreover for every $E<\infty$ there exists $N<\infty$ and compact $K\subseteq X$ such that for any $J$-holomorphic strip $I\times[0,1]$ of energy $\leq E$ with $I$ of length $\geq N$, there is a point in $I\times[0,1]$ which is mapped to $K$.
This notion evidently depends only on the germ of $J$ and $L,L'$ near infinity.

The pair $(L,L')$ is called dissipative iff it admits a dissipative family over $[0,1]$.
Note that if $(L,L')$ is dissipative, then $L\cap L'$ is compact.
\end{definition}

\begin{lemma}\label{pairdissipativemetric}
A family of compatible almost complex structures $J$ parameterized by $[0,1]$ is dissipative for $(L,L')$ if the resulting $\RR$-invariant family on $\RR\times[0,1]$, with boundary components labelled by $L$ and $L'$, respectively, is dissipative, and moreover the distance between $L$ and $L'$ with respect to the metric $\min_t(g_{\omega,J_t})$ is bounded away from zero near infinity.
\end{lemma}

\begin{proof}
Let $K$ be any compact set containing $L\cap L'$ in its interior.
For any $J$-holomorphic map $u$ on $I\times[0,1]$ for an interval $I$ of length $N$, if $u$ maps no point of $I\times[0,1]$ into $K$, then we have
\begin{equation}\label{stripenergybound}
E(u)\geq\int_I\int_{[0,1]}\left|\partial_tu\right|^2\geq\int_I\biggl(\int_{[0,1]}\left|\partial_tu\right|\biggr)^2\geq\const\cdot N,
\end{equation}
where the final lower bound comes from the hypothesis on the distance between $L$ and $L'$ near infinity.
\end{proof}

\begin{lemma}
Any family of cylindrical almost complex structures on a Liouville manifold is dissipative for pairs of cylindrical Lagrangians with compact intersection (and the same for finite products).
\end{lemma}

\begin{proof}
We verify the criterion in Lemma \ref{pairdissipativemetric}.
Distance scales exponentially under the Liouville flow, so as long as cylindrical Lagrangians $L$ and $K$ are disjoint near infinity, their distance is uniformly bounded below.
For products $L_1\times\cdots\times L_n$ and $K_1\times\cdots\times K_n$, the intersection being compact implies each pair $L_i\cap K_i$ has compact intersection, so the previous argument applies, except if possibly $L_i\cap K_i=\varnothing$ for some $i$.
In this latter case, the distance between $L_i$ and $K_i$ is bounded below by some $\epsilon>0$, and this passes to the product.
\end{proof}

Let us now generalize Definition \ref{geomtoabstract}, replacing cylindricity by dissipativity.

\begin{definition}\label{disspativeSSS}
Given a dissipative exact symplectic manifold $(X,\lambda)$, we define an abstract Floer setup $\SSS^\diss(X,\lambda)$ as follows.
We consider the set $\sL$ of exact dissipative Lagrangians.
A tuple $(L_0,\ldots,L_k)$ will be called composable iff it is mutually transverse and every pair is dissipative.
We define Floer data as in Definition \ref{geomtoabstract}, except that instead of cylindricality, we impose the following dissipativity conditions on the families of almost complex structures \eqref{Jfamilyforainfty}:
\begin{enumerate}
\item $J$ should be dissipative at every interior point of a fiber of $\Sbar_{k,1}\to\Rbar_{k,1}$.
\item $J$ should be dissipative at every boundary point of a fiber of $\Sbar_{k,1}\to\Rbar_{k,1}$ (with respect to the Lagrangian labelling that boundary component).
\item $J$ over $\Sbar_{1,1}=[0,1]$ should be dissipative for the pair of Lagrangians labelling the boundary components.
\end{enumerate}
Moreover, in the first two conditions, dissipation should hold uniformly in nearby fibers, in the following sense.
For every $p\in\Sbar_{k,1}$ there should exist a sequence of shells in $X$ such that there exist arbitrarily small neighborhoods $U\subseteq\Sbar_{k,1}$ of $p$ together with neighborhoods $V\subseteq\Rbar_{k,1}$ of the image of $p$ under the map $\pi:\Sbar_{k,1}\to\Rbar_{k,1}$, such that for every compact $K\subseteq X$, there exists a compact $K'\subseteq X$ such that for every $q\in V$, the sum of $\hbar$ for the restriction of $J$ to $\pi^{-1}(q)\cap U$ over the shells outside $K$ but inside $K'$ is $\geq 1$.
\end{definition}

Note that $\SSS^\diss(X,\lambda)$ is functorial under all exact symplectomorphisms $\phi:(X,\lambda)\to(Y,\lambda')$ (meaning $\phi$ is a diffeomorphism and $\phi^*\lambda_Y=\lambda_X+df$ for some smooth function $f$, not necessarily compactly supported).
In practice, we will not consider all dissipative Lagrangians, rather only certain subclasses of interest.

To prove that $\SSS^\diss(X,\lambda)$ is an abstract Floer setup, we should prove that almost complex structures can be constructed by induction and that the moduli spaces they define are compact.
This is done in the next two lemmas.
(That dissipative almost complex structures can be taken to make moduli spaces transverse is immediate from the usual perturbation arguments, since the perturbation takes place over a compact set, while dissipativity is a property near infinity.)

\begin{lemma}\label{dissipativeJinduct}
Dissipative almost complex structures in the sense of Definition \ref{disspativeSSS} can be constructed by induction.
\end{lemma}

\begin{proof}
Begin with the case $k=1$.
Dissipativity of the pair of Lagrangians is precisely the assertion of the existence of a dissipative family over $\Sbar_{1,1}=[0,1]$.

Now consider $k\geq 2$.
By the induction hypothesis, we have already a family defined over a neighborhood of $\partial\Rbar_{k,1}$ by the collaring condition, and is dissipative there.
For the same reason, it is defined in near the punctures on $\Sbar_{k,1}$.
To extend to the interior of $\Sbar_{k,1}$, we cover the remainder of $\Sbar_{k,1}$ by finitely many open sets $Q_a$, and over each we choose a dissipative family $J_a$ (either for $X$ if at an interior point or for $(X,L)$ if at a boundary point labelled by $L$), which exists since $X$ is dissipative as is each $L\subseteq X$.
For each such open set, choose shells making the relevant series of $\hbar$ diverge.
Also choose such shells over the locus where $J$ is already fixed.
We can then delete some of these shells to make the shells from different open subsets of $\Sbar_{k,1}$ disjoint, while maintaining the divergence property.
Requiring $J$ to coincide with $J_a$ over $Q_a$ times the chosen shells ensures $J$ is dissipative.
We can extend $J$ everywhere else arbitrarily using contractibility of the space of compatible almost complex structures.
\end{proof}

\begin{remark}
The dissipation conditions on $J$ in Definition \ref{disspativeSSS} differ in an interesting way from those used in \cite[Definition 4.5]{gpssectorsoc} to construct symplectic cohomology of Liouville sectors.
There, the data of an open covering of the parameter space together with a choice of shells exhibiting dissipation (`dissipation data') was recorded, and was required to be compatible across boundary strata.
The reason the inclusion of such data was necessary ultimately came down to the fact that the moduli spaces of domains in question did not have canonical collars, so agreement of dissipation data was necessary to ensure that the glued families over a neighborhood of the boundary remained dissipative and thus ensure that the analogue of Lemma \ref{dissipativeJinduct} would hold.
In the present situation, we have strip-like coordinates \eqref{coordsI}--\eqref{coordsII} which give canonical collars on $\Rbar_{k,1}$, and hence we can use a much simpler version of dissipativity, namely a local property of the family.
\end{remark}

\begin{lemma}
Dissipative almost complex structures in the sense of Definition \ref{disspativeSSS} make the moduli spaces of disks of energy $\leq E$ compact.
\end{lemma}

\begin{proof}
The argument from \cite[Proposition 3.19]{gpssectorsoc} applies without change.
We outline the argument anyway so as to emphasize that it applies under the present hypothesis (dissipativity) despite the fact that the statement of \cite[Proposition 3.19]{gpssectorsoc} assumes cylindricity.
It suffices to produce an \emph{a priori} $C^0$-estimate, i.e.\ to show that there is a compact subset of $X$ which contains all of the holomorphic curves in question (then the usual Gromov compactness arguments apply).

Metrize the fibers of $\Sbar_{k,1}\to\Rbar_{k,1}$ so that in the `thin parts' the metric is the product metric $I\times[0,1]$ ($I\subseteq\RR$ an interval) in the chosen strip-like coordinates.
Dissipativity in the sense of Definition \ref{LLdissipative} implies there exist $N<\infty$ and compact $K\subseteq X$ such that for any finite strip $I\times[0,1]$ in the thin parts, with $I$ of length $\geq N$, there is a point in $I\times[0,1]$ which $u$ maps inside $K$.
This implies that each point of the domain (fiber of $\Sbar_{k,1}\to\Rbar_{k,1}$) is within a bounded distance of a point which is mapped into $K$ by $u$.
It thus suffices to show that if $u(p)\in K$, then $u(B_\varepsilon(p))\in K'$ for some compact $K'$ depending on $K$ and some absolute $\varepsilon>0$ (depending just on the family of almost complex structures).
For this, we use dissipativity in the sense of Definitions \ref{Xdissipative} and \ref{Ldissipative} near $p$: divergence of the sum of $\hbar$ implies there are finitely many shells outside $K$ whose sum of $\hbar$ exceeds the energy bound $E$, and hence $u(B_\varepsilon(p))$ cannot cross all of them.
\end{proof}

Having shown that $\SSS^\diss(X,\lambda)$ is an abstract Floer setup, we now discuss isotopy invariance of Floer cohomology in the dissipative context.

\begin{lemma-definition}\label{isotopyinvariancedefdissipative}
The Floer cohomology groups arising from $\SSS^\diss(X,\lambda)$ have the same isotopy invariance structure as stated in Lemma-Definition \ref{isotopyinvariancedef} but with the cylindricity assumption dropped.
\end{lemma-definition}

\begin{proof}
All that needs to be observed is that an exact Lagrangian isotopy $L_t$ induces an exact symplectic isotopy $\Phi_t$ (defined by the property that $\Phi_tL_t=L_0$) which is unique up to contractible choice (this amounts to extension of Hamiltonians, which is obvious).
\end{proof}

Isotopy invariance allows us to define $hF^\bullet(L,K)$ for any pair $(L,K)$ which is dissipative (hence, in particular, disjoint at infinity), but not necessarily transverse (apply a compactly supported perturbation to make them transverse, and note that the resulting $hF^\bullet$ object is well defined by isotopy invariance).

We do not know how to upgrade $\SSS^\diss(X,\lambda)$ to an abstract wrapped Floer setup, since we do not know how to wrap dissipative Lagrangians.
We can, however, form an abstract wrapped Floer setup using cylindrical Lagrangians and dissipative almost complex structures as follows.
For any stopped Liouville sector, let $\SSS^{\cyl,\diss}(X,\f)$ denote the abstract Floer setup obtained from $\SSS^\diss(X)$ by restricting to cylindrical Lagrangians which are disjoint from $\f$ at infinity.
Thus $\SSS^\cyl(X,\f)\subseteq\SSS^{\cyl,\diss}(X,\f)\subseteq\SSS^\diss(X)$.
The inclusion $\SSS^\cyl(X,\f)\subseteq\SSS^{\cyl,\diss}(X,\f)$ is a bijection on sets of Lagrangians, respecting composability, so induces an isomorphism on pre-categories $h\F^{\pre}_\SSS$, by invariance of $hF^\bullet$.
The additional data enhancing an abstract Floer setup to an abstract wrapped Floer setup (Definition \ref{abstractwrappedfloersetupdef}) involves only $H^\bullet\F^{\pre}_\SSS$, so the enhancement of $\SSS^\cyl(X,\f)$ to an abstract wrapped Floer setup in \S\ref{wrapdefsec} applies equally to $\SSS^{\cyl,\diss}(X,\f)$.

\begin{lemma}\label{usedissipativeJ}
The inclusion of abstract wrapped Floer setups $\SSS^\cyl(X,\f)\to\SSS^{\cyl,\diss}(X,\f)$ induces an equivalence on wrapped Fukaya categories.
\end{lemma}

\begin{proof}
The definition of $HW^\bullet$ depends only on $H^\bullet\F^{\pre}_\SSS$ and the data of Defintion \ref{abstractwrappedfloersetupdef}.
\end{proof}

\section{Basic properties of wrapped Fukaya categories}

In this section, we prove some basic results about partially wrapped Fukaya categories.

\subsection{Equivalent presentations of the same Fukaya category}\label{variouswrappedsamesec}

\begin{lemma}\label{oppositesw}
There is a canonical equivalence $\W(X^-,\f)=\W(X,\f)^\op$.
\end{lemma}

\begin{proof}
The abstract wrapped Floer setups underlying $\W(X,\f)$ and $\W(X^-,\f)$ are `opposite' for trivial reasons (for a discussion of orientation lines, see \cite[B.7]{sheridanversality}).
In particular, they are both bi-wrapped in the sense of Remark \ref{opposites}, from which the desired equivalence follows.
\end{proof}

\begin{remark}
Additional arguments are required to show that the opposite equivalences in Lemma \ref{oppositesw} are compatible with pushforward functors.
\end{remark}

\begin{definition}\label{trivialinclusiondef}
An inclusion of Liouville sectors is said to be \emph{trivial} if it is isotopic, through inclusions of Liouville sectors, to a symplectomorphism.
\end{definition}

\begin{lemma}[{\cite[Lemma 3.41]{gpssectorsoc}}]\label{winvariancesector}
For a trivial inclusion of Liouville sectors $X\hookrightarrow X'$, the pushforward on wrapped Fukaya categories is a quasi-equivalence.
Thus a deformation of Liouville sectors $\{X_t\}_{t\in[0,1]}$ induces a natural quasi-equivalence $\W(X_0)=\W(X_1)$.
\end{lemma}

\begin{proof}
We claim that $X_{-a}\to X$ induces a quasi-equivalence on wrapped Fukaya categories for $X_{-a}:=e^{-aX_I}(X)$ where $I$ is a fixed defining function for $X$.
This implies that $\W(X_a)\to\W(X_b)$ is a quasi-equivalence for all $a<b$.
Given a small perturbation $X'$ of $X$ sitting inside $X_1$, we can also flow $X'$ in and out using the flow of $I$, to obtain $X'_a$ for $a\in\RR$, so we obtain quasi-equivalences $\W(X_a)\to\W(X_b)$ as well.
Now we can sandwich together
\begin{equation}
\W(X'_{-3})\to\W(X_{-2})\to\W(X'_{-1})\to\W(X)
\end{equation}
to see that $\W(X'_{-3})\to\W(X'_{-1})$ and $\W(X_{-2})\to\W(X)$ being quasi-equivalences implies all functors above are quasi-equivalences (the `two-out-of-six property').

It thus suffices to prove the claim that $\W(X_{-a})\to\W(X)$ is a quasi-equivalence.
This is shown in \cite[Lemma 3.33]{gpssectorsoc} using a careful analysis of the Reeb dynamics.
Here is a more formal argument.
Consider the map $X\to X_{-a}$ defined by flowing along $X_I$ for time $-a$ (extend $I$ to a globally defined linear Hamiltonian on $X$ by cutting off and extending by zero).
We argue that the induced map $\W(X)\to\W(X_{-a})$ is inverse to the pushforward map $\W(X_{-a})\to\W(X)$ (say, for simplicity, at the level of cohomology categories).
Restricting to Lagrangians which are disjoint from a fixed cylindrical neighborhood of $\partial X$ containing the support of $I$, and to elements of $HW^\bullet$ in the image of $HF^\bullet$ of such Lagrangians, both functors are the `identity', hence are certainly inverses to each other.
It thus suffices to show that such Lagrangians and such morphisms cover all of $\W(X)$ and $\W(X_{-a})$.
In other words, we want to know that for any Liouville sector $X$, we can realize any morphism in $\W(X)$ inside $HF^\bullet(L,K)$ for $L,K$ disjoint from a fixed small cylindrical neighborhood of $\partial X$.
To do this, we simply isotope any pair $L,K$ using a defining function (different from the above!) and appeal to isotopy invariance from Lemma-Definition \ref{isotopyinvariancedef} and unitality in Lemma-Definition \ref{continuationdef} (compare Remark \ref{isotopyisomorphism}).
\end{proof}

\begin{lemma}\label{decreasingintersectionfaithful}
If $\f_0\supseteq\f_1\supseteq\cdots$ is a decreasing family of closed subsets of $(\partial_\infty X)^\circ$, then the natural map
\begin{equation} \label{eq: limitstop} 
\varinjlim_iHW^\bullet(L,K)_{X,\f_i}\to HW^\bullet(L,K)_{X,\bigcap_{i=0}^\infty\f_i}
\end{equation}
is an isomorphism.
\end{lemma}

\begin{proof}
The left hand side is computed by the wrapping category $\bigcup_{i=1}^\infty(L\to-)_{\f_i}^+$ since direct limits commute with direct limits, and this is the same as the wrapping category $(L\to-)_{\bigcap_{i=0}^\infty\f_i}^+$.
\end{proof}

\begin{lemma}\label{winvariancemarked}
Let $X$ be a Liouville sector.
A deformation of codimension zero submanifolds-with-corners $\f_t\subseteq(\partial_\infty X)^\circ$ with every $\partial\f_t$ convex
(meaning there is a contact vector field $V_t$ defined in a neighborhood of $\partial\f_t$ which at every point of $\partial\f_t$ is strictly outward pointing with respect to every face)
induces a natural quasi-equivalence $\W(X,\f_0)=\W(X,\f_1)$.
\end{lemma}

\begin{proof}
By the same sandwiching argument, it is enough to show that the map $\W(X,\f)\to\W(X,e^V\f)$ is a quasi-equivalence, where $V$ is any contact vector field inward pointing along $\partial\f$.
When $\partial\f$ is smooth (i.e.\ no corners), this follows from the analysis of the cutoff Reeb vector field from \cite[Lemma 3.33]{gpssectorsoc} cited in the proof of Lemma \ref{winvariancesector}, or from the alternative formal argument in that proof.
For general $\f$ with corners, we may write $\f$ as a decreasing intersection of smoothings $\f^1\supseteq\f^2\supseteq\cdots$ each with convex boundary and appeal to the smooth boundary case and Lemma \ref{decreasingintersectionfaithful}.
\end{proof}

\begin{lemma}\label{faithfulstopping}
Let $X\hookrightarrow X'$ be an inclusion of Liouville sectors and let $\f\subseteq(\partial_\infty X)^\circ$ and $\f'\subseteq(\partial_\infty X')^\circ$ be closed subsets with $\f'\cap\partial_\infty X=\f\cup\partial(\partial_\infty X)$.
Then $\W(X,\f)\hookrightarrow\W(X',\f')$ is fully faithful.
\end{lemma}

\begin{proof}
The hypothesis on $\f$ and $\f'$ implies that wrapping inside $\partial_\infty X\setminus\f$ is the same as wrapping inside $\partial_\infty X'\setminus \f'$.  Holomorphic disks in $X'$ are forced to lie in $X$ by the usual maximum principle argument \cite[Lemma 3.20]{gpssectorsoc}.
\end{proof}

\begin{lemma}\label{horizequivalence}
Let $X\hookrightarrow\bar X$ be the inclusion of a Liouville sector $X$ into its convex completion $\bar X$.
The map $\W(X)\xrightarrow\sim\W(\bar X,\partial_\infty\bar X\setminus(\partial_\infty X)^\circ)$ is a quasi-equivalence.
\end{lemma}

\begin{proof}
The map is fully faithful by Lemma \ref{faithfulstopping}.
To show essential surjectivity, we just need to show that every Lagrangian in $\bar X$ with boundary contained in $(\partial_\infty X)^\circ$ is isotopic to a Lagrangian in $X$.
This statement is invariant under deformation of the pair $(X,\bar X)$, as can be seen from the ``pushing by $X_I$'' argument from the proof of Lemma \ref{winvariancesector}.

By \cite[Lemma 2.32]{gpssectorsoc}, we may deform to the situation to
$X=\bar X\setminus(\RR_{s\geq 0}\times\RR_{\left|t\right|\leq 1}\times F_0)^\circ \subseteq \bar X$.
In local coordinates $\RR_{s\geq 0}\times\RR_{\left|t\right|\leq 1}\times F_0$, the Liouville form is given by $e^s(dt-\lambda)$, which is preserved by the vector field $-\frac\partial{\partial s}+t\frac\partial{\partial t}+Z_\lambda$.
Extending this vector field arbitrarily to a Hamiltonian vector field linear at infinity on $\bar X$, we see that it pushes any Lagrangian in $\bar X$ with boundary disjoint from $\RR_{\left|t\right|\leq 1}\times F_0$ into $X$.
\end{proof}

\begin{corollary}\label{horizsmallstop}
Let $X\hookrightarrow\bar X\hookleftarrow F\times\CC_{\Re\geq 0}$ be the inclusion of a Liouville sector into its convex completion and the complement.
The natural functors
\begin{equation}
\W(X)\xrightarrow\sim\W(\bar X,\partial_\infty(F\times\CC_{\Re\geq 0}))\xrightarrow\sim\W(\bar X,F_0)\xrightarrow\sim\W(\bar X,\cc_F)
\end{equation}
are quasi-equivalences, where $\cc_F\subseteq F=F\times\{\infty\}\subseteq F\times\partial_\infty\CC_{\Re\geq 0}\subseteq\partial_\infty(F\times\CC_{\Re\geq 0})\subseteq\partial_\infty\bar X$.
\end{corollary}

\begin{proof}
The first equivalence is just Lemma \ref{horizequivalence}.
The second two equivalences follow from Lemmas \ref{decreasingintersectionfaithful} and \ref{winvariancemarked} since $\partial_\infty(F\times\CC_{\Re\geq 0})$ can be deformed down to either $F_0$ or $\f$ through codimension zero submanifolds with convex boundary by \cite[Lemma 2.18]{gpssectorsoc}.
\end{proof}

\begin{remark}
Theorem \ref{winvariancestrong} also asserts equivalences of Fukaya categories under deformation of stops.
In contrast to the above, it concerns a situation where there is no obvious functor between the categories in question.
The functor ends up being a certain pushforward, but by a non-cylindrical symplectomorphism.
This requires the use of dissipative Floer data from \S\ref{noncyl}, from which Theorem \ref{winvariancestrong} is in fact an easy corollary.
\end{remark}

\subsection{Invariance from constancy of contact complement} \label{sec: ccc}

Although it is not obvious from its statement, the proof of Theorem \ref{winvariancestrong} will involve the consideration of non-cylindrical Floer data.
Indeed, the equivalence it asserts will simply be the pushforward under a particular non-cylindrical symplectomorphism.

\begin{proof}[Proof of Theorem \ref{winvariancestrong}]
By passing to the convex completion and appealing to Lemma \ref{horizequivalence}, it suffices to treat the case when $X$ is a Liouville manifold.

The hypothesis on $\{\f_t\}_{t\in[0,1]}$ is equivalent to the existence of contact vector fields $V_t$ on $\partial_\infty X\setminus\f_t$, varying smoothly in $t$, such that the flow of $V_t$ defines a contactomorphism $\partial_\infty X\setminus\f_0\to\partial_\infty X\setminus\f_1$.
Using this contact isotopy at infinity to define isotopies of exact cylindrical Lagrangians, we obtain at least an identification of the objects of $\W(X,\f_0)$ with the objects of $\W(X,\f_1)$.

We lift this contact isotopy at infinity to a Hamiltonian symplectomorphism $\Phi:X\to X$ as follows.
Fix a sequence of smooth contact vector fields $V^1_t,V^2_t,\ldots$ on $\partial_\infty X$ which agree with $V_t$ away from smaller and smaller neighborhoods of $\bigcup_{t\in[0,1]}\{t\}\times\f_t$.
Let $\Phi^1_t:X\to X$ $(t\in[0,1]$) be any Hamiltonian isotopy which agrees with $V^1_t$ at infinity.
Iteratively define $\Phi^i_t:X\to X$ by modifying $\Phi^{i-1}_t$ to agree with $V_t^i$ near infinity.
By performing these successive modifications in smaller and smaller neighborhoods of $\bigcup_{t\in[0,1]}\{t\}\times\f_t$ (further and further out at infinity), we ensure that $\Phi_t^i$ converges to a Hamiltonian isotopy $\Phi_t:X\to X$ as $i\to\infty$, in the sense that $\Phi_t^i$ eventually agrees with $\Phi_t$ over the complement of any fixed neighborhood of $\bigcup_{t\in[0,1]}\{t\}\times\f_t$.
We define $\Phi:=\Phi_1:X\to X$ to be the resulting time one flow map.

Now consider defining the wrapped Fukaya categories $\W(X,\f_i)$ using not the abstract wrapped Floer setups $\SSS(X,\f_i)=\SSS^\cyl(X,\f_i)$ from \S\ref{wrapdefsec} based on cylindrical Lagrangians and cylindrical almost complex structures, but rather the larger abstract wrapped Floer setups $\SSS^{\cyl,\diss}(X,\f_i)$ from \S\ref{noncyl} consisting of cylindrical Lagrangians and dissipative almost complex structures.
The resulting wrapped Fukaya categories are the same (Lemma \ref{usedissipativeJ}).

We now claim that $\Phi$ induces an isomorphism of abstract wrapped Floer setups $\SSS^{\cyl,\diss}(X,\f_0)=\SSS^{\cyl,\diss}(X,\f_1)$.
Because of the asymptotic cylindricity property of $\Phi$, it follows that if $L$ is cylindrical and disjoint from $\f_0$ at infinity, then $\Phi(L)$ is cylindrical and disjoint from $\f_1$ at infinity, and conversely.
Certainly $\Phi$ preserves mutual transversality.
Since we use dissipative almost complex structures, it follows that $\Phi$ preserves valid Floer data as well, since the notion of dissipativity depends only on the symplectic structure.
The asymptotic cylindricity property of $\Phi$ implies that it preserves the isotopy invariance isomorphisms of Floer cohomology, hence preserves units and thus continuation maps as well.
We have thus shown that $\Phi$ defines an isomorphism $\SSS^{\cyl,\diss}(X,\f_0)=\SSS^{\cyl,\diss}(X,\f_1)$.

Since $\Phi$ induces an isomorphism of abstract wrapped Floer setups $\SSS^{\cyl,\diss}(X,\f_0)=\SSS^{\cyl,\diss}(X,\f_1)$, it therefore induces an isomorphism of their associated wrapped Fukaya categories $\W(X,\f_0)=\W(X,\f_1)$.
\end{proof}

\begin{example}\label{corefamilycomplementconstant}
Let us show that if $F_t\subseteq(\partial_\infty X)^\circ$ is a deformation of Liouville hypersurfaces then the family of their cores $\f_t$ satisfies the hypothesis of Theorem \ref{winvariancestrong} (though note that in this case the conclusion of Theorem \ref{winvariancestrong} follows much more easily from Corollary \ref{horizsmallstop} and Lemma \ref{winvariancesector}; the whole point of Theorem \ref{winvariancestrong} is that verifying its hypothesis for a given family of stops $\f_t$ may be easier than producing a corresponding family of ribbons $F_t$).
To see that the complement is a locally trivial bundle of contact manifolds, write a neighborhood of $F_t$ in local coordinates $(\RR_z\times F_t,dz+\lambda_t)$, and observe that there is a contact vector field $W_t:=-z\partial_z-Z_{\lambda_t}$ which is complete in the positive direction.
There thus exists a smooth trivialization of the complement of $\bigcup_{t\in[0,1]}\{t\}\times\f_t\subseteq[0,1]\times(\partial_\infty X)^\circ$ which is standard near $\partial(\partial_\infty X)$ and which preserves $W_t$ near $\f_t$.
With respect to this trivialization, we thus have a family of contact structures $\xi_t$ on a manifold-with-boundary $M$, such that $\xi_t$ is constant near $\partial M$, and $\xi_t$ is $W$-invariant near infinity for a fixed vector field $W$ giving $M$ a complete cylindrical structure near infinity.
Recall that for any deformation of contact structures $\xi_t$ on a manifold $M$, there exists a \emph{unique} family of vector fields $V_t$ satisfying $\sL_{V_t}\xi_t=\partial_t\xi_t$ and $V_t\in\xi_t$.
If the flow of $V_t$ is complete, then this provides an isotopy $\Phi_t:M\to M$ with $\Phi_t^\ast\xi_t=\xi_0$.
In our particular example, $V_t$ vanishes near $\partial M$ since $\xi_t$ is constant there, and $V_t$ is $W$-invariant near infinity since $\xi_t$ is $W$-invariant there.
These two conditions clearly imply the flow of $V_t$ is complete, thus concluding the proof.
\end{example}

\subsection{Stop removal is localization at continuation maps}

A fundamental property of partially wrapped Fukaya categories---following directly from their definition---is their localization behavior under removing any portion of a stop.
The general version of this property is as follows:

\begin{lemma} \label{lem: general stop removal}
Let $X$ be a Liouville manifold or sector, with two stops $\f\subseteq\g\subseteq(\partial_\infty X)^\circ$.
Suppose that for every $L\subseteq X$ disjoint from $\g$, there exists a sequence of wrappings $L=L^{(0)}\leadsto L^{(1)}\leadsto\cdots$ in the complement of $\f$, with each $L^{(i)}$ disjoint from $\g$, which is cofinal in $(X,\f)$ (e.g.\ this holds by general position if $\g\setminus\f$ is mostly Legendrian).
Then the natural map $ \W(X, \g) \to \W(X, \f)$ factors through a fully faithful morphism from the quotient: 
\begin{equation} \label{cquotientinproof}
 \W(X, \g) \to \W(X, \g) /  \C_{\f, \g} \hookrightarrow \W(X, \f)
\end{equation}
where $\C_{\f, \g} \subseteq \W(X, \g)$ is the collection of cones 
of continuation maps $L^w\to L$ for all positive wrappings $L\leadsto L^w$ disjoint from $\f$, with $L$ and $L^w$ disjoint from $\g$.
\end{lemma}

\begin{proof}
For a positive isotopy $L\leadsto L^w$ disjoint from $\f$, with $L$ and $L^w$ disjoint from $\g$, the associated continuation map $L^w\to L$ is an isomorphism in $\W(X,\f)$.
Thus the functor $\W(X,\g)\to\W(X,\f)$ sends $\C_{\f,\g}$ to zero objects.
The quotient functor $\W(X,\f)\to\W(X,\f)/\C_{\f,\g}$ is thus a quasi-equivalence \cite[Lemma 3.13]{gpssectorsoc}.
The functor $\W(X,\g)\to\W(X,\f)$ localizes to a functor
\begin{equation}
\W(X,\g)/\C_{\f,\g}\to\W(X,\f)/\C_{\f,\g}\xleftarrow\sim\W(X,\f)
\end{equation}
which (up to formally inverting quasi-equivalences) defines the desired functor \eqref{cquotientinproof}.
Our task is to show this functor (which exists for arbitrary $\f\subseteq\g\subseteq(\partial_\infty X)^\circ$) is fully faithful under the stated hypotheses.

The proof is essentially same as the proof of Lemma \ref{pwrappingworks} that the localization construction yields in \S\ref{abstractfloersec} wrapped Floer cohomology.
Given any $L$ disjoint from $\g$, choose a sequence of wrappings $L=L^{(0)}\leadsto L^{(1)}\leadsto\cdots$ in the complement of $\f$ which are cofinal in $(X,\f)$, such that each $L^{(i)}$ is disjoint from $\g$ (this exists by hypothesis).
We claim that the natural map
\begin{equation}\label{hwtohw}
\varinjlim_iHW^\bullet(L^{(i)},K)_{(X,\g)}\xrightarrow\sim HW^\bullet(L,K)_{(X,\f)}
\end{equation}
is an isomorphism.
Indeed, it is surjective since $\varinjlim_iHF^\bullet(L^{(i)},K)\xrightarrow\sim HW^\bullet(L,K)_{(X,\f)}$ is an isomorphism.
The map $\varinjlim_iHF^\bullet(L^{(i)},K)\to\varinjlim_iHW^\bullet(L^{(i)},K)_{(X,\g)}$ is surjective by cofinality of the sequence $L^{(i)}$: any class in $HW^\bullet(L^{(i)},K)_{(X,\g)}$ is in the image of $HF^\bullet(L^{(i)w},K)$ for some positive wrapping $L^{(i)}\leadsto L^{(i)w}$ in the complement of $\g$, which can be composed with a wrapping $L^{(i)w}\leadsto L^{(j)}$ by cofinality, thus ensuring that this class is realized in $HF^\bullet(L^{(j)},K)$.
Thus \eqref{hwtohw} is an isomorphism.

Since multiplication by a continuation map induces an isomorphism on wrapped Floer cohomology, 
we conclude from \eqref{hwtohw} being an isomorphism that the pro-object $\cdots\to L^{(1)}\to L^{(0)}$ in $\W(X,\g)$ is left $\C$-local in the sense of \cite[Lemma 3.16]{gpssectorsoc} (let $\C=\C_{\f,\g}$).
We now consider
\begin{equation}\label{prolocalsequence}
H^\bullet(\W(X,\g)/\C)(L,K)\xrightarrow\sim\varinjlim_iH^\bullet(\W(X,\g)/\C)(L^{(i)},K)\xleftarrow\sim\varinjlim_iH^\bullet\W(X,\g)(L^{(i)},K).
\end{equation}
The first map is an isomorphism since each cone $L^{(i+1)}\to L^{(i)}$ is in $\C$ \cite[Lemma 3.12]{gpssectorsoc}, and the second map is an isomorphism since the pro-object $\cdots\to L^{(1)}\to L^{(0)}$ is left $\C$-local \cite[Lemma 3.16]{gpssectorsoc}.
Now the sequence \eqref{prolocalsequence} maps to the corresponding sequence with $\W(X,\f)$.
In this latter sequence the maps are also isomorphims, and the rightmost vertical map is obviously an isomorphism.
We conclude that $\W(X, \g) / \C \hookrightarrow \W(X, \f)$ is fully faithful.
\end{proof}

Since the cones of two maps $\alpha$ and $\beta$ together generate the cone of their composition $\alpha\beta$, it is equivalent in Lemma \ref{lem: general stop removal} to quotient by cones of positive isotopies which `generate the monoid of positive wrappings'.
In particular: 

\begin{lemma} \label{lem: transverse crossings generate} 
In the situation of Lemma \ref{lem: general stop removal}, if $\g \setminus \f$ is mostly Legendrian, we may replace $\C_{\f, \g}$ with the cones on
just the continuation maps corresponding to positive wrappings $L\leadsto L^w$ which meet $\g \setminus \f$ exactly once, at a smooth Legendrian point, and transversely.
\end{lemma}

\begin{proof}
Follows from Lemma \ref{genericpositiveisotopy}.
\end{proof}

Theorem \ref{wrapcone}, whose proof occupies essentially all of Sections \ref{lagrdisksection} and \ref{wet}, 
identifies geometric representatives for the cones of interest in the situation of Lemma \ref{lem: transverse crossings generate}.  
Proposition \ref{removecontact} is an instance where there are \emph{no} cones needed:

\begin{proof}[Proof of Proposition \ref{removecontact}]
    It suffices to show by Lemma \ref{lem: general stop removal} that cofinal wrappings in the complement of $\g$ are in fact cofinal in the complement of $\f$ as well.
To show this, simply choose a Reeb vector field which is tangent to $\g\setminus\f$ (which exists since $\g\setminus\f$ is a contact submanifold) and appeal to the cofinality criterion Lemma \ref{cofinalitycriterion}.
\end{proof}

It is not generally true that the map $\W(X, \g) /  \C_{\f, \g} \hookrightarrow \W(X, \f)$ is essentially surjective; e.g. 
if $\g = \partial_\infty X$ and $\f = \emptyset$, this will typically be false.  However: 

\begin{lemma} \label{legendrian stop removal surjective}
If $\g \setminus \f$ is mostly Legendrian, then the map $\W(X, \g) /  \C_{\f, \g} \hookrightarrow \W(X, \f)$ is essentially surjective.
\end{lemma}

\begin{proof}
It is enough to prove that every Legendrian inside $\partial_\infty X$ disjoint from $\f$ can be isotoped disjointly from $\f$ to be moreover disjoint from $\g$.
This is a special case of the first part of Lemma \ref{lem:pushoffcore} (note the differing notation).
\end{proof}

\section{Cobordism attachment and twisted complexes}\label{caatc}

The proof of Proposition \ref{cobordismtwisted} consists of two steps.
First, we show in Lemma \ref{cobordismtwistedyoneda} the desired quasi-isomorphism at the level of Yoneda modules over any finite poset of Lagrangians disjoint at infinity from the cobordism $C$.
Then, we show, using the thinness hypothesis on $C$ and a direct limit argument, that testing against such finite posets is enough to ensure quasi-isomorphism in the wrapped Fukaya category.

\begin{lemma}\label{cobordismtwistedyoneda}
Let $X$ be a Liouville sector, and let $L_1,\ldots,L_n\subseteq X$ be disjoint exact Lagrangians whose primitives vanish at infinity.
Let $C\subseteq S\partial_\infty X$ be an exact Lagrangian cobordism with negative end $\partial_\infty L_1\sqcup\cdots\sqcup\partial_\infty L_n$, such that the primitive $f_C:C\to\RR$ of $\lambda|_C$ satisfies
\begin{equation}
f_C|_{\partial_\infty L_1}<\cdots<f_C|_{\partial_\infty L_n},
\end{equation}
regarding $\partial_\infty L_i$ as the negative ends of $C$.
We denote by $L^t=\#_i^{C_t}L_i$ the result of taking $\coprod_{i=1}^nL_i$ and gluing on the $t$-translate $C_t$ of $C$ at infinity ($L^t$ is defined for sufficiently large $t\in\RR$).

Let $\A$ be a finite decorated poset (of Lagrangians and Floer data on $X$ as in \S\S\ref{abstractfloersec}--\ref{wrapdefsec}) such that the boundary at infinity of each $L\in\A$ is disjoint from the image of $C$ under the projection $S\partial_\infty X\to\partial_\infty X$; let $\A$ also denote the resulting $\ainf$-category.
Moreover, fix Floer data for the poset $\{\A>\ast\}$ (i.e.\ the poset $\A$ union one additional object $\ast$ smaller than everything in $\A$), so that assigning to $*$ one of the Lagrangians $L_1,\ldots,L_n,L^t$ defines left $\A$-modules $CF^\bullet(-,L_1),\ldots,CF^\bullet(-,L_n),CF^\bullet(-,L^t)$.
For sufficiently large $t<\infty$, there is an isomorphism of $\A$-modules
\begin{equation}\label{twistedisomorphismyoneda}
CF^\bullet(-,L^t)=[CF^\bullet(-,L_1)\to\cdots\to CF^\bullet(-,L_n)],
\end{equation}
where the right hand side denotes a twisted complex $(\bigoplus_{i=1}^nCF^\bullet(-,L_i),\sum_{i<j}D_{ij}^{C,t})$.
Moreover, using the same Floer data to define a map $CF^\bullet(-,L^t)\to CF^\bullet(-,L^{t+\varepsilon})$ by counting strips $\RR\times[0,1]$ with moving Lagrangian boundary conditions given by the obvious isotopy $L^t\leadsto L^{t+\varepsilon}$ translating $C$ near infinity, the resulting map respects the filtration induced by \eqref{twistedisomorphismyoneda} and acts as the identity on the associated graded, provided $\varepsilon>0$ is sufficiently small in terms of $t$.
\end{lemma}

Before beginning the proof of Lemma \ref{cobordismtwistedyoneda}, we clarify the meaning of $\#_i^{C_t}L_i$.
Identify a neighborhood of infinity (i.e.\ the positive end) in the symplectization $S\partial_\infty X$ with its image inside $X$ under the canonical embedding.
We translate the cobordism $C$ upwards towards infinity via the Liouville flow.
For sufficiently large $t \in \RR$, the locus where the translated-by-$t$ cobordism $C_t$ is cylindrical on its negative end will overlap with the region where the Lagrangians $L_i$ are cylindrical on their positive ends, and we may glue them using this common locus since by assumption $\partial_{-\infty}C=\partial_\infty L_1\sqcup\cdots\sqcup\partial_\infty L_n$.
We denote the result of this gluing by $\#_i^{C_t}L_i$, which is an exact Lagrangian.
Note that we could just as well attach $C$ at infinity to a single Lagrangian $L\subseteq X$ with $\partial_\infty X=\partial_{-\infty}C$, but in this case the result $\#^CL$ may not be exact.

\begin{proof}
As $t\to\infty$, the Lagrangian $L^t=\#_i^{C_t}L_i$ coincides with $\coprod_{i=1}^nL_i$ on larger and larger compact subsets of $X$.
To compare primitives, fix primitives $f_i:L_i\to\RR$ which vanish at infinity, and fix a primitive $f_C:C\to\RR$.
Let $c_i:=f_C|_{\partial_\infty L_i}\in\RR$, so we have $c_1<\cdots<c_n$.
As we translate $C$ by the Liouville flow, these constants $c_i$ scale exponentially, that is $f_{C_t}|_{\partial_\infty L_i}=e^t\cdot c_i$.
Hence we may choose primitives $f_{L^t}:L^t\to\RR$ such that
\begin{equation}\label{primitivecomparison}
f_{L^t}|_{L_i}=f_i+e^t\cdot c_i,
\end{equation}
where by assumption $c_1<\cdots<c_n$.

We now consider the Floer theory of $L_1$, \ldots, $L_n$, and $L^t$ with our fixed collection of Lagrangians $\A$.
Because all $A\in\A$ are disjoint at infinity from $C$, we have (for sufficiently large $t<\infty$) natural identifications
\begin{equation}\label{intersectionswithsurgery}
A\cap L^t = \coprod_{i=1}^n A \cap L_i
\end{equation}
and thus isomorphisms of abelian groups
\begin{equation}
CF^\bullet(A,L^t)=CF^\bullet(A,L_1)\oplus\cdots\oplus CF^\bullet(A,L_n).
\end{equation}
We now study how the $\ainf$ operations interact with this direct sum decomposition.
Consider a Fukaya $\ainf$ disk giving the $\A$-module structure of $CF^\bullet(-,L^t)$.
Such a disk has one boundary component mapping to $L^t$ and the remaining boundary components mapping to Lagrangians from our fixed collection $\A$.
In view of \eqref{intersectionswithsurgery}, the endpoints of the segment labelled with $L^t$ are mapped to intersections with some $L_i$ for $i=1,\ldots,n$.
Let us label the possible cases $(i,j)$ for $1\leq i,j\leq n$ as on the left side of Figure \ref{figurecobordismdisk}.

\begin{figure}[ht]
\centering
\includegraphics[max width=.95\textwidth]{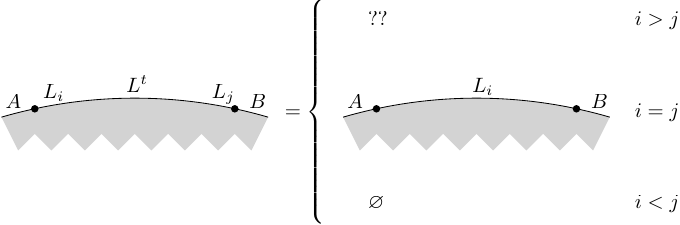}
\caption{Possibilities for a Fukaya $\ainf$ disk with boundary on $L^t$.}\label{figurecobordismdisk}
\end{figure}

The energy of such disks can be calculated as a function of $t$ using \eqref{primitivecomparison}.
In the case $i<j$, we conclude that such disks cannot exist for $t$ sufficiently large, since their energy would be negative.
In the case $i=j$, we conclude that the boundary component labelled with $L^t$ must be mapped entirely to $L_i=L_j$ for $t$ sufficiently large.
Indeed, the energy of such disks is independent of $t$, and the proof of compactness of the moduli spaces of such disks \cite[Proposition 3.19]{gpssectorsoc} based on monotonicity produces an \emph{a priori} $C^0$-estimate depending only on the energy (and on the geometry of the almost complex structures).
Note that how large $t$ must be for this argument to work depends in particular on the actions of the intersections of $L_i$ with the Lagrangians in $\A$ and between the Lagrangians in $\A$ in each other, and hence we are using finiteness of $\A$ in a crucial way here.
We draw no conclusions about the case $i>j$ (there can be many such disks).
We summarize this discussion in Figure \ref{figurecobordismdisk}.
This identification of disks immediately gives the isomorphism of $\A$-modules \eqref{twistedisomorphismyoneda}.

The same argument applies to justify the statement about maps $CF^\bullet(-,L^t)\to CF^\bullet(-,L^{t+\varepsilon})$.
The key assertion about the energy of disks follows in this case from the energy identity for moving Lagrangian boundary conditions \cite[(3.42)]{gpssectorsoc}.
The point is that the contribution of the moving Lagrangian boundary conditions can be made arbitrarily small by taking $\varepsilon>0$ sufficiently small in terms of $t$.
\end{proof}

\begin{proof}[Proof of Proposition \ref{cobordismtwisted}]
Let $\OO$ be a decorated poset of Lagrangians as in \S\S\ref{abstractfloersec}--\ref{wrapdefsec} for defining the partially wrapped Fukaya category, and fix Floer data for defining the $\OO$-modules of $L^t$ and $L_i$ as in the statement of Lemma \ref{cobordismtwistedyoneda}.
Since $C$ is thin, we may assume that all Lagrangians in $\OO$ are disjoint from $C$ at infinity.

Fix a sequence $t_1<t_2<\cdots\to\infty$ and isotopies $L^{t_r}\leadsto L^{t_{r+1}}$ (translation of $C$ as in Lemma \ref{cobordismtwistedyoneda}), thus defining a diagram of $\OO$-modules
\begin{equation}\label{Ltsequence}
CF^\bullet(-,L^{t_1})\xrightarrow{\psi_1}CF^\bullet(-,L^{t_2})\xrightarrow{\psi_2}\cdots.
\end{equation}
Let $\OO_1\subseteq\OO_2\subseteq\cdots\subseteq\OO$ be finite downward closed subposets with $\bigcup_{r=1}^\infty\OO_r=\OO$ such that there are isomorphisms of $\OO_r$-modules
\begin{equation}
CF^\bullet(-,L^{t_r})=[CF^\bullet(-,L_1)\to\cdots\to CF^\bullet(-,L_n)]
\end{equation}
and moreover that the maps $\psi_r$ restricted to $\OO_r$ respect the induced filtrations and act as the identity on the associated graded.
The existence of such a sequence $t_r$ and subposets $\OO_r$ follows from Lemma \ref{cobordismtwistedyoneda}.

The mapping telescope
\begin{equation}
\Biggl[\bigoplus_{r=1}^\infty CF^\bullet(-,L^{t_r})\xrightarrow{\bigoplus(\id_r-\psi_r)}\bigoplus_{r=1}^\infty CF^\bullet(-,L^{t_r})\Biggr]
\end{equation}
models the homotopy colimit of the sequence of $\OO$-modules \eqref{Ltsequence}.
Since the maps in \eqref{Ltsequence} are all quasi-isomorphisms, the inclusion of the first term $CF^\bullet(-,L^{t_1})$ into the mapping telescope is a quasi-isomorphism.

We modify the mapping telescope as follows.
Write $j_r:\OO_r\hookrightarrow\OO$, and note that since $\OO_r\subseteq\OO$ is downward closed, the restriction functor $j_r^\ast$ on modules has a left adjoint $(j_r)_!$ namely ``extension by zero''.
We consider now another mapping telescope
\begin{equation}\label{restrictedtelescope}
\Biggl[\bigoplus_{r=1}^\infty(j_r)_!(j_r)^\ast CF^\bullet(-,L^{t_r})\xrightarrow{\bigoplus(\id_r-\psi_r)}\bigoplus_{r=1}^\infty(j_r)_!(j_r)^\ast CF^\bullet(-,L^{t_r})\Biggr],
\end{equation}
which includes tautologically into the original mapping telescope.
Moreover, this inclusion is a quasi-isomorphism.
Indeed, this can be checked for each object of $\OO$ individually by Lemma \ref{moduleqiso}.
Each such object is in $\OO_r$ for all sufficiently large $r$, which is enough.

The above discussion thus implies that the $\OO$-module $CF^\bullet(-,L^{t_1})$ is quasi-isomorphic to the second mapping telescope \eqref{restrictedtelescope} above.
Now this second mapping telescope has, by construction, a filtration whose subquotients are
\begin{equation}
\Biggl[\bigoplus_{r=1}^\infty(j_r)_!(j_r)^\ast CF^\bullet(-,L_i)\xrightarrow{\bigoplus(\id_r-\id_{r,r+1})}\bigoplus_{r=1}^\infty(j_r)_!(j_r)^\ast CF^\bullet(-,L_i)\Biggr]
\end{equation}
for $i=1,\ldots,n$.
These subquotients are, by the same argument as above, quasi-isomorphic to the $\OO$-modules $CF^\bullet(-,L_i)$.

We have thus produced a (zig-zag of) quasi-isomorphism(s) (note Lemma \ref{moduleqiso}) of $\OO$-modules
\begin{equation}\label{extensioninO}
CF^\bullet(-,L^{t_0})=[CF^\bullet(-,L_1)\to\cdots\to CF^\bullet(-,L_n)].
\end{equation}
To conclude, we just need to localize at $I$.
Localizing on the left by produces a quasi-isomorphism of $\W$-modules.
The inclusion of $\W$ (localization of $\OO$) into the localization of $\{\OO>*\}$ (with $*$ assigned to any of the Lagrangians $L^{t_0},L_1,\ldots,L_n$) is a quasi-equivalence by Lemma \ref{suffwrapequiv}, so the localized modules are thus the Yoneda modules of $L^{t_0},L_1,\ldots,L_n$ (if this argument seems too much like a trick, we invite the reader to review Lemma \ref{representcohomology} and Proposition \ref{kunnethrepresentable}, which give a much `softer' argument for representability which is equally applicable here).
Thus the localization of \eqref{extensioninO} expresses the Yoneda module of $L^{t_0}$ as twisted complex of those of $L_1,\ldots,L_n$.
In view of the Yoneda Lemma \ref{yoneda} and the proof of Lemma \ref{generationqiso}, this produces the desired quasi-isomorphism \eqref{twistedisomorphism}.
\end{proof}

\section{Lagrangian linking disks and surgery at infinity}\label{lagrdisksection}

The purpose of this section is to introduce and perform some basic manipulations with the Lagrangian cobordisms we are interested in, namely the Lagrangian linking disks and the relatively non-exact embedded Lagrangian $1$-handle.
Our definitions of and reasoning about exact Lagrangian submanifolds of symplectizations will be given primarily in terms of the fronts of their Legendrian lifts, however we will also give descriptions in terms of Weinstein handles in some cases.

In particular, we give a careful account of the higher-dimensional version of the series of pictures in Figure \ref{figureonehandlecancelTAKETWO}.

\begin{figure}[ht]
\centering
\includegraphics[max width=.95\textwidth]{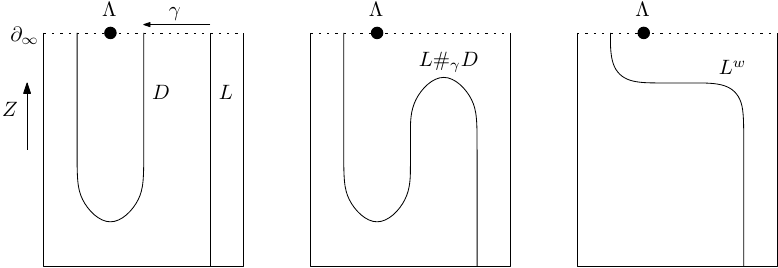}
\caption{This picture is only two-dimensional.
Left: the Lagrangian disk $D$ (an arc) linking a Legendrian submanifold $\Lambda$ (a point) at infinity.
Middle: the result $L\#_\gamma D$ of attaching a `relatively non-exact Lagrangian $1$-handle' with center the indicated Reeb chord $\gamma$ from $L$ to $D$.
Right: the result $L^w$ of positively isotoping $L$ through $\Lambda$, which is evidently isotopic to $L\#_\gamma D$ (in the present two-dimensional picture).}\label{figureonehandlecancelTAKETWO}
\end{figure}

\subsection{Front projections}\label{secfrontprojections}

\begin{figure}[ht]
\centering
\includegraphics[max width=.95\textwidth]{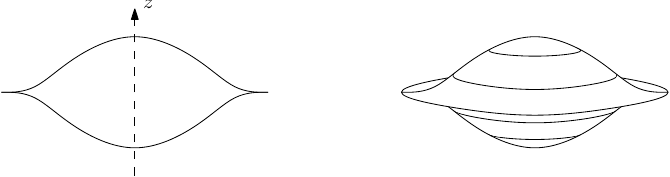}
\caption{The front of the standard Legendrian unknot in contact $\RR^3=\RR_z\times T^\ast\RR$ is illustrated on the left.  The front of the standard Legendrian unknot in contact $\RR^{2n+1}=\RR_z\times T^\ast\RR^n$ is obtained by spinning the picture on the left about the dotted vertical $z$-axis (see the `front spinning' operation in \cite[4.4]{ekholmetnyresullivan}\cite[Example 5.3]{ekholmetnyresabloff}\cite{golovkospinning}); for example, the case of $\RR^5$ is illustrated on the right.}\label{figurelegendrianunknot}
\end{figure}

A function $f:M\to\RR$ gives rise to its graph $\Gamma_f$ inside $\RR_z\times M$ as well as to the graph $\Gamma_{(f,df)}$ of $(f,df)$ inside $(\RR_z\times T^\ast M,dz-\lambda_{T^\ast M})$, which is Legendrian.
The projection $\RR_z\times T^\ast M\to\RR_z\times M$ is called the \emph{front projection}, and the image of a Legendrian is termed its \emph{front}.
For example, $\Gamma_f$ is the front of the Legendrian $\Gamma_{(f,df)}$.

Unlike the case of the graph as just mentioned, usually this front has singularities, as in Figure \ref{figurelegendrianunknot}.
The part of the front which is locally the graph of a smooth function lifts uniquely to a smooth Legendrian.  
The fronts we draw in this paper all have the property that taking the closure of the lift of the smooth locally graphical part 
recovers the original Legendrian (in fact, the only singularities which appear are cusps).

An exact Lagrangian submanifold $L\subseteq(T^\ast M,\lambda_{T^\ast M})$ lifts to a Legendrian submanifold of $(\RR_z\times T^\ast M,dz-\lambda_{T^\ast M})$, and a choice of lift is equivalent (via taking the $z$-coordinate) to a choice of function $g:L\to\RR$ satisfying $dg=\lambda_{T^\ast M}|_L$.
We may thus represent a pair $(L,g)\subseteq(T^\ast M,\lambda_{T^\ast M})$ by drawing its front in $\RR_z\times M$.

To represent exact Lagrangian submanifolds of symplectizations via a front projection, we consider the following setup.
We consider the cotangent bundle $T^\ast(\RR_s\times M)$ equipped with the Liouville vector field $Z_\cyl$ defined by adding to $Z_{T^\ast(\RR_s\times M)}$ the canonical Hamiltonian lift of the vector field $\partial_s$ on the base $\RR_s\times M$ (denote by $\lambda_\cyl$ the associated Liouville form).
In coordinates $T^\ast(\RR_s\times M)=\RR_t\times\RR_s\times T^\ast M$, we have
\begin{align}
\lambda_{T^\ast(\RR_s\times M)}&=\lambda_{T^\ast M}+t\,ds,&Z_{T^\ast(\RR_s\times M)}&=Z_{T^*M}+t\partial_t,\\
\lambda_\cyl&=\lambda_{T^\ast M}+t\,ds-dt,&Z_\cyl&=Z_{T^*M}+t\partial_t+\partial_s.
\end{align}
The Liouville flow identifies $(T^\ast(\RR_s\times M),Z_\cyl)$ with the symplectization of the contact type hypersurface $\{s=0\}=\RR_t\times T^\ast M$, over which the Liouville form $\lambda_\cyl$ restricts to the contact form
\begin{equation}\label{cylissymp}
-dt+\lambda_{T^\ast M}
\end{equation}
(which is the standard contact form on $J^1(M)=\RR_t\times T^\ast M$).

We represent exact Lagrangians $L\subseteq T^\ast(\RR_s\times M)$ via their front inside $\RR_z\times\RR_s\times M$ (note that a Lagrangian submanifold $L\subseteq T^\ast(\RR_s\times M)$ is exact with respect to $\lambda_\cyl$ iff it is exact with respect to $\lambda_{T^\ast(\RR_s\times M)}$).
If such a front is given locally by the graph of a function $g:\RR_s\times M\to\RR_z$, then $g$ is the primitive of $\lambda_{T^\ast(\RR_s\times M)}$, and $f=g-\partial_sg$ is the primitive of $\lambda_\cyl$.
In particular, the Lagrangian $L\subseteq T^\ast(\RR_s\times M)$ is cylindrical precisely where $g=e^sh+\const$ for some function $h:M\to\RR$.
A Legendrian submanifold $\Lambda\subseteq(\RR_t\times T^\ast M,-dt+\lambda_{T^\ast M})$ gives rise to an everywhere cylindrical Lagrangian in its symplectization $(T^\ast(\RR_s\times M),Z_\cyl)$.
In terms of front projections, this corresponds to taking a front inside $\RR_t\times M$ and setting $z=e^st+\const$ to obtain a front inside $\RR_z\times\RR_s\times M$ (representing the cylinder over the original Legendrian equipped with the primitive $f\equiv\const$), see Figure \ref{figurecylinderunknot} (for ease of illustration, we will use the coordinate $r=e^s$ in place of $s$ in figures).

\begin{warning}\label{bigwarning}
This method of representing exact Lagrangian cobordisms in symplectizations by drawing the front of their Legendrian lifts is not standard, and so care is needed in order to interpret the figures correctly.
The key thing to remember is that vertical translation in the $z$-coordinate \emph{does not change} the Lagrangian cobordism being represented, rather it changes the \emph{primitive} of $\lambda_\cyl$ assigned to it.
This is particularly important to remember when looking at cross-sections $s=\const$.
In particular, the cross-sections for $s$ approaching $\pm\infty$ \emph{are not} the fronts of the positive/negative Legendrian ends.
Rather, they are the result of scaling these fronts by $e^s$ and then translating them vertically by the (\emph{locally} constant) value of the primitive of $\lambda_\cyl$ on those ends.
To obtain the front of the positive/negative Legendrian ends, one instead should replace the vertical coordinate $z$, say the graph of a function $g$, with the graph of $\partial_sg$.
Of course, this derivative $\partial_sg$ cannot be seen from a cross-section $r=e^s=\const$ alone.
In particular, we emphasize that the Lagrangian cobordism represented by a front in $\RR_z\times\RR_{r>0}\times M$ generally speaking has \emph{nothing} to do with the Legendrian isotopy obtained by regarding the cross-sections $r=e^s=\const$ as a `movie' of Legendrian fronts.
\end{warning}

\begin{figure}[ht]
\centering
\includegraphics[width=.95\textwidth]{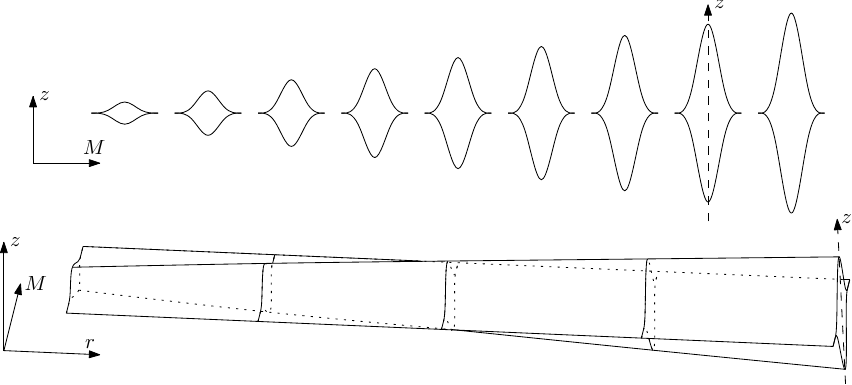}
\caption{Below: the front inside $\RR_z\times\RR_{r>0}\times M$ of the cylinder over the standard Legendrian unknot.  Above: cross-sections $r=e^s=\const$ of the same.  The picure shows $\dim M=1$; in higher dimensions we spin $M$ about the dotted vertical $z$-axis.}\label{figurecylinderunknot}
\end{figure}

It is important to know when a front  in $\RR_z\times\RR_s\times M$ corresponds to an \emph{embedded} Lagrangian (rather than just an immersed Lagrangian).
This holds whenever the Legendrian lift has no Reeb chords, or in other words whenever 
the front has no common tangencies with any of its vertical translates (other than itself).
This is always the case for a front in $\RR_z\times\RR_s\times M$ obtained as $e^s$ times a front in $\RR_t\times M$ representing an embedded Legendrian.

\subsection{Weinstein handles}\label{weinsteinhandles}

We recall the basics of Weinstein handles.
Weinstein handles were introduced by Weinstein \cite{weinstein}, motivated by the construction of Stein structures due to Eliashberg \cite{eliashbergstein} and work of Gromov--Eliashberg \cite{eliashberggromovconvexsymplectic}.
The connection between Weinstein and Stein structures was developed fully by Cieliebak--Eliashberg \cite{cieliebakeliashberg}.
Our present perspective, however, is purely symplectic.

A \emph{Weinstein $k$-handle} of dimension $2n$ ($0\leq k\leq n$) is a germ of a Liouville cobordism near a point $p$ at which $Z$ vanishes and in a neighborhood of which $Z$ is gradient-like for a function $\phi$ having a Morse critical point of index $k$ at $p$.
The \emph{standard Weinstein $k$-handle} is (the germ near zero of)
\begin{multline}
(X_k,\omega_k,Z_k):=\\
\left(\RR^{2k}\times\RR^{2(n-k)},\omega_\std,\sum_{i=1}^k\frac 12(-x_i\partial_{x_i}+3y_i\partial_{y_i})+\sum_{j=1}^{n-k}\frac 12(x_j'\partial_{x_j'}+y_j'\partial_{y_j'})\right).
\end{multline}
A general Weinstein $k$-handle need not be locally modelled on this standard Weinstein $k$-handle, however it can always be canonically deformed to it (we will not use this result logically, and we do not know a reference except in the case $k=n$ \cite[Lemma 6.6]{girouxpardon}).
The fundamentals of Morse theory (reordering of critical values, handle cancellation, handle slides, etc.) for Weinstein handles were developed by Cieliebak--Eliashberg \cite{cieliebakeliashberg}.

Coupled Weinstein handles were introduced more recently by Eliashberg--Ganatra--Lazarev \cite{eliashbergganatralazarev}.
A \emph{coupled Weinstein $(k,\ell)$-handle} ($0\leq\ell\leq k\leq n$) is a Weinstein $k$-handle together with a (germ of) Lagrangian submanifold $L$ passing through $p$ such that $Z$ is tangent to $L$ and the restriction of $\phi$ to $L$ has a Morse critical point of index $\ell$.
The \emph{standard coupled Weinstein $(k,\ell)$-handle} is given by
\begin{equation}
(X_k,\omega_k,Z_k,L_{k,\ell}):=\{y_1=\cdots=y_\ell=x_{\ell+1}=\cdots=x_k=0=y_1'=\cdots=y_{n-k}'\}).
\end{equation}
For example, $L_{n,n}$ is the core of the underlying Weinstein $n$-handle and $L_{n,0}$ is the cocore.
Again, a general coupled Weinstein handle need not be locally modelled on the standard coupled handle, but it can be canonically deformed to the standard one.
Of course, there is no reason to stop with just one Lagrangian---we could just as well consider (as we will later) multiple $Z$-invariant Lagrangians passing through $p$ on which the restriction of $\phi$ is Morse.

An \emph{exact embedded Lagrangian $k$-handle} for $0\leq k<n$ is a coupled Weinstein $(k,k)$-handle followed by a Weinstein $(k+1)$-handle, such that the pair cancel (as Weinstein handles), so the result is an exact Lagrangian cobordism inside a trivial Liouville cobordism (i.e.\ a symplectization); see Dimitroglou Rizell \cite[\S 4]{rizellambientsurgery} or Bourgeois--Sabloff--Traynor \cite{bourgeoissablofftraynor}.
We will not appeal to this notion logically, but it provides relevant context for the constructions of this section.

\subsection{Legendrian linking spheres and Lagrangian linking disks}\label{linkingdisksection}

\begin{figure}[ht]
\centering
\includegraphics[max width=.95\textwidth]{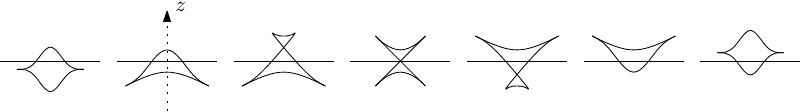}
\caption{The Legendrian linking sphere around a given point of a Legendrian submanifold (seven views, related by Legendrian isotopy).  This is the picture inside contact $3$-space, and the general case is obtained by spinning about the dotted vertical $z$-axis.}\label{figurelinkingsphere}
\end{figure}

\begin{figure}[ht]
\centering
\includegraphics[max width=.95\textwidth]{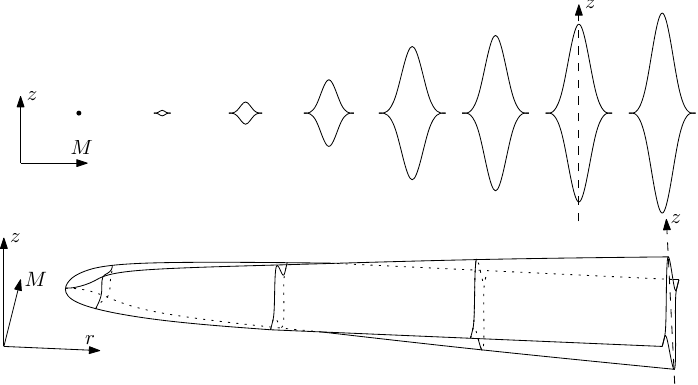}
\caption{Below: the front inside $\RR_z\times\RR_{r>0}\times M$ of the Lagrangian linking disk.  Above: cross-sections $r=\const$ of the same.  The picure shows $\dim M=1$; in higher dimensions we spin $M$ about the dotted vertical $z$-axis.  (Compare Figure \ref{figurecylinderunknot}.)}\label{figurelinkingdisk}
\end{figure}

At any point $p$ of a Legendrian submanifold $\Lambda$ there is a \emph{Legendrian linking sphere} $S_p$ (aka meridian) defined by the picture in Figure \ref{figurelinkingsphere}.
The picture takes place in a contact Darboux chart around $p$ in which $\Lambda$ is the horizontal line.
The Legendrian linking spheres $S_p$ bound \emph{Lagrangian linking disks} $D_p$ in the symplectization obtained by modifying the cylinder $\RR\times S_p$ as in Figure \ref{figurelinkingdisk}.
Note that there are two possible interpretations of Figure \ref{figurelinkingdisk}, namely either as a filling of the leftmost picture or the rightmost picture of Figure \ref{figurelinkingsphere}.
The disks defined by these two interpretations are canonically Lagrangian isotopic (this can be seen by extending the Legendrian isotopy in Figure \ref{figurelinkingsphere} in the natural way, or by comparing with the description in terms of Weinstein handles given below).

The Lagrangian linking disk $D_p$ can be alternatively described as follows in terms of coupled Weinstein handles in the sense of \S\ref{weinsteinhandles}.
Let $p\in\Lambda\subseteq Y$ be a (germ near the point $p$ of a) Legendrian submanifold of a contact manifold $Y^{2n-1}$.
Let $0\leq\ell<n$, and attach to the pair $(Y,\Lambda)$ a coupled $(\ell,\ell)$-handle and a coupled $(\ell+1,\ell+1)$-handle which cancel (as coupled handles) near $p$.
At either such handle, we can consider another Lagrangian passing through the critical point which is transverse to $\Lambda$ and on which the restriction of the Morse function has a critical point of index zero.
This gives a properly embedded Lagrangian $\RR^n$ inside the symplectization of $Y$, and this Lagrangian is our alternative definition of $D_p$.

To see that $D_p$ does not depend on $\ell$ or on whether we take it over the $(\ell,\ell)$-handle or over the $(\ell+1,\ell+1)$-handle, argue as follows.
We consider the everywhere cylindrical Lagrangian $L=S\Lambda\subseteq SY=X$.
We introduce a cancelling pair of coupled Weinstein handles of indices $(\ell,\ell)$ and $(\ell+1,\ell+1)$ on $L$.
Locally, the model for this is as follows.
The pair $(X,L)$ is locally modelled on
\begin{equation}
(T^\ast(\RR_s\times\Lambda),\RR_s\times\Lambda,Z_\cyl=Z_{T^\ast(\RR_s\times\Lambda)}+\pi^\ast\partial_s)
\end{equation}
as described in \S\ref{secfrontprojections}, where we now use $\pi^\ast$ to denote the canonical lift of vector fields on $\RR_s\times\Lambda$ to Hamiltonian vector fields on $T^\ast(\RR_s\times\Lambda)$.
If we deform $\partial_s$ by introducing a pair of cancelling Morse handles of indices $\ell$ and $\ell+1$, the resulting deformation of Hamiltonian lifts realizes the operation of introducing a cancelling pair of coupled Weinstein handles of indices $(\ell,\ell)$ and $(\ell+1,\ell+1)$.
(More precisely, for this to work we may need to work with $\varepsilon>0$ times this deformed vector field to ensure that $Z_\cyl$ is indeed a Weinstein structure.)
In this description, the Lagrangian linking disk $D_p$ is simply the cotangent fiber over one of the zeroes of the deformed vector field on $\RR_s\times\Lambda$.
When the critical points come together and cancel, these two fibers also come together, thus showing that we can define $D_p$ as the fiber over either of them.
To see that the disk $D_p$ is independent of $\ell$, consider adding cancelling Morse handles of indices $\ell,\ell+1,\ell+1,\ell+2$ to $L$ (for example, using the vector field $(x^2+t_1)\partial_t+(y^2+t_2)\partial_y$ with $t_1,t_2<0$).
These handles can be cancelled in two ways, depending on which of the $(\ell+1)$-handles cancels with the $\ell$-handle and which with the $(\ell+2)$-handle (in the example, this corresponds to raising $t_1$ to be positive or raising $t_2$ to be positive).
It follows that the cotangent fibers over each of these handles are all isotopic.
This shows that $D_p$ is independent of $\ell$.

To see that the description of the Lagrangian linking disk in terms of Weinstein handles is indeed equivalent to the picture in Figure \ref{figurelinkingdisk}, argue as follows.
We consider the local model given above with $\ell=0$.
Namely, we consider the Liouville vector field $Z_V$ on $T^\ast(\RR_s\times M)$ given by $Z_{T^\ast(\RR_s\times\Lambda)}$ plus the Hamiltonian lift of a vector field $V$ on $\RR_s\times M$ which is obtained from $\partial_s$ by a compactly supported perturbation (within the class of gradient-like vector fields) which introduces a pair of cancelling zeroes of indices $0$ and $1$, as illustrated in Figure \ref{figurebasecancel}.
We suppose in addition that near the zero of index zero, $V$ is locally smoothly conjugate to $\frac 23\sum_ix_i\partial_{x_i}$, and hence $Z_V$ is given locally by $\frac 23\sum_ix_i\partial_{x_i}+\frac 13\sum_iy_i\partial_{y_i}$.
(This local requirement on $V$ is not compatible with the need, mentioned earlier, to replace $V$ with $\varepsilon\cdot V$ for small $\varepsilon>0$ to ensure that $Z_V$ is Weinstein.  So, to be precise, we must first choose $V$, then scale it down to $\varepsilon\cdot V$, and finally perform a local modification near the index zero critical point to ensure the correct local form.)

\begin{figure}[ht]
\centering
\includegraphics[max width=.95\textwidth]{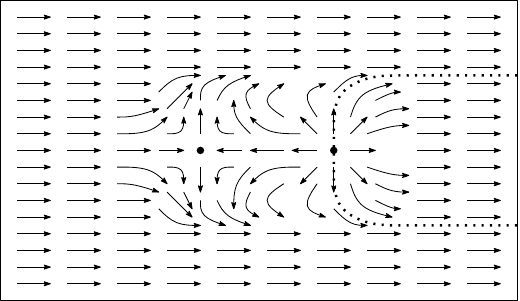}
\caption{The vector field $V$ on $\RR_s\times M$ with cancelling critical points of indices $0$ and $1$, and the hypersurface $H\subseteq\RR_s\times M$ (dotted).}\label{figurebasecancel}
\end{figure}

Consider a hypersurface $H\subseteq\RR_s\times M$ passing through the index $0$ zero of $V$, with $V$ tangent to $H$, and $H$ non-compact only in the $s=+\infty$ direction (see Figure \ref{figurebasecancel}).
The conormal $N^\ast H$ is a $Z_V$-invariant Lagrangian submanifold which intersects $L=\RR\times\Lambda$ cleanly along $H$.
The Lagrangian linking disk as described via Weinstein handles is precisely a small pushoff of $N^*H$, moving in the positive direction at infinity, so as to intersect $\RR\times\Lambda$ exactly once (namely at the zero of $V$ of index $0$).
To draw this pushoff, we first deform $N^\ast H$ to $\widetilde{N^\ast H}$ (which still intersects $L$ cleanly along $H$) as follows.
The conormal $N^\ast H$ is given in local coordinates by $\{x_1=0=y_2=\cdots=y_n\}$, and we define $\widetilde{N^\ast H}$ to be given in local coordinates by $\{x_1-y_1^2=0=y_2=\cdots=y_n\}$ extended globally by $Z_V$-invariance (note that this locus is indeed locally invariant under $Z_V$).
Now the front of $\widetilde{N^\ast H}$ is precisely the picture in Figure \ref{figurelinkingdisk}.
Now the Lagrangian linking disk as described by Weinstein handles (which intersects $L$ once transversally) is a small perturbation of $\widetilde{N^\ast H}$ (which intersects $L$ cleanly along $H$).
This perturbation may be described as simply a small positive or negative pushoff via the Reeb vector field on the local contact sphere near the zero of $Z_V$.
At infinity, the effect on $\widetilde{N^\ast H}$ is the same, simply a small positive or negative Reeb pushoff, which yields the same picture from Figure \ref{figurelinkingdisk} except now perturbed either up or down to coincide at infinity with either the far left or far right picture from Figure \ref{figurelinkingsphere}, respectively.

\subsection{Relatively non-exact embedded Lagrangian \texorpdfstring{$1$}{1}-handle}\label{onehandlesection}

\begin{figure}[ht]
\centering
\includegraphics[max width=.95\textwidth]{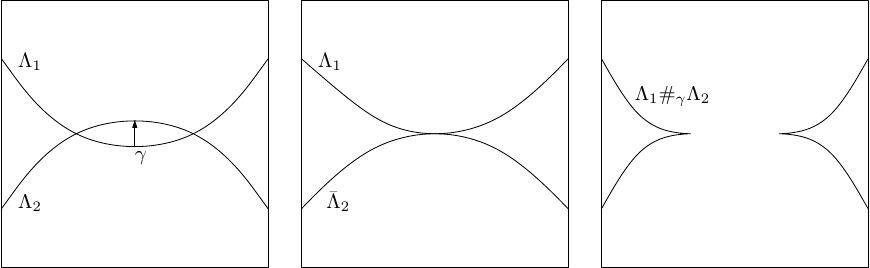}
\caption{Middle: Two Legendrians $\bar\Lambda_1$ and $\bar\Lambda_2$ with a single clean intersection.  Left: Two Legendrians $\Lambda_1=(\bar\Lambda_1)^-$ and $\Lambda_2=(\bar\Lambda_2)^+$, and the obvious short Reeb chord $\gamma$ from $\Lambda_1$ to $\Lambda_2$.  Right: The new Legendrian $\Lambda_1\#_\gamma\Lambda_2$ (topologically the connect sum of $\Lambda_1$ and $\Lambda_2$).  The relatively non-exact embedded Lagrangian $1$-handle has negative end $\Lambda_1\sqcup\Lambda_2$ and positive end $\Lambda_1\#_\gamma\Lambda_2$.  (In higher dimensions, the picture is obtained by spinning about the central vertical $z$-axis.)}\label{figurehandlebeforeafter}
\end{figure}

Fix two Legendrians $\bar\Lambda_1$ and $\bar\Lambda_2$ inside a contact manifold $Y$ and a contact Darboux chart for $Y$ with front projection as in the middle of Figure \ref{figurehandlebeforeafter}.
This chart is the standard neighborhood of a single clean intersection of $\bar\Lambda_1$ and $\bar\Lambda_2$, and we assume $\bar\Lambda_1$ and $\bar\Lambda_2$ are otherwise disjoint (an isolated intersection point $p\in\bar\Lambda_1\cap\bar\Lambda_2$ is `clean' when $T_p\bar\Lambda_1\cap T_p\bar\Lambda_2=0$).
Legendrians $\Lambda_1$ and $\Lambda_2$ are defined as small negative/positive pushoffs of $\bar\Lambda_1$ and $\bar\Lambda_2$, and are illustrated on the left of Figure \ref{figurehandlebeforeafter}.
There is an obvious short Reeb chord $\gamma$ from $\Lambda_1$ to $\Lambda_2$, which we will call the `center'.
Given such data, we define an exact Lagrangian cobordism $L\subseteq SY$ asymptotic at $s=-\infty$ to $\Lambda_1\sqcup\Lambda_2$ and asymptotic at $s=+\infty$ to the Legendrian denoted $\Lambda_1\#_\gamma\Lambda_2$ appearing on the right of Figure \ref{figurehandlebeforeafter}.
We call this $L$ a \emph{relatively non-exact embedded Lagrangian $1$-handle}; topologically, it is simply a $1$-handle.

\begin{remark}
It is tempting to begin the discussion above instead with $\Lambda_1$, $\Lambda_2$, and $\gamma$ and produce from this data the rest of Figure \ref{figurehandlebeforeafter} (as we mistakenly did in an initial version of this text), but let us explain why this is difficult.
First, we should say precisely what sort of input data $(\Lambda_1,\Lambda_2,\gamma)$ one might want to consider.
Certainly $\Lambda_1$ and $\Lambda_2$ should be disjoint Legendrian submanifolds.
The curve $\gamma$ produced above is a positive arc $\gamma:[0,1]\to Y$ from $\gamma(0)\in\Lambda_1$ to $\gamma(1)\in\Lambda_2$ (otherwise disjoint from $\Lambda_1\sqcup\Lambda_2$), and it is moreover equiped with a Lagrangian subbundle $F\subseteq\gamma^\ast\xi$ (`framing') which coincides with $T\Lambda_1$ at $0$ and is transverse to $T\Lambda_2$ at $1$ (namely, $F$ simply `is' $T\Lambda_1$, which makes sense since $\gamma$ is an arbitrarily small perturbation of an isolated clean intersection).
We are thus led to a precise mathematical question: given a pair of disjoint Legendrian submanifolds $\Lambda_1$ and $\Lambda_2$ along with a framed positive arc $\gamma$ from one to the other (in the above sense), does the ensemble $(\Lambda_1,\Lambda_2,\gamma,F)$ come, uniquely up to isotopy, from a pair of cleanly intersecting Legendrians $\bar\Lambda_1$ and $\bar\Lambda_2$ via the construction above?
The answer is surely negative.
One can indeed perform an isotopy supported in a neighborhood of $\gamma$ which brings $\Lambda_1$ forward to intersect $\Lambda_2$ cleanly, however such an isotopy will twist a lot, producing framings which are very `large' (the set of framings up to isotopy is a $\ZZ$-torsor, and large means towards positive infinity, for a certain convention on sign).
In general, if a given framing is achievable (by any isotopy), then so is any larger framing (one can always add positive twists).
There is, however, no reason to expect existence or uniqueness for arbitrary framings.
As pointed out by one of the referees, this discussion is closely related to the notion of a `contractible' Reeb chord from Ekholm--Honda--K\'{a}lm\'{a}n \cite[Definition 6.13]{ekholmhondakalman}.
\end{remark}

\begin{remark}
It is important to point out that the relatively non-exact embedded Lagrangian 1-handle $L$ is not ``local'' to the clean intersection point of $\bar\Lambda_1$ and $\bar\Lambda_2$.
That is, $L$ \emph{does not} coincide with the cylinder over $\Lambda_1\sqcup\Lambda_2$ away from the Darboux chart in Figure \ref{figurehandlebeforeafter}.
Rather, away from this chart, the cobordism $L$ is an arbitrarily small, but necessarily nontrivial, perturbation of the cylinder over $\Lambda_1\sqcup\Lambda_2$, supported away from both positive and negative infinity.
This is forced by reasons of action: the restriction of the primitive of $\lambda|_L$ to the negative end $\partial_{-\infty}L=\Lambda_1\sqcup\Lambda_2$ necessarily takes different constant values on $\Lambda_1$ and $\Lambda_2$ (see Remark \ref{nonexactforced}), and if $L$ were cylindrical outside the region illustrated in Figure \ref{figurehandlebeforeafter}, these two constants would both coincide with the constant value of the primitive on the connected positive end $\partial_\infty L=\Lambda_1\#_\gamma\Lambda_2$.
\end{remark}

\begin{figure}[ht]
\centering
\includegraphics[width=.95\textwidth]{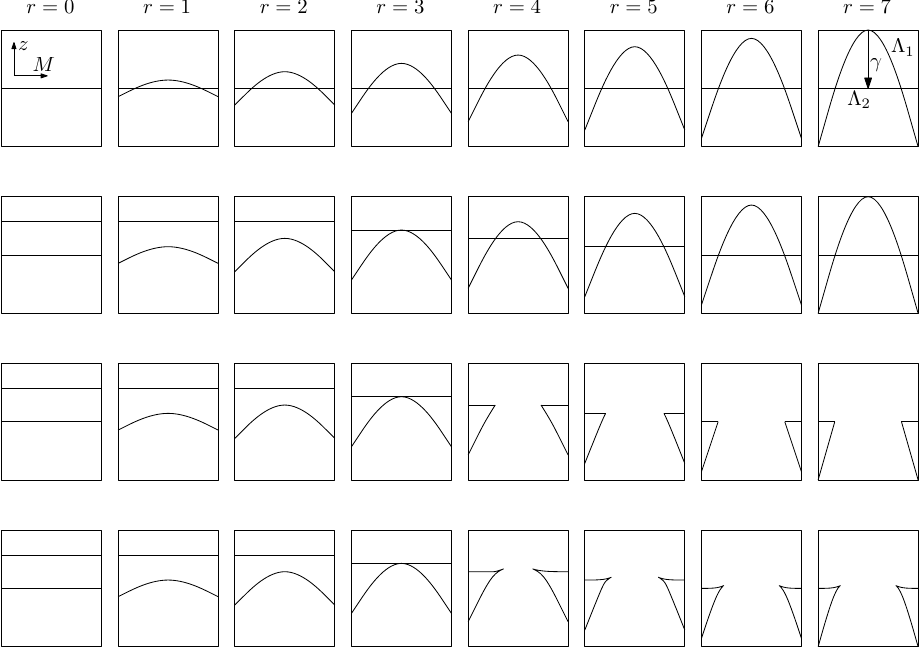}
\caption{Each row is the cross-sections $\{r=\const\}$ of a front inside $\RR_z\times\RR_{r>0}\times M$ (the first column $\{r=0\}$ illustrates the $r\to 0^+$ limit of the front).  In higher dimensions, we spin each cross section about the central vertical $z$-axis.  The top row represents the front of the cylinder over $\Lambda_1\sqcup\Lambda_2$, while the bottom row represents the front of the relatively non-exact embedded Lagrangian $1$-handle $L$.}\label{figureonehandle}
\end{figure}

\begin{figure}[ht]
\centering
\includegraphics[width=.95\textwidth]{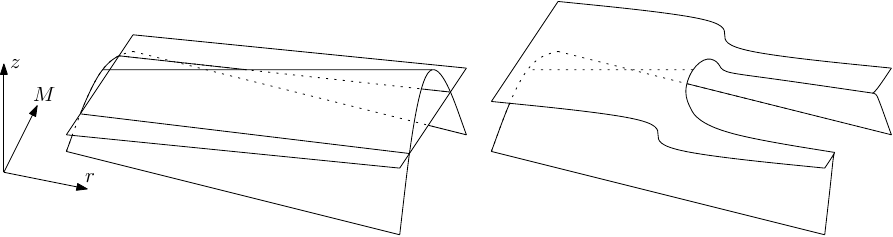}
\caption{Left: Front of the cylinder over $\Lambda_1\sqcup\Lambda_2$.  Right: Front of the relatively non-exact embedded Lagrangian $1$-handle $L$.  We emphasize that the negative Legendrian ends of the Lagrangian cobordisms represented by these two fronts are the \emph{same} (compare Warning \ref{bigwarning}); both are the Legendrian $\Lambda_1\sqcup\Lambda_2$ whose front appears on the left of Figure \ref{figurehandlebeforeafter}.}\label{figureonehandlesecond}
\end{figure}

Let us now define the relatively non-exact embedded Lagrangian $1$-handle $L$ associated to $\bar\Lambda_1,\bar\Lambda_2\subseteq Y$ as above.
We define $L$ by the illustrations in Figures \ref{figureonehandle} and \ref{figureonehandlesecond}, which we now explain in detail (the apparent asymmetry between $\Lambda_1$ and $\Lambda_2$ is only for ease of illustration).
As in \S\ref{secfrontprojections}, we are considering fronts inside $\RR_z\times\RR_{r>0}\times M$ to describe exact Lagrangians inside $T^\ast(\RR_{r>0}\times M)$ with a choice of primitive.
Recall (from the discussion in \S\ref{secfrontprojections}) that the front inside $\RR_z\times\RR_{r>0}\times M$ of an exact cylindrical Lagrangian inside $T^\ast(\RR_{r>0}\times M)$ is given by $z=e^st+\const$ (equivalently, $rt+\const$) in terms of the front inside $\RR_t\times M$ of the corresponding Legendrian inside $\RR_t\times T^*M$.

The construction of $L$ begins with the cylinder over $\Lambda_1\sqcup\Lambda_2$, illustrated in the first row of Figure \ref{figureonehandle} and on the left side of Figure \ref{figureonehandlesecond}.
The vertical coordinate is reversed in comparison with Figure \ref{figurehandlebeforeafter} due to \eqref{cylissymp} being negative of the standard contact form on $\RR_t\times T^\ast M$.
Recall that the primitive $f$ is given by $g-\partial_sg$ where $g$ is the $z$-coordinate, so the primitive vanishes on both components.

We now perturb the cylinder over $\Lambda_1\sqcup\Lambda_2$ to obtain the exact Lagrangian inside $T^\ast(\RR_{r>0}\times M)$ whose front inside $\RR_z\times\RR_{r>0}\times M$ is illustrated in the second row of Figure \ref{figureonehandle}.
Specifically, we lift (increase the vertical $z$-coordinate) of the component of the front corresponding to $\Lambda_2$ near $s=-\infty$, while keeping it fixed near $s=+\infty$.
Explicitly, we add $\varepsilon\cdot\varphi(s)$ to the $z$-coordinate, for some smooth function $\varphi:\RR\to\RR$ with $\varphi(s)=1$ near $s=-\infty$ and $\varphi(s)=0$ near $s=+\infty$, for some choice of sufficiently small $\varepsilon>0$.
Some remarks are in order about this perturbation.
First of all, while the perturbation of the \emph{front} is nontrivial near $s=-\infty$, the induced perturbation of \emph{exact Lagrangians} is trivial near $s=-\infty$ (vertical translation of the front corresponds to adding a constant to the primitive assigned to the exact Lagrangian; compare Warning \ref{bigwarning}).
Thus we have defined a perturbation of the cylinder over $\Lambda_1\sqcup\Lambda_2$ which is fixed near $s=\pm\infty$.
The second remark is that this perturbation is \emph{not} supported inside (the cone over) the local Darboux chart appearing in the illustrations (indeed, the function $\varphi$ depends only on $s$, not where we are in $M$); this is not a problem since we may draw the same picture in $\RR_z\times\RR_{r>0}\times\Lambda_2$, which describes the cone over a Weinstein neighborhood of $\Lambda_2$.
The remaining steps in the construction of $L$ will take place entirely within the illustrated local chart.

We now remove part of the exact Lagrangian whose front appears in the second row of Figure \ref{figureonehandle} to obtain the exact Lagrangian whose front appears in the third row.
As illustrated, we remove the part of the front where the component coming from $\Lambda_1$ lies `above' (i.e.\ larger $z$-coordinate) the component coming from $\Lambda_2$.
The result is an exact Lagrangian with boundary (the boundary corresponds to the `corner' of the front).
Finally, we replace the `corner' of the front with cusps to obtain the front appearing in the fourth row of Figure \ref{figureonehandle} and on the right side of Figure \ref{figureonehandlesecond}.
In terms of exact Lagrangians, the effect is to glue in a band connecting the two boundary components.
This final row defines the relatively non-exact embedded Lagrangian 1-handle $L$.
It is straightforward to check that $L$ is indeed embedded (i.e.\ that its front is nowhere tangent to any of its nontrivial vertical translations).

\begin{remark}\label{nonexactforced}
Note that the primitive $f:L\to\RR$ of $\lambda|_L$ satisfies
\begin{equation}
f|_{\{-\infty\}\times\Lambda_2}>f|_{\{-\infty\}\times\Lambda_1}
\end{equation}
(this is the origin of the term ``relatively non-exact'' used to describe $L$).
The pair $(L,f)$ cannot be deformed to satisfy $f|_{\{-\infty\}\times\Lambda_2}=f|_{\{-\infty\}\times\Lambda_1}$.
Indeed, if such a deformation were to exist, then Proposition \ref{cobordismtwisted} would yield a direct sum decomposition $L\#_\gamma K=L\oplus K$ in the wrapped Fukaya category, and there are simple counterexamples to this statement.
Note that a similar phenomenon was observed by Chantraine \cite{chantrainerelnonexact} using a construction of Lagrangian cobordisms to which ours is surely related (and Remark \ref{onehandleandpolterovich} possibly goes towards this).
\end{remark}

\begin{remark}
The relatively non-exact embedded Lagrangian $1$-handle should be contrasted with the exact embedded Lagrangian $k$-handles recalled in \S\ref{weinsteinhandles}.
The latter are exact, local, and their attachment does not alter the corresponding object of the wrapped Fukaya category by Proposition \ref{cobordismtwisted}.
\end{remark}

\begin{remark}\label{onehandleandpolterovich}
Using the pictures above, one can check (though we will not use this fact) that attaching a relatively non-exact embedded Lagrangian $1$-handle followed by attaching an exact embedded Lagrangian $(n-1)$-handle can be alternatively described as wrapping $\Lambda_1$ through $\Lambda_2$ to create a single transverse double point and then resolving that double point via Polterovich surgery.
\end{remark}

\begin{remark}\label{onehandlesmall}
The relatively non-exact embedded Lagrangian $1$-handle is `thin' in the sense of Proposition \ref{cobordismtwisted} by Lemma \ref{lem:pushoffcore} since it is contained the cone over a small neighborhood of $\bar\Lambda_1\cup\bar\Lambda_2$ which has a ribbon given by the plumbing of their cotangent bundles.
\end{remark}

\subsection{Wrapping through a Legendrian}\label{wrapishandlesection}

We now prove Proposition \ref{wrapishandle}, whose setup we briefly recall.
We consider an exact cylindrical Lagrangian $L\subseteq X$ and a Legendrian $\Lambda\subseteq\partial_\infty X$.
We consider an exact cylindrical Lagrangian isotopy $L\leadsto L^w$ which is positive at infinity (a `wrapping') and which passes through $\Lambda$ exactly once, transversally, at a point $p\in\Lambda$.
We also consider the exact cylindrical Lagrangian $L\#_\gamma D$ obtained by attaching a relatively non-exact embedded Lagrangian $1$-handle to $L\sqcup D$, where $D\subseteq X$ denotes the linking disk of $\Lambda$ at $p$ (the reader may refer to the discussion immediately following the statement of Lemma \ref{cobordismtwistedyoneda} for a precise explanation of the `attach at infinity' operation).

\begin{figure}[ht]
\centering
\includegraphics[max width=.95\textwidth]{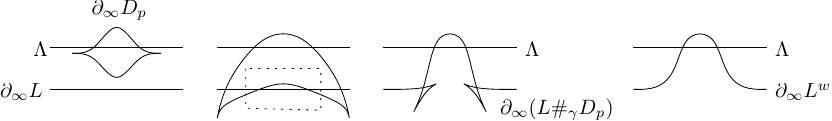}
\caption{The attaching locus of the relatively non-exact embedded Lagrangian $1$-handle (dotted box), and the proof that $\partial_\infty(L\#_\gamma D)=\partial_\infty L^w$.}\label{figureconnectsumwithdisk}
\end{figure}

\begin{figure}[ht]
\centering
\includegraphics[max width=.95\textwidth]{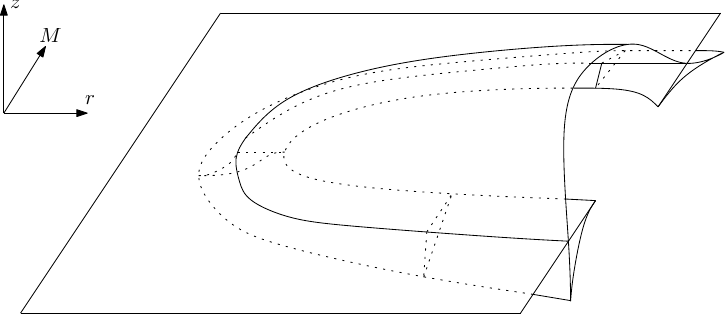}
\caption{The result of attaching a relatively non-exact Lagrangian $1$-handle with a Lagrangian linking disk (with inverted vertical coordinate; compare with the sign of \eqref{cylissymp}).}\label{figurezeroonehandles}
\end{figure}

\begin{proof}[Proof of Proposition \ref{wrapishandle}]
We are to show that $L^w$ and $L\#_\gamma D$ are isotopic as exact cylindrical Lagrangians.
It is enough to consider a single local model.

Figure \ref{figureconnectsumwithdisk} illustrates where the relatively non-exact embedded Lagrangian $1$-handle is attached, and proves the desired statement at the level of the contact boundary, namely $\partial_\infty L^w=\partial_\infty(L\#_\gamma D_p)$.
The leftmost diagram in Figure \ref{figureconnectsumwithdisk} shows the stop $\Lambda$, the boundary $\partial_\infty L$ of our Lagrangian, and the boundary $\partial_\infty D_p$ of the linking disk.
The next diagram in Figure \ref{figureconnectsumwithdisk} (related by Legendrian isotopy to the one on its left) shows the contact Darboux chart along which we attach a relatively non-exact Lagrangian $1$-handle (compare with the left of Figure \ref{figurehandlebeforeafter}).
The right two diagrams in Figure \ref{figureconnectsumwithdisk} (related by Legendrian isotopy) show the result after attaching the $1$-handle, which is evidently the same as is obtained by simply passing $\partial_\infty L$ through $\Lambda$ in the positive direction to obtain $\partial_\infty L^w$.

To show the full statement $L^w=L\#_\gamma D_p$, we also argue via picture.
Combining $D_p$ as illustrated on the bottom of Figure \ref{figurelinkingdisk} with the $1$-handle as illustrated on the right of Figure \ref{figureonehandlesecond}, we obtain Figure \ref{figurezeroonehandles} as an illustration of $L\#_\gamma D_p$.
To explain further: the very ``first'' (minimal $r$) cusp point in Figure \ref{figurezeroonehandles} is the beginning of the linking disk, and slicing at a slightly larger value of $r$ produces the second front in Figure \ref{figureconnectsumwithdisk}; the part above this slice (larger $r$) is the $1$-handle.
Now we apply a parameterized Legendrian Reidemeister I move to Figure \ref{figurezeroonehandles} to obtain $L^w$.
\end{proof}

\section{Wrapping exact triangle and stop removal} \label{wet}

We now unfold the cohomological consequences of the geometry of linking disks.

\begin{proof}[Proof of Proposition \ref{handlecone}]
We consider applying Proposition \ref{cobordismtwisted} to the relatively non-exact embedded Lagrangian $1$-handle attached at infinity to $L\sqcup K$ to form $L\#_\gamma K$.
The relatively non-exact embedded Lagrangian $1$-handle is thin in the required sense by Remark \ref{onehandlesmall}.
Proposition \ref{cobordismtwisted} also requires that the primitives of $\lambda$ restricted to $L$ and $K$ be constant at infinity, which we can achieve by performing a small perturbation by Lemma \ref{cylvanishinf}.
Proposition \ref{cobordismtwisted} now produces an exact triangle
\begin{equation}\label{trianglewithoutgamma}
L\to K\to L\#_\gamma K\to,
\end{equation}
and our task is to show that the map $L\to K$ can be taken to be (the morphism in the wrapped Fukaya category corresponding to) the short Reeb chord $\gamma$.

Morally speaking, the reason one expects this to be true is that in the limit as the $1$-handle is pushed to infinity, the analysis of disks in Figure \ref{figurecobordismdisk} should admit a strengthening illustrated in Figure \ref{figuresurgerydisk} in that the count of disks in the remaining case $(K,L)$ should coincide with the count of disks with the segment labelled with $L\#_\gamma K$ replaced by two segments labelled $K$ and $L$ separated by a puncture asymptotic at infinity to $\gamma$.
This is, however, purely motivation.

\begin{figure}[ht]
\centering
\includegraphics[max width=.95\textwidth]{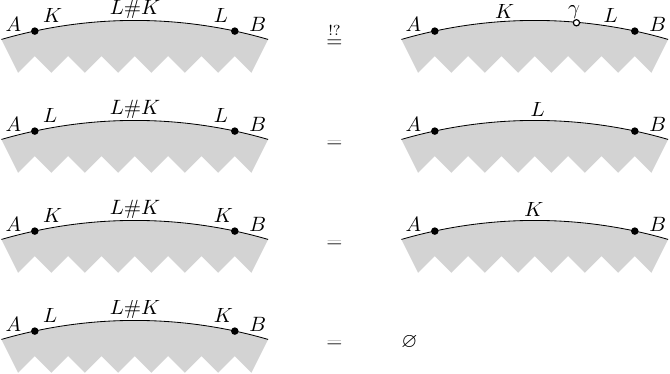}
\caption{Possibilities for a Fukaya $\ainf$ disk with boundary on $L\#_\gamma K$.}\label{figuresurgerydisk}
\end{figure}

Introduce a stop at a small negative pushoff of the plumbing of $\partial_\infty L$ and $\partial_\infty K$ (compare Remark \ref{onehandlesmall}).
Now consider testing the exact triangle \eqref{trianglewithoutgamma} against $L^w$, a small positive pushoff of $L$ intersecting $K$ exactly once, corresponding to $\gamma$.
We obtain a long exact sequence
\begin{equation}
HF^\bullet(L^w,L)\to HF^\bullet(L^w,K)\to HF^\bullet(L^w,L\#_\gamma K)\to
\end{equation}
(note that when any of $L$, $K$, or $L\#_\gamma K$ is wrapped backwards, it immediately falls into the new stop, and hence $HW^\bullet$ is $HF^\bullet$).
The connecting homomorphism $L\to K$ we are looking for is thus simply the image in $HF^\bullet(L^w,K)$ of the continuation map in $HF^\bullet(L^w,L)$.
Now $HF^\bullet(L^w,K)$ is freely generated by a single intersection point corresponding to $\gamma$.
This proves the desired result up to an unknown integer multiple.

\begin{figure}[ht]
\centering
\includegraphics[max width=.95\textwidth]{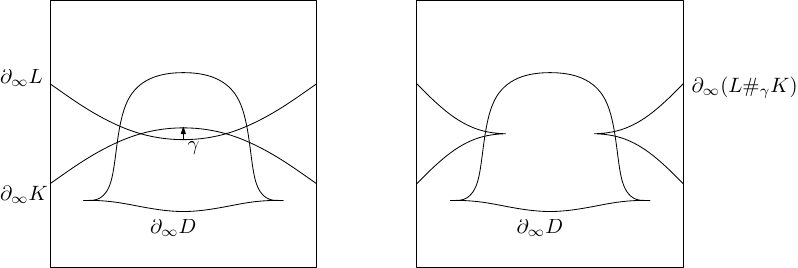}
\caption{Left: The contact boundary of $L$ and $K$, together with the Lagrangian disk $D$ linking both of them.  Right: The contact boundary of $L\#_\gamma K$, which is evidently unlinked from $D$.  (In higher dimensions, the picture is obtained by spinning about the central vertical $z$-axis.)}\label{figuretestwithsurgery}
\end{figure}

To fix the unknown integer multiple, we test $L$, $K$, and $L\#_\gamma K$ against a Lagrangian disk $D$ linking both $L$ and $K$ near $\gamma$ (see Figure \ref{figuretestwithsurgery}).
Clearly $HF^\bullet(D,L)=HF^\bullet(D,K)=\ZZ$ as both Floer complexes are generated by a single intersection point.
On the other hand, $HF^\bullet(D,L\#_\gamma K)=0$ since $\partial_\infty D$ is unlinked with $\partial_\infty(L\#_\gamma K)$ and hence $D$ can be disjoined from $L\#_\gamma K$ (keeping boundaries at infinity disjoint).
By introducing an auxiliary stop at a small positive Reeb pushoff of $\partial_\infty D$, these three $HF^\bullet$ groups are in fact $HW^\bullet$ (because $\partial_\infty D$ falls immediately into the stop, and this wrapping is cofinal by Lemma \ref{cofinalitycriterion}).
It follows that the connecting homomorphism in $HW^\bullet(L,K)$ indeed equals $\pm\gamma$ in the category with the auxiliary stop, and this implies the same in the category without the auxiliary stop simply by pushing forward.
\end{proof}

\begin{proof}[Proof of Theorem \ref{wrapcone}]
Propositions \ref{handlecone} and \ref{wrapishandle} combine to show that $D$ is quasi-isomorphic to the cone on \emph{some} morphism $a:L^w\to L$, and our goal is to show that this morphism can be taken to be the continuation map $L^w\to L$.
Introduce a stop at a small positive pushoff of $\partial_\infty L^w$, and consider testing the exact triangle (in the wrapped Fukaya category with this additional stop) against a smaller positive pushoff $L^{ww}$ of $L^w$.
Since wrapping $L^{ww}$ directly into the stop is cofinal by Lemma \ref{cofinalitycriterion}, this yields an exact triangle
\begin{equation}
HF^\bullet(L^{ww},L^w)\xrightarrow{\cdot a}HF^\bullet(L^{ww},L)\to HF^\bullet(L^{ww},D_p)\to.
\end{equation}
Since $HF^\bullet(L^{ww},D_p)=0$ (they are disjoint), we conclude that multiplication by $a$ is an isomorphism.
Now the groups $HF^\bullet(L^{ww},L^w)$ and $HF^\bullet(L^{ww},L)$ are canonically isomorphic, and we may simply write them as $HF^\bullet(L^+,L)$, the Floer cohomology of $L$ with an unspecified small positive pushoff $L^+$ thereof (recall Lemma-Definition \ref{isotopyinvariancedef} and the surrounding discussion).
This group $HF^\bullet(L^+,L)$ is an algebra with respect to Floer composition.
Since multiplication by $a\in HF^\bullet(L^+,L)$ is an isomorphism on $HF^\bullet(L^+,L)$, we conclude that $a\in HF^\bullet(L^+,L)$ is a unit.
Now the quasi-isomorphism type of the cone $[L^w\xrightarrow aL]$ is unchanged by multiplying $a$ by a unit in either $HW^\bullet(L^w,L^w)$ or $HW^\bullet(L,L)$.
Thus if $a\in HF^\bullet(L^+,L)$ is a unit, we may replace $a$ with the identity $\1_L\in HF^\bullet(L^+,L)$ which is by definition the continuation map in $HF^\bullet(L^w,L)$.
\end{proof}

\begin{proof}[Proof of Theorem \ref{stopremoval}]
Essential surjectivity is Lemma \ref{legendrian stop removal surjective}.
According to Lemma \ref{lem: transverse crossings generate}, we therefore have an equivalence 
$\W(X,\g)/ \C_{\f, \g} \to \W(X, \f)$, where  $\C_{\f, \g}$ is the collection of 
cones on continuation maps on wrappings which meet $\g \setminus \f$ exactly once,
at a smooth Legendrian point, and transversely.  Theorem \ref{wrapcone} (which we have just proven) 
shows that these cones are isomorphic to linking disks.
\end{proof}

\begin{remark}
It follows from Theorem \ref{wrapcone} that linking disks to $\g \setminus \f$ vanish in $\W(X, \f)$.
If one follows through the proof, the argument for this boils down to saying that they are cones on continuation maps $L^w\to L$ for isotopies $L\leadsto L^w$ passing once through $\g\setminus\f$, and these continuation maps are, essentially by definition, isomorphisms in $\W(X,\f)$.
One can also argue more directly:
each such disk $D$ is contained in a small neighborhood of a point of $\partial_\infty X\setminus\f$, 
and hence every Lagrangian $L$ has a cofinal sequence of wrappings which are disjoint from $D$.
Alternatively, one could argue that $D$ has a cofinal sequence of wrappings disjoint from any fixed $L$, 
or one could argue that $D$ is contained in a halfspace $\CC_{\Re\geq 0}\times\CC^{n-1}$ and appeal to Corollary \ref{zeroobject}.
\end{remark}

\section{Products and cylindrization} \label{sec: products}

The theme of this section is that products of objects which are cylindrical at infinity need not be cylindrical at infinity, yet can often be made to be so by suitable small perturbations.
We show how to do this perturbation in cases of interest to us.

\subsection{Products of Liouville sectors and stopped Liouville manifolds}\label{productofwithstops}

We begin with a general discussion of products of Liouville sectors and products of stopped Liouville manifolds.
We then compare these two product operations, arguing that the product of Liouville sectors is a special case of the product of stopped Liouville manifolds.
Note that we do \emph{not} discuss products of stopped Liouville sectors.

Given two stopped Liouville manifolds $(X,\f)$ and $(Y,\g)$, their product is defined to be
\begin{equation}\label{productofpairs}
(X,\f)\times(Y,\g):=(X\times Y,(\f\times\cc_Y)\cup(\f\times\g\times\RR)\cup(\cc_X\times\g)).
\end{equation}
To interpret the stop on $X\times Y$, recall that $\partial_\infty(X\times Y)$ is covered by $\partial_\infty X\times Y$ and $X\times\partial_\infty Y$, which overlap over $\partial_\infty X\times\partial_\infty Y\times\RR$.
Note that the product stop can change drastically as the Liouville forms on $X$ and $Y$ are deformed or as the stops $\f$ and $\g$ undergo ambient contact isotopy (or the sort of deformation appearing in Theorem \ref{winvariancestrong} or Lemmas \ref{decreasingintersectionfaithful} and \ref{winvariancemarked}).

Given two Liouville sectors $X$ and $Y$, we can consider their product $X\times Y$.
This product can fail to be a Liouville sector for two reasons.
First, it will have corners unless one or both of $X$ and $Y$ is a Liouville manifold.
Second, and rather more seriously, the product may fail to be cylindrical at infinity.
It is cylindrical at infinity provided the factors $X$ and $Y$ satisfy a technical condition: their Liouville vector fields must be everywhere tangent to their boundary.
In this case, any cylindrical smoothing of the corners of $X\times Y$ yields a Liouville sector by \cite[Remark 2.12]{gpssectorsoc}, which we denote by $(X\times Y)^\sm$.
The technical condition of the Liouville vector field being everywhere tangent to the boundary is harmless: every Liouville sector can be deformed to satisfy this condition, and in fact this deformation is unique up to contractible choice \cite[Lemma 2.11 and Proposition 2.28]{gpssectorsoc}.
Thus, up to canonical deformation, the product of any pair of Liouville sectors is a Liouville sector.

We now compare the product of Liouville sectors with the product of their (stopped) convex completions.
Recall that any Liouville sector $X$ gives rise, up to contractible choice, to an inclusion $X\hookrightarrow(\bar X,\f)$ where $\bar X$ is a Liouville manifold (termed the convex completion of $X$) and $\f\subseteq\partial_\infty X$ is a stop (this is the setup of Corollary \ref{horizsmallstop}).

\begin{proposition}  \label{prop: product sectors versus stopped}
Let $X$ and $Y$ be Liouville sectors.
The canonical inclusions $X\hookrightarrow(\bar X,\f)$ and $Y\hookrightarrow(\bar Y,\g)$ into stopped convex completions can be chosen so that their product
\begin{equation}
X\times Y\hookrightarrow(\bar X,\f)\times(\bar Y,\g)
\end{equation}
is, after smoothing the corners of the domain, also the canonical inclusion of a Liouville sector into its stopped convex completion.
\end{proposition} 

\begin{proof} 
We may choose coordinates of the form:
\begin{align}
F\times(T^\ast\RR_{\geq 0},Z_{T^\ast\RR_{\geq 0}}+\pi^\ast\varphi(s_1)\partial_{s_1})&\to\bar X\phantom{\bar Y}\qquad\phantom{Y}X=\bar X\setminus(F\times T^\ast\RR_{>1}),\\
G\times(T^\ast\RR_{\geq 0},Z_{T^\ast\RR_{\geq 0}}+\pi^\ast\varphi(s_2)\partial_{s_2})&\to\bar Y\phantom{\bar X}\qquad\phantom{X}Y=\bar Y\setminus(G\times T^\ast\RR_{>1}),
\end{align}
where $F$ and $G$ are Liouville manifolds, $\pi^\ast$ denotes the lift from vector fields on $\RR$ to Hamiltonian vector fields on $T^\ast\RR$, and $\varphi:\RR\to[0,1]$ is smooth and satisfies $\varphi(s)=0$ for $s\leq 2$ and $\varphi(s)=1$ for $s\geq 3$.
The product $\bar X\times\bar Y$ is thus equipped with a chart of the form
\begin{equation}
F\times G\times(T^\ast\RR^2_{\geq 0},Z_{T^\ast\RR^2_{\geq 0}}+\pi^\ast[\varphi(s_1)\partial_{s_1}+\varphi(s_2)\partial_{s_2}])\to\bar X\times\bar Y.
\end{equation}
Now as illustrated in Figure \ref{figuretwovectorfields}, the vector field $\varphi(s_1)\partial_{s_1}+\varphi(s_2)\partial_{s_2}$ may be deformed over a compact subset of the interior of $\RR^2_{\geq 0}$ to a vector field of the form $\varphi(s)\partial_s$ for some coordinates $(s,\theta)$ on $\RR^2_{\geq 0}\setminus\{(0,0)\}$.
The locus $s\leq 1$ in this deformation is thus a smoothing $(X\times Y)^\sm$ of the corners of $X\times Y$, and its complement can be described as
\begin{equation}\label{productfiber}
\Bigl[F\times Y\underset{F\times G\times T^\ast[0,1]}\cup X\times G\Bigr]\times(T^\ast\RR_{\geq 0},Z_{T^\ast\RR_{\geq 0}}+\pi^\ast\varphi(s)\partial_s).
\end{equation}
We will denote this deformation of $\bar X\times\bar Y$ by $\overline{(X\times Y)^\sm}$, since the above discussion shows it is the convex completion of $(X\times Y)^\sm$ (we have thus shown that convex completion commutes with product of Liouville sectors, up to canonical deformation).  

\begin{figure}[ht]
\centering
\includegraphics[max width=.95\textwidth]{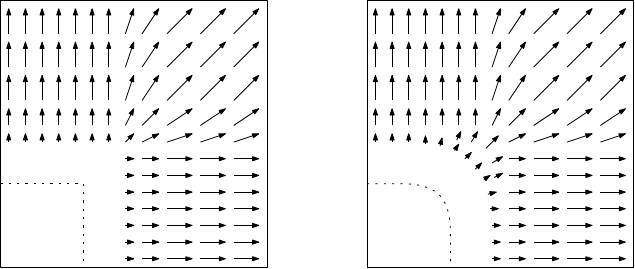}
\caption{Left: The vector field $\varphi(s_1)\partial_{s_1}+\varphi(s_2)\partial_{s_2}$ on $\RR^2$ defining the Liouville structure on $T^\ast\RR^2_{\geq 0}\times F\times G$ inside the product $\bar X\times\bar Y$ (the dotted line indicates the boundary of $X\times Y$).  Right: The deformed vector field $\varphi(s)\partial_s$ defining the Liouville structure on $T^\ast\RR^2_{\geq 0}\times F\times G$ which defines what we call $\overline{(X\times Y)^\sm}$ (the dotted line indicates a smoothing $(X\times Y)^\sm$ of $X\times Y$).  Note that the deformation is supported in a compact subset of $\RR^2_{\geq 0}$, disjoint from the boundary.}\label{figuretwovectorfields}
\end{figure}

We write $H$ for the first factor in \eqref{productfiber}; it is a Liouville manifold. 
The deformation $\bar X\times\bar Y$ to $\overline{(X\times Y)^\sm}$ is supported away from $H$, so we may
regard $H$ as living in $\partial_\infty (\bar X\times\bar Y)$ as well.  
It is evident from the construction that $\cc_H$ is the product stop.
\end{proof}

\subsection{Products of Lagrangians}\label{productcylindrical}

If $L\subseteq X$ and $K\subseteq Y$ are cylindrical,
the product Lagrangian 
$L\times K\subseteq X\times Y$ inside the product Liouville manifold $(X\times
Y,\lambda_X+\lambda_Y)$ need not be, and typically is not.

The goal of this subsection is to describe a deformation of the product $L\tildetimes K\subseteq X\times Y$ (called the \emph{cylindrization} of $L\times K$) which is cylindrical, for any pair $L$ and $K$ satisfying the (mild) assumption that both primitives $f_L$ and $f_K$ (of $\lambda_X|_L$ and $\lambda_Y|_K$, respectively) are compactly supported.
That this assumption does not result in any loss of generality is guaranteed by:

\begin{lemma}\label{cylvanishinf}
For any exact cylindrical Lagrangian $L\subseteq X$, there exists a compactly supported exact Lagrangian isotopy $L\leadsto L'$ such that $\lambda_X|_{L'}$ has a compactly supported primitive $f_{L'}$.
\end{lemma}

\begin{proof}
Fix a primitive $f_L$ of $\lambda_X|_L$.
Let $H:X\to\RR$ be a compactly supported Hamiltonian which equals $1$ over a large 
compact subset of $X$.
For any $a\in\RR$ with $\left|a\right|$ sufficiently small, applying the (forwards or backwards) Hamiltonian flow of $H$ to $L$ for small time defines a compactly supported isotopy $L\leadsto L'$ and a primitive $f_{L'}$ satisfying $f_{L'}=f_L+a$ near infinity.
By conjugating such an isotopy by the Liouville flow (pushing it towards infinity), we may in fact achieve $f_{L'}=f_L+a$ for arbitrary $a\in\RR$.
Finally, perform such an isotopy separately in each of the non-compact ends of $L$.
\end{proof}

Let $L\subseteq X$ and $K\subseteq Y$ be two exact cylindrical Lagrangians inside Liouville sectors $(X,\lambda_X)$ and $(Y,\lambda_Y)$ (whose Liouville vector fields are tangent to their respective boundaries).
Since $L$ and $K$ are exact, there exist primitives $f_L:L\to\RR$ and $f_K:K\to\RR$ satisfying $df_L=\lambda_X|_L$ and $df_K=\lambda_Y|_K$.
We say $L$ (resp.\ $K$) is \emph{strongly exact} iff $\lambda_X|_L\equiv 0$ (resp.\ $\lambda_Y|_K\equiv 0$), equivalently, iff $Z_X$ (resp.\ $Z_Y$) is everywhere tangent to $L$ (resp.\ $K$).
If both $L$ and $K$ are strongly exact, then the product Lagrangian $L\times K\subseteq X\times Y$ inside the product Liouville manifold $(X\times Y,\lambda_X+\lambda_Y)$ is also strongly exact, hence, in particular, cylindrical.

We now define the cylindrization $L\tildetimes K\subseteq X\times Y$, assuming only that $f_L$ and $f_K$ are both compactly supported.
Fix extensions $f_L:X\to\RR$ and $f_K:Y\to\RR$ which vanish near the boundary and are supported inside subdomains $X_0\subseteq X$ and $Y_0\subseteq Y$ (respectively) whose completions are $X$ and $Y$.
We now consider the Liouville form
\begin{equation}\label{deformedliouville}
\lambda_X+\lambda_Y-d(f_L\phi_K)-d(\phi_Lf_K),
\end{equation}
where $\phi_L:X\to[0,1]$ and $\phi_K:Y\to[0,1]$ are smooth functions which vanish over $X_0$ and $Y_0$ (and near $\partial X$ and $\partial Y$), are $Z$-invariant near infinity, and equal $1$ over a neighborhood of $\partial_\infty L$ and $\partial_\infty K$, respectively.
Note that the restriction of \eqref{deformedliouville} to $L\times K$ is compactly supported, and hence the associated Liouville vector field is tangent to $L\times K$ outside a compact set.

We now claim that for suitable choices of $\phi_L$ and $\phi_K$, this deformed Liouville form \eqref{deformedliouville} remains convex outside $X_0\times Y_0$.
More precisely, the associated Liouville vector field is outward pointing along $\partial(X_0\times Y_0)$ and exhibits its exterior as the positive half of a symplectization.
To see this, we calculate the Liouville vector field corresponding to \eqref{deformedliouville} to be
\begin{equation}\label{deformedliouvillevectorfield}
[Z_X+\phi_KX_{f_L}+X_{\phi_L}f_K]+[Z_Y+\phi_LX_{f_K}+X_{\phi_K}f_L].
\end{equation}
Consider now the positive flow of this vector field starting at a point of $(X\setminus X_0)\times Y$, so the $X$-component of the vector field is given by $Z_X+X_{\phi_L}f_K$ (note that $\phi_KX_{f_L}$ vanishes over this locus since $f_L$ is supported inside $X_0$).
Now we note that $X_{(e^{NZ_X})_\ast\phi_L}=e^{-N}(e^{NZ_X})_\ast X_{\phi_L}$, so replacing $\phi_L$ with $(e^{NZ_X})_\ast\phi_L$ and taking $N\to\infty$, the term $X_{\phi_L}f_K$ becomes negligible.
Thus for $N<\infty$ sufficiently large, starting at any point outside $X_0\times Y$, this flow is complete and escapes to infinity.
A symmetric argument implies applies to the case where one starts at a point outside $X \times Y_0$ (after replacing $\phi_K$ with $(e^{NZ_Y})_\ast\phi_K$ for $N$ sufficiently large).
Hence, implicitly fixing such a sufficiently large $N$ and replacing $\phi_L$ and $\phi_K$ as above, we see that outside $X_0 \times Y_0$ the flow is complete and escapes to infinity.
Moreover, the same holds at every point of the linear interpolation between the product Liouville vector field $Z_X+Z_Y$ and \eqref{deformedliouvillevectorfield}.

Since the linear deformation of Liouville forms between $\lambda_X+\lambda_Y$
and \eqref{deformedliouville} (where we have replaced $\phi_L$ and $\phi_K$ with their pushforwards under $e^{NZ_X}$ and $e^{NZ_Y}$ for a fixed large $N$) maintains convexity outside
$X_0\times Y_0\subseteq X\times Y$, it is necessarily induced by a Hamiltonian symplectomorphism
$\Phi:X\times Y\to X\times Y$, fixed over $X_0\times Y_0$, so that
$\Phi^\ast(\lambda_X+\lambda_Y)=\eqref{deformedliouville}$.
Concretely, $\Phi$ is characterized uniquely by the properties that $\Phi=\id$ over $X_0\times Y_0$ and $\Phi_\ast\eqref{deformedliouvillevectorfield}=Z_X+Z_Y$.
Since $L\times K$ is cylindrical with respect to \eqref{deformedliouville}, it follows that its image
\begin{equation}
L\tildetimes K:=\Phi(L\times K)\subseteq(X\times Y,\lambda_X+\lambda_Y)
\end{equation}
is cylindrical.
Note that by taking $\phi_L$ and $\phi_K$ to be supported in small neighborhoods of $\partial_\infty L$ and $\partial_\infty K$ (in particular, vanishing over large compact subsets of $X$ and $Y$), we can ensure that $\Phi$ is supported in an arbitrarily small neighborhood of $((L\cap X_0)\times\partial_\infty K)\cup(\partial_\infty L\times(K\cap Y_0))$ (in particular, $\Phi=\id$ and $L\tildetimes K=L\times K$ over arbitrarily large compact subsets of $X\times Y$).
Note also that the choices going into the definition of the cylindrization
$L\tildetimes K$ form a contractible space. Although we will suppress the
dependence of $L\tildetimes K$ on all of these choices, it will be important
to study the dependence of this construction on $N$ (which measures the
proximity to infinity of the cylindrization) in order to control
holomorphic disks and positivity of wrappings.

\subsection{Wrapping product Lagrangians}\label{wrappingproduct}

We will need to know that formation of the cylindrized product $L\tildetimes K$ respects positivity of isotopies near infinity.
More precisely, given two Lagrangian isotopies $\{L_t\}_{t\in[0,1]}$ and $\{K_t\}_{t\in[0,1]}$ which are positive near infinity (meaning $\partial_t\partial_\infty L_t$ and $\partial_t\partial_\infty K_t$ are positively transverse to the respective contact distributions), we would like to know that the isotopy $\{L_t\tildetimes K_t\}_{t\in[0,1]}$ is positive near infinity (at least for controlled choice of cylindrization $L_t\tildetimes K_t$ of the product $L_t\times K_t$).

Fix functions $f_L^t:X\to\RR$ and $f_K^t:Y\to\RR$ satisfying $df_L^t|_{L_t}=\lambda_X|_{L_t}$ and $df_K^t|_{K_t}=\lambda_Y|_{K_t}$, supported inside $X_0$ and $Y_0$ (and away from $\partial X$ and $\partial Y$), respectively.
Also fix $\phi_L^t:X\to\RR$ and $\phi_K^t:Y\to\RR$ as before.
We define the cylindrized isotopy $L_t\tildetimes K_t$ as before, using the Liouville form \eqref{deformedliouville} with $\phi_L$ and $\phi_K$ replaced with their pushforwards under $e^{NZ_X}$ and $e^{NZ_Y}$; we emphasize that this cylindrized isotopy depends on a choice of sufficiently large $N$.

\begin{lemma}\label{productpositive}
If $\{L_t\}_{t\in[0,1]}$ and $\{K_t\}_{t\in[0,1]}$ are each positive isotopies, then for $N<\infty$ sufficiently large (depending on the entire isotopy $\{L_t \times K_t\}_{t \in [0,1]}$), the cylindrized isotopy $L_t\tildetimes K_t:=\Phi_t(L_t\times K_t)$ is positive at infinity.
\end{lemma}

\begin{proof}
By symmetry, it suffices to make a calculation in the chart $X\times(Y\setminus Y_0)$, in which the product Liouville form is given by $\lambda_X+\lambda_Y$ and the deformed Liouville form is given by $\lambda_X+\lambda_Y-d(f_L^t(e^{NZ_Y})_\ast\phi_K^t)$.
The corresponding Liouville vector fields are given, respectively, by
\begin{align}
&Z_X+Z_Y,\\
&Z_X+Z_Y+X_{f_L^t}(e^{NZ_Y})_\ast\phi_K^t+e^{-N}f_L^t(e^{NZ_Y})_\ast X_{\phi_K^t}.
\end{align}
To understand what happens in the limit $N\to\infty$, we pull back under $e^{NZ_Y}$, to obtain, respectively,
\begin{align}
\label{productlimit}&Z_X+Z_Y,\\
\label{deformedlimit}&Z_X+Z_Y+X_{f_L^t}\phi_K^t+e^{-N}f_L^tX_{\phi_K^t}.
\end{align}
In these coordinates, the term with the factor $e^{-N}$ evidently becomes negligible as $N\to\infty$, and hence the defining relations $\Phi_t=\id$ over a large compact subset and $(\Phi_t)_\ast\eqref{deformedlimit}=\eqref{productlimit}$ show that as $N\to\infty$, the limit of $\Phi_t$ exists (converging smoothly) and is the identity on the $Y$-coordinate.
The limit of $(e^{-NZ_Y})_\ast(L_t\tildetimes K_t)$ therefore also exists.
The pulled back Liouville form $\lambda_X+e^{N}\lambda_Y$ converges (after rescaling by $e^{-N}$) in the limit to $\lambda_Y$.
We thus conclude that evaluating this limiting form on the limiting deformed isotopy $(e^{-NZ_Y})_\ast(L_t\tildetimes K_t)$ only sees the $\partial_\infty K_t$ factor (since the limiting $\Phi_t$ is the identity on the $Y$-coordinate), which by assumption moves positively.
We conclude that for sufficiently large $N<\infty$, the isotopy $L_t\tildetimes K_t$ is positive at infinity.
\end{proof}

Note that the size of $N<\infty$ needed to ensure positivity of the cylindrized isotopy $L_t\tildetimes K_t$ depends on the entire isotopies $\{L_t\}_{t\in[0,1]}$ and $\{K_t\}_{t\in[0,1]}$, and in particular cannot be made to coincide with a given previously chosen $N$ for $L_0\tildetimes K_0$ or $L_1\tildetimes K_1$.
This means that arguments involving cylindrized product isotopies require some care (though no serious issues will arise).
It also means that the above is not sufficient for ensuring positivity of cylindrized product isotopies over noncompact parameter spaces (as are needed, for example, to describe cofinal wrappings); this will be resolved in \S\ref{productcofinal} by taking $N$ to be a function of $t$ and refining the above analysis.

\subsection{Cofinality of product wrappings}\label{productcofinal}

Here we show that product wrappings are cofinal amongst all wrappings.  The cofinality criterion of Lemma \ref{cofinalitycriterion} will be essential.

\begin{proposition} \label{product wrappings cofinal} 
Let $X$ and $Y$ be Liouville manifolds and $L^0\subseteq X$ and $K^0\subseteq Y$ cylindrical Lagrangians.
There exist cofinal wrapping isotopies $\{L^t\}_{t\geq 0}$ and $\{K^t\}_{t\geq 0}$ whose cylindrized product isotopy $\{L^t\tildetimes K^t\}_{t\geq 0}$ (for certain appropriate cylindrization data) is a cofinal wrapping of $L^0\tildetimes K^0$.
\end{proposition}

\begin{proof}
Fix contact forms $\beta_X$ and $\beta_Y$ on $\partial_\infty X$ and $\partial_\infty Y$, respectively, and choose wrappings $\{L^t\}_{t\geq 0}$ and $\{K^t\}_{t\geq 0}$ which simply follow the Reeb vector fields $\R_{\beta_X}$ and $\R_{\beta_Y}$ at infinity.
We will show that the product wrapping $\{L^t\tildetimes K^t\}_{t\geq 0}$ satisfies the cofinality criterion (Lemma \ref{cofinalitycriterion}) with respect to the ``product contact form''
\begin{equation}\label{productcontactform}
\beta_{X\times Y}:=\min(\beta_X+\lambda_Y,\lambda_X+\beta_Y).
\end{equation}
(Recall that $\partial_\infty(X\times Y)$ is covered by two charts $\partial_\infty X\times Y$ and $X\times\partial_\infty Y$, and interpret $\beta_X+\lambda_Y$ and $\lambda_X+\beta_Y$ as contact forms on each of these charts, respectively; concretely $\beta_{X\times Y}$ is the contact form on $\partial_\infty(X\times Y)$ which corresponds to the Liouville subdomain $X_0\times Y_0\subseteq X\times Y$, where $X_0\subseteq X$ and $Y_0\subseteq Y$ are the Liouville subdomains corresponding to $\beta_X$ and $\beta_Y$, respectively.)

We saw in \S\ref{wrappingproduct} that the wrapping $\{L^t\tildetimes K^t\}_{t\geq 0}$ is at least positive, or rather that for any $T$ that there exists cylindrization data making $\{L^t\tildetimes K^t\}_{t \in [0, T]}$ positive.
Here we show, by quantitatively refining the same argument, that one can find cylindrization data making the entire isotopy $\{L^t\tildetimes K^t\}_{t \geq 0 }$ positive with moreoever a positive lower bound on $\beta_{X\times Y}(\partial_t\partial_\infty(L^t\tildetimes K^t))$ (which immediately implies the cofinality criterion). 
In the notation of the previous subsection, we specify the perturbation $L^t\tildetimes K^t$ by fixing $f_L^t$, $f_K^t$ and $\phi_L^t$, $\phi_K^t$ as before and by specifying a function $N(t)$ in place of the constant $N$ used previously.
Over the unperturbed locus $\partial_\infty L_t\times\partial_\infty K_t\times\RR\subseteq\partial_\infty X\times\partial_\infty Y\times\RR$ (rather, the unperturbed locus is an arbitrarily large compact subset of this), the product contact form $\beta_{X\times Y}$ is given by $e^{\min(0,-s)}\beta_X+e^{\min(0,s)}\beta_Y$, and its evaluation on $\partial_t(\partial_\infty L_t\times\partial_\infty K_t\times\RR)$ is thus given by
\begin{equation}
e^{\min(0,-s)}\beta_X(\partial_t\partial_\infty L_t)+e^{\min(0,s)}\beta_Y(\partial_t\partial_\infty K_t)\geq\begin{cases}\beta_X(\partial_t\partial_\infty L_t)&s\leq 0,\\\beta_Y(\partial_t\partial_\infty K_t)&s\geq 0.\end{cases}
\end{equation}
In fact, both terms on the right are simply $1$, since by definition $\partial_\infty L_t$ and $\partial_\infty K_t$ follow the Reeb flows of $\beta_X$ and $\beta_Y$, respectively.

We now consider the perturbed locus where $L^t\tildetimes K^t$ differs from $L^t\times K^t$.
Following the reasoning from \S\ref{wrappingproduct} surrounding \eqref{productlimit}--\eqref{deformedlimit}, we may observe that
\begin{equation}
\beta_{X\times Y}(\partial_t\partial_\infty(L_t\tildetimes K_t))=\beta_Y(\partial_t\partial_\infty K_t)+O(e^{-N(t)}(1+\left|N'(t)\right|)),
\end{equation}
where the constant in the $O(\cdot)$ depends on $t$.
The first term again equals $1$, and a function $N:\RR_{\geq 0}\to\RR$ makes the second term negligible iff it solves the differential inequality
\begin{equation}\label{diffineq}
\left|N'(t)\right|+1\leq\varepsilon(t)e^{N(t)},
\end{equation}
where $\varepsilon:\RR_{\geq 0}\to\RR_{>0}$ is a continuous function depending on the isotopies $L^t$, $K^t$, the primitives $f_L^t$, $f_K^t$, and the functions $\phi_L^t$, $\phi_K^t$.

To solve the differential inequality \eqref{diffineq}, first note that it makes sense for Lipschitz functions $N$, and within this class of functions, if $N_1$ and $N_2$ are solutions then so is $\min(N_1,N_2)$.
Therefore it suffices to find, for arbitrarily large $T<\infty$, solutions $N:[0,T)\to\RR_{>0}$ which go to infinity at $T$ (then just take the minimum of all such).
Such solutions may be constructed by taking $N\equiv a$ over $[0,t_0]$ and then continuing for $t>t_0$ by taking $N'(t)$ as large as possible subject to \eqref{diffineq} (for sufficiently large $a<\infty$, this solves \eqref{diffineq} and diverges to infinity in finite time).
This produces the desired cylindrized isotopy $\{L_1^t\tildetimes K_1^t\}_{t\geq 0}$ for which $\beta_{X\times Y}(\partial_t\partial_\infty(L_t\tildetimes K_t))$ is bounded below, and hence by Lemma \ref{cofinalitycriterion} is cofinal.
\end{proof}

We now extend this reasoning to the stopped case:

\begin{proposition} \label{product wrappings cofinal stopped}
Consider stopped Liouville manifolds $(X,\f)$ and $(Y,\g)$ and cylindrical Lagrangians $L^0 \subseteq X \setminus \f$ and $K^0 \subseteq Y \setminus \g$.
There there exist cofinal wrapping isotopies $\{L^t\}_{t\geq 0} \subseteq X \setminus \f$ and $\{K^t\}_{t\geq 0} \subseteq Y \setminus \g$ whose cylindrized product isotopy $\{L^t\tildetimes K^t\}_{t\geq 0}$ (for certain appropriate cylindrization data) is a cofinal wrapping of $L^0\tildetimes K^0$ in
$(X, \f) \times (Y, \g)$.
\end{proposition}

\begin{proof}
Choose smooth contact Hamiltonians on $\partial_\infty X$ and $\partial_\infty Y$ vanishing precisely along $\f$ and $\g$, and let $\beta_X$ and $\beta_Y$ be the corresponding contact forms on $\partial_\infty X\setminus\f$ and $\partial_\infty Y\setminus\g$, respectively (so the Reeb vector fields of $\beta_X$ and $\beta_Y$ are complete).
Let $\{L_1^t\}_{t\geq 0}$ and $\{K_1^t\}_{t\geq 0}$ be positive isotopies which, outside fixed compact subsets of $X$ and $Y$, coincide with the flow by (the linear Hamiltonian corresponding to) the chosen contact vector fields on $\partial_\infty X$ and $\partial_\infty Y$.
We now consider again the product contact form \eqref{productcontactform}.
As before, this product contact form $\beta_{X\times Y}$ corresponds to the contact type hypersurface $\partial(X_0\times Y_0)$, where $X_0\subseteq X$ and $Y_0\subseteq Y$ are the subsets corresponding to $\beta_X$ and $\beta_Y$, respectively.
Note that now $X_0$ and $Y_0$ are not compact, rather they approach $\partial_\infty X$ and $\partial_\infty Y$ along $\f$ and $\g$, respectively.
Their product $X_0\times Y_0$ approaches $\partial_\infty(X\times Y)$ along the product stop 
$\h := (\cc_X\times\g)\cup(\f\times\g\times\RR)\cup(\f\times\cc_Y) \subset \partial_\infty(X\times Y)$, 
so the product contact form $\beta_{X\times Y}$ from \eqref{productcontactform} is a contact form on $\partial_\infty(X\times Y)\setminus\h$.
Now the same reasoning as before shows that there is a lower bound on the evaluation of $\beta_{X\times Y}$ on the $t$-derivative of the product isotopy $\partial_t\partial_\infty(L_1^t\tildetimes K_1^t)$, and hence the product isotopy $\{L_1^t\tildetimes K_1^t\}_{t\geq 0}$ is cofinal.
\end{proof}

\section{K\"unneth} \label{sec: kunneth}

In this section we prove Theorem \ref{kunneth}.
That is, we construct a fully faithful K\"unneth functor for stopped Liouville manifolds:
\begin{equation} \label{eq: stopped kunneth}
\W(X,\f)\otimes\W(Y,\g)\to\W((X,\f)\times(Y,\g))
\end{equation}
which on objects takes $L \otimes K \mapsto L \tilde \times K$, and we formally deduce from this a similar functor for Liouville sectors.
For notational convenience, we will abbreviate $\W_X=\W(X,\f)$, $\W_Y=\W(Y,\g)$, and $\W_{X\times Y}=\W((X,\f)\times(Y,\g))$ for most of this section.

Let us outline the basic ideas and difficulties which go into the construction of the K\"unneth functor \eqref{eq: stopped kunneth}.
At the level of Lagrangian Floer cohomology, for choices of complex structure $J_X, J_Y$, 
one has an obvious identification at the chain level
\begin{equation}\label{kunnethcomplexproductisointro}
CF^\bullet(L_1,L_2;J_X)\otimes CF^\bullet(K_1,K_2;J_Y)=CF^\bullet(L_1\times K_1,L_2\times K_2;J_X\times J_Y).
\end{equation}
There are two primary difficulties in upgrading this tautological identification to an $\ainf$-bilinear-functor $\W_X\otimes\W_Y\to\W_{X\times Y}$.
First, and most obviously, the right hand side is \emph{not} a morphism complex in $\W_{X\times Y}$, since neither the product Lagrangians $L_i\times K_i$ nor the product almost complex structure $J_X\times J_Y$ is cylindrical on $X\times Y$ (and wrapping $L_1,K_1$ in $X,Y$, respectively, is not exactly the same as wrapping in $X\times Y$).
Second, and less obviously, while the above tautological isomorphism is compatible with $\mu^1$ and $\mu^2$, it does not play well with $\mu^k$ for $k\geq 3$, due to the fact that the moduli spaces of disks $\Rbar_{k,1}$ are then positive dimensional and so the moduli spaces of disks in $X\times Y$ are the \emph{fiber} product over $\Rbar_{k,1}$ of those for $X$ and $Y$.
This is an incarnation of the closely related fact that there is no canonical tensor product of $\ainf$-categories $\W_X\otimes\W_Y$, while there is for dg-categories.
We separate these two difficulties by constructing the K\"unneth functor \eqref{eq: stopped kunneth} as a composition
\begin{equation}
\W_X\otimes\W_Y\to\W_{X\times Y}^\pr\to\W_{X\times Y}
\end{equation}
where $\W_{X\times Y}^\pr$ is like $\W_{X\times Y}$ but built using product Lagrangians, product Floer data, and product wrappings.
The first difficulty is involved in constructing the second functor, and the second difficulty in the first functor.

To construct the first functor $\W_X\otimes\W_Y\to\W_{X\times Y}^\pr$, we use Yoneda.
That is, we define a functor $\T:\W_X\otimes\W_Y\to\Mod\W_{X\times Y}^\pr$ and prove that it lands inside representable modules.
This $\T$ is, in other words, a $(\W_{X\times Y}^\pr,\W_X\otimes\W_Y)$-trimodule, and the basic claim is that $\T(-,L,K)\in\Mod\W(X\times Y)^\pr$ is representable.
This assertion is more or less tautological: it is represented by $L\times K$.
Note that representability is a separate assertion for each individual pair $(L,K)$, and this is why proving it is easier than constructing directly an $\ainf$-bilinear-functor $\W_X\otimes\W_Y\to\W_{X\times Y}^\pr$, which would involve operations for many pairs $(L_i,K_i)$ defined via more complicated moduli spaces.
The induced functor $\W_X\otimes\W_Y\to\W_{X\times Y}^\pr$ is thus somewhat inexplicit, so one must check that it agrees on cohomology with the tautological isomorphisms \eqref{kunnethcomplexproductisointro}.

The second functor $\W_{X\times Y}^\pr\to\W_{X\times Y}$ is cylindrization of Lagrangians as in \S\ref{sec: products}.
The key is to choose a countable set of Lagrangians defining $\W_{X\times Y}^\pr$ for which there exists a symplectomorphism $\Phi$ which simultaneously cylindrizes all of them.
The desired functor is then simply pushforward under $\Phi$.
Of course, $\Phi$ will not send product of cylindrical Floer data on $X\times Y$ to cylindrical Floer data on $X\times Y$.
This becomes irrelevant by using dissipative Floer data as in \S\ref{noncyl}.
Full faithfulness of this functor follows from cofinality of product wrappings, proved in \S\ref{productcofinal}.

To relate the negative pushoff of the diagonal $\Delta^-\in\W_{X^-\times X}$ and the diagonal $(\W_X,\W_X)$-bimodule, the main point is that pulling back wrapped Floer cochains with $\Delta$ is the same as plugging in $\Delta$ to the trimodule $\T$, which is tautologically the diagonal bimodule by `seam erasing'.

\subsection{Product Floer data}\label{Wprod}

In this subsection, we define $\W_{X\times Y}^\pr$, the `split' wrapped Fukaya category of $(X,\f)\times(Y,\g)$, for stopped Liouville manifolds $(X,\f)$ and $(Y,\g)$.
The adjective `split' and the superscript `$\pr$' indicate that it is defined using Lagrangians, Floer data, and wrappings which are products of cylindrical (rather than cylindrical on the product, which would give $\W_{X\times Y}$, the usual wrapped Fukaya category of $(X,\f)\times(Y,\g)$).
The definition of $\W_{X\times Y}^\pr$ follows \S\ref{wrapdefsec} closely.

\begin{definition}
We define an abstract Floer setup $\SSS_{X\times Y}^\pr$ as in Definition \ref{geomtoabstract}, with the following differences.
An element of the set $\sL$ of Lagrangians is a pair of cylindrical Lagrangians $L\subseteq X$ and $K\subseteq Y$ (often written simply as $L\times K\subseteq X\times Y$).
A tuple $(L_0\times K_0,\ldots,L_k\times K_k)$ will be called composable iff $(L_0,\ldots,L_k)$ is mutually transverse and so is $(K_0,\ldots,K_k)$.
We use product almost complex structures, i.e.\ we choose families of cylindrical almost complex structures on $X$ and on $Y$, and we consider disks holomorphic with respect to their product.
The moduli spaces of holomorphic disks in $X\times Y$ with boundary on $L_0\times K_0,\ldots,L_r\times K_r$ are compact by the usual monotonicity and uniformly bounded geometry arguments (uniformly bounded geometry is preserved under taking products, and the distance between $L\times K$ and $L'\times K'$ near infinity is bounded below since the same holds for the pairs $(L,L')$ and $(K,K')$).
\end{definition}

\begin{lemma}
Generic product almost complex structures achieve transversality.
\end{lemma}

\begin{proof}
(This is a known statement for which we do not know a reference.)
To prove generic transversality for a given class of pseudo-holomorphic maps $u:C\to W$ with respect to domain-dependent almost complex structures, the key (compare \cite[(9k)]{seidelbook}) is to show that for every such pseudo-holomorphic map $u:C\to W$ with respect to an almost complex structure $J$, the map
\begin{equation}\label{gensurj}
\End^{0,1}_J(T_{u(p)}W)\xrightarrow{\circ(du)_p}\Hom^{0,1}_J(T_pC,T_{u(p)}W)
\end{equation}
is surjective for an open subset of $p\in C$ meeting every component of $C$.
It is well known and evident that this surjectivity holds whenever $(du)_p\ne 0$ (and thus generic transversality holds except possibly for curves with constant components, which need to be analyzed separately).

In the present setting of $W=X\times Y$, generic transversality for \emph{product} almost complex structures amounts to surjectivity of \eqref{gensurj} after restricting the domain to
\begin{equation}
\End^{0,1}_J(T_{u(p)}X)\oplus\End^{0,1}_J(T_{u(p)}Y)\subseteq\End^{0,1}_J(T_{u(p)}W).
\end{equation}
The result is simply the direct sum of the maps \eqref{gensurj} for $\pi_X\circ u$ and $\pi_Y\circ u$ separately.
It is thus surjective whenever $d\pi_X\circ(du)_p$ and $d\pi_Y\circ(du)_p$ are \emph{both} nonzero.
In the contrary case where one vanishes over some non-empty open set, unique continuation forces $\pi_X\circ u$ or $\pi_Y\circ u$ to be constant on an entire connected component of $C$.
When $k\geq 2$ (i.e.\ three boundary Lagrangians), this is disallowed by our assumption that the Lagrangians be mutually transverse (hence have no triple intersections).
In the remaining case $k=1$, we have constant strips sent to transverse intersections of Lagrangians, and these are cut out transversally by calculation.
\end{proof}

\begin{lemma}\label{kunnethproducthomologytrivial}
There is a canonical isomorphism $hF^\bullet(L\times K,L'\times K')=hF^\bullet(L,L')\otimes hF^\bullet(K,K')$.
It is compatible with product, in the sense that it identifies the composition map for $(L_1\times K_1,L_2\times K_2,L_3\times K_3)$ with the tensor product of the composition maps for $(L_1,L_2,L_3)$ and $(K_1,K_2,K_3)$.
\end{lemma}

\begin{proof}
Fix almost complex structures $J_X:[0,1]=\Sbar_{1,1}\to\J(X)$ and $J_Y:[0,1]=\Sbar_{1,1}\to\J(Y)$ used to define $CF^\bullet(L,L')$ and $CF^\bullet(K,K')$.
Their product $J_X\times J_Y$ is valid Floer data for $(L\times K,L'\times K')$.
As is well understood, there is a tautological isomorphism of complexes
\begin{equation}\label{kunnethcomplexproductiso}
CF^\bullet(L_1,L_2;J_X)\otimes CF^\bullet(K_1,K_2;J_Y)=CF^\bullet(L_1\times K_1,L_2\times K_2;J_X\times J_Y).
\end{equation}
Namely, the intersection points generating both sides are in obvious bijection with each other.
To compare the differentials, note that a holomorphic map into the product is simply a pair of holomorphic maps into both factors.
Transversality on each of the factors (including for the constant strips at intersection points $L_1\cap L_2$ and $K_1\cap K_2$) implies transversality in the product.

Similar arguments apply to the continuation maps relating different choices of Floer data for $(L,L')$ and $(K,K')$, which imply that \eqref{kunnethcomplexproductiso} induces a well defined isomorphism $hF^\bullet(L\times K,L'\times K')=hF^\bullet(L,L')\otimes hF^\bullet(K,K')$ in the homotopy category.
Compatibility with Floer product is similar.
\end{proof}

\begin{remark}\label{donottensorhomology}
Lemma \ref{kunnethproducthomologytrivial} provides a homomorphism $HF^\bullet(L,L')\otimes HF^\bullet(K,K')\to HF^\bullet(L\times K,L'\times K')$ which is \emph{not} in general an isomorphism due to torsion issues.
\end{remark}

As in Lemma-Definition \ref{isotopyinvariancedef}, $hF^\bullet(L\times K,L'\times K')$ is invariant under isotopies of $L,K,L',K'$ for which the pairs $(L,L')$ and $(K,K')$ stay disjoint at infinity.
This allows us to define $hF^\bullet(L\times K,L'\times K')$ for \emph{all} $L,K,L',K'$ by positively perturbing $L$ and $K$.
Moreover, these isotopy invariance isomorphisms are compatible with the isomorphisms from Lemma \ref{kunnethproducthomologytrivial} by the same argument, and hence Lemma \ref{kunnethproducthomologytrivial} applies to all $L,K,L',K'$.
That is, $h\F_{X\times Y}^\pr=h\F_X\otimes h\F_Y$.
In particular, the algebra object
\begin{equation}
hF^\bullet(L\times K,L\times K):=hF^\bullet(L^+\times K^+,L\times K)
\end{equation}
is well defined and coincides under Lemma \ref{kunnethproducthomologytrivial} with $hF^\bullet(L,L)\otimes hF^\bullet(K,K)$.
Unitality of $HF^\bullet(L,L)$ and $HF^\bullet(K,K)$ does not trivially imply unitality of $HF^\bullet(L\times K,L\times K)$ due to Remark \ref{donottensorhomology}, but does so by the following argument.

\begin{lemma}\label{continuationhomotopyunit}
If $HF^\bullet(L,L)$ is unital, then multiplication by $\1_L\in HF^0(L,L)$ on $hF^\bullet(L,L)$ is the identity.
\end{lemma}

\begin{proof}
Multiplication by $\1_L$ on $hF^\bullet(L,L)$ is a quasi-isomorphism, hence is an isomorphism by cofibrancy \cite[Lemma 3.6]{gpssectorsoc}.
Since $\1_L\circ\1_L=\1_L$, multiplication by $\1_L$ squares to itself.
Any automorphism which squares to itself is the identity.
\end{proof}

\begin{lemma}\label{productunits}
The algebra $HF^\bullet(L\times K,L\times K)$ is unital with unit $\1_{L\otimes K}$ the image of $\1_L\otimes\1_K$, and the modules $HF^\bullet(L\times K,L'\times K')$ and $HF^\bullet(L'\times K',L\times K)$ over it are unital.
\end{lemma}

\begin{proof}
Multiplication by $\1_L\otimes\1_K$ acts as the identity on $HF^\bullet(L,L')\otimes HF^\bullet(K,K')$, but note that Lemma \ref{kunnethproducthomologytrivial} \emph{does not} equate this with $HF^\bullet(L\times K,L'\times K')$ (see Remark \ref{donottensorhomology}).
Instead, we note that by Lemma \ref{continuationhomotopyunit}, multiplication by $\1_L$ acts as the identity on $hF^\bullet(L,L')$ (and analogously for $\1_K$), and hence $\1_L\otimes\1_K$ acts as the identity on $hF^\bullet(L,L')\otimes hF^\bullet(K,K')=hF^\bullet(L\times K,L'\times K')$, hence on its cohomology as well.
Taking $L=L'$ and $K=K'$ gives unitality of the algebra $HF^\bullet(L\times K,L\times K)$, and the general case is module unitality.
\end{proof}

Given the unitality from Lemma \ref{productunits}, continuation maps for $\SSS_{X\times Y}^\pr$ are defined as in Lemma-Definition \ref{continuationdef}.
We consider the wrapping categories $R_{L\times K}=R_L\times R_K$, which satisfy the factorization property by general position.
We denote the resulting wrapped Floer cohomology groups by
\begin{equation}\label{wprodmorcohomologydirlim}
HW^\bullet(L\times K,L'\times K')^\pr=\varinjlim_{\begin{smallmatrix}L^w\in R_L\\K^w\in R_L\end{smallmatrix}}HF^\bullet(L^w\times K^w,L'\times K').
\end{equation}
The argument from Lemma \ref{wraprightlocal} shows the right locality property, so we have defined an abstract wrapped Floer setup $\SSS_{X\times Y}^\pr$.
We denote its wrapped Fukaya category by $\W_{X\times Y}^\pr$.

Having defined $\W_{X\times Y}^\pr$, let us make a first step towards comparing it with $\W_X\otimes\W_Y$.
Although the symbol $\W_X\otimes\W_Y$ has no meaning for us on its own (other than being simply a pair of $\ainf$-categories), we can make sense out of its cohomology category: an object of $H^\bullet(\W_X\otimes\W_Y)$ is a pair $L\in\W_X$ and $K\in\W_Y$, and the group of morphisms $(L,K)\to(L',K')$ is the cohomology of the tensor product $\W_X(L,L')\otimes\W_Y(K,K')$, with composition as expected.
We can calculate this cohomology as follows.

\begin{lemma}\label{tensorofqisolemma}
The natural map
\begin{equation}\label{tensorofqiso}
\varinjlim_{\begin{smallmatrix}L^w\in R_L\\K^w\in R_K\end{smallmatrix}}H^\bullet(hF^\bullet(L^w,L')\otimes hF^\bullet(K^w,K'))\to\varinjlim_{\begin{smallmatrix}L^w\in R_L\\K^w\in R_K\end{smallmatrix}}H^\bullet(\W_X(L^w,L')\otimes\W_Y(K^w,K'))
\end{equation}
is an isomorphism.
\end{lemma}

\begin{proof}
Choose cofinal sequences $L=L_0\leadsto L_1\leadsto\cdots$ and $K=K_0\leadsto K_1\leadsto\cdots$ in $R_L$ and $R_K$, respectively, so by cofinality it suffices to consider the map
\begin{equation}\label{tensorofqisoseq}
\varinjlim_{i,j}H^\bullet(hF^\bullet(L_i,L')\otimes hF^\bullet(K_j,K'))\to\varinjlim_{i,j}H^\bullet(\W_X(L_i,L')\otimes\W_Y(K_j,K'))
\end{equation}
The map $\varinjlim_iHF^\bullet(L_i,L')\to\varinjlim_iH^\bullet\W_X(L_i,L')$ is an isomorphism by Lemma \ref{pwrappingworks}.
Realize this map on the chain level using mapping telescopes and choices of cycles lifting the continuation maps $L_{i+1}\to L_i$.
The result is a quasi-isomorphism of cofibrant complexes, hence a homotopy equivalence \cite[Lemma 3.6]{gpssectorsoc}.
The tensor product of this homotopy equivalence with the corresponding one for $(K_j,K')$ is thus a homotopy equivalence, hence in particular a quasi-isomorphism.
This tensor product of mapping telescopes calculates the domain of \eqref{tensorofqiso}.
\end{proof}

The domain of \eqref{tensorofqiso} is, in view of \eqref{wprodmorcohomologydirlim} and Lemma \ref{kunnethproducthomologytrivial}, the morphism group in $H^\bullet\W_{X\times Y}^\pr$.
The codomain of \eqref{tensorofqiso} is of course just $H^\bullet(\W_X(L,L')\otimes\W_Y(K,K'))$.
Thus Lemma \ref{tensorofqisolemma} defines an equivalence of categories
\begin{equation}\label{kunnethproducthomologycategories}
H^\bullet(\W_X\otimes\W_Y)=H^\bullet\W_{X\times Y}^\pr
\end{equation}
(compatibility with composition is straightforward).
In the next subsection, we will lift this equivalence to an $\ainf$-bilinear-functor $\W_X\otimes\W_Y\to\W_{X\times Y}^\pr$.

\subsection{K\"unneth trimodule} \label{kunnethtrimodule}

Here we construct the functor $\W_X\otimes\W_Y\to\W_{X\times Y}^\pr$.
We do so by constructing a $(\W_{X\times Y}^\pr,\W_X\otimes\W_Y)$-trimodule $\T$, that is an $\ainf$-bilinear-functor
\begin{equation}
\T:\W_X\otimes\W_Y\to\Mod\W_{X\times Y}^\pr,
\end{equation}
and then showing that it lands inside representable functors, i.e.\ that $\T(-,L,K)$ is represented by $L\times K$.
We also show that this functor induces the equivalence of cohomology categories \eqref{kunnethproducthomologycategories}.
The construction of $\T$ will use pseudo-holomorphic quilted strips following Ma'u \cite{mau} and Ganatra \cite{ganatrawrapcy}.

\begin{figure}[ht]
\centering
\includegraphics[max width=.95\textwidth]{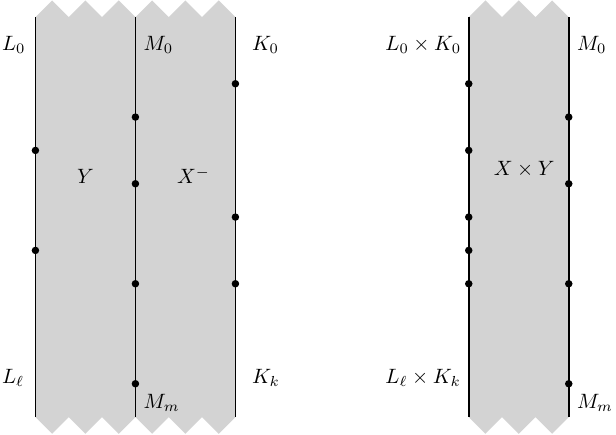}
\caption{Holomorphic maps used to define the $(\W_{X\times Y}^\pr,\W_X\otimes\W_Y)$-trimodule $\T$.}\label{figuretrimodule}
\end{figure}

To define the trimodule $\T$, we extend the abstract wrapped Floer setups underlying $\W_X$, $\W_Y$, and $\W_{X\times Y}^\pr$ as follows.
We declare a tuple $(M_m,\ldots,M_0,L_0,\ldots,L_\ell,K_0,\ldots,K_k)$ ($M_i\in\sL_{X\times Y}^\pr$, $L_i\in\sL_X$, $K_i\in\sL_Y$) to be composable iff when we write $M_i=L_i'\times K_i'$, the tuples $(L_m',\ldots,L_0',L_0,\ldots,L_k)$ and $(K_m',\ldots,K_0',K_0,\ldots,K_k)$ are both mutually transverse.
We then consider Floer data for strips as in Figure \ref{figuretrimodule} (see \cite[\S 5.3]{gpssectorsoc} for more details on the relevant moduli spaces of domains).
Such a strip can be thought of either as a pair of maps to $X^-$ and $Y$ as on the left of Figure \ref{figuretrimodule} or as a single map to $X\times Y$ as on the right (via folding the strip).
Floer data consists of compatible families of cylindrical almost complex structures on $X$ and $Y$ over each half of the strip, respectively ($s$-invariant in the thin parts of the strip, with respect to fixed universal strip-like coordinates).
We restrict the abstract Floer setups $\SSS_X$, $\SSS_Y$, and $\SSS_{X\times Y}^\pr$ to use negative strip-like coordinates \eqref{coordsII} which extend to biholomorphisms as in Remark \ref{striplikebottombiholospecial}, and we require that on the trimodule domain strips, we use the tautological strip-like coordinates at $s=\pm\infty$ (these assumptions ensure that gluing via the strip-like coordinates defines collars for the relevant moduli space of domains).
As a map to $X\times Y$, the folded strip should be holomorphic with respect to the product of these almost complex structures.
As a pair of maps to $X^-$ and $Y$, the left half (mapping to $Y$) should be holomorphic, and the right half (mapping to $X^-$) should be antiholomorphic.
The usual monotonicity and uniformly bounded geometry arguments from \cite[Proposition 3.19]{gpssectorsoc} apply to show that such (possibly broken) holomorphic strips map to a fixed compact subset of the target, and hence the (Gromov--Floer compactified) moduli spaces of such strips are indeed compact.
Choosing Floer data generically we obtain transversality, and counting the
zero-dimensional moduli spaces defines moduli counts
\begin{multline}
CF^\bullet(M_m,M_{m-1})\otimes\cdots\otimes CF^\bullet(M_1,M_0)\otimes CF^\bullet(M_0,L_0\times K_0)\\
\otimes CF^\bullet(L_0,L_1)\otimes\cdots\otimes CF^\bullet(L_{\ell-1},L_\ell)\otimes CF^\bullet(K_0,K_1)\otimes\cdots\otimes CF^\bullet(K_{k-1},K_k)\\
\to CF^\bullet(M_m,L_\ell\times K_k)[1-m-\ell-k]
\end{multline}
satisfying the identities to define an $\ainf$-trimodule.

Now suppose $P_X$, $P_Y$, and $P_{X\times Y}^\pr$ are decorated posets as in \S\ref{abstractfloersec} for the abstract wrapped Floer setups underlying $\W_X$, $\W_Y$, and $\W_{X\times Y}^\pr$.
Suppose further that all tuples $(M_m,\allowbreak\ldots,\allowbreak M_0,\allowbreak L_0,\allowbreak\ldots,\allowbreak L_\ell,\allowbreak K_0,\allowbreak\ldots,\allowbreak K_k)$ with $K_0>\cdots>K_k\in P_X$, $L_0>\cdots>L_\ell\in P_Y$, and $M_m>\cdots>M_0\in P_{X\times Y}^\pr$ are composable in the sense of the previous paragraph.
In this case, Floer data as above can be chosen for all such tuples by induction, and we thus obtain an $(\OO_{X\times Y}^\pr,\OO_X\otimes\OO_Y)$-trimodule $\Q$.
It is straightforward to choose $P_X$, $P_Y$, and $P_{X\times Y}^\pr$ to be countable, cofinite, and sufficiently wrapped, such that moreover the above transversality condition is satisfied (first choose $P_X$ and $P_Y$, and then incorporate the transversality condition into the inductive construction of the countable cofinite sufficiently wrapped $P_{X\times Y}^\pr$ from Lemma \ref{pcanwrap}).
Now localizing $\Q$ defines a $(\W_{X\times Y}^\pr,\W_X\otimes\W_Y)$-trimodule $\T={}_{(I_{X\times Y}^\pr)^{-1}}\Q_{I_X^{-1},I_Y^{-1}}$.

Let us calculate the cohomology of $\T$.
The cohomology of $\Q$ is evident: $H^\bullet\Q(M,L,K)=HF^\bullet(M,L\times K)$.
The cohomology of its localization $_{(I_{X\times Y}^\pr)^{-1}}\Q$ is calculated by the direct limit over $P_{X\times Y}^\pr$-wrapping sequences by Lemma \ref{pwrappingworks}.
Thus $H^\bullet{}_{(I_{X\times Y}^\pr)^{-1}}\Q(M,L,K)=HW^\bullet(M,L\times K)^\pr$ \eqref{wprodmorcohomologydirlim}.
By the equivalence \eqref{kunnethproducthomologycategories}, it follows that ${}_{(I_{X\times Y}^\pr)^{-1}}\Q$ is $I_X$- and $I_Y$-local on the right, so the localization map ${}_{(I_{X\times Y}^\pr)^{-1}}\Q\to{}_{(I_{X\times Y}^\pr)^{-1}}\Q_{I_X^{-1},I_Y^{-1}}$ is a quasi-isomorphism by \cite[Lemma 3.13]{gpssectorsoc}, and thus $H^\bullet\T(M,L,K)=HW^\bullet(M,L\times K)^\pr$ as well.

The trimodule $\T$ may be viewed as a functor
\begin{equation}\label{trimoduleasfunctor} \T: \W_X\otimes\W_Y\to\Mod\W_{X\times Y}^\pr, \end{equation}
and our next task is to show that it lands inside the representable modules $\W_{X\times Y}^\pr\subseteq\Mod\W_{X\times Y}^\pr$.

\begin{proposition}\label{kunnethrepresentable}
The functor $\T$ \eqref{trimoduleasfunctor} lands in representable modules, and the resulting functor (up to inverting quasi-equivalences)
\begin{equation} \label{eq: naive kunneth}
\W_X\otimes\W_Y\to\W_{X\times Y}^\pr
\end{equation}
agrees on $H^\bullet$ categories with the equivalence \eqref{kunnethproducthomologycategories}.
\end{proposition}

\begin{proof}
We first argue that the $\W_{X\times Y}^\pr$-module $\T(-,L,K)$ is represented by $L\times K\in\W_{X\times Y}^\pr$ (or, more precisely, an object isomorphic to it!).
Choose $L'\times K'\in\W_{X\times Y}^\pr$ (i.e.\ an element of our chosen $P_{X\times Y}^\pr$) transverse to $L\times K$, and choose an element of
\begin{equation}
H^0\T(L'\times K',L,K)=HW^0(L'\times K',L\times K)^\pr
\end{equation}
which is an isomorphism $L'\times K'\xrightarrow\sim L\times K$ in $H^\bullet\W_{X\times Y}^\pr$.
Multiplication with this class in $H^0\T(L'\times K',L,K)$ defines (up to homotopy) a map of $\W_{X\times Y}^\pr$-modules
\begin{equation}\label{representqiso}
\W_{X\times Y}^\pr(-,L'\times K')\to\T(-,L,K)
\end{equation}
which on cohomology is the isomorphism $HW^\bullet(-,L'\times K')^\pr=HW^\bullet(-,L\times K)^\pr$ given by composing with $L'\times K'\xrightarrow\sim L\times K$.
Thus $\T(-,L,K)$ is representable (note that the argument above depended only on the fact that the $H^\bullet\W_{X\times Y}^\pr$-module $H^\bullet\T(-,L,K)$ is representable; this argument applies to any $\ainf$-module, see Lemma \ref{representcohomology}).

To check the action of the resulting functor \eqref{eq: naive kunneth} on a given class in $HF^\bullet(L_1\times K_1,L_2\times K_2)$, note that such a class determines a map $\T(-,L_1,K_1)\to\T(-,L_2,K_2)$ (up to homotopy), and we just need to figure out where this class is sent under the representing quasi-isomorphisms \eqref{representqiso}.
We should thus consider the diagram
\begin{equation}
\begin{tikzcd}
\W_{X\times Y}^\pr(-,L_1'\times K_1')\ar[r]&\T(-,L_1\times K_1)\ar[d]\\
\W_{X\times Y}^\pr(-,L_2'\times K_2')\ar[r]&\T(-,L_2\times K_2)
\end{tikzcd}
\end{equation}
where the horizontal maps are multiplication by isomorphisms $L_1'\times K_1'\xrightarrow\sim L_1\times K_1$ and $L_2'\times K_2'\xrightarrow\sim L_2\times K_2$ in $H^\bullet\W_{X\times Y}^\pr$.
On cohomology, this diagram consists of the induced isomorphisms $HW^\bullet(-,L_1'\times K_1')^\pr=HW^\bullet(-,L_1\times K_1)^\pr$ and $HW^\bullet(-,L_2'\times K_2')^\pr=HW^\bullet(-,L_2\times K_2)^\pr$ on the top and bottom rows, respectively, and the map $HW^\bullet(-,L_1\times K_1)^\pr\to HW^\bullet(-,L_2\times K_2)^\pr$ given by multiplication by our chosen class in $HF^\bullet(L_1\times K_1,L_2\times K_2)$.
Thus the induced map $HW^\bullet(-,L_1'\times K_1')^\pr\to HW^\bullet(-,L_2'\times K_2')^\pr$ is also multiplication by this same class, which is enough (for example, by looking at where it sends the identity map of $L_1'\times K_1'$).
\end{proof}

\begin{lemma}\label{trimodulewelldefined}
The functor \eqref{eq: naive kunneth} is well defined up to quasi-isomorphism.
\end{lemma}

\begin{proof}
We are to show that the quasi-isomorphism class of the $(\W_{X\times Y}^\pr,\W_X\otimes\W_Y)$-trimodule $\T$ is independent of the choice of $P_X$, $P_Y$, $P_{X\times Y}^\pr$, and trimodule Floer data.

Let $(P_X,P_Y,P_{X\times Y}^\pr)$ and $(P_X',P_Y',P_{X\times Y}^{\pr\prime})$ be two triples equipped with trimodule Floer data.
Apply Lemma \ref{pcanwrap} to $P_X\sqcup P_X'$ and $P_Y\sqcup P_Y'$ to produce sufficiently wrapped countable cofinite decorated posets $P_X''$ and $P_Y''$, respectively.
Also construct a decorated poset $P_{X\times Y}^{\pr\prime\prime}$ by applying the construction of Lemma \ref{pcanwrap} to $P_{X\times Y}^\pr\sqcup P_{X\times Y}^{\pr\prime}$ as follows.
We add $\ZZ_{\geq 0}$ (ordered appropriately), we color it with elements of $P_{X\times Y}^\pr\sqcup P_{X\times Y}^{\pr\prime}$, and we label it accordingly with Lagrangians constructed out of cofinal sequences in wrapping categories.
The factorization property Definition \ref{abstractwrappedfloersetupdef}\ref{perturbation} (compare Remark \ref{pcanwrapcountable}) allows us to ensure that these new Lagrangians labelling $\ZZ_{\geq 0}$ satisfy the transversality conditions necessary to ensure that $(P_X,P_Y,P_{X\times Y}^\pr\sqcup\ZZ_{\geq 0})$, $(P_X',P_Y',P_{X\times Y}^{\pr\prime}\sqcup\ZZ_{\geq 0})$, and $(P_X'',P_Y'',\ZZ_{\geq 0})$ are all suitable for defining the trimodule.

Now choose trimodule Floer data for the triple $(P_X'',P_Y'',\ZZ_{\geq 0})$, and consider its restriction to $(P_X,P_Y,\ZZ_{\geq 0})$.
Then choose trimodule Floer data on $(P_X,P_Y,P_{X\times Y}^\pr\sqcup\ZZ_{\geq 0})$ restricting to the already fixed trimodule Floer data on $(P_X,P_Y,P_{X\times Y}^\pr)$ and on $(P_X,P_Y,\ZZ_{\geq 0})$.
With such trimodule Floer data, we have
\begin{equation}
(P_X,P_Y,P_{X\times Y}^\pr)\sim(P_X,P_Y,P_{X\times Y}^\pr\sqcup\ZZ_{\geq 0})\sim(P_X,P_Y,\ZZ_{\geq 0})\sim(P_X'',P_Y'',\ZZ_{\geq 0})
\end{equation}
where $\sim$ means that the resulting functors \eqref{eq: naive kunneth} are quasi-isomorphic.
By symmetry, we have $(P_X',P_Y',P_{X\times Y}^{\pr\prime})\sim(P_X'',P_Y'',\ZZ_{\geq 0})$ as well, so $(P_X,P_Y,P_{X\times Y}^\pr)\sim(P_X',P_Y',P_{X\times Y}^{\pr\prime})$ as desired.
\end{proof}

\subsection{Cylindrization} \label{kunnethcylindrization}

We will now construct a fully faithful functor
\begin{equation}
\W_{X\times Y}^\pr\to\W_{X\times Y}
\end{equation}
sending $L\times K$ to its cylindrization $L\tildetimes K$.
This functor is defined by pushing forward under an exact symplectomorphism which simultaneously cylindrizes a certain carefully chosen set of product Lagrangians used to define $\W_{X\times Y}^\pr$ (there is, of course, no symplectomorphism which simultaneously cylindrizes \emph{all} product Lagrangians).
So that Floer data can be pushed forward under arbitrary exact symplectomorphisms, we assume throughout this subsection the use of dissipative Floer data in the sense of \S\ref{noncyl} in the abstract wrapped Floer setups underlying $\W_{X\times Y}$ and $\W_{X\times Y}^\pr$.
We also tacitly assume all Lagrangians are equipped with a \emph{compactly supported} primitive of the restriction of the Liouville form, as is necessary to define cylindrizations of product Lagrangians as in \S\ref{productcylindrical} (products of such Lagrangians span all isomorphism classes in $\W_{X\times Y}^\pr$ by Lemma \ref{cylvanishinf}).

We begin by considering the effect of cylindrization of product Lagrangians on Floer cohomology.
Recall from \S\ref{productcylindrical} that the cylindrization $L\tildetimes K$ of a product Lagrangian $L\times K$ is its image under a certain exact symplectic isotopy (well defined up to contractible choice) supported in an arbitrarily small neighborhood of $(\partial_\infty L\times\cc_Y)\sqcup(\cc_X\times\partial_\infty K)$.
The cylindrization $L\tildetimes K$ is, in particular, well defined as an object of $\W_{X\times Y}$.

When $(L,L')$ and $(K,K')$ are disjoint at infinity and the isotopies cylindrizing $L\times K$ and $L'\times K'$ are taken to have sufficiently small support (hence, in particular, to have disjoint support), the cylindrizations $L\tildetimes K$ and $L'\tildetimes K'$ are disjoint at infinity.
Thus when $(L,L')$ and $(K,K')$ are transverse, their cylindrizations $L\tildetimes K$ and $L'\tildetimes K'$ are as well, and hence $HF^\bullet(L\tildetimes K,L'\tildetimes K')$ is well defined, due to the isotopy invariance isomorphisms on $HF^\bullet$ from Lemma-Definition \ref{isotopyinvariancedef}.
Moreover, since cylindrization is itself an exact Lagrangian isotopy, isotopy invariance in the dissipative context (Lemma-Definition \ref{isotopyinvariancedefdissipative}) implies that:

\begin{lemma}\label{cylindrizefloerequivalencehtpy}
For $(L,L')$ and $(K,K')$ transverse, there is a canonical equivalence $HF^\bullet(L\times K,L'\times K')=HF^\bullet(L\tildetimes K,L'\tildetimes K')$.
These equivalences are compatible with the isotopy invariance isomorphisms on both sides, and with Floer multiplication $\mu^2$.
\qed
\end{lemma}

We now compare continuation maps.
Lemma \ref{productpositive} guarantees that for appropriate choice of cylindrization data, the cylindrized product of small positive pushoffs of $L$ and $K$ is a small positive pushoff of the cylindrized product of $L$ and $K$.
It also follows that the resulting equivalence from Lemma \ref{cylindrizefloerequivalencehtpy}
\begin{equation}\label{endokunneth}
HF^\bullet(L^+\times K^+,L\times K)=HF^\bullet((L\tildetimes K)^+,L\tildetimes K)
\end{equation}
is an equivalence of algebras.
In particular, this identification respects units: $\1_{L\times K}=\1_{L\tildetimes K}$.

\begin{corollary}\label{cylindrizecontinuation}
For any pair of positive isotopies $L\leadsto L^w$ and $K\leadsto K^w$ the identification
\begin{equation}
HF^\bullet(L^w\times K^w,L\times K)=HF^\bullet(L^w\tildetimes K^w,L\tildetimes K)
\end{equation}
identifies the continuation map of $(L\leadsto L^w)\times(K\leadsto K^w)$ with the continuation map of $L\tildetimes K\leadsto L^w\tildetimes K^w$ (which is, up to contractible choice of cylindrization data, a positive isotopy by Lemma \ref{productpositive}).
\end{corollary}

\begin{proof}
This follows immediately from the definition of continuation maps in terms of units, the fact that the identifications in Lemma \ref{cylindrizefloerequivalencehtpy} respect multiplication, and the resulting fact that \eqref{endokunneth} respects units.
\end{proof}

Consider the functor $H^\bullet\F_{X\times Y}^\pr\to H^\bullet\W_{X\times Y}$ sending $L\times K$ to $L\tildetimes K$ and acting on morphisms for pairs $(L,L')$ and $(K,K')$ (not necessarily transverse) by sending $HF^\bullet(L\times K,L'\times K'):=HF^\bullet(L^+\times K^+,L'\times K')$ to $HF^\bullet(L^+\tildetimes K^+,L'\tildetimes K')$ using Lemma \ref{cylindrizefloerequivalencehtpy}.
This functor sends continuation maps to isomorphisms by Corollary \ref{cylindrizecontinuation}.
Thus by the universal property of localization Lemma \ref{HWlocalizationproperty}, it determines a unique (up to unique isomorphism) functor
\begin{equation}\label{kcylindrizecohomology}
H^\bullet\W_{X\times Y}^\pr\to H^\bullet\W_{X\times Y}.
\end{equation}

\begin{corollary}\label{cylff}
The functor \eqref{kcylindrizecohomology} is fully faithful.
\end{corollary}

\begin{proof}
We wish to show to be isomorphisms the maps
\begin{equation}\label{cylindrizewrapcohomology}
HW^\bullet(L_1\times K_1,L_2\times K_2)^\pr\to HW^\bullet(L_1\tildetimes K_1,L_2\tildetimes K_2),
\end{equation}
comprising the functor \eqref{kcylindrizecohomology}.
Choose cofinal wrappings $L_1^t$ and $K_1^t$, and use Propositions \ref{product wrappings cofinal}--\ref{product wrappings cofinal stopped} to choose cylindrization data defining $L_1^t\tildetimes K_1^t$ so that it is also a cofinal positive isotopy.
Choosing these generically allows enough transversality to define an isomorphism
\begin{equation}\label{cylindrizedirlimwrap}
\varinjlim_tHF^\bullet(L_1^t\times K_1^t,L_2\times K_2)\to\varinjlim_tHF^\bullet(L_1^t\tildetimes K_1^t,L_2\tildetimes K_2)
\end{equation}
using the isomorphisms from Lemma \ref{cylindrizefloerequivalencehtpy} and the agreement of continuation maps from Corollary \ref{cylindrizecontinuation}.
Now there is a natural map from \eqref{cylindrizedirlimwrap} to \eqref{cylindrizewrapcohomology}, which is an isomorphism by cofinality of $L_1^t$, $K_1^t$, and $L_1^t\tildetimes K_1^t$.
\end{proof}

Our final task is to lift \eqref{kcylindrizecohomology} to an $\ainf$-functor.

Consider a decorated poset $P$ for the abstract wrapped Floer setup underlying $\W_{X\times Y}^\pr$, together with a choice of \emph{cylindrization data} for every $L\times K\in P$ in the following sense.
Cylindrization data for $L\times K$ consists functions $(f_L,\phi_L,f_K,\phi_K)$ as in \eqref{deformedliouville} subject to the following support requirements.
The functions $\phi_L$ and $\phi_K$ must be supported inside \emph{specified} \emph{standard} neighborhoods $N_{\partial_\infty L}\subseteq X$ of $\partial_\infty L$ and $N_{\partial_\infty K}\subseteq Y$ of $\partial_\infty K$.
Moreover, as $L\times K$ ranges over all of $P$, the standard neighborhoods $N_{\partial_\infty L}\subseteq X$ are required to be disjoint and locally finite (and the same on $Y$).
Finally, the support of the products $f_L\phi_K$ and $\phi_Lf_K$ must be contained in specified disjoint neighborhoods of $\cc_X\times\partial_\infty Y$ and $\partial_\infty X\times\cc_Y$, respectively.
These conditions together guarantee that the deformation terms to the Liouville form \eqref{deformedliouville} have disjoint support and hence can be added simultaneously.
We thus obtain a single exact symplectic isotopy $\Phi$ which simultaneously cylindrizes all Lagrangians in $P$.
The pushforward $\Phi(P)$ is thus a decorated poset for the abstract wrapped Floer setup underlying $\W_{X\times Y}$ (note that Floer data remains dissipative under pushforward since dissipativity is symplectomorphism invariant).

Unfortunately, a given decorated poset $P$ for the abstract wrapped Floer setup underlying $\W_{X\times Y}^\pr$ need not admit cylindrization data in the above sense.
Indeed, the existence of cylindrization data implies, for instance, that for every $L\times K$ appearing in $P$, there exists a neighborhood of $\partial_\infty L$ which is disjoint from $\partial_\infty L'$ for every other $L'\times K'$ in $P$, a rather strong condition.
So, we need an argument to produce sufficiently wrapped decorated posets which admit cylindrization data.

\begin{lemma}\label{cancylindrizeposet}
There exists a sufficiently wrapped countable cofinite decorated poset $P$ together with cylindrization data as above, such that $P$ contains all isotopy classes of pairs $(L,K)$.
\end{lemma}

\begin{proof}
We appeal to the construction in Lemma \ref{pcanwrap}.
That is, we set $P=\ZZ_{\geq 0}$, which we color by isotopy classes of pairs $(L,K)$ so that each color appears infinitely often.
For each such color, we choose a cofinal sequence in the wrapping category $R_{L\times K}:=R_L\times R_K$, call it $L_0\times K_0\leadsto L_1\times K_1\leadsto\cdots$.
If $i_0^{L\times K}<i_1^{L\times K}<\cdots\in\ZZ_{\geq 0}$ denotes the sequence colored by $(L,K)$, we label $i_r^{L\times K}$ with $A_r$ for some choice of factorizations $L_{n_r}\times K_{n_r}\to A_r\to L_{n_{r+1}}\times K_{n_{r+1}}$ in $R_{L\times K}$ and indices $n_0<n_1<\cdots$, chosen by the following inductive procedure.
The factorization property allows us to choose $A_r$ so that it is transverse to all Lagrangians assigned previously to indices $<i_r^{L\times K}$, but it does not address the key condition that it be disjoint from the previously fixed standard neighborhoods of these Lagrangians.
To achieve this disjointness property, we use Lemma \ref{lem:pushoffcore}, which says that every Lagrangian admits cofinal wrappings disjoint at infinity from any fixed finite number of standard neighborhoods of Legendrians.
This means that for every index $n_r<\infty$, there exists a larger $n_{r+1}<\infty$ and a factorization $L_{n_r}\times K_{n_r}\to B_r\to L_{n_{r+1}}\times K_{n_{r+1}}$ in $R_{L\times K}$ in which $B_r$ satisfies the required disjointness property.
We now let $A_r$ be a perturbation of $B_r$.
\end{proof}

We now show how to obtain the desired $\ainf$-functor $\W_{X\times Y}^\pr\to\W_{X\times Y}$ lifting \eqref{kcylindrizecohomology} from any sufficiently wrapped countable decorated poset which admits cylindrization data.
Pushing forward $P$ under the symplectomorphism determined by the cylindrization data gives a decorated poset $\tilde P$ for $\W_{X\times Y}$.

There is hence a tautological isomorphism $\OO_P=\OO_{\tilde P}$.
The diagram
\begin{equation}
\begin{tikzcd}
H^\bullet\OO_P\ar[r,equal]\ar[d]&H^\bullet\OO_{\tilde P}\ar[d]\\
H^\bullet\F_{X\times Y}^\pr\ar[r]&H^\bullet\W_{X\times Y}
\end{tikzcd}
\end{equation}
commutes by definition of the bottom arrow.
As noted earlier, the bottom arrow factors uniquely through $H^\bullet\W_{X\times Y}^\pr$, so we can put this category in the lower left corner as well.
Thus morphisms in $H^0\OO_P$ which are sent to isomorphisms in $H^\bullet\W_{X\times Y}^\pr$ are sent to morphisms in $H^0\OO_{\tilde P}$ which are sent to isomorphisms in $H^\bullet\W_{X\times Y}$.
There is thus an induced map of localizations $\W_P\to W_{\tilde P}$.
Since $P$ is sufficiently wrapped and contains all isotopy classes of pairs of Lagrangians, we have $\W_P=\W_{X\times Y}^\pr$.
By including $\tilde P$ into a sufficiently wrapped countable cofinite decorated poset, we obtain a functor $\W_{\tilde P}\to\W_{X\times Y}$.
Combining these, we obtain a functor
\begin{equation}\label{kcylindrize}
\W_{X\times Y}^\pr\to\W_{X\times Y}.
\end{equation}
On cohomology categories, it obviously coincides with \eqref{kcylindrizecohomology} (factor any morphism in $H^\bullet\W_{X\times Y}^\pr=H^\bullet\W_P$ into a composition of morphisms in $\OO_P$ and their inverses using Lemma \ref{pwrappingworks}).
In particular, it is fully faithful by Corollary \ref{cylff}.

\begin{remark}
Note that we have \emph{not} shown the $\ainf$-functor \eqref{kcylindrize} to be independent of the choice of $P$ and cylindrization data (although we certainly expect it to be).
What makes this independence difficult is that the cylindrization functor requires choosing a single isotopy $\Phi$ for all Lagrangians in $P$.
Independence would presumably be easy given an adaptation of the above argument in which we can choose different isotopies for each of the Lagrangians in $P$.
The most straightforward way of doing this leads to considering the pseudo-holomorphic curve equation with Hamiltonian term, to which the dissipative setup of \S\ref{noncyl} does not immediately apply (but one could certainly try to generalize it).
\end{remark}

\begin{proof}[Proof of Theorem \ref{kunneth}]
The K\"unneth functor for stopped Liouville manifolds is the composition of the equivalence $\W_X\otimes\W_Y\xrightarrow\sim\W_{X\times Y}^\pr$ from Proposition \ref{kunnethrepresentable} and the functor $\W_{X\times Y}^\pr\hookrightarrow\W_{X\times Y}$ \eqref{kcylindrize}, which is fully faithful by Corollary \ref{cylff}.
This induces a K\"unneth functor for Liouville sectors using Proposition \ref{prop: product sectors versus stopped} and Corollary \ref{horizsmallstop}.
\end{proof}

\subsection{Stabilization by \texorpdfstring{$T^\ast[0,1]$}{T*[0,1]} or \texorpdfstring{$(\CC,\{\pm\infty\})$}{(C,\{+-infty\})}}\label{stabilizationsubsection}

The K\"unneth functor in the form of Theorem \ref{kunneth} proved in the previous subsections gives rise, as a special case, to functors
\begin{align}
\W(X)&\hookrightarrow\W(X\times T^\ast[0,1]),\\
\W(X,\f)&\hookrightarrow\W((X,\f)\times(\CC,\{\pm\infty\})),
\end{align}
(the former for Liouville sectors, the latter for stopped Liouville manifolds).
To define these functors from the K\"unneth bilinear functors
\begin{align}
\W(X)\otimes\W(T^\ast[0,1])&\hookrightarrow\W(X\times T^\ast[0,1]),\\
\W(X,\f)\otimes\W(\CC,\{\pm\infty\})&\hookrightarrow\W((X,\f)\times(\CC,\{\pm\infty\})),
\end{align}
it suffices to fix functors
\begin{align}
\label{ZtoTstar}\ZZ&\to\W(T^\ast[0,1]),\\
\label{ZtoCstop}\ZZ&\to\W(\CC,\{\pm\infty\}),
\end{align}
where $\ZZ$ denotes the $\ainf$-category with a single object $\ast$ with endomorphism algebra $\ZZ$, sending this object to $\fiber\in\W(T^\ast[0,1])$ and $i\RR\in\W(\CC,\{\pm\infty\})$, respectively.
The construction of such functors \eqref{ZtoTstar}--\eqref{ZtoCstop} can be
obtained by sending $\ast$ to a fixed representative $K$ of $\fiber \in
\W(T^\ast[0,1])$ (respectively of $i \RR \in\W(\CC,\{\pm\infty\})$), and the
generator $1 \in \ZZ$ to the strict unit in $\Hom(K,K)$ (note that our construction
of partially wrapped Fukaya categories produces categories which are strictly unital).

\subsection{Pulling back the diagonal}\label{diagonalsection}

In this section we prove Proposition \ref{kunnethdiagonal}, which identifies the pullback of (an appropriate perturbation of) the diagonal under the K\"unneth functor with the diagonal bimodule.

Up until now, it has sufficed in the foundations of wrapped Fukaya categories \S\S\ref{abstractfloersec}--\ref{wrapdefsec} to work with sufficiently wrapped decorated posets.
By contrast, our arguments here will require decorated posets which are sufficiently bi-wrapped in the sense of Remark \ref{opposites} (and thus we will be using the correspondingly the bi-wrapped variant of Lemma \ref{pcanwrap} throughout this subsection).
Informally speaking, that means that, rather than just `wrapping the first factor forwards', we will also need to `wrap the second factor backwards'.
This is not so surprising given that to even state Proposition \ref{kunnethdiagonal} requires knowing that $\W(X^-,\f)=\W(X,\f)^\op$ (Lemma \ref{oppositesw}), an assertion which manifestly requires a comparison between wrapping the first and second factors.

Given a stopped Liouville manifold $(X,\f)$, we let $\Delta=\Delta_X\subseteq X^-\times X$ denote the geometric diagonal; it is Lagrangian.
Grading/orientation data for the diagonal is defined following \cite[Remark 3.0.7]{wehrheimwoodwardquilted}.
When $\f\ne\varnothing$, the geometric diagonal $\Delta$ runs into the product stop at infinity, hence is not an object of $\W_{X^-\times X}$ (recall our running abbreviation throughout this section $\W_{X^-\times X}=\W((X^-,\f)\times(X,\f))$).
To be more precise, the boundary at infinity of $\Delta$ is given by
\begin{equation}
\partial_\infty\Delta_X=0\times\Delta_{\partial_\infty X}\subseteq\RR\times\partial_\infty X^-\times\partial_\infty X\subseteq\partial_\infty(X^-\times X).
\end{equation}
Unless $\f=\varnothing$, this intersects the product stop, which inside $\RR\times\partial_\infty X^-\times\partial_\infty X$ is given by $\RR\times\f\times\f$.

A \emph{positive/negative pushoff} of $\Delta$ is a positive/negative cylindrical isotopy $\Delta_r$ starting at $\Delta_0=\Delta$ and defined for small $r\geq 0$, such that $\Delta_r$ is disjoint from the product stop for all $r>0$.
We will also abuse terminology and refer to the single Lagrangian $\Delta^\pm=\Delta_r$ for any $r>0$ (rather than the entire isotopy) as a positive/negative pushoff of $\Delta$.

Any exact cylindrical symplectomorphism $\Phi:X\to X$ has a graph $\Gamma_\Phi=\{(p,\Phi(p))\in X^-\times X\}\subseteq X^-\times X$ which is an exact cylindrical Lagrangian.
The graph of the identity is the diagonal $\Delta$, and every exact cylindrical perturbation of $\Delta$ is the graph of an exact cylindrical perturbation of the identity.
The graph $\Gamma_\Phi\subseteq X^-\times X$ is disjoint from the product stop precisely when $\f\cap\Phi(\f)=\varnothing$.
Thus a positive/negative pushoff of $\Delta$ is the same (up to contractible choice) as a positive/negative contact isotopy $\Psi_r$ of $\partial_\infty X$ with $\f\cap\Psi_r(\f)=\varnothing$ for $r>0$.
There may be many inequivalent ways (or no way) of positively/negatively displacing an arbitrary closed set $\f$, however we do have the following:

\begin{lemma}\label{diagonalpushoff}
A choice of ribbon for $\f$ determines, canonically up to contractible choice, positive/negative pushoffs $\Delta^\pm\subseteq X^-\times X$ disjoint from the product stop at infinity.
\end{lemma}

\begin{proof}
As discussed above, it is equivalent to show that there is a canonical up to contractible choice way of positively/negatively displacing the stop $\f\subseteq\partial_\infty X$.
Let $F_0\subseteq\partial_\infty X$ be a ribbon for $\f$ along with a contact form $\alpha$ defined in its neighborhood for which $(F_0,\alpha|_{F_0})$ is a Liouville domain.
The Reeb flow of $\alpha$ is transverse to $F_0$ (this is equivalent to non-degeneracy of $d\alpha|_{F_0}$), and thus provides the desired positive/negative displacement of $\f$.
Given $F_0$, the space of allowable $\alpha$ is convex, hence contractible.
\end{proof}

Henceforth the notation $\Delta^\pm$ shall refer to any fixed choice of positive/negative pushoff of $\Delta$.
Since they are disjoint from the product stop, they are well defined objects $\Delta^\pm\in\W_{X^-\times X}$.
When $\f=\varnothing$, of course $\Delta^+=\Delta^-=\Delta$ in $\W_{X^-\times X}$.

We now define modules $CW^\bullet(\Delta,-)$ and $CW^\bullet(\Delta^\pm,-)$ on the category $\W_{X^-\times X}$ (the latter two will be the Yoneda modules of $\Delta^\pm\in\W_{X^-\times X}$).
For any decorated poset $P$ for the abstract wrapped Floer setup underlying $\W_{X^-\times X}$, we can add an object $*$ greater than everything to obtain a poset $\{*>P\}$.
We label $*$ with $\Delta$ or $\Delta^\pm$, and we extend Floer data.
Everything is cylindrical, so the abstract wrapped Floer setup from \S\ref{wrapdefsec} suffices.
We do need to add the restriction that $(\Delta,L_0,\ldots,L_k)$ be mutually transverse for totally ordered subsets of $\{*>P\}$ starting at $*$.
Choosing $P$ which is sufficiently bi-wrapped (using the bi-wrapped variant of Lemma \ref{pcanwrap} from Remark \ref{opposites}) and localizing yields a category with semi-orthogonal decomposition $\langle\ZZ,\W_{X^-\times X}\rangle$ encoding the relevant right $\W_{X^-\times X}$-module.
Given any two such $P$ and $P'$, we can include $P\sqcup P'$ into a sufficiently bi-wrapped $P''$, which shows that the resulting modules are well-defined up to quasi-isomorphism (compare Lemma \ref{trimodulewelldefined}).

The negative isotopy $\Delta\leadsto\Delta^-$ has an associated continuation map in $HF^0(\Delta,\Delta^-)$ from Lemma-Definition \ref{continuationdef}.
This continuation map induces a map of $\W_{X^-\times X}$-modules
\begin{equation}
CW^\bullet(\Delta^-,-)\to CW^\bullet(\Delta,-)
\end{equation}
(define both modules simultaneously using Floer data for the poset $\{\Delta>\Delta^->P\}$).

\begin{lemma}\label{diagonalperturbationworks}
The continuation map $CW^\bullet(\Delta^-,-)\to CW^\bullet(\Delta,-)$ is a quasi-isomorphism of $\W_{X^-\times X}$-modules.
\end{lemma}

\begin{proof}
We just need to check that
\begin{equation}
HW^\bullet(\Delta^-,L)\to HW^\bullet(\Delta,L)
\end{equation}
is an isomorphism for cylindrical $L\subseteq X^-\times X$, where both sides are defined by negatively wrapping the second factor.
Let $\Delta_r$ be a negative isotopy beginning with $\Delta_0=\Delta$, so $\Delta_r$ is a negative pushoff $\Delta^-$ for every $r>0$.
Multiplication by the continuation map in $HF^0(\Delta_s,\Delta_r)$ defines an isomorphism $HW^\bullet(\Delta_r,L)\to HW^\bullet(\Delta_s,L)$ for $0<s<r$.
It thus suffices to show that
\begin{equation}
\varinjlim_{r\downarrow 0}HW^\bullet(\Delta_r,L)\to HW^\bullet(\Delta,L)
\end{equation}
is an isomorphism.
Crucially, direct limits commute with direct limits, so this map is the direct limit over negative wrappings $L^{-w}$ of $L$ of
\begin{equation}
\varinjlim_{r\downarrow 0}HF^\bullet(\Delta_r,L^{-w})\to HF^\bullet(\Delta,L^{-w}).
\end{equation}
For fixed $L^{-w}$ disjoint from $\partial_\infty\Delta$ at infinity, the maps $HF^\bullet(\Delta_r,L^{-w})\to HF^\bullet(\Delta,L^{-w})$ are isomorphisms for sufficiently small $r>0$, so we are done (as such $L^{-w}$ are certainly cofinal in the negative wrapping category of $L$).
\end{proof}

We can also define a module $CW^\bullet(\Delta,-)$ over the category $\W_{X^-\times X}^\pr$ by the same method as above, this time using product of cylindrical Floer data as used to define $\W_{X^-\times X}^\pr$.
To show that such Floer data provides compactness, we observe that product of cylindrical almost complex structures are cylindrical near $\partial_\infty\Delta$, hence are dissipative for $\Delta$ (see \S\ref{noncyl}).
Dissipativity of a pair $(\Delta,L\times K)$ with respect to product of cylindrical almost complex structures follows from a lower bound on the distance between $L$ and $K$ near infinity, which for $L$ and $K$ cylindrical is equivalence to compactness of $L\cap K$.
This is a generic condition, hence poses no issue.

\begin{lemma}\label{diagonalcylin}
For suitable choices of Floer data, the cylindrization functor $\W_{X^-\times X}^\pr\to\W_{X^-\times X}$ from \S\ref{kunnethcylindrization} pulls back the $\W_{X^-\times X}$-module $CW^\bullet(\Delta,-)$ to the $\W_{X^-\times X}^\pr$-module of the same name.
\end{lemma}

\begin{proof}
Observe that we could define these modules more generally using dissipative Floer data.
Using the inductive procedure in Remark \ref{opposites} and Lemma \ref{cancylindrizeposet}, we can build a sufficiently bi-wrapped decorated poset $P$ for $\W_{X^-\times X}^\pr$ which admits cylindrization data.
Further, take $P$ transverse to $\Delta$ so that we can build the $\W_{X^-\times X}^\pr$-module $CW^\bullet(\Delta,-)$ using Floer data for $\{\Delta>P\}$.
Now push forward the Floer data under the cylindrizing isotopy which defines the cylindrization functor, and include the resulting pushed forward poset into a sufficiently wrapped poset for $\W_{X^-\times X}$, and extend the pushed forward Floer data to define the $\W_{X^-\times X}$-module $CW^\bullet(\Delta,-)$.
Since the cylindrizing isotopy is supported in a neighborhood of $(\cc_{X^-}\times\partial_\infty X)\sqcup(\partial_\infty X^-\times\cc_X)$, it is, in particular, supported away from $\partial_\infty\Delta$.
\end{proof}

\begin{proposition}\label{diagonaltrimodule}
The pullback of the $\W_{X^-\times X}^\pr$-module $CW^\bullet(\Delta,-)$ under the functor $\W_{X^-}\otimes\W_X\to\W_{X^-\times X}^\pr$ from \S\ref{kunnethtrimodule} is the diagonal bimodule of $\W_X$.
\end{proposition}

\begin{proof}
Let us denote the functor from \S\ref{kunnethtrimodule} by $k:\W_{X^-}\otimes\W_X\to\W_{X^-\times X}^\pr$.
It is defined by the diagram
\begin{equation}
\begin{tikzcd}
\W_{X^-}\otimes\W_X\ar[r,"k"]\ar[rd,"\T"]&\W_{X^-\times X}^\pr\ar[d,"{Z\mapsto\Hom(-,Z)}"]\\
{}&\Mod\W_{X^-\times X}^\pr
\end{tikzcd}
\end{equation}
or, symbolically, $\T(-,L,K)=\W_{X^-\times X}^\pr(-,k(L,K))$.
In other words, the pullback under $k$ of any Yoneda module $\Hom(M,-)$ over $\W_{X^-\times X}^\pr$ is $\T(M,-,-)$.

We would like to show that the pullback under $k$ of $CW^\bullet(\Delta,-)$ is $\T(\Delta,-,-)$.
Note that this does \emph{not} follow from taking $M=\Delta$ above, since $\Delta$ is not an object of $\W_{X^-\times X}^\pr$ (which means moreover that $\T(\Delta,-,-)$ does not even have a definition \emph{a priori}).

To resolve this difficulty, we consider the category ${}^\Delta\W_{X^-\times X}^\pr$ with semi-orthogonal decomposition $\langle\ZZ,\W_{X^-\times X}^\pr\rangle$ encoding the $\W_{X^-\times X}^\pr$-module $CW^\bullet(\Delta,-)$ (recall that defining such a category is how this module was defined).
We extend the trimodule $\T$ to a $({}^\Delta\W_{X^-\times X}^\pr,\W_{X^-}\otimes\W_X)$-trimodule ${}^\Delta\T$ as follows.
We define ${}^\Delta\Q$ by counting the same pseudo-holomorphic curves as before, using product of cylindrical Floer data (note that we have already seen above that pairs $(\Delta,L\times K)$ are dissipative in this context provided $L\cap K$ is compact).
We now fix sufficiently bi-wrapped posets $P_{X^-}$, $P_X$, and $P_{X^-\times X}^\pr$, and we define ${}^\Delta\T$ to be the localization ${}_{(I_{X^-\times X}^\pr)^{-1}}({}^\Delta\Q)_{I_{X^-}^{-1},I_X^{-1}}$.
The cohomology of ${}^\Delta\Q$ is simply Floer cohomology.
Localizing ${}^\Delta\Q$ on the right by $I_{X^-}$ and $I_X$ is calculated by negatively wrapping the latter two arguments of ${}^\Delta\Q$ by Lemma \ref{pwrappingworks} since $P_{X^-}$ and $P_X$ are sufficiently op-wrapped.
The result is wrapped Floer cohomology $HW^\bullet(M,L\times K)^\pr$ defined by negatively wrapping $L$ and $K$.
Further localizing on the left by $C_{X^-\times X}^\pr$ does nothing to the cohomology since $HW^\bullet(-,L\times K)^\pr$ is left $I_{X^-\times X}^\pr$-local.
Now ${}^\Delta\T(-,L,K)$ is represented by $L\times K$ by the argument from Proposition \ref{kunnethrepresentable} (recall that this argument involved only cohomology).
We thus have a diagram
\begin{equation}
\begin{tikzcd}
\W_{X^-}\otimes\W_X\ar[r,"{}^\Delta k"]\ar[rd,"{}^\Delta\T"]&{}^\Delta\W_{X^-\times X}^\pr\ar[d,"{Z\mapsto\Hom(-,Z)}"]\\
{}&\Mod{}^\Delta\W_{X^-\times X}^\pr
\end{tikzcd}
\end{equation}
for some functor ${}^\Delta k$.
By definition, we have $({}^\Delta k)^*CW^\bullet(\Delta,-)={}^\Delta\T(\Delta,-,-)$.
We are interested instead in $k^*CW^\bullet(\Delta,-)$.

To relate $k^*CW^\bullet(\Delta,-)$ with ${}^\Delta\T(\Delta,-,-)$, we consider the maps
\begin{equation}
\begin{tikzcd}[row sep=small]
CW^\bullet(\Delta,-)\otimes_{\W_{X^-\times X}^\pr}CW^\bullet(-,k(-,-))\ar[d,equal]\ar[r,"\sim"]&CW^\bullet(\Delta,k(-,-))\\
CW^\bullet(\Delta,-)\otimes_{\W_{X^-\times X}^\pr}\T(-,-,-)\ar[d]\\
CW^\bullet(\Delta,-)\otimes_{{}^\Delta\W_{X^-\times X}^\pr}{}^\Delta\T(-,-,-)\ar[r,"\sim"]&{}^\Delta\T(\Delta,-,-)
\end{tikzcd}
\end{equation}
where the horizontal maps are quasi-isomorphisms by \cite[Lemma 3.7]{gpssectorsoc}.
Since $\T$ is the restriction of ${}^\Delta\T$, which is represented by ${}^\Delta k$, which lands in the full subcategory $\W_{X^-\times X}^\pr\subseteq{}^\Delta\W_{X^-\times X}^\pr$, the vertical map above has the form
\begin{equation}
\C(X,-)\otimes_\A\C(-,K)\to\C(X,-)\otimes_\C\C(-,K)
\end{equation}
for $\A\subseteq\C$ a full subcategory containing $K$, and such a map is a quasi-isomorphism since both sides map quasi-isomorphically to $\C(X,K)$ by \cite[Lemma 3.7]{gpssectorsoc}.

It remains to show that ${}^\Delta\T(\Delta,-,-)$ is the diagonal bimodule of $\W_X$, which is a variant of the standard fact that one can erase diagonal seam conditions when counting holomorphic quilts.
As shown above, the localization map ${}^\Delta\Q(\Delta,-,-)_{I_{X^-}^{-1},I_X^{-1}}\to{}^\Delta\T(\Delta,-,-)$ is a quasi-isomorphism, so it suffices to work with the former.
Now the decorated posets $P_{X^-}$ and $P_X$ together with the bimodule ${}^\Delta\Q(\Delta,-,-)$ define an $\ainf$-category $\langle\OO_{P_{X^-}}^\op,\OO_{P_X}\rangle_{{}^\Delta\Q(\Delta,-,-)}$ (semi-orthogonal gluing along the bimodule).
The crucial observation is that
\begin{equation}
\langle\OO_{P_{X^-}}^\op,\OO_{P_X}\rangle_{{}^\Delta\Q(\Delta,-,-)}=\OO_{\{P_{X^-}^\op>P_X\}}
\end{equation}
for certain Floer data decorating $\{P_{X^-}^\op>P_X\}$ extending the decorations on $P_{X^-}$ and $P_X$ (namely by taking the Floer data used to define ${}^\Delta\Q(\Delta,-,-)$, observing that the middle seam on the left of Figure \ref{figuretrimodule} can simply be eliminated since it is labelled with the diagonal, being careful of course to use Floer data for ${}^\Delta\Q(\Delta,-,-)$ so that the resulting Floer data on the `unseamed' strip is smooth!); this is where the choice of grading/orientation data for $\Delta$ is used, compare \cite[Lemma 3.0.12]{wehrheimwoodwardquilted}.
This identification is compatible with localization by $I_{X^-}$ and $I_X$, that is we have
\begin{equation}
\langle\W_{P_{X^-}}^\op,\W_{P_X}\rangle_{{}^\Delta\Q(\Delta,-,-)_{I_{X^-}^{-1},I_X^{-1}}}=\OO_{\{P_{X^-}^\op>P_X\}}[I_{X^-}^{-1},I_X^{-1}].
\end{equation}
By including $\{P_{X^-}^\op>P_X\}$ into something sufficiently wrapped, we obtain a map from the left hand side to $\W_X$.
The restriction of this map to $\W_{P_X}$ is the `identity', as is its restriction to $\W_{P_{X^-}}^\op$.
Moreover, its restriction to the bimodule of morphisms from the latter to the former is a quasi-isomorphism since both domain and target are wrapped Floer cohomology.
This identifies ${}^\Delta\Q(\Delta,-,-)_{I_{X^-}^{-1},I_X^{-1}}$ with the diagonal bimodule of $\W_X$, so we are done.
\end{proof}

\begin{proof}[Proof of Proposition \ref{kunnethdiagonal}]
Combine Lemmas \ref{diagonalperturbationworks} and \ref{diagonalcylin} with Proposition \ref{diagonaltrimodule}.
\end{proof}

\subsection{Fukaya--Seidel categories}\label{fscomparison}

In this subsection, we explain how the methods of \S\ref{kunnethcylindrization} can be applied without much modification to compare models of the so-called \emph{Fukaya--Seidel category}.

Throughout this paper, we use the term `Fukaya--Seidel category' interchangeably with `partially wrapped Fukaya category' for stopped Liouville manifolds $(E,\f)$ in which the stop, at least morally, takes the form $\f=\{-\infty\}\times\cc_F$ for some chart $\CC_{\leq 0}\times F\subseteq E$ which extends to an exact symplectic fibration $\pi:E\to\CC$.
This `cylindrical' model from \S\ref{wrapdefsec} (so-called since it uses cylindrical Lagrangians) is, however, superficially rather different from the various constructions termed the `Fukaya--Seidel category' in the literature \cite[(15a)]{seidelbook}\cite[\S 2]{maydanskiyseidel}\cite{girouxpardon}\cite[(5b)]{seidellefschetzvi} which rely instead on `thimble-like' Lagrangians.
We will not attempt to summarize these definitions, nor to formulate a general framework encompassing all of them.
Rather, we will consider just a certain very simple setup, leaving it to the reader to make the necessary adjustments (not necessarily all that trivial) in their geometric setup of interest.

We fix a Liouville manifold $E$ and a proper codimenion zero Liouville embedding $\sigma:F\times\CC_{\Re\geq 0}\hookrightarrow E$.
We will consider the stop $\f=\cc_F\times\{i\infty\}\subseteq\partial_\infty(F\times\CC_{\Re\geq 0})\subseteq\partial_\infty E$.

\begin{definition}
A Lagrangian $L\subseteq E$ is called \emph{thimble-like} (with respect to $\sigma$) when outside of a compact subset of $E$ it coincides with a disjoint union $\bigsqcup_{\theta\in\Theta}L_\theta\times[e^{i\theta}\RR_{\geq 0}]$ for some finite subset $\Theta=\Theta_L\subseteq(-\pi,\pi)=(\partial_\infty\CC_{\Re\geq 0})^\circ$ and non-empty Lagrangians $L_\theta\subseteq F$.
\end{definition}

A thimble-like Lagrangian $L$ whose fibers $L_\theta$ are exact has a cylindrization $\tilde L$, since the definition of cylindrization in \S\ref{productcylindrical} is a local operation which we can just do in the product chart $F\times\CC_{\Re\geq 0}$.

We now form an abstract wrapped Floer setup using thimble-like Lagrangians (similar to that in \S\ref{Wprod}) and dissipative almost complex structures.
The exact symplectic manifold $E$ is dissipative since it is Liouville.
Exact thimble-like Lagrangians are dissipative since they are isotopic to cylindrical Lagrangians (namely their cylindrizations), which are dissipative.
Pairs of thimble-like Lagrangians with disjoint sets of asymptotic angles are dissipative since they can be simultaneously cylindrized.
(One is of course free to use dissipative almost complex structures which in addition satisfy other nice properties, such as a maximum principle for pseudo-holomorphic curves, provided one proves they exist.)
This gives an abstract Floer setup following Definition \ref{geomtoabstract}, in which a tuple $(L_0,\ldots,L_k)$ is composable iff it is mutually transverse and the sets of asymptotic angles $\Theta_{L_0},\ldots,\Theta_{L_n}$ are disjoint.

Isotopy invariance Lemma-Definition \ref{isotopyinvariancedefdissipative} allows us to define $hF^\bullet(L,K)$ for arbitrary (not necessarily transverse) pairs of thimble-like Lagrangians as in \S\ref{wrapdefsec}, where $L^+$ is defined by positively perturbing $\Theta_L$ (e.g.\ using a rotational Hamiltonian on $\CC$, restricted to $\CC_{\Re\geq 0}$, and extended to the rest of $E$ arbitrarily).
The algebras $HF^\bullet(L,L)$ are identified with $HF^\bullet(\tilde L,\tilde L)$ by isotopy invariance, and hence they are unital (note that unitality is a \emph{property}, and one can almost certainly prove it without appealing to cylindrization).
Now Lemma-Definition \ref{continuationdef} defines continuation maps for positive isotopies of thimble-like Lagrangians (an isotopy of thimble-like Lagrangians is deemed positive when the set of asymptotic angles $\Theta$ moves in the positive direction inside $\partial_\infty\CC_{\Re\geq 0}$).

We claim that the result is an abstract wrapped Floer setup.
We define the wrapping category of a thimble-like Lagrangian as in \S\ref{wrappingcatsec}, namely as a slice category in the category whose objects are thimble-like Lagrangians and whose morphisms are positive isotopies modulo deformation rel endpoints.
Wrapped Floer cohomology (by definition the direct limit over wrapping) $HW^\bullet(L,K)$ is given by $HF^\bullet(L^w,K)$ for any positive isotopy $L\leadsto L^w$ with the property that all elements of $\Theta_{L^w}$ lie above $\Theta_K$ (further positive isotopy of $L^w$ is then necessarily disjoint from $K$ at infinity, hence leaves $HF^\bullet(L^w,K)$ invariant by isotopy invariance).
Right locality now follows as in Lemma \ref{wraprightlocal}.
We have thus defined an abstract wrapped Floer setup, and we denote its wrapped Fukaya category by $\FS(E,\sigma)$.

\begin{proposition}
There is a canonical fully faithful embedding $\FS(E,\sigma)\hookrightarrow\W(E,\f)$.
\end{proposition}

\begin{proof}
Apply the argument of the previous subsection \S\ref{kunnethcylindrization} in the product chart $F\times\CC_{\Re\geq 0}\subseteq E$.
\end{proof}

It often happens that $\W(E,\f)$ is generated by cylindrizations of thimble-like Lagrangians, making the embedding $\FS(E,\sigma)\hookrightarrow\W(E,\f)$ a pre-triangulated equivalence.
For example, this is the case when $(E,\f)$ is a Lefschetz fibration by Corollaries \ref{lefschetzgeneration} and \ref{lefschetzgenerationII} (modulo a comparison, similar to \cite[\S 6.2]{girouxpardon}, between the cylindrical cocores which we have been calling `Lefschetz thimbles' and the thimble-like `true' Lefschetz thimbles).

\section{Generation by cocores and linking disks}\label{generationsection}

\subsection{Products of cocores and linking disks}\label{cocorediskkunneth}

For the purpose of proving Theorem \ref{generation}, we require an understanding of how cocores and linking disks are transformed under taking products.

Obviously, for Weinstein manifolds $X$ and $Y$, the product $X\times Y$ is again Weinstein, and the product of a cocore in $X$ and a cocore in $Y$ is a cocore in $X\times Y$.
Moreover, every cocore in $X\times Y$ is of this form.
Note that cocores are strongly exact (the Liouville form vanishes identically on them) and hence cylindrization need not be discussed.
We now wish to generalize this discussion to stopped Weinstein manifolds.

Let $(X,\f)$ and $(Y,\g)$ be stopped Weinstein manifolds satisfying the hypothesis of Theorem \ref{generation}, fixing decompositions $\f=\f^\subcrit\cup\f^\crit$ and $\g=\g^\subcrit\cup\g^\crit$.
The product $X\times Y$ is a Weinstein manifold, and the product stop
\begin{equation}
\h:=(\cc_X\times\g)\cup(\f\times\g\times\RR)\cup(\f\times\cc_Y)
\end{equation}
is mostly Legendrian, admitting a natural decomposition $\h=\h^\subcrit\cup\h^\crit$.
To see this, first write $\cc_X=\cc_X^\subcrit\cup\cc_X^\crit$, where $\cc_X^\subcrit$ (resp.\ $\cc_X^\crit$) is the union of the cores of the subcritical (resp.\ critical) handles (and similarly for $\cc_Y$).
We may thus define
\begin{equation}
\h^\crit:=(\cc_X^\crit\times\g^\crit)\cup(\f^\crit\times\g^\crit\times\RR)\cup(\f^\crit\times\cc_Y^\crit),
\end{equation}
which is an open subset of $\h$.
Clearly $\h^\crit$ is a (locally closed) Legendrian submanifold (using the fact that the cocores of the critical handles of $X$ and $Y$ are disjoint at infinity from $\f$ and $\g$, respectively), and one may also check that $\h^\subcrit:=\h\setminus\h^\crit$ is closed and is a countable union of locally closed isotropic submanifolds.
Note also that cocores of $X\times Y$, being cocores of $X$ times cocores of $Y$, are disjoint from the product stop $\h$.

Let us now compare products of cocores/linking disks of $(X,\f)$ and $(Y,\g)$ with the cocores/linking disks of $(X\times Y,\h)$.
Obviously, products of cocores are cocores, and this accounts for all cocores in the product, just as in the case without stops.

Let us now consider products of cocores and linking disks.
Consider the linking disk inside $X$ at a given point of $\f^\crit$.
Recall from \S\ref{linkingdisksection} that this linking disk admits the following description in terms of Weinstein handles.
On the Lagrangian cylinder $\RR\times\f^\crit\subseteq X$, we introduce a pair of cancelling Weinstein handles of indices $0$ and $1$ (more precisely, coupled $(0,0)$- and $(1,1)$-handles), and the linking disk is a Lagrangian plane invariant under the Liouville flow intersecting $\RR\times\f^\crit$ precisely at the critical point of index zero.
The introduction of these handles increases the core $\cc_X$, however note that the new points of $\cc_X$ are entirely contained in $\RR\times\f$, so in particular the product stop $\h$ remains the same.
We now take this deformation of Liouville forms on $X$ and multiply with $Y$, noting that by the previous sentence, the product stop $\h$ remains the same.
The result is a deformation of the Liouville form on $X\times Y$ which introduces a pair of cancelling critical points of index $k$ and $k+1$ for every Weinstein $k$-handle of $Y$.
When $k$ is the critical index for $Y$, these critical points lie on $\RR\times\f^\crit\times\cc_Y^\crit$.
It follows that the product of a linking disk of $\f^\crit$ and a cocore of $Y$ is the linking disk of the corresponding component of $\f^\crit\times\cc_Y^\crit$.

Finally, let us consider products of linking disks and linking disks.
We consider the same description of the linking disks as before, namely in terms of a deformation of the Liouville forms on $X$ and $Y$ to introduce cancelling Weinstein handles of indices $0$ and $1$ on $\RR\times\f^\crit$ and $\RR\times\g^\crit$.
Now the result on the product is that the Liouville form is deformed near $\RR\times\f^\crit\times\g^\crit\times\RR$ to introduce four cancelling Weinstein handles of indices $0$, $1$, $1$, and $2$.
The product of the linking disks is now a Lagrangian plane invariant under the Liouville flow intersecting $\RR\times\f^\crit\times\g^\crit\times\RR$ precisely at the critical point of index $0$, and this is exactly the linking disk at the corresponding point of $\f^\crit\times\g^\crit\times\RR$.

\subsection{Proof of generation}

\begin{proof}[Proof of Theorem \ref{generation}]
Fix a stopped Weinstein manifold $(X,\f)$, and consider the K\"unneth embedding
\begin{align}
\W(X,\f)&\hookrightarrow\W(X\times\CC,\h)\\
L&\mapsto L\tildetimes i\RR
\end{align}
where
\begin{equation}
\h=(\cc_X\times\{\pm\infty\})\cup(\f\times\RR)\subseteq\partial_\infty(X\times\CC)=X\times\partial_\infty\CC\underset{\partial_\infty X\times\partial_\infty\CC\times\RR}\cup\partial_\infty X\times\CC.
\end{equation}
It is enough to show that the image under K\"unneth of any $L\in\W(X,\f)$ is generated by the images of the linking disks and cocores (see Lemma \ref{generationqiso}).
By the discussion in \S\ref{cocorediskkunneth}, these images are precisely the linking disks of the product stop $\h$.
Hence by the wrapping exact triangle Theorem \ref{wrapcone} and general position (Lemma \ref{genericpositiveisotopy}), it suffices to show that $L\times i\RR$ can be isotoped through $\h$ to a zero object.
For such an isotopy, we can simply take (the cylindrization of) $L$ times an isotopy of $i\RR$ through the stop at $+\infty\in\partial_\infty\CC$ to an arc contained in the lower half-plane $\CC_{\Im\leq 0}$, which is a zero object (and hence whose cylindrized product with $L$ is also a zero object by virtue of the K\"unneth functor).
\end{proof}

\begin{lemma}\label{halfplanevanish}
For any Liouville sector $X$, we have $\W(X\times\CC_{\Re\geq 0})=0$.
\end{lemma}

\begin{proof}
The Hamiltonian function $(\Re)^2$ is positive and linear (for the radial Liouville structure on $\CC$), generating the family of `shear' symplectomorphisms (in fact, Liouville isomorphisms) of $\CC$ given by $(x+iy)\mapsto(x+i(y+tx))$.
For any two cylindrical Lagrangians $L,K\subseteq X\times\CC_{\Re\geq 0}$, we have $e^{NX_{(\Re)^2}}L\cap K=\varnothing$ for sufficiently large $N<\infty$.
Indeed, the function $\frac yx$ is Liouville invariant (hence is bounded on any cylindrical Lagrangian) and is acted on by translation by the Hamiltonian vector field $X_{(\Re)^2}=x\partial_y$.
\end{proof}

\begin{corollary}\label{zeroobject}
If $L\subseteq X$ is in the image of any inclusion $Y\times\CC_{\Re\geq 0}\hookrightarrow X$, then $L$ is a zero object in $\W(X)$ (and in fact in $\W(X,\f)$ provided $\f\cap\partial_\infty(Y\times\CC_{\Re\geq 0})^\circ=\varnothing$).
\end{corollary}

\begin{proof}
If $L$ is contained in the \emph{interior} of the image of $Y\times\CC_{\Re\geq 0}$, then the conclusion holds since $\W(Y\times\CC_{\Re\geq 0})=0$ by Lemma \ref{halfplanevanish} and any $\ainf$-functor (in particular $\W(Y\times\CC_{\Re\geq 0})\to\W(X)$) sends zero objects to zero objects.
In general, simply isotope $L$ into the interior of $Y\times\CC_{\Re\geq 0}$ (e.g.\ using the Hamiltonian flow of a defining function $I$ for $Y\times\CC_{\Re\geq 0}$) and recall that isotopies of exact Lagrangians induce quasi-isomorphisms of objects in the wrapped Fukaya category.
\end{proof}

\begin{proof}[Proof of Theorem \ref{halfplanegeneration}]
Given any $L\subseteq F\times\CC_{\Re\geq 0}$ missing $\Lambda$ at infinity, the isotopy from the proof of Lemma \ref{halfplanevanish} eventually pushes $L$ to a zero object in $\W(F\times\CC_{\Re\geq 0},\Lambda)$.
Perturbing this isotopy to intersect $\Lambda$ only by passing through $\Lambda^\crit$ transversally via Lemma \ref{genericpositiveisotopy} and appealing to the wrapping exact triangle Theorem \ref{wrapcone} proves the claim.
\end{proof}

\begin{proof}[Proof of Corollary \ref{lefschetzgeneration}]
The Fukaya--Seidel category of our Lefschetz fibration $\pi:\bar X\to\CC$ is the wrapped Fukaya category of the Liouville sector obtained from $F\times\CC_{\Re\geq 0}$ by attaching critical Weinstein handles along the vanishing cycles inside (or, rather, their lifts to) $F\times\partial_\infty\CC_{\Re\geq 0}$.
Denoting by $\Lambda$ the union of these attaching loci, there is a functor
\begin{equation}
\W(F\times\CC_{\Re\geq 0},\Lambda)\to\W((F\times\CC_{\Re\geq 0})\cup_\Lambda(\text{handles})).
\end{equation}
The image of this functor generates the target by Theorem \ref{generation}, since the linking disks to $\Lambda$ are sent to the cocores of the added handles.
On the other hand, the domain category is generated by these linking disks by Theorem \ref{halfplanegeneration}, and the cocores of the added handles are precisely the Lefschetz thimbles.
\end{proof}

\begin{proof}[Proof of Theorem \ref{generationpractice}]
We follow the proof of Theorem \ref{generation}.
To see that the product stop $\h=(\cc_{X,\f}\times\{\pm\infty\})\cup(\f\times\RR)$ is mostly Legendrian, note that the Liouville vector field of $X$ is necessarily tangent to the smooth Lagrangian locus $\cc_{X,\f}^\crit\subseteq\cc_{X,\f}$, and hence $\cc_{X,\f}^\crit\times\{\pm\infty\}\subseteq X\times\partial_\infty\CC$ is Legendrian.
If $\f^\crit\subseteq\f$ is the smooth Legendrian locus, then $\f^\crit\times\RR\subseteq\partial_\infty X\times\CC$ is Legendrian.
The remainder of the product stop is $\h^\subcrit=(\cc_{X,\f}^\subcrit\times\{\pm\infty\})\cup(\f^\subcrit\times\RR)$, which is certainly covered by the smooth image of a second countable manifold of dimension less than Legendrian.
It thus suffices to show that for each generalized cocore $L\subseteq X$ intersecting $\cc_{X,\f}^\crit$ transversally at $p$, the cylindrized product $L\tildetimes i\RR$ is isomorphic in $\W(X\times\CC,\h)$ to the linking disk at $p\times\infty$.

We discuss first the case that the generalized cocore in question $L$ is invariant under the Liouville flow.
Consider isotoping the linking disk $i\RR$ of $(\CC,\{\pm\infty\})$ to make it (transversally) pass through $\infty$ (and thus become a zero object).
The (cylindrized) product of this isotopy with $L$ passes through the product stop $\h$ once, transversally, at $p\times\infty$ (note that this intersection occurs far away from where the product Lagrangian needs to be cylindrized, since $L$ is Liouville invariant).
In view of the wrapping exact triangle and the fact (which follows from K\"unneth) that the product of anything with a zero object is again a zero object, we conclude that $L\tildetimes i\RR$ is isomorphic in $\W(X\times\CC,\h)$ to the linking disk of the product stop at $p\times\infty$, as desired.

For case of general (i.e.\ not necessarily globally Liouville invariant) $L$, fix a primitive $f:L\to\RR$ of $\lambda_X|_L$.
Using the argument from the proof of Lemma \ref{cylvanishinf}, we may assume that $f$ is compactly supported and that $f$ vanishes at the transverse intersection point $L\cap\cc_{X,\f}=\{p\}$.
We may thus extend $f$ to all of $X$ so as to be compactly supported and to vanish on a neighborhood of $\cc_{X,\f}$, except possibly in a small neighborhood of $p$, where we may only demand that it vanish on $\cc_{X,\f}$ itself.
We now consider the family of stopped Liouville manifolds
\begin{equation}
(X\times\CC,\lambda_X-t\,df+\lambda_\CC,\partial_\infty(\cc_{X,\f}\times\RR))
\end{equation}
for $t\in[0,1]$.
Note that, by definition, $X_f$ vanishes over $\cc_{X,\f}$ except in a neighborhood of $p$ where it can be nonzero but tangent to $\cc_{X,\f}$; hence $(\cc_{X,\f}\times\RR)$ is indeed cylindrical, so the definition of the stop above makes sense (note that we make no claims about the core of $(X,\lambda_X-df)$).
Moreover, this family of stopped Liouville manifolds is nice at infinity: the only changes at infinity occurs where the stop is just a smooth Legendrian, so the picture at infinity simply undergoes a global contact isotopy.
In particular, the wrapped Fukaya categories are canonically identified, so to show agreement of $L\tildetimes i\RR$ with the linking disk in the category at $t=0$, it suffices to show agreement in the category at $t=1$.
Now $L$ is invariant under the Liouville flow of the Liouville form $\lambda_X-df$, and hence the argument from the previous paragraph applies.
\end{proof}

\begin{proof}[Proof of Corollary \ref{weinsteinkunneth}]
(Cylindrized) products of generalized cocores are generalized cocores, so $\W((X,\f)\times(Y,\g))$ is generated by product Lagrangians by Theorem \ref{generationpractice}.
\end{proof}

\begin{proof}[Proof of Corollary \ref{homologicallysmooth}]
By Corollary \ref{weinsteinkunneth}, the perturbed diagonal $\Delta^-\subseteq(X,\f)\times(X^-,\f)$ is generated by (cylindrizations of) product Lagrangians.
We now pull back under the K\"unneth embedding of $(X,\f)\times(X^-,\f)$.
The perturbed diagonal $\Delta^-$ pulls back to the diagonal bimodule by Proposition \ref{kunnethdiagonal}, and a cylindrized product pulls back to a tensor product of Yoneda modules by full faithfulness of K\"unneth.
\end{proof}

Recall that we call a Liouville manifold $X$ inessential when (after possibly deforming the Liouville form) there exists a closed contact submanifold of the contactization $X\times\RR$ which contains $\cc_X\times 0$.

\begin{lemma}\label{stabinessential}
$W\times\CC$ is inessential for any Liouville manifold $W$.
\end{lemma}

\begin{proof}
We are to show that $\cc_W\times 0\times 0$ inside $(W\times\CC\times\RR,\lambda_W+\lambda_\CC+dt)$ is contained in a closed contact submanifold.
Let $W_0\subseteq W$ be a Liouville domain completing to $W$, and consider the boundary $\partial(W_0\times D^2)$ (or rather its boundary after smoothing corners).
This times $0\in\RR$ is a closed contact submanifold of $W\times\CC\times\RR$, and so it suffices to describe an ambient contact isotopy which sends $\cc_W\times 0\times 0$ into it.

For such an isotopy, take the identity on $W$ times a contact isotopy of $\CC\times\RR$ sending $0\times 0$ to $1\times 0$ (either one should consider such an isotopy which, near the point $0\times 0$ as it moves, preserves the contact \emph{form}, or one should make the general observation that a contact isotopy of $(Y,\xi)$ is the same thing as a Liouville isotopy of its symplectization, which we can multiply by $W$ to obtain the symplectization of the contact manifold $W\times Y$ for any Liouville manifold $W$, so contact isotopies of $Y$ induce contact isotopies of $W\times Y$).
\end{proof}

\begin{proof}[Proof of Proposition \ref{removeinessential}]
For simplicity of notation, let us ignore $\f$.
Instead of $H_0$, we can use its contactization $H_0\times[0,1]$ as the stop.
We then have functors
\begin{equation}
\W(X,H_0\times[0,1])\to\W(X,Q)\to\W(X,\cc_H\times\{{\textstyle\frac 12}\})\to\W(X)
\end{equation}
where $Q\subseteq H_0\times[0,1]$ is a closed contact submanifold containing $\cc_X\times\{\frac 12\}$.
The composition of the first two of these functors is an equivalence by shrinking the stop, and the composition of the last two is fully faithful by Proposition \ref{removecontact}.
It follows (the `two-out-of-six property') that all functors above are fully faithful.
\end{proof}

\begin{proof}[Proof of Corollary \ref{inessentialvanish}]
The K\"unneth embedding $\W(X)\to\W(X\times(\CC,\{\pm\infty\}))$ is fully faithful, and stop removal $\W(X\times(\CC,\{\pm\infty\}))\to\W(X\times\CC)$ is fully faithful by Proposition \ref{removeinessential}.
The composition $\W(X)\to\W(X\times\CC)$ is thus fully faithful, yet it sends $L$ to $L\tildetimes[i\RR]$ which is a zero object of the target by K\"unneth since $[i\RR]\in\W(\CC)$ is a zero object.
\end{proof}

\begin{proof}[Proof of Corollary \ref{lefschetzgenerationII}]
We are to show that the wrapped Fukaya category of
\begin{equation}
(X,\cc_F)=(F\times(\CC,\{-\infty\}))\cup_\Lambda(\text{handles})
\end{equation}
is generated by the cocores of the handles.
The core of $X$ is $\RR_{\leq 0}\times\cc_F$ union the cores of the handles.
Multiply $X$ by $(\CC,\{\pm\infty\})$, and consider the product stop.
Removing the cores of the handles times $\{\pm\infty\}$ from this product stop quotients by the stabilizations of the cocores by stop removal Theorem \ref{stopremoval}.
What remains of the product stop is $\cc_F$ times an interval, which is a core of $F\times\CC$.
Since $F\times\CC$ is inessential by Lemma \ref{stabinessential}, removing it is fully faithful by Proposition \ref{removeinessential}.
We have thus removed the entire stop, leaving us with $\W(X\times\CC)$, which vanishes by Corollary \ref{inessentialvanish} since $X\times\CC$ is inessential by Lemma \ref{stabinessential}.
We conclude that the quotient of $\W(X)$ by the cocores of the handles is zero, so $\W(X)$ is split-generated by the cocores.
On the other hand, the cocores form an exceptional collection (the boundary at infinity of $X$ is $F_0\times[0,1]$, so Reeb dynamics are trivial), and it is well known that everything split-generated by an exceptional collection is in fact generated by it (for a proof, see \cite[Lemma A.10]{gpswrappedconstructible}).
\end{proof}

\section{Stopped inclusions}\label{stoppedinclusionssection}

We now explore geometric conditions under which the pushforward functor on wrapped Fukaya categories $\W(X)\to\W(X')$ induced by an inclusion of Liouville sectors $X\hookrightarrow X'$ is fully faithful.
A simple example of such a situation, namely when the wrapping in $X'$ was literally the same as the wrapping in $X$, was explored in Lemma \ref{faithfulstopping}.
More generally, since wrapped Floer cohomology can be computed by wrapping \emph{only one factor}, it is enough to assume that when Lagrangians inside $X$ are wrapped inside $X'$, they never re-enter $X$ after leaving.
We introduce notions of \emph{forward stopped} and \emph{tautologically forward stopped} (and dually backward stopped) inclusions of Liouville sectors which are geometric conditions sufficient to ensure this, and hence to ensure full faithfulness of the pushforward functor on wrapped Fukaya categories.
This notion depends only on the contact geometry of the boundary at infinity.

\subsection{Review of convex hypersurfaces in contact manifolds}\label{convexreview}

Recall that a compact cooriented hypersurface $H$ inside a cooriented contact manifold $Y$ is called \emph{convex} iff there exists a contact vector field defined near $H$ which is positively transverse to $H$.
Such a contact vector field determines a partition
\begin{equation}\label{Vpartition}
H=H_+\cup_\Gamma H_-
\end{equation}
where $\Gamma\subseteq H$ (called the \emph{dividing set}) is a hypersurface transverse to the characteristic foliation (in particular, avoiding its singularities) and all of the positive (resp.\ negative) tangencies of $\xi$ to $H$ are contained in $H_+$ (resp.\ $H_-$).
Namely, $H_\pm\subseteq H$ are the loci where the contact vector field $V$ is positively/negatively transverse to $\xi$, and they meet along the locus $\Gamma\subseteq H$ where $V\in\xi$.
The dividing set $\Gamma$ 
is a transversely cut out hypersurface inside $H$ for \emph{every} transverse contact vector field $V$, and is
a contact submanifold of $Y$.
The unique contact forms $\lambda_\pm$ on $Y$ defined near $H_\pm$ by $\lambda_\pm(V)=\pm 1$ restrict to Liouville forms on $H_\pm$ whose Liouville vector fields are tangent to the characteristic foliation.
With these Liouville forms, $H_\pm$ are Liouville manifolds, whose contact boundaries are naturally identified with $\Gamma$.

The partition \eqref{Vpartition} depends on a choice of transverse contact vector field $V$.  For a more invariant notion, we may pass to the cores
$\cc_{H_\pm}$  of $H_\pm$, which do not depend on the choice of $V$.  The complement $H\setminus(\cc_{H_+}\cup\cc_{H_-})$ equipped with its characteristic foliation is (non-canonically) diffeomorphic to $\RR\times\Gamma$ equipped with the foliation by $\RR\times\{p\}$.
We thus obtain the following canonical intrinsically defined structure: 
\begin{equation}\label{betterpartition}
\cc_{H_+},\cc_{H_-}\subseteq H\qquad H\setminus(\cc_{H_+}\cup\cc_{H_-})\to\Gamma,
\end{equation}
where the second map is projection along the characteristic foliation (i.e.\ the ``contact reduction'' of the `even contact manifold' $H\setminus(\cc_{H_+}\cup\cc_{H_-})$).
In particular, the contact manifold $\Gamma$ is defined intrinsically in terms of $H$.
The image of any section of the projection $H\setminus(\cc_{H_+}\cup\cc_{H_-})\to\Gamma$ is the dividing set for some transverse contact vector field $V$.
The characteristic foliation of $H$ is cooriented (by the coorientations on $H$ and $\xi$), and let us fix conventions so that this coorientation is outward pointing along $\partial H_+$.
In particular, we have the following useful result:

\begin{lemma}\label{domainsinsideconvex}
A subdomain $A\subseteq H$ is the ``positive part'' of some partition \eqref{Vpartition} associated to a transverse contact vector field iff $\cc_{H_+}\subseteq A\subseteq H\setminus\cc_{H_-}$ and the characteristic foliation is outwardly transverse to $\partial A$.
More generally, a subdomain $A\subseteq H$ is a Liouville subdomain of the positive part of some partition \eqref{Vpartition} iff $A\subseteq H\setminus\cc_{H_-}$ and the characteristic foliation is outwardly transverse to $\partial A$.\qed
\end{lemma}

\begin{example}\label{convexexample}
Consider a Liouville pair $(\bar X,F)$.
Choose a contact form near $F_0$ whose restriction to $F_0$ makes it a Liouville domain; this determines coordinates $\RR_t\times F_0\hookrightarrow\partial_\infty\bar X$ near $t=0$ in which the contact form is given by $dt+\lambda_{F_0}$.
As in \cite[Definition 2.14]{gpssectorsoc}, removing $(\RR_{s>0}\times\RR_{\left|t\right|<\varepsilon}\times F_0^\circ,e^s(dt+\lambda_{F_0}))$ from $\bar X$ yields a Liouville sector $X$.
The boundary of $\partial_\infty X$ is thus $\partial(\RR_{\left|t\right|\leq\varepsilon}\times F_0)$, which admits the transverse contact vector field $t\frac\partial{\partial t}+Z_{F_0}$.
The dividing set is thus $\{t=0\}$, and the positive/negative cores are $\{t=\pm\varepsilon\}\times\cc_{F_0}$.
Any Liouville subdomain of $F_0$, embedded inside either side $\{t=\pm\varepsilon\}\times F_0$, is a Liouville subdomain of the positive/negative parts of $\partial(\RR_{\left|t\right|\leq\varepsilon}\times F_0)$.
\end{example}

\subsection{Forward/backward stopped cobordisms and inclusions} \label{forwardstoppedsec} 

We now define `forward stopped' and `tautologically forward stopped' inclusions of (stopped) Liouville sectors.  
Being (tautologically) forward stopped depends only on the contact boundary, and so we first define the corresponding notions for contact cobordisms.
The dual notions of being (tautologically) backward stopped simply mean (tautologically) forward stopped after negating the Liouville form / the coorientation of the contact structure.

A \emph{contact cobordism} shall mean a contact manifold-with-boundary $M$ whose boundary $\partial M$ is compact, convex, and has a specified decomposition into open pieces $\partial M=\partial_\inn M\sqcup\partial_\out M$, with $\partial_\out M$ cooriented in the outward direction and $\partial_\inn M$ in the inward direction.
Note that $\partial M$ is required to be compact, but $M$ itself is not.
Both $\partial_\inn M$ and $\partial_\out M$ have positive and negative cores $\cc_{(\partial_\inn M)_\pm}$ and $\cc_{(\partial_\out M)_\pm}$, respectively.
Wrapping in $M$ will tend to push towards $\cc_{(\partial_\out M)_+}\sqcup\cc_{(\partial_\inn M)_-}$ and away from $\cc_{(\partial_\inn M)_+}\sqcup\cc_{(\partial_\out M)_-}$ (a convenient class of positive contact vector fields which does precisely this is constructed in \cite[\S 2.9]{gpssectorsoc}).

\begin{figure}[ht]
\centering
\includegraphics[max width=.95\textwidth]{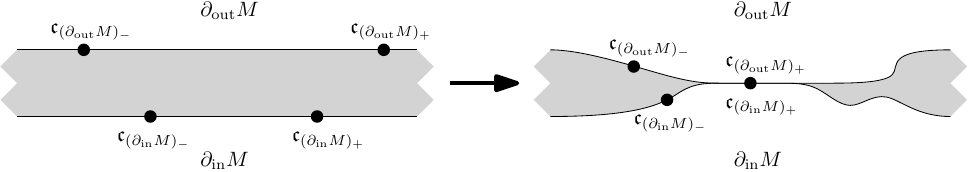}
\caption{Deformation of $M$ so that $\partial_\inn M$ and $\partial_\out M$ touch over $\cc_{(\partial_\inn M)_+}$.}\label{figuretautstop}
\end{figure}

\begin{definition}\label{tautforwardstoppeddef}
A contact cobordism $M$ is called \emph{tautologically forward stopped} iff it admits a compactly supported deformation (through contact cobordisms) which makes $\partial_\inn M$ touch $\partial_\out M$ over a neighborhood of the incoming positive core $\cc_{(\partial_\inn M)_+}$ so that it is disjoint from the outgoing negative core $\cc_{(\partial_\out M)_-}$ (see Figure \ref{figuretautstop}).
(This final ``pinched'' object is not strictly speaking a manifold-with-boundary, so to make this discussion precise, one may work with contact cobordisms equipped with a germ of codimension zero embedding into a contact manifold.)
\end{definition}

The contact cobordisms which we are interested in arise from inclusions of (possibly stopped) Liouville sectors.
Recall that the contact boundary $\partial_\infty X$ of a Liouville sector $X$ is a compact cooriented contact manifold with convex boundary $\partial\partial_\infty X$.
We coorient this boundary in the outward direction, so positive wrapping inside $X$ will tend to push Lagrangians towards the positive core $\cc_{(\partial\partial_\infty X)_+}$.
In the presence of a stop $\f\subseteq(\partial_\infty X)^\circ$ which does not approach the boundary $\partial\partial_\infty X$, the relevant cooriented contact manifold with convex boundary is $\partial_\infty X\setminus\f$.
For an inclusion of Liouville sectors $X\hookrightarrow X'$, the difference of their boundaries $M=\partial_\infty(X'\setminus X^\circ)$ is a contact cobordism, with $\partial_\inn M=\partial\partial_\infty X$ and $\partial_\out M=\partial\partial_\infty X'$.
More generally, for an inclusion of stopped Liouville sectors $(X,\f)\hookrightarrow(X',\f')$ satisfying the (rather strong) condition that $\f=\f'\cap(\partial_\infty X)^\circ$ and $\f'$ does not approach $\partial_\infty X$ or $\partial_\infty X'$, we may consider the contact cobordism $M=\partial_\infty(X'\setminus X^\circ)\setminus\f'$.

\begin{definition}\label{fsforsectors}
An inclusion of stopped Liouville sectors $(X,\f)\hookrightarrow(X',\f')$ (satisfying $\f=\f'\cap(\partial_\infty X)^\circ$ and that $\f'$ does not approach $\partial_\infty X$ or $\partial_\infty X'$) is called tautologically forward stopped iff the associated contact cobordism is.
\end{definition}

\begin{example}
Consider a Liouville sector $X$ including into a slight enlargement $X^+$ of itself, defined by flowing out by the Hamiltonian flow of a defining function for unit time.
Then $M=\partial_\infty(X^+\setminus X^\circ)$ is a contact cylinder, and shrinking $X^+$ back down to $X$ defines a pinching of $M$ which shows that $X\hookrightarrow X^+$ is tautologically forward stopped.
Hence every trivial inclusion of Liouville sectors $X\hookrightarrow X'$ is tautologically forward stopped (a trivial inclusion may, by definition, be deformed to $X\hookrightarrow X^+$).
\end{example}

\begin{proposition} \label{forwardexample}
Any inclusion of Liouville sectors $X\hookrightarrow X'$ may be factored into the composition of a 
forward stopped inclusion $X\hookrightarrow\check X'$ and a stop removal  $\check X'\hookrightarrow X'$.
\end{proposition}
\begin{proof}
We consider the Liouville hypersurface in $\partial_\infty X'$ defined as a small outward pushoff of some choice of positive part $(\partial\partial_\infty X)_+\subseteq\partial\partial_\infty X$ (in the sense of \eqref{Vpartition}).
Removing from $X'$ a neighborhood of this Liouville hypersurface as in Example \ref{convexexample} defines a Liouville sector $\check X'$.
Translating this excised neighborhood of the Liouville hypersurface back towards $\partial\partial_\infty X$ shows that the inclusion $X\hookrightarrow\check X'$ is tautologically forward stopped.
\end{proof}

\begin{example}\label{boundaryfiberinclusionstopping}
Consider a Liouville sector $X$ with $F$ (a component of) its symplectic boundary.
There is an inclusion of Liouville sectors $F \times T^*[0,1] \to X$, giving coordinates on a neighborhood of (a component of) the boundary of $X$.
This inclusion is not usually tautologically forward stopped: $\partial(F \times T^*[0,1]) = F \times (T^*_0 [0,1]  \cup T^*_1 [0,1])$, and only one of these components is tautologically forward stopped using $\partial X$.
We can stop the other component by adding a stop to $X$ as in Proposition \ref{forwardexample}.
Equivalently, we could apply Proposition \ref{forwardexample} directly to the inclusion $F \times T^*[0,1] \to \bar X$ where $\bar X$ is the convex completion of $X$ along the boundary component $F$.
The result is a tautologically forward stopped inclusion from $F\times T^\ast[0,1]$ into $\bar X$ stopped at $F_0$ and a small positive pushoff $F_0^+$ thereof.
Combining this with the K\"unneth stabilization functor, we obtain (using Corollary \ref{stoppedff} below) a fully faithful embedding $\W(F)\hookrightarrow\W(\bar X,F\cup F^+)$.
\end{example}

The properties of tautologically forward stopped inclusions which are relevant to controlling wrapping are in fact shared
by a more general class we term simply forward stopped inclusions.  In addition to being more general, the notion of forward stopped inclusions
will be more immediately connected to the geometry of wrapping.  However, in applications, we have found that the 
forward stopped inclusions of interest are in fact tautologically forward stopped.

\begin{figure}[ht]
\centering
\includegraphics[max width=.95\textwidth]{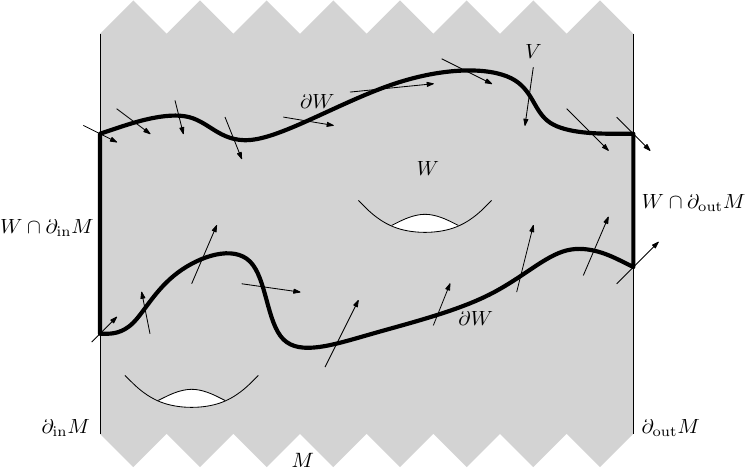}
\caption{Schematic diagram of the codimension zero submanifold-with-boundary $W$ and the contact vector field $V$ certifying that a contact cobordism $M$ is forward stopped.}\label{figureforwardstopped}
\end{figure}

\begin{definition}\label{forwardstoppeddef}
A contact cobordism $M$ is called \emph{forward stopped} iff there exists a closed subset $W\subseteq M$ and a positive contact vector field $V$ defined over $\Nbd\partial W$ (where $\partial W\subseteq M$ denotes the boundary in the point set topological sense) satisfying the following conditions (in which case we say $M$ is \emph{forward stopped by} $W$ or $(W,V)$ or that $W$ or $(W,V)$ is a \emph{forward stopping witness for $M$}):
\begin{enumerate}
\item\label{bwcpt}$\partial W$ is compact.
\item$W$ is sent into itself by the forward flow of $V$, namely for every trajectory $\gamma:[0,\varepsilon)\to M$ of $V$ with $\gamma(0)\in W$, we have $\gamma(t)\in W$ for all $t\in[0,\varepsilon)$.
\item\label{Vinward}$V$ is inward pointing along $\partial_\inn M$.
\item\label{Voutward}$V$ is outward pointing along $\partial_\out M$.
\item\label{bminusgood}$W\cap\partial_\inn M\subseteq\partial_\inn M$ is the ``positive piece'' $(\partial_\inn M)_+$ in some partition \eqref{Vpartition} of $\partial_\inn M$.
\item\label{bplusgood}$W\cap\partial_\out M\subseteq\partial_\out M$ is a Liouville subdomain of the ``positive piece'' $(\partial_\out M)_+$ in some partition \eqref{Vpartition} of $\partial_\out M$.
\end{enumerate}
(Recall Lemma \ref{domainsinsideconvex} for equivalent formulations of \ref{bminusgood}--\ref{bplusgood}.)
The adjective `forward stopped' may also be applied to inclusions of stopped Liouville sectors as in Definition \ref{fsforsectors}.
\end{definition}

The point of Definition \ref{forwardstoppeddef} is that nothing escapes from $W$ under the flow of suitably chosen positive contact vector fields (a precise such statement is given in Proposition \ref{forwardcofinal} and its proof).

\begin{lemma}\label{extendV}
Given a contact cobordism $M$ and a forward stopping witness $(W,V)$, there exist contact vector fields $T_\inn$ and $T_\out$ inwardly/outwardly transverse to $\partial_\inn M$ and $\partial_\out M$ which coincide with $V$ over $\Nbd(\partial_\inn M\cap\partial W)$ and $\Nbd(\partial_\out M\cap\partial W)$, respectively.
\end{lemma}

\begin{proof}
This follows from conditions \ref{Vinward}--\ref{bplusgood} and the discussion in \S\ref{convexreview}.
\end{proof}

\begin{proposition}\label{stoppeddeformation}
The property of being forward stopped is invariant under compactly supported deformation of contact cobordisms.
\end{proposition}

Despite its aesthetic appeal, from a logical standpoint this result is unnecessary in view of \cite[Lemmas 2.9 and 3.41]{gpssectorsoc}.

\begin{proof}
A deformation of $M$ is smoothly trivial, so we may regard $M$ as fixed and only the contact structure $\xi^t$ as varying.
Fix a forward stopping witness $(W^0,V^0)$ for $(M,\xi^0)$.
Let $T_\inn^t$ and $T_\out^t$ be a family of inward/outward pointing contact (with respect to $\xi^t$) vector fields defined over $\Nbd\partial_\inn M$ and $\Nbd\partial_\out M$, respectively, so that at time $t=0$ they agree with $V^0$ as in Lemma \ref{extendV}.
By modifying our smooth trivialization, we may assume that $T_\inn^t$ and $T_\out^t$ are independent of $t$, so let us just denote them by $\frac\partial{\partial r}$.

Let $\hat M$ denote the completion of $M$ obtained by gluing on cylindrical ends $\RR_{\leq 0}\times\partial_\inn M$ and $\RR_{\geq 0}\times\partial_\out M$ via $\frac\partial{\partial r}$; we may extend $\xi_t$ to be $\frac\partial{\partial r}$-invariant in these ends.
We may also extend $W^0\subseteq M$ to $\hat W^0\subseteq\hat M$ by gluing on $\RR_{\leq 0}\times(W\cap\partial_\inn M)$ and $\RR_{\geq 0}\times(W\cap\partial_\out M)$.
Note that the family $(M,\xi^t)$ is contactomorphic to $M_R:=M\cup([-N,0]\times\partial_\inn M)\cup([0,N]\times\partial_\out M)$ equipped with $\xi^t$ (to see this, it is enough to cut off the contact vector field $\frac\partial{\partial r}$ to make it supported in a small neighborhood of $\partial M$ and then use its flow).
Hence it is enough to show that $(M_R,\xi^1)$ is forward stopped.
Note that $(M_R,\xi^0)$ is forward stopped, as witnessed by $\hat W^0$ and the splicing of $V^0$ with $\frac\partial{\partial r}$.

By Gray's theorem, there is a unique family of vector fields $X_t\in\xi_t$ satisfying $\sL_{X_t}\xi_t=\frac\partial{\partial t}\xi_t$; uniqueness shows that $X_t$ commutes with $\frac\partial{\partial r}$ in the cylindrical ends.
If $M$ is non-compact, then note that $X_t$ are supported away from this non-compactness, since by assumption our given deformation is compactly supported.
We may now simply flow the stopping witness for $(M_R,\xi^0)$ under $X_t$ to obtain one for $(M_R,\xi^1)$ (we should take $R$ large enough so that the flowed witness remains cylindrical outside $M_R$).
\end{proof}

\begin{proposition}\label{tautisstopped}
A tautologically forward stopped contact cobordism may be deformed (away from infinity) to be forward stopped.
\end{proposition}

\begin{proof}
We will show that (a small deformation of) the limiting pinched cobordism $M$ from Definition \ref{tautforwardstoppeddef} is forward stopped.
By translating $\partial_\inn M$ outward by its transverse contact vector field $V$, we obtain a genuine contact cobordism $M'$.
Let $W\subseteq M'$ be the cylinder over (i.e.\ the sweepout under $V$ of) a small Liouville domain neighborhood of $\cc_{(\partial_\inn M)_+}\subseteq\partial_\inn M$.
Using Lemma \ref{domainsinsideconvex}, we see that $W\cap\partial_\inn M'$ is indeed the positive piece $(\partial_\inn M')_+$ in some partition \eqref{Vpartition} and that $W\cap\partial_\out M'$ is a Liouville subdomain of $(\partial_\out M')_+$ in some partition.
The vector field $V$ is tangent to $\partial W$ and thus $W$ is a stopping witness for $M$.
\end{proof}

We now argue that being forward stopped tells us something about wrapping, namely, that if $Y\hookrightarrow Y'$ is forward stopped, then if we wrap a given Legendrian submanifold of $Y$ inside $Y'$, once it exits $Y$ it stays within $W$ and never returns to $Y$.
More precisely, such wrappings (those which stay inside $W$ and never return to $Y$) are cofinal in all wrappings (wrappings which exit $W$ can of course exist; the claim is simply that they can be ignored for the purposes of calculating wrapped Floer cohomology).

\begin{proposition}\label{forwardcofinal}
Suppose $Y\hookrightarrow Y'$ is forward stopped by $W$, and let $\Lambda\subseteq Y^\circ$ be a compact Legendrian submanifold.
Consider wrappings $\Lambda\leadsto\Lambda^w$ inside $Y^\circ$ followed by wrappings $\Lambda^w\leadsto\Lambda^{ww}$ inside $(Y')^\circ$ supported inside (i.e.\ fixed outside) $W$ union any fixed small neighborhood of $W\cap\partial Y$ inside $Y$ and of $W\cap\partial Y'$.
The collection of all such compositions $\Lambda\leadsto\Lambda^{ww}$ is cofinal in the wrapping category of $\Lambda$ inside $Y'$.
\end{proposition}

\begin{proof}
Fix a positive contact vector field $V$ defined over $\Nbd\partial W$ satisfying the conditions of Definition \ref{forwardstoppeddef}.
Fix contact vector fields $T$ and $T'$ outwardly transverse to $\partial Y$ and $\partial Y'$ as in Lemma \ref{extendV}.
We use these transverse contact vector fields to fix coordinates near $\partial Y$ and $\partial Y'$ as in \cite[\S 2.9]{gpssectorsoc}.
We also fix contact forms near $\partial Y$ and $\partial Y'$ as in \cite[\S 2.9 Equations (2.22) and (2.26)]{gpssectorsoc} which evaluate to $1$ on $T$ and $T'$ over $\Nbd(W\cap\partial Y)$ and $\Nbd(W\cap\partial Y')$.
In particular, over $\Nbd(W\cap\partial Y)$ and $\Nbd(W\cap\partial Y')$, we have coordinates $\RR_{t\geq 0}\times(W\cap\partial Y)$ and $\RR_{t\geq 0}\times(W\cap\partial Y')$ (respectively) in which $T$ and $T'$ equal $-\frac\partial{\partial t}$ and the contact form equals $\lambda-dt$ and $\lambda'-dt$ for Liouville forms $\lambda$ and $\lambda'$ on $W\cap\partial Y$ and $W\cap\partial Y'$.

\begin{figure}[ht]
\centering
\includegraphics[max width=.95\textwidth]{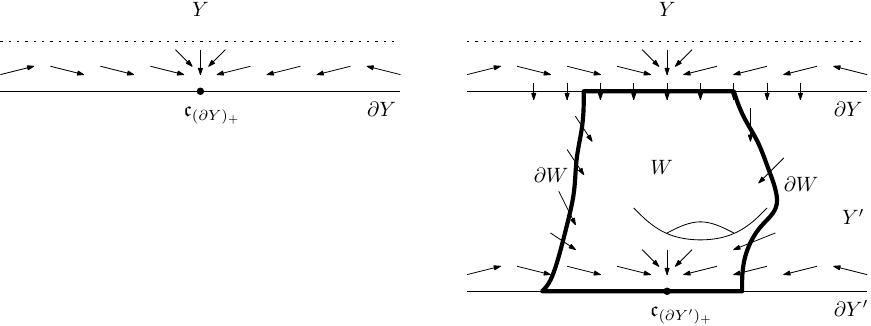}
\caption{The positive contact vector fields $V_1$ (left) and $V_2$ (right).}\label{figureforwardvectorfields}
\end{figure}

We define a positive contact vector field $V_1$ on $Y$ as follows.
Near $\partial Y$, we define $V_1$ by a contact Hamiltonian of the form $M(t)$ satisfying $M(0)=M'(0)=0$ and $M(t),M'(t)>0$ for $t>0$, thus of the form \cite[\S 2.9 Equations (2.23) and (2.29)]{gpssectorsoc} with the resulting excellent dynamics near $\partial Y$.
In particular, in the coordinates $\RR_{t\geq 0}\times(W\cap\partial Y)$ over $\Nbd(W\cap\partial Y)$, we have
\begin{equation}
V_1=-M'(t)Z_\lambda-M(t)\frac\partial{\partial t}.
\end{equation}
We extend $V_1$ to all of $Y$ to be complete (this is possible since $\partial W$ is compact: use the argument in the proof of \cite[Lemma 3.29]{gpssectorsoc}).
The vector field $V_1$ is illustrated on the left of Figure \ref{figureforwardvectorfields}.

We now define a positive contact vector field $V_2$ on $Y'$.
The contact vector field $V_2$ is obtained from $V_1$ by modifying $M(t)$ in a very small neighborhood of $\partial Y$ to not decay to zero but rather to a very small constant $\varepsilon>0$.
Thus $V_2$ equals $-\varepsilon\frac\partial{\partial t}=\varepsilon T=\varepsilon V$ over $\Nbd(\partial Y\cap\partial W)$, and we extend $V_2$ to all of $\Nbd\partial W$ as $\varepsilon\cdot V$.
Near $\partial Y'$ we declare $V_2$ to again be given by a contact Hamiltonian of the form $N(t)$ satisfying $N(0)=N'(0)=0$ and $N(t),N'(t)>0$ for $t>0$ and equalling $\varepsilon$ outside a very small neighborhood of $t=0$.
Examining the resulting form
\begin{equation}
V_2=-N'(t)Z_{\lambda'}-N(t)\frac\partial{\partial t}
\end{equation}
over $\Nbd(W\cap\partial Y')$, we see that the flow of $V_2$ attracts towards $W\cap\partial Y'$ in a neighborhood.
We extend $V_2$ to the rest of $Y'$ to be complete.
The vector field $V_2$ is illustrated on the right of Figure \ref{figureforwardvectorfields}.

We now analyze the effect of flowing a given compact Legendrian $\Lambda\subseteq Y^\circ$ under $V_1$ and/or $V_2$.
The image of $\Lambda$ under the forward flow of $V_1$ is bounded away from $\partial Y$ except for a neighborhood of $\cc_{(\partial Y)_+}\subseteq(W\cap\partial Y)$ (this follows from the explicit form of $V_1$ near $\partial Y$, see \cite[\S 2.9 Equations (2.23) and (2.29) and Proposition 2.35]{gpssectorsoc}).
Hence, fixing such a $\Lambda$, we may choose the modification $V_2$ of $V_1$ to take place in a sufficiently small neighborhood of $\partial Y$ so as to ensure that the forward flow of $\Lambda$ under $V_2$ exits $Y$ only along (the interior of) $W\cap\partial Y$.
Now $V_2$ is weakly inward pointing along $\partial W$, so the forward flow of $\Lambda$ under $V_2$ either stays entirely within $Y$ or exits $Y$ along $W\cap\partial Y$ and then stays entirely within $W$ (except over a small neighborhood of $W\cap\partial Y'$, where it is not necessarily weakly inward pointing along $\partial W$, but nevertheless attracts towards $W\cap\partial Y'$).
Furthermore, the same holds for the flow of $\Lambda$ under $aV_2+(1-a)V_1$ for any $a\in[0,1]$.

With our vector fields $V_1$ and $V_2$ defined and their dynamics understood, we may now conclude the proof.
The family of wrappings $\Lambda\leadsto e^{NV_2}\Lambda$ is cofinal, since $V_2$ is a complete positive contact vector field on $(Y')^\circ$ (use Lemma \ref{cofinalitycriterion}).
These wrappings may be factorized as $\Lambda\leadsto e^{NV_1}\Lambda\leadsto e^{NV_2}\Lambda$, where the second positive isotopy is $e^{N(aV_2+(1-a)V_1)}\Lambda$ for $a\in[0,1]$.
The first isotopy $\Lambda\leadsto e^{NV_1}\Lambda$ clearly takes place entirely inside $Y$.
The second isotopy $e^{N(aV_2+(1-a)V_1)}\Lambda$ is supported inside $W$ union a small neighborhood of $W\cap\partial Y$ inside $Y$ and of $W\cap\partial Y'$.
Indeed, the only points which move during this isotopy are those corresponding to trajectories which have reached a place where $V_1$ and $V_2$ differ, and such trajectories can only end inside $W$ union a small neighborhood of $W\cap\partial Y$ inside $Y$ and of $W\cap\partial Y'$.
\end{proof}

\subsection{Forward stopped implies fully faithful} 

\begin{corollary}\label{stoppedff}
Let $(X,\f)\hookrightarrow(X',\f')$ be an inclusion of stopped Liouville sectors which is forward stopped by $W$.
The pushforward functor $\W(X,\f)\hookrightarrow\W(X',\f')$ is fully faithful,
and its image is left-orthogonal to every $K\subseteq X'$ (disjoint from $\f'$ at infinity) which is disjoint from $X$ and disjoint at infinity from $W$.
\end{corollary}

\begin{proof}
For full faithfulness, we just need to argue that $HW^\bullet(L,K)_{X,\f}\to HW^\bullet(L,K)_{X',\f'}$ is a quasi-isomorphism.
According to Proposition \ref{forwardcofinal}, wrapping $L$ inside $X'$ away from $\f'$ may be described by first wrapping $L\leadsto L^w$ inside $X$ away from $\f$, and then wrapping $L^w\leadsto L^{ww}$ inside $W$ (union a small neighborhood of $\partial_-W$).
The second wrapping is disjoint from $\partial_\infty K\subseteq\partial_\infty X$, and hence the map $HF^\bullet(L^w,K)\xrightarrow\sim HF^\bullet(L^{ww},K)$ is an isomorphism by Lemma-Definition \ref{continuationdef}.
Taking the direct limit over pairs of wrappings $L\leadsto L^w\leadsto L^{ww}$ yields the map $HW^\bullet(L,K)_{X,\f}\xrightarrow\sim HW^\bullet(L,K)_{X',\f'}$ which is thus also an isomorphism.

The argument for orthogonality is the same: if $K\cap X = \varnothing$ and $\partial_\infty K \cap W = \varnothing$,
then every $L\subseteq X$ has cofinal wrappings away from $K$, and hence the direct limit computing $HW^\bullet(L, K)$ vanishes.
\end{proof}

\begin{example}
Given an inclusion of Liouville sectors $X \hookrightarrow X'$, consider the factoring constructed in Proposition \ref{forwardexample} 
into a forward stopped inclusion $X\hookrightarrow\check X'$ and a stop removal $\check X'\hookrightarrow X'$.  
Correspondingly the functor $\W(X) \to \W(X')$ factors into a fully faithful morphism
$\W(X) \hookrightarrow \W(\check X')$, followed by the map 
$\W(\check X') \to \W(X')$, which is a localization if $(\partial\partial_\infty X)_+$ is Weinstein by Theorem \ref{stopremoval}).
\end{example} 

\section{Pushout formulae} \label{pushoutsec}

In this section we derive various pushout formulae for partially wrapped Fukaya categories from Corollary \ref{stoppedff} (full faithfulness and semi-orthogonality for forward stopped inclusions).

\subsection{Viterbo restriction}\label{viterborestrictionsection}

We implement a proposal of Sylvan \cite{sylvantalk} for defining a restriction functor (conjecturally agreeing with Abouzaid--Seidel's Viterbo restriction functor \cite{abouzaidseidel} on the domain of the latter).

Let $X^\inn_0\subseteq X_0$ be an inclusion of Liouville domains, with completions $X^\inn$ and $X$.
Given any Liouville manifold $W$ and an embedding of $X_0$ as a Liouville hypersurface inside $\partial_\infty W$, we may consider the diagram
\begin{equation}\label{viterbogeneraldiagramfirst}
\begin{tikzcd}
\W(X)\ar{r}&\W(W,\cc_X)\ar{d}\\
\W(X^\inn)\ar{r}&\W(W,\cc_{X^\inn}),
\end{tikzcd}
\end{equation}
where the horizontal arrows are given by the composition of the K\"unneth stabilization functor $\W(X)\hookrightarrow\W(X\times T^\ast[0,1])$ with pushing forward under the tautological inclusion $X\times T^\ast[0,1]\hookrightarrow(W,\cc_X)$ near the stop $\cc_X$.
Given this diagram, it is natural to ask whether there is a natural functor $\W(X)\to\W(X^\inn)$ making the diagram commute.
Sylvan conjectured that Abouzaid--Seidel's (partially defined) functor $\W(X)\to\W(X^\inn)$ is such a functor.
Sylvan also observed that, in the opposite direction, one may, under certain assumptions, use the diagram \eqref{viterbogeneraldiagramfirst} in the special (in fact, universal) case $W=X\times\CC_{\Re\geq 0}$ to define \emph{a priori} a restriction functor $\Tw\W(X)\to\Tw\W(X^\inn)$, and conjectured that this recovers Abouzaid--Seidel's restriction functor on the domain of the latter.

To spell this out, observe that the inclusions of Liouville sectors
\begin{align}
\label{ffviterboI}X\times T^\ast[0,1]&\hookrightarrow(X\times\CC_{\Re\geq 0})\setminus\Nbd(X_0\times\{\infty\}),\\
\label{ffviterboII}X^\inn\times T^\ast[0,1]&\hookrightarrow(X\times\CC_{\Re\geq 0})\setminus\Nbd(X_0^\inn\times\{\infty\}),
\end{align}
are tautologically forward stopped (by following the trivial open book for the contact boundary of $X\times\CC$), and hence induce fully faithful pushforward functors by Corollary \ref{stoppedff}.
K\"unneth stabilization is always fully faithful, so we have a diagram whose horizontal arrows are fully faithful embeddings:
\begin{equation}\label{viterbospecialdiagram}
\begin{tikzcd}
\W(X)\ar[hook]{r}&\W(X\times\CC_{\Re\geq 0},\cc_X)\ar{d}\\
\W(X^\inn)\ar[hook]{r}&\W(X\times\CC_{\Re\geq 0},\cc_{X^\inn})
\end{tikzcd}
\end{equation}
(where we have used Corollary \ref{horizsmallstop} to relate the categories on the right with those on the right of \eqref{ffviterboI}--\eqref{ffviterboII}).
If $X^\inn$ is Weinstein, then Theorem \ref{halfplanegeneration} implies that the bottom horizontal arrow is essentially surjective on twisted complexes, and is thus a pre-triangulated equivalence.
Hence inverting it defines a restriction functor
\begin{equation}\label{restrictionfunctor}
\Tw\W(X)\to\Tw\W(X^\inn)
\end{equation}
making $\Tw\eqref{viterbospecialdiagram}$ commute.
In fact, by pushing forward under the canonical embedding $X\times\CC_{\Re\geq 0}\to W$ near a Liouville hypersurface $X_0\hookrightarrow\partial_\infty W$, we see that this restriction functor \eqref{restrictionfunctor} makes $\Tw\eqref{viterbogeneraldiagramfirst}$ commute as well.

\begin{remark}
The assumption that $X^\inn$ is Weinstein is used only to ensure essential surjectivity of $\W(X^\inn)\to\W(X\times\CC_{\Re\geq 0},\cc_{X^\inn})$; the restriction functor \eqref{restrictionfunctor} is defined whenever this holds.
More generally, one may define a ``Viterbo $(\W(X),\W(X^\inn))$-bimodule'' from \eqref{viterbospecialdiagram} by pulling back the diagonal bimodule from $\W(X\times\CC_{\Re\geq 0},\cc_{X^\inn})$ and ask whether it is representable (possibly by a twisted complex) and thus defines a functor \eqref{restrictionfunctor}.
It seems plausible that one could use geometric arguments to show this representability for any inclusion of Liouville domains $X^\inn_0\subseteq X_0$, possibly leading to a favorable comparison with the construction of Abouzaid--Seidel \cite{abouzaidseidel}.
\end{remark}

If we assume that both $X^\inn_0$ and the cobordism $X_0\setminus X^\inn_0$ are Weinstein (up to deformation), then we can say more about the restriction functor \eqref{restrictionfunctor} and the commutative diagrams \eqref{viterbogeneraldiagramfirst}:

\begin{proposition}
Suppose both $X^\inn_0$ and $X_0\setminus X^\inn_0$ are Weinstein.
The restriction functor \eqref{restrictionfunctor} is the quotient by the cocores of $X$ not in $X^\inn$, and the diagram
\begin{equation}\label{viterbogeneraldiagram}
\begin{tikzcd}
\Tw \W(X)\ar{r}\ar{d}[swap]{\eqref{restrictionfunctor}}&\Tw \W(W,\f\cup\cc_X)\ar{d}\\
\Tw \W(X^\inn)\ar{r}&\Tw \W(W,\f\cup\cc_{X^\inn})
\end{tikzcd}
\end{equation}
is a homotopy pushout for any stopped Liouville sector $(W,\f)$ with an embedding $X_0\hookrightarrow(\partial_\infty W)^\circ\setminus\f$ as a Liouville hypersurface.
\end{proposition}

\begin{proof}
Since both $X$ and $X^\inn$ are Weinstein, both horizontal functors in \eqref{viterbospecialdiagram} are essentially surjective on twisted complexes by Theorem \ref{halfplanegeneration} and are thus pre-triangulated equivalences.
The right vertical functor is the quotient by the linking disks of $\cc_X\setminus\cc_{X^\inn}$ by Theorem \ref{stopremoval}.
Since these are precisely the images of the cocores of $X$ not in $X^\inn$, it follows by definition that \eqref{restrictionfunctor} is the quotient by the same.

Now going in the reverse direction, the vertical maps in \eqref{viterbogeneraldiagram} are, respectively, quotienting by the cocores of $X$ not in $X^\inn$ and quotienting by their images the linking disks to $\cc_X\setminus\cc_{X^\inn}$.
We thus conclude that \eqref{viterbogeneraldiagram} is a homotopy pushout by Lemma \ref{hocolimlocalcommute}.
\end{proof}

\subsection{Sector gluing}\label{mvthmproofsec}

\begin{proof}[Proof of Theorem \ref{mvthm}]
Our first task is simply to setup nice coordinates near the splitting hypersurface $X_1\cap X_2$.
To construct these coordinates, we will deform the Liouville form on $X$ near the splitting hypersurface; this is permitted since such a deformation does not affect the validity of the statement we are trying to prove.
The discussion surrounding \cite[Proposition 2.25]{gpssectorsoc} shows that there exist coordinates near $X_1\cap X_2$ of the form $(F\times T^\ast[0,1],\lambda_F+\lambda_{T^\ast[0,1]}+df)$ where $f:F\times T^\ast[0,1]\to\RR$ has proper support over $T^\ast[0,1]$ and $f=f_\pm:F\to\RR$ away from a compact subset of $T^\ast[0,1]$ and $X_1\cap X_2=F\times N^\ast\{\frac 12\}$.
By assumption, the Liouville manifold $(F,\lambda_F)$ is Weinstein up to deformation, namely there exists a compactly supported function $g:F\to\RR$ such that $(F,\lambda_F+dg)$ is Weinstein.
We may, of course, absorb $g$ into $f$ so that $(F,\lambda_F)$ is itself Weinstein.
We now consider the obvious linear deformation of Liouville forms between $\lambda_F+\lambda_{T^\ast[0,1]}+df$ and $\lambda_F+\lambda_{T^\ast[0,1]}+d(\varphi(t)f)$ where $\varphi:[0,1]\to[0,1]$ is smooth and equals $1$ near the boundary and equals $0$ near $t=\frac 12$.
A calculation as in \cite[Proposition 2.28]{gpssectorsoc} shows that convexity at infinity is preserved under this deformation.
Since $\varphi$ vanishes in a neighborhood of $t=\frac 12$, in this deformed Liouville manifold there now exist coordinates
\begin{equation}\label{splittingcoords}
(F\times T^\ast[0,1],\lambda_F+\lambda_{T^\ast[0,1]})\to(X,\lambda_X)
\end{equation}
strictly respecting the Liouville form, where $(F,\lambda_F)$ is Weinstein, and locally $X_1=\{t\leq\frac 12\}$ and $X_2=\{t\geq\frac 12\}$.

Let $X_i^+$ denote a small enlargement of $X_i^+$, namely locally $X_1^+=\{t\leq\frac 47\}$ and $X_2^+=\{t\geq\frac 37\}$, so $X_1^+\cap X_2^+=F\times T^\ast[\frac 37,\frac 47]$ (see Figure \ref{figuregluingdiagram}).
There is thus a diagram of $\ainf$-categories, well defined up to quasi-equivalence,
\begin{equation}\label{secondpushout}
\begin{tikzcd}
\W(F\times T^\ast[\frac 37,\frac 47])\ar{r}\ar{d}&\W(X_1^+,\rr_1)\ar{d}\\
\W(X_2^+,\rr_2)\ar{r}&\W(X,\rr_1\cup\rr_2)
\end{tikzcd}
\end{equation}
(note that the stop $\rr=\rr_1\cup\rr_2$ is disjoint from the chart \eqref{splittingcoords} near $X_1\cap X_2$ constructed above).
It is this diagram of $\ainf$-categories which we would like to show is an almost homotopy pushout.

\begin{figure}[ht]
\centering
\includegraphics[max width=.95\textwidth]{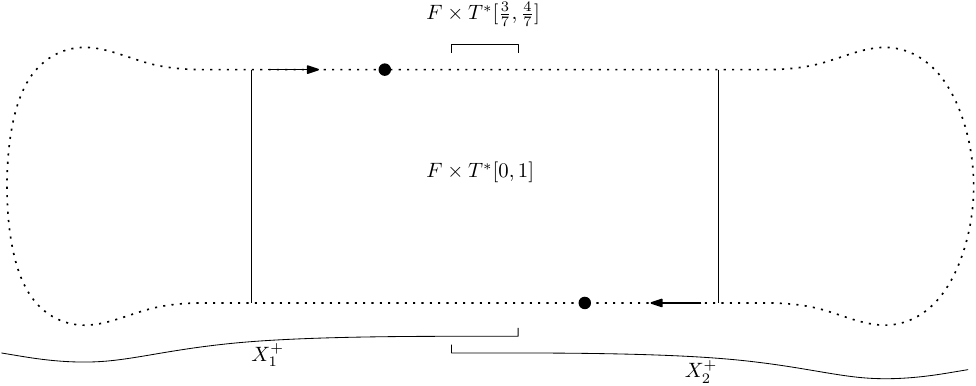}
\caption{Liouville sectors $X_i^+$ near the chart \eqref{splittingcoords}, with two dots indicating the location of the stops imposed to obtain $\check X_i^+$.
The arrows indicate the direction of the Reeb flow.}\label{figuregluingdiagram}
\end{figure}

We now introduce more stops into the picture, namely at $F_0\times\{\infty\cdot dt\}\times\{t=\frac 27\}$ and $F_0\times\{-\infty\cdot dt\}\times\{t=\frac 57\}$, where $F_0\subseteq F$ denotes any Liouville domain whose completion is $F$.
We denote the resulting Liouville sectors by $\check X$ and $\check X_i^+$; note that we still have $\check X_1^+\cap\check X_2^+=X_1^+\cap X_2^+=F\times T^\ast[\frac 37,\frac 47]$ since the added stops lie outside the region $[\frac 37,\frac 47]$.
There is thus a corresponding diagram of $\ainf$-categories
\begin{equation}\label{firstpushout}
\begin{tikzcd}
\W(F\times T^\ast[\frac 37,\frac 47])\ar{r}\ar{d}&\W(\check X_1^+,\rr_1)\ar{d}\\
\W(\check X_2^+,\rr_2)\ar{r}&\W(\check X,\rr_1\cup\rr_2).
\end{tikzcd}
\end{equation}
There is a map of diagrams $\eqref{firstpushout}\to\eqref{secondpushout}$.
By stop removal Theorem \ref{stopremoval} (after shrinking the stops modelled on $F_0$ down to their cores), each of the four maps comprising this map of diagrams is a quotient by the linking disks of the stops.
Thus by Lemma \ref{hocolimlocalcommute}, to show that \eqref{secondpushout} is an almost homotopy pushout, it suffices to show that \eqref{firstpushout} is an almost homotopy pushout.

We now argue that \eqref{firstpushout} is an almost homotopy pushout by appealing to Proposition \ref{semiorthhocolim} (in other words, we wish to show that \eqref{firstpushout} has, up to pre-triangulated equivalence, the same shape as in Example \ref{semiorthsquare}).
Let us first argue that the functors comprising \eqref{firstpushout} are all fully faithful.
By Corollary \ref{stoppedff}, it is enough to check that each of the corresponding inclusions of Liouville sectors is tautologically backward stopped.
We consider the inclusion $(\check X_1^+,\rr_1)\hookrightarrow(\check X,\rr_1\cup\rr_2)$.
The incoming core of the boundary of $\partial_\infty\check X_1^+$ is given by $\cc_F\times\{-\infty\cdot dt\}\times\{t=\frac 47\}$, and dragging the stop $F_0\times\{-\infty\cdot dt\}\times\{t=\frac 57\}$ from $t=\frac 57$ to $t=\frac 47$ shows that this inclusion is tautologically backward stopped (compare Proposition \ref{forwardexample} and Example \ref{boundaryfiberinclusionstopping}).
Identical considerations show that all the inclusions of Liouville sectors forming the diagram \eqref{firstpushout} are tautologically backward stopped, and so the functors are fully faithful.

Let us now ask for the left-orthogonal complements of $\W(F\times T^\ast[\frac 37,\frac 47])$ inside $\W(\check X_i^+,\rr_i)$ for $i=1,2$.
The Lagrangians inside $\{t\leq\frac 17\}$ (for $i=1$) or $\{t\geq\frac 67\}$ (for $i=2$) are left-orthogonal to $\W(F\times T^\ast[\frac 37,\frac 47])$ by Corollary \ref{stoppedff}.
Together with $\W(F\times T^\ast[\frac 37,\frac 47])$, these Lagrangians generate $\W(\check X_i^+,\rr_i)$ for $i=1,2$ by the wrapping exact triangle (push any Lagrangian until it lies in $\{t\leq\frac 17\}$ or $\{t\geq\frac 67\}$; this costs some number of linking disks, but these are generated by the subcategories in question by applying K\"unneth to $T^*[0,1]$ with a stop times $F$).
The Lagrangians inside $\{t\leq\frac 17\}$ and $\{t\geq\frac 67\}$ thus split-generate the ``right-new'' objects.
Their images in $\W(\check X,\rr_1\cup\rr_2)$ are mutually orthogonal by Corollary \ref{stoppedff}, which finishes off the verification of the hypotheses of Proposition \ref{semiorthhocolim}.
We have thus shown that \eqref{firstpushout} (and hence \eqref{secondpushout}) is an almost homotopy pushout, as desired.
\end{proof}

\begin{remark}
The wrapping analysis in the proof above in fact shows that the full subcategory of $\W(\check X,\rr_1\cup\rr_2)$
spanned by the images of $\W(X,\rr_1), \W(X,\rr_2), \W(F)$ is the \emph{Grothendieck construction} of the diagram
\begin{equation}\label{htpycolim}
\W(X_1,\rr_1)\leftarrow\W(F)\to\W(X_2,\rr_2),
\end{equation}
i.e. the semi-orthogonal gluing of $\W(F)$ with $\W(X_1,\rr_1)\sqcup\W(X_2,\rr_2)$ along (the disjoint union of) the pullback of the diagonal bimodules.  
That is, the space of morphisms from an object of $\W(F)$ to an object of $\W(X_i,\rr_i)$ is given by the morphism space in 
$\W(X_i,\rr_i)$ from the image of the first object to the second.

The homotopy colimit of a diagram such as \eqref{htpycolim} can be described concretely as the localization of the Grothendieck construction at the morphisms 
$L\to Q_i(L)$ for $L\in\W(F)$ corresponding 
to the identity map of $Q_i(L)$, where $Q_i$ ($i=1,2$) denotes the two functors in \eqref{htpycolim} (see Definition \ref{hocolimdef}).
It is equivalent to consider a generating set of $L\in\W(F)$, e.g.\ cocores, in which case one can check (e.g.\ using K\"unneth) that these cones are precisely the linking disks to the stop removed by the functor $\W(\check X,\rr_1\cup\rr_2)\to\W(X,\rr_1\cup\rr_2)$.
This provides an alternative proof of Theorem \ref{mvthm}, though it is not materially different (compare the discussion in Example \ref{formalreverse}) and requires a bit more work to justify completely.
\end{remark}

\begin{remark}
The pushout formula of Theorem \ref{mvthm} is a special case of the descent formula of Theorem \ref{weinsteindescent}.  One
reason we have proven Theorem \ref{mvthm} separately is that it does not depend on the new geometric notions
introduced in \S\ref{admsubsec}.
\end{remark}

\subsection{Weinstein handle attachment} \label{handles}

\begin{proof}[Proof of Corollary \ref{handleeffect}]
Fix a stopped Liouville sector 
$(X^\inn,\f^\inn)$ of dimension $2n$, and let $\Lambda^{k-1}\subseteq(\partial_\infty
X^\inn)^\circ\setminus\f^\inn$ be a parameterized isotropic sphere with a
trivialization of its symplectic normal bundle (i.e.\ the data for attaching a
Weinstein $k$-handle).  Let $X$ be the result of attaching a Weinstein
$k$-handle to $X^\inn$ along $\Lambda$.  There is an open embedding
$\partial_\infty X^\inn\setminus\Lambda\hookrightarrow\partial_\infty X$, and
we let $\f$ denote the image of $\f^\inn$ under this embedding.  We may view
$X$ as a gluing or union in the sense of Theorem \ref{mvthm} (see also 
\cite[\S 3.1]{eliashbergweinsteinrevisited}) of:
\begin{enumerate}
\item$X_1:=X^\inn_\Lambda$, defined
as $X^\inn$ with an added stop along $\Lambda$ (or rather the corresponding
Liouville sector $X^\inn \setminus \Nbd \Lambda$; note that $\Lambda$ comes with a canonical Weinstein ribbon $T^\ast\Lambda\times\CC^{n-k}$ since its symplectic normal bundle is trivialized as part of the handle attaching data), and
\item$X_2:=T^\ast B^k\times\CC^{n-k}$, along
\item$T^\ast (S^{k-1} \times (-\varepsilon, \varepsilon)) \times \CC^{n-k}$.
\end{enumerate}
It follows that there is an almost homotopy pushout
\begin{equation}\label{handlepushout}
\begin{tikzcd}
\W(T^\ast(S^{k-1}\times[0,1])\times\CC^{n-k})\ar{r}\ar{d}&\W(X_\Lambda^\inn,\f^\inn)\ar{d}\\
\W(T^\ast B^k\times\CC^{n-k})\ar{r}&\W(X,\f).
\end{tikzcd}
\end{equation}
We now analyze the critical and subcritical cases separately.

If $k<n$ (i.e.\ $\Lambda$ is subcritical), then
$\W(T^\ast(S^{k-1}\times[0,1])\times\CC^{n-k})=0=\W(T^\ast B^k\times\CC^{n-k})$ since $n-k>0$
(one argument goes via Theorem \ref{generation}, which shows they are generated
by the empty set, compare Example \ref{cotangentgeneration}), so \eqref{handlepushout} reduces to a fully faithful embedding $\W(X_\Lambda^\inn,\f^\inn)\hookrightarrow\W(X,\f)$.
On the other hand, stop removal Theorem \ref{stopremoval} in the subcritical case implies that $\W(X_\Lambda^\inn, \f^\inn)\xrightarrow\sim\W(X^\inn, \f^\inn)$ is a quasi-equivalence, and hence, by inverting this latter quasi-equivalence, we obtain a fully faithful embedding
\begin{equation}\label{subcriticalembedding}
\W(X^\inn, \f^\inn) \hookrightarrow \W(X, \f).
\end{equation}
Geometrically, this functor first perturbs a Lagrangian in $X^\inn$ to avoid $\Lambda$ and then completes it inside $X$.
If $X^\inn$ is a Weinstein manifold and $\f^\inn$ is mostly Legendrian, then using Theorem \ref{generation} we see that \eqref{subcriticalembedding} is in fact a quasi-equivalence.

Let us now consider the case $k=n$ (i.e.\ $\Lambda$ is Legendrian), so we may write \eqref{handlepushout} as
\begin{equation}\label{handlepushoutLegendrian}
\begin{tikzcd}
\W(T^\ast(S^{k-1}\times[0,1]))\ar{r}\ar{d}&\W(X_\Lambda^\inn,\f^\inn)\ar{d}\\
\W(T^\ast B^k)\ar{r}&\W(X,\f).
\end{tikzcd}
\end{equation}
For $k=n>1$, the categories $\W(T^\ast B^k)$ and $\W(T^\ast(S^{k-1}\times[0,1]))$ are both generated by a cotangent fiber (see Example \ref{cotangentgeneration}), and the endomorphism algebras of these cotangent fibers are $\ZZ$ and $C_{-\bullet}(\Omega S^{n-1})$, respectively, by Abbondandolo--Schwarz \cite{abbondandoloschwarz} and Abouzaid \cite{abouzaidtwisted} (along with Theorem \ref{kunneth}).
We thus obtain the desired pushout \eqref{handlepushoutstatement}.
The images of the cotangent fibers inside $\W(X_\Lambda^\inn,\f^\inn)$ and $\W(X,\f)$ are the linking disk $D\subseteq X_\Lambda^\inn$ of $\Lambda$ and the $\cocore\subseteq X$, respectively, giving \eqref{handlepushoutalgebras}.
When $k=n=1$, the only difference is that instead $\W(T^\ast(S^{k-1}\times[0,1]))=\ZZ\sqcup\ZZ$.
\end{proof}

\section{Sectorial hypersurfaces, coverings, and corners}\label{admsubsec}

We introduce and study sectorial hypersurfaces, sectorial coverings, and Liouville sectors with corners, establishing some basic properties.
This section may be regarded as a continuation of \cite[\S 2]{gpssectorsoc}.

\subsection{Sectorial hypersurfaces} \label{sectorialhypersurfaces}

\begin{lemma}\label{coisotropic}
Let $H_1,\ldots,H_n\subseteq X$ be a collection of cleanly intersecting hypersurfaces in a symplectic manifold.
The following are equivalent:
\begin{enumerate}
\item All multiple intersections $H_{i_0}\cap\cdots\cap H_{i_k}$ are coisotropic.
\item All pairwise intersections $H_i\cap H_j$ are coisotropic.
\item We have $C_i\subseteq TH_j$ over $H_i\cap H_j$, where $C_i$ denotes the characteristic foliation of $H_i$.
\end{enumerate}
\end{lemma}

\begin{proof}
It is equivalent to prove the corresponding statement for tuples of codimension one subspaces $V_1,\ldots,V_n$ inside a symplectic vector space $W$.
Pass to the $\omega$-orthogonal complements $C_1,\ldots,C_n\subseteq W$, which are lines inside $W$.
The three conditions in question can now be stated as follows:
\begin{enumerate}
\item All sums $C_{i_0}+\cdots+C_{i_k}\subseteq W$ are isotropic.
\item All sums $C_i+C_j$ are isotropic.
\item All pairs $C_i$ and $C_j$ are $\omega$-orthogonal.
\end{enumerate}
The equivalence of these is now obvious.
\end{proof}

Recall that when working in a Liouville manifold-with-boundary with Liouville vector field $Z$, we write $\Nbd^Z$ to indicate a $Z$-invariant neighborhood. 

\begin{definition}\label{sectorial}
Let $X$ be a Liouville manifold-with-boundary.
A collection of cylindrical hypersurfaces $H_1,\ldots,H_n\subseteq X$ will be called \emph{sectorial} iff their characteristic foliations are $\omega$-orthogonal (over their intersections) and there exist functions $I_i:\Nbd^ZH_i\to\RR$ (linear near infinity) satisfying:
\begin{equation}\label{admidentities}
dI_i|_{C_i}\ne 0,\qquad dI_i|_{C_j}=0\text{ for }i\ne j,\qquad\{I_i,I_j\}=0.
\end{equation}
We also allow immersed cylindrical hypersurfaces $H\to X$, with $I$ now defined on $\Nbd^ZH$ regarded as an immersed codimension zero submanifold of $X$, and the subscripts $i$ in the identities \eqref{admidentities} now indexing the ``local branches'' of $H$.
\end{definition}

\begin{remark}
A word of caution is in order: we do \emph{not} show that the space of tuples $\{I_i:\Nbd^ZH_i\to\RR\}_i$ for a given sectorial collection $H_1,\ldots,H_n$ is contractible (or even connected).
Some care is thus warranted regarding whether certain constructions/results depend on a choice of $\{I_i:\Nbd^ZH_i\to\RR\}_i$.
\end{remark}

\begin{lemma}
Sectorial hypersurfaces $H_1,\ldots,H_n\subseteq X$ are mutually transverse, that is $\codim(TH_{i_1}\cap\cdots\cap TH_{i_k})=k$ over $H_{i_1}\cap\cdots\cap H_{i_k}$, and these multiple intersections are coisotropic.
\end{lemma}

\begin{proof}
It follows from the conditions $dI_i|_{C_i}\ne 0$ and $dI_i|_{C_j}=0$ for $i\ne j$ that the characteristic lines $C_i$ are linearly independent, which implies mutual transversality of their symplectic orthogonal complements $TH_i$.
That the intersections are coisotropic follows from Lemma \ref{coisotropic}.
\end{proof}

\begin{example}
A Liouville manifold-with-boundary $X$ is a Liouville sector iff $\partial X$ is sectorial.
\end{example}

\begin{example}\label{cotangentsectorialII}
Let $Q$ be a manifold, and let $G_1,\ldots,G_n\subseteq Q$ be a collection of mutually transverse hypersurfaces.
Their inverse images inside $T^\ast Q$ form a sectorial collection of hypersurfaces.
Namely, we may take $I_i$ to be the Hamiltonian lifts of vector fields $V_i$ on $Q$ where $V_i$ is transverse to $G_i$ and tangent to $G_j$ for $j\ne i$.
More generally, the same holds for the inverse image in $T^\ast Q$ of any self-transverse immersed hypersurface $G\looparrowright Q$.
\end{example}

\begin{lemma}\label{conjugate}
Let $H_1,\ldots,H_n\subseteq X$ be a sectorial collection of hypersurfaces, with a choice of $I_i:\Nbd^ZH_i\to\RR$.
There exist unique functions $t_i:\Nbd^ZH_i\to\RR$ satisfying $Zt_i\equiv 0$ near infinity, $H_i=\{t_i=0\}$, $\{I_i,t_j\}=\delta_{ij}$, and $\{t_i,t_j\}=0$.
\end{lemma}

\begin{proof}
The properties $H_i=\{t_i=0\}$ and $X_{I_i}t_i\equiv 1$ define $t_i:\Nbd^ZH_i\to\RR$ uniquely.
Differentiating $X_{I_i}t_i\equiv 1$ by $Z$ yields $X_{I_i}(Zt_i)\equiv 0$, which together with $Zt_i=0$ over $H_i$ near infinity implies that $Zt_i=0$ near infinity.
The remaining properties follow from a standard calculation as we now recall.
The Jacobi identity $\{I_i,\{I_j,t_j\}\}+\{I_j,\{t_j,I_i\}\}+\{t_j,\{I_i,I_j\}\}=0$ implies $\{I_j,\{t_j,I_i\}\}=0$.
Now $X_{t_j}$ spans the characteristic foliation of $H_j$, so since $dI_i|_{C_j}=0$ for $i\ne j$, we have $\{t_j,I_i\}=0$ over $H_j$, so $\{I_j,\{t_j,I_i\}\}=0$ implies that in fact $\{t_j,I_i\}=0$ everywhere for $i\ne j$.
Jacobi again $\{t_i,\{t_j,I_j\}\}+\{t_j,\{I_j,t_i\}\}+\{I_j,\{t_i,t_j\}\}=0$ now yields $X_{I_j}\{t_i,t_j\}=0$, and symmetrically we have $X_{I_i}\{t_i,t_j\}=0$.
Thus it is enough to show that $\{t_i,t_j\}=0$ over $H_i\cap H_j$, and here it follows since $X_{t_i}\in C_i$ which is tangent to $H_j$.
\end{proof}

\begin{lemma}\label{rksplittinglemma}
The map
\begin{equation}\label{rkfibration}
(I_{i_1},t_{i_1},\ldots,I_{i_k},t_{i_k}):\Nbd^Z(H_{i_1}\cap\cdots\cap H_{i_k})\to T^*\RR^k
\end{equation}
is a symplectic fibration whose symplectic connection is flat, thus giving a symplectic product decomposition
\begin{equation}\label{rksplitting}
X=F\times T^*\RR^k
\end{equation}
identifying $Z$-invariant neighborhoods of $H_{i_1}\cap\cdots\cap H_{i_k}$ and $F\times T^*_0\RR^k$.
\end{lemma}

\begin{proof}
The Hamiltonian vector fields $X_{I_{i_1}},\ldots,X_{I_{i_k}},X_{t_{i_1}},\ldots,X_{t_{i_k}}$ are linearly independent and span a symplectic subspace of $TX$, in view of their Poisson brackets.
This subspace is symplectically orthogonal to the kernel of the differential of the map $(I_{i_1},t_{i_1},\ldots,I_{i_k},t_{i_k})$, so we conclude that this map is a submersion with symplectic fibers, and that these vector fields span the horizontal distribution of the induced symplectic connection.
Moreover, these vector fields are the horizontal lifts of the corresponding vector fields on the target, so since they commute (by their Poisson brackets), we conclude that the connection is flat.
\end{proof} 

\begin{remark}\label{newfromoldsectorial}
In the coordinates \eqref{rksplitting}, the sectorial hypersurfaces are simply the inverse images of the coordinate hyperplanes in $\RR^k$.
This gives a particularly simple way of ``smoothing corners'' or doing other local modifications to a sectorial collection of hypersurfaces near a given stratum $H_{i_1}\cap\cdots\cap H_{i_k}$: choose some other collection of mutually transverse hypersurfaces in $\RR^k$ and pull back.
The new collection stays sectorial: we may take the new $X_{I_i}$ to be the Hamiltonian lifts of vector fields on $\RR^k$ which are transverse to our new set of hypersurfaces (compare Example \ref{cotangentsectorialII}).
\end{remark}

\subsection{Straightening the Liouville form} \label{subsec:sectorstraightening} 

We now study the interaction of the splitting \eqref{rksplitting} with Liouville forms.

\begin{lemma}\label{rkliouville}
With respect to the coordinates \eqref{rksplitting}, we have
\begin{equation}\label{liouvillesplitting}
\lambda_X=\lambda_F+\lambda_{T^*\RR^k}+df
\end{equation}
for a Liouville form $\lambda_F$ on $F$ and a function $f:F\times T^*\RR^k\to\RR$ supported in $F_0\times T^*\RR^k$ for some compact $F_0\subseteq F$.
This function $f$ locally factors through $F$ wherever $ZI_i=I_i$ and $Zt_i=0$ (and hence $f$ may be taken to be zero iff $ZI_i=I_i$ and $Zt_i=0$ everywhere).
\end{lemma}

\begin{proof}
Let $\lambda_F$ be the Liouville form on $F$ obtained as the restriction of $\lambda_X$ to $F\times\{0\}=I_{i_1}^{-1}(0)\cap t_{i_1}^{-1}(0)\cap\cdots\cap I_{i_k}^{-1}(0)\cap t_{i_k}^{-1}(0)$.
Since $\omega_X=\omega_F+\omega_{T^*\RR^k}$ from Lemma \ref{rksplittinglemma}, the difference between $\lambda_X$ and $\lambda_F+\lambda_{T^*\RR^k}$ is closed.
To check that this difference is exact, it is enough to check it on $F\times\{0\}$ (where it \emph{a fortiori} vanishes) since the inclusion of $F\times\{0\}$ into $F\times T^*\RR^k$ is a homotopy equivalence.
We fix a choice of $f$ by requiring that $f$ vanish on $F\times\{0\}$.

The identities $ZI_i=I_i$ and $Zt_i=0$ are equivalent to the assertion that the projection of $Z_X=Z_F+Z_{T^*\RR^k}-X_f$ to $T^*\RR^k$ equals $Z_{T^*\RR^k}$, which is in turn equivalent to $f$ locally factoring through the projection to $F$.
Since $ZI_i=I_i$ and $Zt_i=0$ near infinity, we conclude that $f$ is supported inside $F_0\times T^*\RR^k$ for some compact $F_0\subseteq F$.
\end{proof}

Note that the case $k=1$ (studied in \cite[\S 2]{gpssectorsoc}) and the case $k\geq 2$ of Lemma \ref{rkliouville} have slightly differing behavior.
When $k\geq 2$, the cotangent bundle $T^*(-\varepsilon,\varepsilon)^k$ is connected at infinity, which means that we have a well-defined compactly supported function $f_\infty:F\to\RR$, so we may redefine $\lambda_F$ as $\lambda_F+df_\infty$, so that now in \eqref{liouvillesplitting} we have $f$ having compact support.
In contrast, when $k=1$, there are two compactly supported functions $f_{\pm\infty}:F\to\RR$.

We now show how to deform the Liouville form so that it is split $\lambda_X=\lambda_F+\lambda_{T^*\RR^k}$ with respect to the coordinates \eqref{rksplitting} (equivalently, so that $ZI_i=I_i$ and $Zt_i=0$ everywhere).
We will do this in two separate steps, corresponding to the two cases $k=1$ and $k\geq 2$ with differing behavior above.

\begin{lemma}\label{goodcoordsnearintersections}
For any sectorial collection of hypersurfaces $H_1,\ldots,H_n\subseteq X$ and choice of functions $I_i:\Nbd^ZH_i\to\RR$, there exists a compactly supported $f:X\to\RR$ such that the deformed Liouville vector field $Z=Z_{\lambda+df}$ satisfies $ZI_i=I_i$ and $Zt_i=0$ in a neighborhood of $H_i\cap H_j$ for every $i\ne j$.
\end{lemma}

\begin{proof}
Work by induction on strata $H_I:=\bigcap_{i\in I}H_I$, ordered by reverse inclusion of subsets $I\subseteq\{1,\ldots,n\}$ of cardinality $\geq 2$.
Thus we assume the conclusion hold over a neighborhood of $H_J$ for all $J\supsetneqq I$, and we try to achieve the desired conclusion over a neighborhood of $H_I$.
In a neighborhood of $H_I$, we have coordinates $(X,\lambda_X)=(F\times T^*\RR^I,\lambda_F+\lambda_{T^*\RR^I}+df)$ where $f$ is properly supported over $T^*\RR^I$ and independent of the $T^*\RR^I$ coordinate near infinity (in the cotangent directions).
Since $\left|I\right|\geq 2$ so $T^*\RR^I$ is connected at infinity, by modifying $\lambda_F$ we may assume that $f$ has compact support.
Now since the desired conclusion holds in a neighborhood of $H_J$ for $J\supsetneqq I$, we conclude that $f$ vanishes indentically in a neighborhood of $H_I\cap H_j$ for all $j\notin I$.
We may thus simply add $f\cdot\prod_{i\in I}\varphi(t_i)$ to the Liouville form of $X$ (for a suitable cutoff function $\varphi:\RR\to\RR_{\geq 0}$ equalling $1$ near zero and of small support) to achieve the desired conclusion over a neighborhood of $H_I$.
\end{proof}

\begin{lemma}\label{makelinear}
For any sectorial collection of hypersurfaces $H_1,\ldots,H_n\subseteq X$ and choice of functions $I_i:\Nbd^ZH_i\to\RR$, there exists $f:X\to\RR$ such that the deformed Liouville vector field $Z=Z_{\lambda+df}$ satisfies $ZI_i=I_i$ and $Zt_i=0$ (equivalently, the coordinates \eqref{rkfibration}--\eqref{rksplitting} give a splitting of Liouville forms $\lambda_X=\lambda_F+\lambda_{T^*\RR^k}$).

The deformation $Z_{\lambda+r\,df}$, while not necessarily compactly supported, is cylindrical at infinity in the sense that there is a diffeomorphism $X=Y\times\RR_{\geq 0}$ near infinity such that $Z_{\lambda+r\,df}$ is outward pointing along $Y\times\{s\}$ for all $r\in[0,1]$.
\end{lemma}

\begin{proof}
Appealing first to Lemma \ref{goodcoordsnearintersections}, we reduce to the case that the desired conclusion already holds in a neighborhood of every pairwise intersection $H_i\cap H_j$ for $i\ne j$.
In a neighborhood of $H_i$, we have coordinates $(X,\lambda_X)=(F\times T^*\RR,\lambda_F+\lambda_{T^*\RR}+df)$ where $f$ is properly supported over $T^*\RR$ and independent of the $T^*\RR$ coordinate near infinity (in the cotangent directions).
In a neighborhood of $H_i\cap H_j$ for any $j\ne i$, we have that $f$ vanishes identically.
We now simply consider the deformation $\lambda_F+\lambda_{T^*\RR}+d((1-r\varphi(t_i))f)$ for $r\in[0,1]$ (where $\varphi:\RR\to\RR_{\geq 0}$ equals $1$ near zero and has small support).
These deformations associated to the various $H_i$ have disjoint support, so we may perform them simultaneously.
The associated family of Liouville vector fields is given by
\begin{equation}
Z_F+Z_{T^*\RR}-(1-r\varphi(t_i))X_f+rf\varphi'(t_i)X_{t_i}.
\end{equation}
Let us show that this vector field is outward pointing along $\partial(F_0\times\{\left|I\right|\leq N\})$ for $F_0\subseteq F$ a large Liouville domain and sufficiently large $N<\infty$.
The vector field $Z_F+Z_{T^*\RR}$ is certainly outward pointing along $\partial(F_0\times\{\left|I\right|\leq N\})$.
The vector field $X_f$ is tangent to $F$ factor, and vanishes outside a sufficiently large compact subset of $F$, so it is tangent to $\partial(F_0\times\{\left|I\right|\leq N\})$.
The vector field $X_{t_i}$ equals $-\frac\partial{\partial I}$ in the $T^*\RR$ factor, however the factors $rf\varphi'(t_i)$ are bounded, so this term is negligible compared to $Z_{T^*\RR}=I\frac\partial{\partial I}$.
We conclude that during our prescribed deformation, the Liouville vector field remains outward pointing along $\partial(F_0\times\{\left|I\right|\leq N\})$.
Finally, we should show how to patch together these contact type hypersurfaces (defined over each $\Nbd H_i$) globally.
This is straightforward since $Z$ is not changing near the pairwise intersections $H_i\cap H_j$, so near such and higher intersections, we may take any contact hypersurfaces we like.
\end{proof}

\begin{remark}[Sectorial hypersurfaces and products]\label{productofwithcorners}
Given sectorial hypersurfaces $H_1,\ldots,H_n\subseteq X$ and $H_1',\ldots,H_{n'}'\subseteq X'$, we would like to know that their product
\begin{equation}\label{sectorialproduct}
H_1\times X',\ldots,H_n\times X',X\times H_1',\ldots,X\times H_{n'}'\subseteq X\times X'
\end{equation}
is again sectorial.
So that these product hypersurfaces are cylindrical, we must, at a minimum, assume that the Liouville vector fields on $X$ and $X'$ are everywhere tangent to the $H_i$ and $H_i'$.
If, in addition, the functions $I_i$ and $I_i'$ satisfy $ZI_i=I_i$ and $Z'I_i'=I_i'$ everywhere, then their pullbacks to $X\times X'$ show that the collection of hypersurfaces \eqref{sectorialproduct} is indeed sectorial.
These hypotheses on the input hypersurfaces are precisely the properties ensured by applying Lemma \ref{makelinear}.
Thus whenever we want to consider products of sectorial hypersurfaces, we will first apply Lemma \ref{makelinear} to each factor.
\end{remark}

\subsection{Liouville sectors with corners}\label{pairsectorgluing}

\begin{definition}\label{admcorners}
A \emph{Liouville sector with (sectorial) corners} is a Liouville manifold-with-corners whose boundary, viewed as an immersed hypersurface, is sectorial.
\end{definition}

We will usually drop the qualifier `sectorial', except in the following remark.

\begin{remark}\label{smoothadmcorners}
A Liouville sector with sectorial corners is, in particular, a \emph{Liouville sector with (naive) corners}, namely a Liouville manifold-with-corners for which there exists an outward pointing Hamiltonian vector field defined near the boundary which is linear at infinity (indeed, a Hamiltonian giving such a vector field on a Liouville sector with sectorial corners is given by $\sum_i\varphi(t_i)I_i$ for a cutoff function $\varphi$ supported near zero).
Smoothing the corners of a Liouville sector with naive corners $X$ yields a Liouville sector $X^\sm$ \cite[Remark 2.12]{gpssectorsoc}.
A Liouville sector with sectorial corners is a much stronger notion than a Liouville sector with naive corners (e.g.\ the corner strata of the latter need not even be coisotropic).
Since we will not need the notion of a Liouville sector with naive corners in this paper, we shorten `Liouville sector with sectorial corners' to `Liouville sector-with-corners'.
\end{remark}

The wrapped Fukaya category of a Liouville sector-with-corners $X$ may be defined either as that of its smoothing: $\W(X):=\W(X^\sm)$ (which is well-defined in view of Lemma \ref{winvariancesector}), or, equivalently, and somewhat more intrinsically, as the wrapped Fukaya category of its interior (which is an \emph{open Liouville sector}, see Remark \ref{openliouvillesectors}).

\begin{example}
A product of Liouville sectors, each of which has been straightened as in Remark \ref{productofwithcorners}, is a Liouville sector-with-corners.
More generally, the same holds for products of Liouville sectors-with-corners.
\end{example}

Recall that a Liouville sector $X$ has $\partial X = F \times \RR$ whose symplectic reduction $F$, termed the symplectic boundary of $X$, is a Liouville manifold.
We extend this terminology to sectors-with-corners: each corner stratum of $X$ is coisotropic, and its symplectic reduction (namely the $F$ factor appearing in \eqref{rksplitting}) is called a \emph{symplectic boundary stratum} of $X$.
These symplectic boundary strata are themselves Liouville sectors-with-corners (the $I_i$ and $t_i$ descend in view of their Poisson brackets).
We will refer to the maximal proper symplectic boundary strata (i.e.\ those of real codimension two) as \emph{symplectic (boundary) faces}.

We now consider various gluing operations for Liouville sectors-with-corners.
Let us first recall the situation for Liouville sectors.
Given Liouville sectors $X$ and $Y$ with common symplectic boundary $F$, we may glue them to obtain a Liouville manifold $X\#_FY:=X \cup_{F \times \RR} Y$.
The common boundary, now in the interior, is a sectorial hypersurface.  
The same operation may also be described in terms of the corresponding Liouville pairs $(\bar X,F)$ and $(\bar G,F)$, the result being denoted by $(\bar X,F)\#_F(\bar G,F)$.
A direct construction of $\#$ in the language of Liouville pairs may be found in \cite[\S 3.1]{eliashbergweinsteinrevisited}, and reasoning as in \cite[\S 2]{gpssectorsoc} shows it is equivalent to the connect sum of Liouville sectors.

We now extend this picture to Liouville sectors-with-corners in various ways.
We limit our discussion to the case that the boundary consists of exactly two faces meeting precisely along the corner locus.

\begin{construction}\label{cornerdescription}
Items of the following two kinds can each be used to produce one of the other:
\begin{itemize}
\item A Liouville sector-with-corners $X$, with symplectic faces $F, G$, which in turn have common symplectic boundary $P$.
\item A Liouville sector $X^\sm$ with symplectic boundary expressed as $F \#_P G$.
\end{itemize}
\end{construction}

\begin{proof}[Construction]
We write $\partial X = H_1 \cup H_2$ where $H_1 = F \times \RR$, $H_2 = G \times \RR$, and $H_1 \cap H_2 = P \times \RR^2$. 
We apply Lemma \ref{makelinear} so that the coordinates \eqref{rksplitting} strictly respect Liouville forms; explicitly, these are a compatible collection of coordinates
\begin{align}
\label{cornernbhdfirst}F\times T^*\RR_{\leq 0}&\hookrightarrow X,\\
G\times T^*\RR_{\leq 0}&\hookrightarrow X,\\
\label{cornernbhdP}P\times T^*\RR_{\leq 0}^2&\hookrightarrow X,\\
\label{cornernbhdF}P\times T^*\RR_{\leq 0}&\hookrightarrow F,\\
\label{cornernbhdlast}P\times T^*\RR_{\leq 0}&\hookrightarrow G,
\end{align}
where $F$ and $G$ are Liouville sectors and $P$ is a Liouville manifold (though we write $T^*\RR_{\leq 0}$, we really mean $T^* (-\varepsilon,0]$ for small $\varepsilon>0$).
The first three of these are coordinates near $H_1$, $H_2$, and $H_1\cap H_2$, and the last two are coordinates near $\partial F$ and $\partial G$, respectively.
The coordinates $P\times T^*\RR_{\leq 0}^2\hookrightarrow X$ account for the entire corner locus of $X$, so we may describe a smoothing of this corner simply by smoothing the corner of $\RR_{\leq 0}^2$.

The smoothing of this corner may be described as follows (a very similar discussion appeared earlier in \S\ref{productofwithstops}, in particular around Figure \ref{figuretwovectorfields}).
Let $\bar X$ denote the result of gluing onto $X$ (via \eqref{cornernbhdfirst}--\eqref{cornernbhdlast}) copies of
\begin{align}
&F\times(T^\ast\RR_{\geq 0},Z_{T^\ast\RR_{\geq 0}}+\pi^\ast\varphi(s)\partial_s),\\
&G\times(T^\ast\RR_{\geq 0},Z_{T^\ast\RR_{\geq 0}}+\pi^\ast\varphi(s)\partial_s),\\
&P\times(T^\ast\RR^2_{\geq 0},Z_{T^\ast\RR^2_{\geq 0}}+\pi^\ast\varphi(s_1)\partial_{s_1}+\pi^\ast\varphi(s_2)\partial_{s_2}),
\end{align}
where $\pi^\ast$ denotes the lift from vector fields on a manifold to Hamiltonian vector fields on its cotangent bundle, and $\varphi:\RR\to[0,1]$ is smooth and satisfies $\varphi(s)=0$ for $s\leq 2$ and $\varphi(s)=1$ for $s\geq 3$.
Now as illustrated in Figure \ref{figuretwovectorfieldsreprise}, the vector field $\varphi(s_1)\partial_{s_1}+\varphi(s_2)\partial_{s_2}$ may be deformed over a compact subset of the interior of $\RR^2_{\geq 0}$ to a vector field of the form $\varphi(s)\partial_s$ for some coordinates $(s,\theta)$ on $\RR^2_{\geq 0}\setminus\{(0,0)\}$.
The locus $s\leq 1$ in this deformation is thus a smoothing $X^\sm$ of the corners of $X$, and its complement can be described as
\begin{equation}\label{productfiberreprise}
\Bigl[F\underset{P\times T^\ast[0,1]}\cup G\Bigr]\times(T^\ast\RR_{\geq 0},Z_{T^\ast\RR_{\geq 0}}+\pi^\ast\varphi(s)\partial_s).
\end{equation}
We will denote this deformation of $\bar X$ by $\overline{X^\sm}$, since the above discussion shows it is the convex completion of $X^\sm$.
Now the smoothed Liouville sector $X^\sm$ has boundary neighborhood coordinates
\begin{equation}
(F\#_PG)\times T^*\RR_{\leq 0}\hookrightarrow X^\sm
\end{equation}
where $F\#_PG$ denotes the Liouville manifold from \eqref{productfiberreprise}, namely obtained by gluing $F$ and $G$ together along $P\times T^*(-\varepsilon,\varepsilon)$, either side of which is embedded into $F$ and $G$ via \eqref{cornernbhdF}--\eqref{cornernbhdlast}.

\begin{figure}[ht]
\centering
\includegraphics[max width=.95\textwidth]{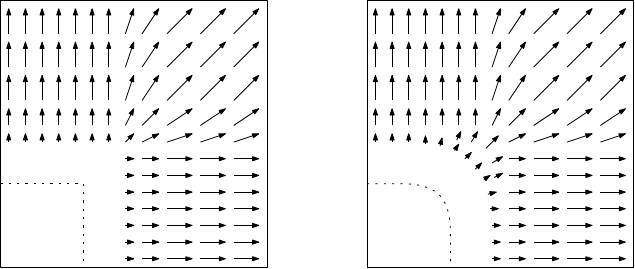}
\caption{Left: The vector field $\varphi(s_1)\partial_{s_1}+\varphi(s_2)\partial_{s_2}$ on $\RR^2$ defining the Liouville structure on $P\times T^\ast\RR^2_{\geq 0}$ (the dotted line indicates the boundary of $X$).  Right: The deformed vector field $\varphi(s)\partial_s$ defining the Liouville structure on $P\times T^\ast\RR^2_{\geq 0}$ which defines what we call $\overline{X^\sm}$ (the dotted line indicates a smoothing $X^\sm$ of $X$).  Note that the deformation is supported in a compact subset of $\RR^2_{\geq 0}$, disjoint from the boundary.}\label{figuretwovectorfieldsreprise}
\end{figure}

\begin{figure}[ht]
\centering
\includegraphics[max width=.95\textwidth]{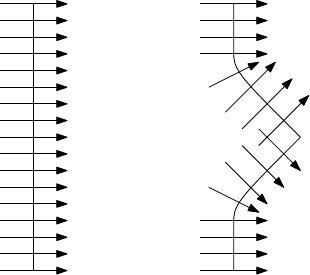}
\caption{Introducing a corner near the origin to turn $\RR\times\RR_{\leq 0}$ (left) into $A$ (right).}\label{figureintroducecorner}
\end{figure}

This operation works in reverse as well: given a Liouville sector $Y$ with boundary neighborhood coordinates $Q\times T^*\RR_{\leq 0}\hookrightarrow Y$ and a splitting of $Q$ as $F\#_PG$, we may introduce a corner into the boundary as follows to producing a Liouville sector-with-corners $X$ with $X^\sm=Y$.
Concretely, near $P$, there is a neighborhood in $Q$ of the form $P \times T^* \RR$ (rather $P \times T^*(-\epsilon, \epsilon)$), and hence a neighborhood in $X$ of the form $P \times T^*(\RR \times \RR_{\leq 0})$
(rather $P \times T^*((-\epsilon, \epsilon) \times (-\epsilon, 0])$).
We now modify $\RR\times\RR_{\geq 0}$ near the origin so as to introduce a corner (see Figure \ref{figureintroducecorner}), and we replace our local chart $P\times T^*(\RR\times\RR_{\leq 0})$ with $P\times T^*A$ to define our desired $X$.
To see that $X$ has sectorial corners, simply note that the Hamiltonian vector field on $Y$ transverse to its boundary is the Hamiltonian lift of $\frac\partial{\partial t}$, so the desired vector fields on $X$ can also be defined on $A$ so that they agree with $\frac\partial{\partial t}$ away from a neighborhood of the origin (see Figure \ref{figureintroducecorner}) and then lifted.
\end{proof}

We now discuss a version of the boundary connect sum $\#$ construction for gluing along a shared subsector of the boundary, rather than the whole boundary. 

\begin{construction} \label{glueandsplit}
Items of the following three kinds can each be used to produce one of the other:
\begin{itemize}
\item A Liouville sector $X$ with a hypersurface $H\subseteq X$, meeting the boundary transversally, with $\{H,\partial X\}$ sectorial, such that $H$ separates $X$ into two pieces $X_1$ and $X_2$ meeting precisely along $H$.
\item A pair of Liouville sectors-with-corners $X_1$ and $X_2$ with two boundary faces each $\partial^j X_i$, $j = 1,2$, and an identification of the symplectic reductions of $\partial^1 X_1$ and $\partial^1 X_2$.
\item A pair of Liouville pairs $(Y_1,Q_1)$ and $(Y_2,Q_2)$ with a common Liouville subsector $Q_1\hookleftarrow P\hookrightarrow Q_2$.
\end{itemize}
In this correspondence, $(Y_i,Q_i)$ is the Liouville pair corresponding to the Liouville sector $(X_i)^{\sm}$ (the rounding of the sector-with-corners $X_i$).
\end{construction}

\begin{proof}[Construction]
The passage between the first two inputs in either direction is evident (splitting along $H$ and gluing $\partial^1 X_1$ to $\partial^1 X_2$).
Beginning with the second type of input, $X_1$ and $X_2$ with $H = \partial^1 X_1 = \partial^1 X_2$, 
we apply Lemma \ref{makelinear} so as to obtain boundary neighborhood coordinates \eqref{cornernbhdfirst}--\eqref{cornernbhdlast} on $X_1$ and $X_2$, say with $F=F_1=F_2$ corresponding to $H$ and 
with $G_i$ the remaining the two pieces (i.e.\ in the context of the first type of input, the symplectic reductions of the two pieces into which $H$ splits $\partial X$).
Passing to the equivalent descriptions of $X_i$ in terms of Liouville pairs given by Construction \ref{cornerdescription}, we obtain $(Y_i,Q_i)$ together with a common Liouville subsector $Q_1\hookleftarrow P\hookrightarrow Q_2$.

Now suppose we are given the third type of input, Liouville pairs $(Y_i,Q_i)$ with a common Liouville subsector $Q_1\hookleftarrow P\hookrightarrow Q_2$.
The boundary $\partial P = R \times \RR$ separates $Q_i$ into $P\#_R(Q_i\setminus P^\circ)$.
We may thus apply Construction \ref{cornerdescription} to $(Y_i,P\#_R(Q_i\setminus P^\circ))$ to obtain Liouville sectors-with-corners, each with a boundary face with neighborhood $P\times T^*\RR_{\geq 0}$.
This yields the second type of input, and gluing along these common faces yields the first.
\end{proof}

Going forward, given a pair of Liouville pairs $(Y_1, Q_1)$ and $(Y_2, Q_2)$ along with a common subsector $Q_1\hookleftarrow P\hookrightarrow Q_2$, we shall call the result of passing to the first item in Construction \ref{glueandsplit} the \emph{gluing of $(Y_1, Q_1)$ and $(Y_2, Q_2)$ along $P$}, denoted
\begin{equation}
    (Y_1, Q_1) \#_P (Y_2, Q_2).
\end{equation}
We will use the same notation and terminology, the gluing of $X_1$ and $X_2$ along $P$
\begin{equation}
    X= X_1 \#_{P} X_2 := X_1 \cup_{P \times \RR} X_2
\end{equation}
for the passage from the second type of input of Construction \ref{glueandsplit} to the first, where $\partial^1 X_1 = \partial^1 X_2 = P \times \RR$.

\subsection{Boundary cores of sectors with corners} \label{sec:cornerboundarycores}

To understand when an inclusion of a Liouville sector $X$ into a larger sector is forward stopped, we use the positive/negative cores $\cc_{(\partial\partial_\infty X)_\pm}$, which were defined in \S\ref{convexreview}.
In order to have a similar understanding of the case when $X$ is a sector-with-corners, we will round corners in a particular way 
and calculate the positive/negative cores of the chosen rounding, as subsets of the original cornered $\partial_\infty X$.
We content ourselves with the case when $X$ has exactly two faces as in Construction \ref{cornerdescription}.

To begin the discussion, let us recall from \S\ref{convexreview} the situation for Liouville sectors (without corners).
So, suppose $X$ is a Liouville sector, with coordinates $F\times T^*\RR_{\leq 0}$ given near its boundary, where $F$ is a Liouville manifold.
Then we have
\begin{equation}
\partial\partial_\infty X=F\bigcup_{\RR\times\partial_\infty F} F,
\end{equation}
and $\cc_{(\partial\partial_\infty X)_\pm}$ are, respectively, the two copies of $\cc_F$ sitting inside the two copies of $F$ inside $\partial\partial_\infty X$.

Now consider the case that $X$ is a Liouville sector-with-corners, with exactly two faces.  Using Lemma \ref{makelinear} if necessary, we fix coordinates near its 
boundary given by $F\times T^*\RR_{\leq 0}$ and $G\times T^*\RR_{\leq 0}$ overlapping over $P\times T^*\RR_{\leq 0}^2$, where 
$F$ and $G$ are Liouville sectors both of which have coordinates $P\times T^*\RR_{\leq 0}$ near their boundary as in \eqref{cornernbhdfirst}--\eqref{cornernbhdlast}.
The locus $\partial\partial_\infty X$ is the union of $F\bigcup_{\RR\times\partial_\infty F}F$ and $G\bigcup_{\RR\times\partial_\infty G}G$.

As explained in Construction \ref{cornerdescription}, 
choosing a rounding of the corner of $\RR_{\leq 0}^2$ determines a corresponding rounding $X^\sm$, a neighborhood of whose boundary is given by $(F \#_P G)\times T^* \RR_{\leq 0}$.  
The cores $\cc_{(\partial\partial_\infty X^\sm)_\pm}$ are evidently given abstractly by $\cc_F \cup_{\cc_P} \cc_G$, i.e.\ the relative cores of $F$ and $G$ attached along their common boundary, $\cc_P$.
Let us discuss how they are embedded.

Away from the corner locus, the positive/negative cores $\cc_{(\partial\partial_\infty X^\sm)_\pm}$ are naturally identified, as in \S\ref{convexreview}, with two copies of $\cc_F$ and $\cc_G$.

It remains to understand
what happens in a neighborhood of the corner locus, where we may work in the local model $P\times T^*\RR_{\leq 0}^2$.
Before smoothing, the contribution of this local model to $\partial\partial_\infty X$ is given by
\begin{equation}\label{localcornerbdrymodel}
    \partial\RR_{\leq 0}^2\times\Biggl[\partial_\infty P\times\RR^2\bigcup_{(\RR\times\partial_\infty P)\times S^1}P\times S^1\Biggr].
\end{equation}
After rounding the corner of $\partial\RR_{\leq 0}^2$, 
the positive/negative cores 
lie in the second term $P\times S^1$ and are given by $\cc_P$ times the codirections in $S^1$ which are outward/inward conormal to (the rounded) $\partial\RR_{\leq 0}^2$. 
(Compare the discussion in Construction \ref{cornerdescription} for another look at this geometry.)
This completes the description of the $\cc_{(\partial\partial_\infty X^\sm)_\pm}$.

It is natural to ask whether we can describe the limiting behavior of 
these as we undo the rounding of the corner.  It is clear from the above description that this limit has the same description as in the rounded case, 
so long as by ``outward/inward conormal'' to a corner, we understand the appropriate quarter cocircle over the corner of $\RR_{\leq 0}^2$.
We henceforth \emph{define} the positive/negative cores of $\partial\partial_\infty X$ by this limit.

\subsection{Sectorial coverings} \label{sectorialcoverings}

\begin{definition}[Sectorial covering]\label{sectorialcoveringdef}
Let $(X, \partial X)$ be a Liouville manifold-with-boundary.
Suppose $X = X_1 \cup \cdots \cup X_n$, where each $X_i$ is a manifold-with-corners with precisely two faces $\partial^1X_i:=X_i\cap\partial X$ and the point set topological boundary $\partial^2X_i$ of $X_i\subseteq X$, meeting along the corner locus $\partial X\cap\partial^2 X_i=\partial^1X_i\cap\partial^2X_i$.
Such a covering $X=X_1\cup\cdots\cup X_n$ is called \emph{sectorial} iff the collection of hypersurfaces $\partial X,\partial^2X_1,\ldots,\partial^2X_n$ is sectorial.
(Note that this means, in particular, that $X$ and $X_1,\ldots,X_n$ are Liouville sectors.)
\end{definition}

\begin{remark}\label{sectorialcoveringotherdefs}
There are many possible variations on the above definition which also deserve the name `sectorial covering'---the key point is just that the collection of all the boundaries $\partial X,\partial X_1,\ldots,\partial X_n$ should be sectorial.
For example, we could allow $X$ and the $X_i$ to have more corners.
We could also insist on no corners: require $X_i\subseteq X$ to be disjoint from $\partial X$, require $\partial X,\partial X_1,\ldots,\partial X_n$ to be sectorial, and require $X_1\cup\cdots\cup X_n\hookrightarrow X$ to be a trivial inclusion (this comes at the cost of $X_1,\ldots,X_n$ not literally covering $X$).
For the purposes of this paper, we work with Definition \ref{sectorialcoveringdef} as stated above.
The various possible alternative definitions are all related by smoothing of corners, and hence our main results will continue to apply in these more general settings.
\end{remark}

\begin{example}\label{twoballscover}
Consider two balls $B_1,B_2\subseteq\RR^n$ whose boundaries are transverse.
The cover $T^*(B_1\cup B_2)=T^*B_1\cup T^*B_2$ does \emph{not} satisfy Definition \ref{sectorialcoveringdef}, since the boundary of $T^*(B_1\cup B_2)$ is not smooth.
To make this example conform to Definition \ref{sectorialcoveringdef}, we should take $X$ to be the cotangent bundle of a smoothing of $B_1\cup B_2$, and we should take $X_1$ and $X_2$ to be bounded by \emph{disjoint} slight perturbations (in $X$) of $(\partial B_1)\cap B_2$ and $B_1\cap(\partial B_2)$, respectively.
In this way, $X=X_1\cup X_2$ coincides with $T^*(B_1\cup B_2)=T^*B_1\cup T^*B_2$ except in (the inverse image of) a small neighborhood of $\partial B_1\cap\partial B_2$.
Similar constructions are often necessary when working with sectorial coverings in the sense of Definition \ref{sectorialcoveringdef}.
(In the `no corner' definition from Remark \ref{sectorialcoveringotherdefs}, we would just say that $T^*B_1\cup T^*B_2$ is a sectorial covering of $X$.)
\end{example}

\begin{example}\label{glueandsplitcovering}
The setup of Construction \ref{glueandsplit} gives a sectorial covering $X=X_1^+\cup X_2^+$, where $X_i^+$ is a slight enlargment of $X_i$, so that $X_1^+\cap X_2^+=H\times(-\varepsilon,\varepsilon)$.
\end{example}

\begin{example}
Let $X$ be a Liouville sector, and let $\partial X,H_1,\ldots,X_m$ be sectorial.
These divide $X$ into some number of connected components.
Suppose that the closure of each such component is embedded, i.e.\ consists of at most one orthant in every choice of local coordinates \eqref{rksplitting}.
Defining $X_1,\ldots,X_n$ to be appropriate slight enlargements of smoothings of corners of these closures, we see that $X_1,\ldots,X_n$ is a sectorial cover of $X$.
\end{example}

\begin{example}
Let $X$ be a Liouville sector-with-corners, and smooth its boundary to obtain $X^\sm$.
The symplectic boundary of $X^\sm$ may be described by the natural generalization of Construction \ref{cornerdescription}.
This symplectic boundary has a sectorial covering by $F_i$, each of which is (a slight enlargement of) a symplectic boundary face of $X$ (with smoothed corners).
\end{example}

Recall that given sectorial hypersurfaces $H_1,\ldots,H_n\subseteq X$, we get Liouville manifolds $F_{i_1,\ldots,i_k}$ from Lemma \ref{rksplittinglemma} and \ref{rkliouville}.
In fact, we have the following finer structure (which, for convenience, we describe only in the case of sectorial coverings).
For any sectorial covering $X=X_1\cup\cdots\cup X_n$, stratify $X$ by strata
\begin{equation}\label{coveringstrata}
X_{I, J, K} = \bigcap_{i\in I}X_i\cap\bigcap_{j\in J}\partial X_j\setminus\bigcup_{k\in K}X_k
\end{equation}
ranging over all decompositions $I \sqcup J \sqcup K = \{1, \ldots, n\}$.
The closure of each $X_{I, J, K}$ is a submanifold-with-corners, whose symplectic reduction is a Liouville sector-with-corners.
Indeed, this follows since the relevant $I_i$ descend to the symplectic reduction in view of their Poisson brackets.

\begin{definition}
A sectorial covering will be called \emph{Weinstein} when the convex completions of all of the symplectic reductions of strata \eqref{coveringstrata} are (up to deformation) Weinstein.
\end{definition}

\begin{lemma} \label{lem:morecovers} 
Given a sectorial covering $X=X_1\cup\cdots\cup X_n$, the following are also sectorial covers:
\begin{align}
\label{auxcoverI}X&=(X_1\cup X_2)\cup X_3\cup\cdots\cup X_n\\
\label{auxcoverII}X_1&=(X_1\cap X_2)\cup\cdots\cup(X_1\cap X_n)
\end{align}
(after smoothing the appropriate corners).
Moreover, if the original cover was Weinstein, so are those described above.
\end{lemma}

\begin{proof}
The new coverings are sectorial by Remark \ref{newfromoldsectorial}.
To see that the Weinstein property is preserved, we should note that, following the geometry of Constructions \ref{cornerdescription}--\ref{glueandsplit}, effect on the strata is to perform connect sum along boundary faces, which preserves being Weinstein \cite[\S 3.1]{eliashbergweinsteinrevisited}.
\end{proof}

\section{Stopping and the \texorpdfstring{$A_2$}{A\_2} sector} \label{subsec:stopa2}

We write $A_2$ for the Liouville sector associated to the stopped Liouville manifold $(\CC,\{e^{2\pi ij/3}\cdot\infty\}_{j=0,1,2})$.

\begin{remark}
The notation $A_2$ is due to the fact $\W(A_2)$  is equivalent to the category of perfect modules over the $A_2$ quiver $\bullet \to \bullet$. 
This category  has a semi-orthogonal decomposition
into two copies of the category of perfect modules over $\bullet$.
Geometrically, this semi-orthogonal decomposition can be seen by contemplating the forward stopped-ness and stopping witnesses for the inclusions of $T^* [0,1]$ around any two of the stops.
\end{remark}

Here, we will be interested in products $A_2 \times Q$, for a Liouville sector $Q$ with symplectic boundary $P$.  
By applying Lemma \ref{makelinear} to $A_2$ and $Q$, we may ensure that this product $Q\times A_2$ is a Liouville sector-with-corners (see Remark \ref{productofwithcorners}).
In fact for $A_2$, rather than abstractly applying Lemma \ref{makelinear}, we will simply use the Liouville vector field illustrated on the left of Figure \ref{figurea2deformation}.
The symplectic faces of this product are $Q\sqcup Q\sqcup Q$ (with a natural cyclic order) and $A_2\times P$.

Given a Liouville sector $X$ with symplectic boundary $F$, along with a sector embedding $Q\hookrightarrow F$, we may use 
Construction \ref{glueandsplit} to form $X \#_Q (Q \times A_2)$ (implicitly we are passing 
$Q \times A_2$ or $(X,F)$ through the first or second bullet point of Construction  \ref{cornerdescription} respectively in order for the result can be fed into the second respectively third input of Construction \ref{glueandsplit}).  As there are three $Q$ faces of $Q \times A_2$, we may glue up to three 
such (sectors associated to the) pairs $(X,F)$ into $Q \times A_2$ in this way.

\begin{figure}[ht]
\centering
\includegraphics[max width=.95\textwidth]{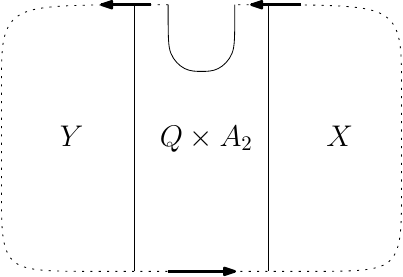}
\caption{Gluing Liouville sectors with $Q\times A_2$ in between.  The arrows indicate the direction of the Reeb flow.}\label{figurea2middle}
\end{figure}

We will be particularly interested in the gluing of two such pairs illustrated in Figure \ref{figurea2middle}, namely $X\#_Q(Q\times A_2)\#_QY$ (with the cyclic ordering illustrated), for Liouville sectors $X$ and $Y$ together with embeddings of $Q$ into their symplectic boundaries.
Of particular importance is the resulting pair of inclusions
\begin{equation}\label{a2gluing}
X\hookrightarrow X\#_Q(Q\times A_2)\#_QY\hookleftarrow Y.
\end{equation}

\begin{proposition}\label{a2stops}
The inclusion $X \hookrightarrow X\#_Q(Q\times A_2)\#_Q Y$ is tautologically forward stopped, and the 
stopping witness produced by Proposition \ref{tautisstopped} is disjoint from $Y$.
\end{proposition}

\begin{proof}
To begin let us fix precisely the geometric objects under consideration.
We consider $X$ and $Y$ to be their cornered versions from the middle bullet of Construction \ref{glueandsplit}, i.e.\ they each have a boundary face with neighborhood $Q\times T^*\RR_{\leq 0}$ where they are glued to $Q\times A_2$ to form $X\#_Q(Q\times A_2)\#_QY$ (so $\#_Q$ is $\cup_{Q\times\RR}$).
The main part of our argument proceeds by deforming $\partial_{\infty}(Q \times A_2)$ to make $\partial \partial_{\infty} X$ and $\partial\partial_{\infty}(X \#_Q (Q \times A_2) \#_Q Y)$ touch (in a way avoiding $\partial \partial_\infty Y$).
To conclude, we (simultaneously) smooth corners and note that the (smoothed) deformation exhibits the desired tautological forward stopping property (this smoothing step is necessary since the formulation of Definition \ref{tautforwardstoppeddef} does not allow corners).

We begin with the Liouville vector field on $A_2$ illustrated on the left of Figure \ref{figurea2deformation}.
This model of $A_2$ has boundary neighborhood coordinates $T^*\RR_{\leq 0}$ (times three), and thus we may consider its product with $Q$ (also assumed to have coordinates $T^*\RR_{\leq 0}\times P$ near its boundary, for $P$ a Liouville manifold) which is a Liouville sector-with-corners $Q\times A_2$, suitable for gluing to $X$ and $Y$ at the locations indicated on the (right and bottom of the) left half of Figure \ref{figurea2deformation}.

\begin{figure}[ht]
\centering
\includegraphics[max width=.95\textwidth]{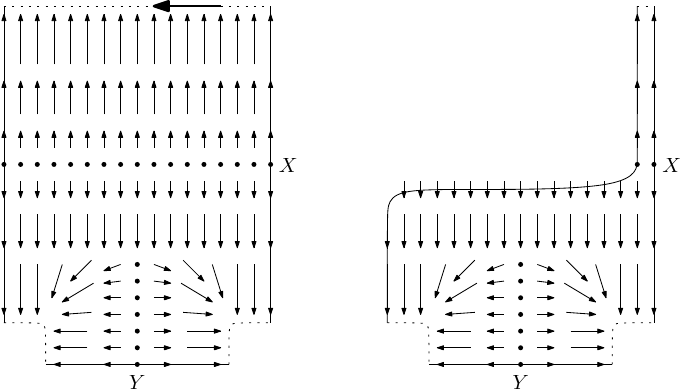}
\caption{A deformation of the $A_2$ Liouville sector, starting from the left at $r=0$, going to the right at $r=1-\varepsilon$, and limiting at $r=1$ to having the upper end pinched.  The large arrow indicates the direction of the Reeb flow.}\label{figurea2deformation}
\end{figure}

Consider now the contact boundary
\begin{equation}
\partial_\infty(Q\times A_2)=(\partial_\infty Q\times A_2)\bigcup_{\partial_\infty Q\times\partial_\infty A_2\times\RR}(Q\times\partial_\infty A_2).
\end{equation}
This is a contact manifold with convex corners in the sense that its boundary has two faces, each face has a transverse contact vector field which is tangent to the other face, and these vector fields commute (compare Definition \ref{admcorners}; the formula from Remark \ref{smoothadmcorners} implies that the boundary is convex, after smoothing corners).

We would now like to deform $\partial_\infty(Q\times A_2)$ by executing the deformation $\{A_2^r\}_{r\in[0,1]}$ illustrated in Figure \ref{figurea2deformation}.
Although $Q\times A_2^r$ is not cylindrical at infinity since the Liouville vector field is not everywhere tangent to $\partial A_2^r$, we will nevertheless make sense of $\partial_\infty(Q\times A_2^r)$ as a deformation at infinity provided we fix a choice of contact form $\alpha_Q$ on $\partial_\infty Q$.
Recall that the open inclusion $\partial_\infty Q\times A_2\hookrightarrow\partial_\infty(Q\times A_2)$ depends on a choice of contact form on $\partial_\infty Q$.
Fix any such contact form $\alpha_Q$, so $\alpha_Q+\lambda_{A_2}$ is a contact form on $\partial_\infty Q\times A_2$.
Now for $A_2^r$, 
consider the deformation of $\partial_\infty Q \times A_2$ (inside $\partial_\infty (Q \times A_2)$) given by $(\partial_\infty Q\times A_2^r,\alpha_Q+\lambda_{A_2^r})$.
At infinity this deformation simply shrinks one of the ends of $A_2$, and hence this deformation can be extended to all of $\partial_\infty (Q \times A_2)$ by, in the $Q\times  \partial_\infty A_2$ portion of $\partial_\infty (Q \times A_2)$, shrinking the end of $A_2$ by $Q \times \partial_\infty A_2^r$.
This gives the desired definition of $\partial_\infty(Q\times A_2^r)$, which we can further glue to obtain a deformation $\partial_\infty(X\#_Q(Q\times A_2^r)\#_QY)$.

Let us argue that $\partial\partial_\infty(Q\times A_2^r)$ and $\partial\partial_\infty(X\#_Q(Q\times A_2^r)\#_QY)$ remain convex during this deformation.
These boundaries have corners, so let us instead prove the stronger result that they have convex corners (which implies convexity by the construction of Remark \ref{smoothadmcorners}).
Recall that the vector fields demonstrating that $Q\times A_2$ has sectorial corners are simply those lifted from $Q$ and $A_2$, which strictly preserve the Liouville forms.
We may assume that our chosen contact form $\alpha_Q$ on $\partial_\infty Q$ is also preserved near the boundary.
Now to extend these vector fields to the deformations with parameter $r$, simply observe that over the deformed boundary component of $A_2^r$, the Liouville form is still locally $-s\,dt$ (i.e.\ the Liouville form on $T^*\RR$) and we may take our vector field to be $-\frac\partial{\partial t}$, where $t$ is the horizontal coordinate and $s$ is the vertical coordinate in Figure \ref{figurea2deformation}.
As this vector field strictly preserves the Liouville form, it defines a contact vector field on $(\partial_\infty Q\times A_2^r,\alpha_Q+\lambda_{A_2^r})$, thus also on $\partial\partial_\infty(Q\times A_2^r)$ and on $\partial\partial_\infty(X\#_Q(Q\times A_2^r)\#_QY)$, commuting with the transverse vector fields for the other boundary faces.

Finally, let us argue that this deformation of the inclusion $\partial_\infty X\hookrightarrow\partial_\infty(X\#_Q(Q\times A_2)\#_QY)$ into $\partial_\infty X\hookrightarrow\partial_\infty(X\#_Q(Q\times A_2^1)\#_QY)$ fulfills Definition \ref{tautforwardstoppeddef} (after smoothing), thus verifying that this inclusion is tautologically forward stopped.
First, note that this deformation causes both boundary faces of $\partial_\infty X$ to touch the boundary of $\partial_\infty(X \#_Q (Q \times A_2^1) \#_Q Y)$.
Smoothing corners, the discussion surrounding \eqref{localcornerbdrymodel} implies that the outgoing core of $\partial\partial_\infty X$ meets $\partial_\infty(\partial T^*\RR_{\leq 0}\times F)$ only along the positive conormal at $0\in\RR_{\leq 0}$, which implies it is contained in the deformed boundary $\partial\partial_\infty(X\#_Q(Q\times A_2^1)\#_QY)$.
It remains to show that it is disjoint from the incoming core of $\partial\partial_\infty(X\#_Q(Q\times A_2^1)\#_QY)$.
In fact, the outgoing core of $\partial\partial_\infty X$ is a subset of the outgoing core of the deformed boundary $\partial\partial_\infty(X\#_Q(Q\times A_2^1)\#_QY)$, so is in particular disjoint from its incoming core.

Finally, note that throughout this deformation (and its smoothing), $\partial_\infty X$ remains separated from $\partial_\infty Y$, hence the stopping witness produced by Proposition \ref{tautisstopped} is disjoint from $\partial_\infty Y$ as well.
\end{proof}

We have the following categorical consequences.  

\begin{corollary}\label{a2semiorthogonal}
In the setting of \eqref{a2gluing}, the functors
\begin{equation}
\W(X)\to \W( X\#_Q(Q\times A_2)\#_Q Y)\leftarrow\W(Y)
\end{equation}
are fully faithful, and $\W(X)$ is left-orthogonal to $\W(Y)$ inside $\W(X\#_Q(Q\times A_2)\#_QY)$. 
More generally, the same holds if we add stops inside $(\partial_\infty X)^\circ$ and $(\partial_\infty Y)^\circ$ not approaching the boundary.
\end{corollary}

\begin{proof}
Combine Proposition \ref{a2stops} and Corollary \ref{stoppedff}.
\end{proof}

\begin{example}\label{stopdoubling}
Corollary \ref{a2semiorthogonal} allows us to generalize Example \ref{boundaryfiberinclusionstopping} from Liouville manifolds to Liouville sectors.
Let $F$ be a Liouville sector with symplectic boundary $Q$.
For an embedding $\bar F_0\hookrightarrow\partial_\infty X$ as a Liouville hypersurface, we consider the gluing $(X,\bar F)\#_F(F\times A_2)$, which can be equivalently described as $(X,F\#_QF)$, for an embedding $(F\#_QF)_0\hookrightarrow\partial_\infty X$ as a Liouville hypersurface (the ``doubling'' of $\bar F_0\hookrightarrow\partial_\infty X$ along $F\subseteq\bar F$).
Corollary \ref{a2semiorthogonal} shows that the natural embeddings
\begin{equation}
F\times T^*[0,1]\hookrightarrow(X,\bar F)\#_F(F\times A_2)
\end{equation}
coming from the unglued faces of $F\times A_2$ both induce fully faithful functors on wrapped Fukaya categories.
\end{example}

\section{Sectorial descent}\label{descent-sec}

We begin by discussing how stop removal (Theorem \ref{stopremoval}) and generation (Theorem \ref{generation}) apply in the case of Liouville sectors-with-corners:

\begin{corollary}\label{cornergenerationstopremoval}
Let $X$ be a Liouville sector-with-corners, whose boundary is the union of two faces meeting along the corner locus, fixing notation as in Construction \ref{cornerdescription}, so there is a commutative diagram
\begin{equation}
\begin{tikzcd}
\W(P)\ar{r}\ar{d}&\W(F)\ar{d}\\
\W(G)\ar{r}&\W(X)
\end{tikzcd}
\end{equation}
If (the convex completions of) $X$, $F$, $G$, and $P$ are all Weinstein up to deformation, then we have:
\begin{itemize}
\item$\W(X)$ is generated by the cocores of $X$ and the images of the cocores of $F$, $G$, and $P$.
\item For $\bar X$ the convex completion of $X$, the functor $\W(X)\to\W(\bar X)$ is the localization at the union of the images of $\W(F)$, $\W(G)$, and $\W(P)$.
\item For $\bar X^F$ the convex completion of $X$ along only the face corresponding to $F$, 
the functor $\W(X)\to\W(\bar X^F)$ is the localization at the image of $\W(F)$.
\end{itemize}
\end{corollary}

\begin{proof}
Given coordinates \eqref{cornernbhdfirst}--\eqref{cornernbhdlast}, we may perform the convex completion operation of gluing on $T^*\RR_{\geq 0}\times F$, $T^*\RR_{\geq 0}\times G$, and $T^*\RR_{\geq 0}^2\times P$ where $T^*\RR_{\geq 0}$ is given the Liouville vector field $Z_{T^*\RR_{\geq 0}}+\pi^*(\varphi(s)\partial_s)$ as in Construction \ref{cornerdescription}.
As explained below \eqref{productfiberreprise}, this operation coincides (up to deformation) with the convex completion of the smoothing of $X$.

The Liouville manifold $\bar X$ is equipped with a natural stop given by $\cc_F\cup(\RR\times\cc_P)\cup\cc_G$ (two copies of the relative cores of $F$ and $G$, glued along their common copy of $\RR\times\cc_P$), such that the pushforward
\begin{equation}\label{cornerstostops}
\W(X)\xrightarrow\sim\W(\bar X,\cc_F\cup(\RR\times\cc_P)\cup\cc_G)
\end{equation}
is a quasi-equivalence by Corollary \ref{horizsmallstop}.

Let us now assume $\bar X$, $\bar F$, $\bar G$, and $P$ are Weinstein up to deformation, and let us choose nice Liouville forms as follows.
By assumption, there exists $f$ such that $\lambda_P+df$ is Weinstein, so by adding $d(\varphi(t_1)\varphi(t_2)f)$ to $\lambda_X$ in coordinates \eqref{cornernbhdP}, we may assume that $\lambda_P$ is itself Weinstein.
Now $F$ has boundary neighborhood coordinates \eqref{cornernbhdF}, and we modify the Liouville form in these coordinates from $\lambda_P+\lambda_{T^*\RR_{\leq 0}}$ by adding $\pi^*(1-\varphi(t))\frac\partial{\partial t}$ where $\varphi:\RR_{\leq 0}\to\RR_{\geq 0}$ equals $1$ near zero and is supported near zero.
Now the locus $\{t\leq -1\}$ has convex boundary, and the completion along this boundary is $\bar F$.
There is thus a compactly supported $f$ (supported away from the boundary neighborhood coordinates) such that $\lambda_F+df$ is Weinstein.
We may add this to the Liouville form on $X$ (and similarly for $G$).
We may now apply Theorem \ref{generation} to $\W(\bar X,\cc_F\cup(\RR\times\cc_P)\cup\cc_G)$ and conclude that $\W(X)$ is generated by cocores of $X$ and the stabilizations of the cocores of $F$, $G$, and $P$ (which are the linking disks to the stop).

In view of the equivalence \eqref{cornerstostops}, stop removal implies that $\W(X)\to\W(\bar X)$ is the localization at the images of $\W(F)$, $\W(G)$, and $\W(P)$, as these precisely account for all the linking disks to the stop appearing in \eqref{cornerstostops}.
We may also convex complete along only one face, namely we may glue on $T^*\RR_{\geq 0}\times F$ (only) to $X$ to obtain what we might write as $\bar X^F$.
To apply stop removal to the functor $\W(X)\to\W(\bar X^F)$, we complete (fully) and add stops.
Namely, we consider the diagram
\begin{equation}
\begin{tikzcd}
\W(X)\ar{d}\ar{r}{\sim}&\W(\bar X,\cc_F\cup(\RR\times\cc_P)\cup\cc_G)\ar{d}\\
\W(\bar X^F)\ar{r}{\sim}&\W(\bar X,\cc_G).
\end{tikzcd}
\end{equation}
The horizontal arrows are quasi-equivalences, and stop removal shows that the right vertical arrow is localization at the image of $\W(F)$, and hence the same applies to $\W(X)\to\W(\bar X^F)$.
\end{proof}

\begin{proposition}\label{reducetotwo}
The case $n=2$ of Theorem \ref{weinsteindescent} implies the general case.
\end{proposition}

\begin{proof}
We argue by induction for $n\geq 3$.

For $I\subseteq\{1,\ldots,n\}$, we use the shorthand $X_I = \bigcap_{i \in I} X_i$ (so $X_\varnothing=X$).
There is a diagram of Liouville sectors $\{X_I\}_I$ over the poset $2^{\{1,\ldots,n\}}$ of subsets of $\{1,\ldots,n\}$ (ordered by reverse inclusion, so $\varnothing$ is maximal).
Let $\Sigma_n = \{ I \subseteq \{1, \ldots, n\} \,|\, I \neq \varnothing\}$ be the poset of non-empty subsets of $\{1, \ldots, n\}$, so $2^{\{1,\ldots,n\}}=\Sigma_n^\triangleright$ and hence we get (see \eqref{univhocolimfunctor}) a canonical map
\begin{equation}\label{reductiontarget}
\hocolim_{I\in\Sigma_n}\W(X_I)\to\W(X).
\end{equation}
Our aim is to show that this map is a pre-triangulated equivalence.

We decompose the poset
    \begin{equation}
        \Sigma_n =  P  \cup Q
    \end{equation}
    into $P := \{ I \in \Sigma_n \,|\, n \in I\}$ (those subsets which
    contain $n$) and $Q := \{ I \in \Sigma_n \,|\, I \neq \{n\}\}$ (those
    subsets which are not equal to the singleton $\{n\}$), which intersect in
    $R := P \cap Q = \{ J \cup \{n\} \,|\, J \in \Sigma_{n-1}\}$ (those
    subsets containing $n$ and at least one other element).
In other words, this is the decomposition of $\Sigma_n$ associated to the map $\Sigma_n\to\Sigma_2=(\bullet\leftarrow\bullet\to\bullet)$ defined by partitioning $\{1,\ldots,n\}$ into $\{1,\ldots,n-1\}$ and $\{n\}$.
The hypotheses of Lemma \ref{hocolimdecomposition} apply to this decomposition, showing that the natural square
\begin{equation}\label{hocolimdecompositionfukaya}
\begin{tikzcd}
\displaystyle\hocolim_{I \in R} \W(X_I)\ar{r}\ar{d}&\displaystyle\hocolim_{I \in P} \W(X_I)\ar{d}\\
\displaystyle\hocolim_{I \in Q} \W(X_I)\ar{r}&\displaystyle\hocolim_{I \in \Sigma_n} \W(X_I)
\end{tikzcd}
\end{equation}
is a homotopy pushout.

We now relate this homotopy pushout square \eqref{hocolimdecompositionfukaya} to the square associated to the decomposition
\begin{equation}
X = (X_1 \cup \cdots \cup X_{n-1})\cup X_n.
\end{equation}
Namely, there is a natural (strictly commuting!) map from \eqref{hocolimdecompositionfukaya} to
\begin{equation}\label{coarseningdecompositionfukaya}
\begin{tikzcd}
\W((X_1\cup\cdots\cup X_{n-1})\cap X_n)\ar{r}\ar{d}&\W(X_n)\ar{d}\\
\W(X_1\cup\cdots\cup X_{n-1})\ar{r}&\W(X)
\end{tikzcd}
\end{equation}
given on each corner of the square by the map \eqref{univhocolimfunctor} (more precisely, \eqref{hocolimdecompositionfukaya} maps to the localization of \eqref{coarseningdecompositionfukaya} at the identity morphisms).

Now \eqref{coarseningdecompositionfukaya} is a homotopy pushout by the case $n=2$ of Theorem \ref{weinsteindescent} (note that the cover $X = (X_1 \cup \cdots \cup X_{n-1})\cup X_n$ is Weinstein since the cover $X=X_1\cup\cdots\cup X_n$ is).
Hence to show that \eqref{reductiontarget} (which is the lower right component of the map $\eqref{hocolimdecompositionfukaya}\to\eqref{coarseningdecompositionfukaya}$) is a pre-triangulated equivalence, it is enough to show that the maps in each of the other three corners of $\eqref{hocolimdecompositionfukaya}\to\eqref{coarseningdecompositionfukaya}$ are pre-triangulated equivalences.
The map
\begin{equation}
\hocolim_{I \in P} \W(X_I)\to\W(X_n)
\end{equation}
is a quasi-equivalence since $\{n\}\in P$ is a maximal element.
The map
\begin{equation}
\hocolim_{I \in R} \W(X_I)\to\W((X_1\cup\cdots\cup X_{n-1})\cap X_n)
\end{equation}
is the descent map (note that $R=\Sigma_{n-1}$) associated to the cover of $(X_1\cup\cdots\cup X_{n-1})\cap X_n$ by $X_1\cap X_n,\ldots,X_{n-1}\cap X_n$ (which is Weinstein since $X=X_1\cup\cdots\cup X_n$ is), and hence is a pre-triangulated equivalence by the induction hypothesis (Theorem \ref{weinsteindescent} in the case $n-1$).
Finally, to analyze
\begin{equation}
\hocolim_{I \in Q} \W(X_I)\to\W(X_1\cup\cdots\cup X_{n-1}),
\end{equation}
note that $\Sigma_{n-1}\subseteq Q$ is cofinal in the sense of satisfying the hypotheses of Lemma \ref{cofinalityspecialcase}, and hence this is the same as 
\begin{equation}
\hocolim_{I \in \Sigma_{n-1}} \W(X_I)\to\W(X_1\cup\cdots\cup X_{n-1}),
\end{equation}
which is again a pre-triangulated equivalence by the induction hypothesis.
\end{proof}

\begin{proposition}\label{reducetotwowiththinintersection}
The case $n=2$ of Theorem \ref{weinsteindescent} is implied by the special case of two-element covers obtained by splitting along a sectorial hypersurface (Example \ref{glueandsplitcovering}).
\end{proposition}

\begin{figure}[ht]
\centering
\includegraphics[max width=.95\textwidth]{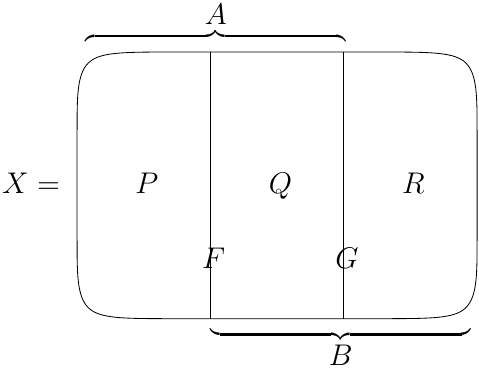}
\caption{A general two-element sectorial cover of $X=P\cup Q\cup R$ by $A=P\cup Q$ and $B=Q\cup R$.}\label{figuretwocoverreduction}
\end{figure}

\begin{proof}
Consider the situation in Figure \ref{figuretwocoverreduction}, namely $X=P\cup Q\cup R$ is covered by $A=P\cup Q$ and $B=Q\cup R$, which is the structure of a general two-element cover.
Note that in a general two-element sectorial cover $X=A\cup B$, the boundaries $\partial A$ and $\partial B$ cannot intersect, since if they did then $A$ and $B$ would not cover $X$ (compare Example \ref{twoballscover}).
Let a neighborhood of $P\cap Q$ be $F\times T^*[0,1]$ and a neighborhood of $Q\cap R$ be $G\times T^*[0,1]$.

Now we have the following diagram of $\ainf$-categories:
\begin{equation}
\begin{tikzcd}
\W(F)\ar{r}\ar{d}&\W(Q)\ar{r}\ar{d}&\W(B)\ar{d}\\
\W(P)\ar{r}&\W(A)\ar{r}&\W(X).
\end{tikzcd}
\end{equation}
The left square and the composite square are both associated to splitting along a sectorial hypersurface, and the right square is the one associated to our original arbitrary covering $X=A\cup B$.
We are thus done by Proposition \ref{compositepushout}.
\end{proof}

\begin{proof}[Proof of Theorem \ref{weinsteindescent}]
By Propositions \ref{reducetotwo} and \ref{reducetotwowiththinintersection}, it is enough to consider two-element covers obtained by splitting along a sectorial hypersurface (i.e.\ those from Example \ref{glueandsplitcovering}).
In other words, we are in the geometric setup of Construction \ref{glueandsplit}, namely we have Liouville pairs $(X,F)$ and $(Y,G)$ with a common Liouville subsector $F\supseteq Q\subseteq G$ (where $Q$ has boundary neighborhood coordinates $P\times T^*\RR_{\leq 0}$), where (the convexifications of) all of $X$, $Y$, $Q$, $F\setminus Q$, $G\setminus Q$, and $P$ are Weinstein (up to deformation), and we must show that the diagram
\begin{equation}\label{postlocalpush}
\begin{tikzcd}
\W(Q)\ar{r}\ar{d}&\W(X,F)\ar{d}\\
\W(Y,G)\ar{r}&\W((X,F)\#_Q(Y,G))
\end{tikzcd}
\end{equation}
is a homotopy pushout.

\begin{figure}[ht]
\centering
\includegraphics[max width=.95\textwidth]{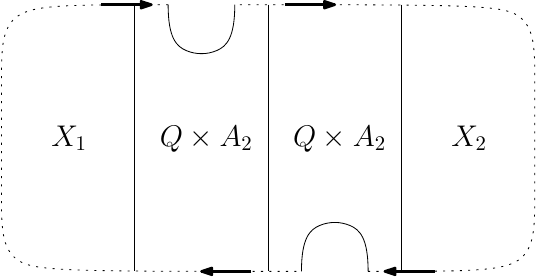}
\caption{Auxiliary geometric setup to prove homotopy pushout property.  The arrows indicate the direction of the Reeb flow.}\label{figurestoppedpushout}
\end{figure}

To show that \eqref{postlocalpush} is a homotopy pushout, we consider a modified geometric setup where $(X,F)$ is replaced with $(X,F)\#_Q(Q\times A_2)$, and similarly for $(Y,G)$.
In other words, we consider the square
\begin{equation}\label{prelocalpush}
\begin{tikzcd}
\W(Q)\ar{r}\ar{d}&\W((X,F)\#_Q(Q\times A_2))\ar{d}\\
\W((Q\times A_2)\#_Q(Y,G))\ar{r}&\W((X,F)\#_Q(Q\times A_2)\#_Q(Q\times A_2)\#_Q(Y,G))
\end{tikzcd}
\end{equation}
associated to the geometry illustrated in Figure \ref{figurestoppedpushout}.
Now Corollary \ref{a2semiorthogonal} provides a semi-orthogonal decomposition of $\W((X,F)\#_Q(Q\times A_2))$ into $\W(X,F)$ and $\W(Q)$ (that these generate follows from Theorem \ref{generation} and Corollary \ref{cornergenerationstopremoval}), and similarly for $\W((Q\times A_2)\#_Q(Y,G))$ (into $\W(Y,G)$ and $\W(Q)$) and $\W((X,F)\#_Q(Q\times A_2)\#_Q(Q\times A_2)\#_Q(Y,G))$ (into $\W(X,F)$, $\W(Y,G)$, and $\W(Q)$).
The square \eqref{prelocalpush} is thus of the shape considered in Example \ref{semiorthsquare}, and thus Proposition \ref{semiorthhocolim} applies to show that \eqref{prelocalpush} is a homotopy pushout.

We now aim to deduce that \eqref{postlocalpush} is a homotopy pushout by constructing a map $\eqref{prelocalpush}\to\eqref{postlocalpush}$ which is a localization and appealing to Lemma \ref{hocolimlocalcommute}.
The inclusion $A_2\hookrightarrow T^*[0,1]$ given by completing/convexifying/capping off the free boundary component of each $A_2$ defines the desired map $\eqref{prelocalpush}\to\eqref{postlocalpush}$.
In this inclusion $A_2\hookrightarrow T^*[0,1]$, there is one linking disk (arc) $\gamma$ in $A_2$ which is sent to a zero object, namely a parellel copy of the boundary component which gets capped off.
Now in the product $Q\times A_2$, we may consider Lagrangians $L\times\gamma$, where $L\subseteq Q$.
We claim that $\eqref{prelocalpush}\to\eqref{postlocalpush}$ quotients at precisely these objects $L\times\gamma$ by stop removal Theorem \ref{stopremoval} (note that this is a separate statement for each corner of the square).
This was shown in Corollary \ref{cornergenerationstopremoval}, so we are done.
\end{proof}

\appendix

\section{\texorpdfstring{$\ainf$}{A-infty}-categorical results}\label{ainftyappendix}

In this appendix, we record proofs of some foundational results about $\ainf$-categories.
The first few of these are well-known over a field; we give proofs in the case of coefficients in a general commutative ring (with cofibrancy assumptions as in \cite[\S 3.1]{gpssectorsoc}). In contrast with the body of this paper, our convention here in the appendix is that all morphisms considered between $\ainf$ categories are genuine $\ainf$ functors (rather than formal compositions of genuine functors and formal inverses of quasi-equivalences).
\subsection{Quasi-isomorphisms of \texorpdfstring{$\ainf$}{A-infty}-modules}

There are \emph{a priori} two different notions of a map $f:\M\to\N$ of $\C$-modules being a quasi-isomorphism.
On the one hand, one could work internally to the dg-category $\Mod\C$ and declare $f$ to be a quasi-isomorphism iff it induces an isomorphism in the cohomology category $H^0(\Mod\C)$ (equivalently, it is has a inverse up to homotopy).
On the other hand, one could work ``pointwise on $\C$'' and declare $f$ to be a quasi-isomorphism iff $f:\M(X)\to\N(X)$ is a quasi-isomorphism for every $X\in\C$.
Fortunately, these two notions turn out to be equivalent:

\begin{lemma}\label{moduleqiso}
A map $f:\M\to\N$ of $\C$-modules has a homotopy inverse iff it is a quasi-isomorphism pointwise on $\C$.
\end{lemma}

\begin{proof}
Obviously if $f$ has a homotopy inverse then it is a quasi-isomorphism pointwise.
The point is to prove the reverse direction, namely that if $f$ is a quasi-isomorphism pointwise, then it has an inverse up to homotopy.

Let $f:\M\to\N$ be given, and suppose $f:\M(X)\to\N(X)$ is a quasi-isomorphism for every $X\in\C$.
We construct $g:\N\to\M$ such that $gf\simeq\id_\M$.
This is enough to prove the desired result, since applying the same assertion to $g$, we find $h:\M\to\N$ with $hg\simeq\id_\N$, so $f=\id_\N f\simeq hgf\simeq h\id_\M=h$, and hence $f\simeq h$ and $g$ are inverses up to homotopy.

We are looking for
\begin{align}
g=\prod_{k\geq 0}g_k&\in\prod_{\begin{smallmatrix}k\geq 0\\X_0,\ldots,X_k\in\C\end{smallmatrix}}\Hom(\C(X_k,X_{k-1})[1]\otimes\cdots\otimes\C(X_1,X_0)[1]\otimes\N(X_0),\M(X_k))\\
w=\prod_{k\geq 0}w_k&\in\prod_{\begin{smallmatrix}k\geq 0\\X_0,\ldots,X_k\in\C\end{smallmatrix}}\Hom(\C(X_k,X_{k-1})[1]\otimes\cdots\otimes\C(X_1,X_0)[1]\otimes\M(X_0),\M(X_k))
\end{align}
such that $g$ is a cocycle of degree zero and $w$ is of degree $-1$ and satisfies $dw=gf-\id_\M$.
We construct $g$ and $w$ by induction on $k$.
For the $k=0$ step, we let $g_0$ be a homotopy inverse to $f_0$ (under our cofibrancy assumptions, a quasi-isomorphism is a homotopy equivalence \cite[Lemma 3.6]{gpssectorsoc}), and we let $w_0$ be any homotopy between $g_0f_0$ and the identity.

For the $k\geq 1$ inductive step, we are seeking to find $g_k$ and $w_k$ solving $(dg)_k=0$ and $(gf-dw)_k=0$.
First observe that (under the induction hypothesis)
\begin{equation}
(dg)_k\in\prod_{X_0,\ldots,X_k\in\C}\Hom(\C(X_k,X_{k-1})[1]\otimes\cdots\otimes\C(X_1,X_0)[1]\otimes\N(X_0),\M(X_k))
\end{equation}
is always a cocycle (for any choice of $g_k$), and moreover its class in cohomology is independent of $g_k$.
To find a $g_k$ for which $(dg)_k=0$, it is thus necessary and sufficient to show that this ``obstruction class'' $[(dg)_k]$ vanishes.
The identity $dw=gf-\id_\M$ in degrees $<k$ gives, upon applying $d$, the relation among ``obstruction classes'' $[(dg)_k]f_0+g_0[(df)_k]-[(d\id_\M)_k]=0$.
The obstruction classes of $\id_\M$ and $f$ both vanish since \emph{a fortiori} they are both cycles.
Since $f_0$ is a homotopy equivalence, we conclude that the obstruction class of $g$ vanishes, as desired.
This shows that we can choose $g_k$ to satisfy $(dg)_k=0$.

We now turn to the second equation $(gf-dw)_k=0$.
We have chosen $g$ to satisfy $(dg)_k=0$, so $gf-dw$ is a cocycle (up to degrees $\leq k$) in the morphism complex (for any choice of $w_k$).
Since $gf-dw=\id_\M$ in degrees $<k$, we conclude that $(gf-dw)_k$ is a cocycle (for any $w_k$).
Since $f_0$ is a homotopy equivalence, we may modify the cohomology class of this cocycle arbitrarily by adding to $g_k$ an appropriate cocycle (note that this operation preserves $(dg)_k=0$).
Performing such a modification of $g_k$ so that $(gf-dw)_k$ becomes null-homologous, we may now choose $w_k$ to ensure that $(gf-dw)_k=0$ as desired.
\end{proof}

\subsection{Yoneda lemma}

Recall that for left $\C$-modules $\M$ and $\N$, the mapping complex from $\M$ to $\N$ in the dg-category $\Mod\C$ is defined as
\begin{equation}\label{hombarcomplex}
\Hom_\C(\M,\N):=\prod_{\begin{smallmatrix}p\geq 0\\X_0,\ldots,X_p\in\C\end{smallmatrix}}\Hom(\C(X_p,X_{p-1})[1]\otimes\cdots\otimes\C(X_1,X_0)[1]\otimes\M(X_0),\N(X_p)).
\end{equation}
Recall also that there is a functor
\begin{align}
\C&\to\Mod\C,\\
X&\mapsto\C(-,X),
\end{align}
known as the Yoneda functor.

\begin{lemma}\label{yoneda}
For any left $\C$-module $\M$, the natural map $\M(X)\to\Hom_\C(\C(-,X),\M(-))$ is a quasi-isomorphism.
In particular (taking $\M=\C(-,Y)$), the Yoneda functor is fully faithful.
\end{lemma}

\begin{proof}
This argument is formally dual to \cite[Lemma 3.7]{gpssectorsoc}.
We consider the mapping cone
\begin{equation}
\prod_{\begin{smallmatrix}p\geq 0\\X=X_0,\ldots,X_p\in\C\end{smallmatrix}}\Hom(\C(X_p,X_{p-1})[1]\otimes\cdots\otimes\C(X_1,X_0)[1],\M(X_p))
\end{equation}
(note that the index $p$ above is shifted by one compared to \eqref{hombarcomplex}).
If $\C$ and $\M$ are strictly unital, then $f\mapsto f(-\otimes\1_X)$ is a contracting homotopy of this complex.
In the general cohomologically unital case, the argument is as follows.

Let $f=\prod_{p\geq 0}f_p$ be a cocycle in the complex above.
Let $m\geq 0$ be the smallest index such that $f_m\ne 0$.
Note that this implies that
\begin{equation}
f_m\in\prod_{X=X_0,\ldots,X_m\in\C}\Hom(\C(X_m,X_{m-1})\otimes\cdots\otimes\C(X_1,X_0),\M(X_m))
\end{equation}
is a cocycle.
Fix a cocycle $\1_X\in\C(X,X)$ representing the cohomological unit, and consider $g:=f(-\otimes\1_X)$.
Using the fact that $f$ is a cocycle, we may calculate that $(dg)_{m-1}=0$ and
\begin{equation}
(dg)_m=f_m(-\otimes\mu^2(-,\1_X)).
\end{equation}
Since $\mu^2(-,\1_X)$ is homotopic to the identity map \cite[Lemma 3.8]{gpssectorsoc}, the right hand side is chain homotopic to $f_m$.
We thus conclude that $(f-dg)_i=0$ for $i<m$ and $(f-dg)_m$ represents zero in cohomology.
Thus by further adding to $f-dg$ a coboundary, we can find a cocycle $f^+$ cohomologous to $f$ satisfying $(f^+)_i=0$ for $i\leq m$.
Iterating this procedure to take $m$ to infinity, we conclude that $f$ is cohomologous to zero.
\end{proof}

A $\C$-module is called \emph{representable} iff it lies in the essential image of the Yoneda embedding $\C\hookrightarrow\Mod\C$.
We record here the simple fact that representability is a cohomological question.

\begin{lemma}\label{representcohomology}
A $\C$-module $\M$ is representable iff the $H^\bullet\C$-module $H^\bullet\M$ is representable.
\end{lemma}

\begin{proof}
The nontrivial direction is to show that if $H^\bullet\M$ is representable, then so is $\M$.
Fix an isomorphism $H^\bullet\M(-)=H^\bullet\C(-,X)$ as $H^\bullet\C$-modules.
We consider the class $[e]\in H^0\M(X)$ corresponding to the identity $\1_X\in H^0\C(X,X)$.
Multiplication by $[e]$ defines a map of $H^\bullet\C$-modules
\begin{equation}\label{repcohiso}
H^\bullet\C(-,X)\to H^\bullet\M(-).
\end{equation}
This map is an isomorphism since under the identification of $H^\bullet\M(-)$ with $H^\bullet\C(-,X)$, it is simply the endomorphism of $H^\bullet\C(-,X)$ given by multiplication by $\1_X$.
Lift $[e]$ to a cycle $e\in\M(X)$ of degree zero, so now multiplication by $e$ defines a map of $\C$-modules
\begin{equation}
\C(-,X)\to\M(-)
\end{equation}
which reduces to \eqref{repcohiso} on cohomology, hence is a quasi-isomorphism `objectwise' on $\C$, hence a quasi-isomorphism in $\Mod\C$ by Lemma \ref{moduleqiso}.
\end{proof}

\subsection{Generation and quasi-equivalences}

For an $\ainf$-category $\C$, denote by $\left|\C\right|$ the set of quasi-isomorphism classes of objects (i.e.\ isomorphism classes of objects in the cohomology category $H^0\C$).
An $\ainf$-functor $F:\C\to\D$ induces a map $\left|\C\right|\to\left|\D\right|$, which is injective if $F$ is fully faithful and is surjective if $F$ is essentially surjective.
In particular, the tautological inclusion $\C\hookrightarrow\Tw\C$ induces an inclusion $\left|\C\right|\subseteq\left|\Tw\C\right|$.

\begin{lemma}\label{generationqiso}
Given a collection of objects $\A\subseteq\C$, the subset $\left|\Tw\A\right|\subseteq\left|\Tw\C\right|$ consisting of objects which are quasi-isomorphic to twisted complexes of objects in $\A$ depends only on $\left|\A\right|\subseteq\left|\C\right|$.
For a quasi-equivalence $F:\C\to\D$, we have $F(\left|\Tw\A\right|)=\left|\Tw F(\A)\right|$.
\end{lemma}

\begin{proof}
Given two objects $X,Y\in\C$, we may form $\cone(f):=[X[1]\xrightarrow fY]\in\Tw\C$ for any degree zero cycle $f\in\C(X,Y)$.
If $f,f'\in\C(X,Y)$ are cohomologous, then $\cone(f)$ and $\cone(f')$ are quasi-isomorphic objects of $\Tw\C$.
More generally, if $X\simeq X'$ and $Y\simeq Y'$ are quasi-isomorphic and $f\in\C(X,Y)$ and $f'\in\C(X',Y')$ are degree zero cycles whose classes coincide under the identification $H^0(\C(X,Y))=H^0(\C(X',Y'))$, then $\cone(f)$ and $\cone(f')$ are quasi-isomorphic objects of $\Tw\C$.
The result now follows by induction.
\end{proof}

\subsection{Homotopy colimits of \texorpdfstring{$\ainf$}{A-infty}-categories}\label{ainfhocolimconstruction}

We provide here an explicit construction of homotopy colimits of $\ainf$-categories (adapting \cite[\S A.3.5]{luriehttpub} to the $\ainf$ setting), and we derive some of their basic abstract properties.
For a discussion of how the notion we define here compares with other notions of homotopy colimits, see Remark \ref{inftycathocolim}.
As a reminder, all of the $\ainf$-categories in this paper are small (i.e.\ they have a \emph{set} of objects), and by a `commuting diagram of $\ainf$-categories' we always mean a \emph{strictly commuting diagram} (rather than only commuting up to specified natural quasi-isomorphism of functors).

We begin by introducing the \emph{Grothendieck construction} (or `lax homotopy colimit') of a diagram of $\ainf$-categories $\{\C_\sigma\}_{\sigma\in\Sigma}$ indexed by a poset $\Sigma$ (compare Thomason \cite{thomason}, Lurie \cite[Definition A.3.5.11]{luriehttpub}).
To set the stage, recall that the \emph{semi-orthogonal gluing} $\langle\C,\D\rangle_\B$ of $\ainf$-categories $\C$ and $\D$ along the $(\C,\D)$-bimodule $\B$ is, by definition, the $\ainf$-category whose objects are $\Ob\C\sqcup\Ob\D$, with $\C$ and $\D$ full subcategories, with $\B$ as the bimodule of morphisms from objects of $\C$ to objects of $\D$, and zero morphism spaces from $\D$ to $\C$.
The Grothendieck construction of a diagram of $\ainf$-categories of the form $\C_0\to\cdots\to\C_p$ is defined inductively by
\begin{align}
\Groth(\C)&:=\C,\\
\Groth(\C_0\xrightarrow{F_1}\cdots\xrightarrow{F_p}\C_p)&:=\Groth(\C_0\xrightarrow{F_1}\cdots\xrightarrow{F_{p-2}}\C_{p-2}\xrightarrow{F_{p-1}}\langle\C_{p-1},\C_p\rangle_{\C_p(F_p(-),-)}).
\end{align}
Let us now consider the functoriality of the Grothendieck construction.
Fix a weakly order preserving map of totally ordered sets $i:\{0<\cdots<q\}\to\{0<\cdots<p\}$ covered by a map of diagrams $(\C_0\to\cdots\to\C_q)\to i^*(\D_0\to\cdots\to\D_p)$, and let us define a functor $\Groth(\C_0\to\cdots\to\C_q)\to\Groth(\D_0\to\cdots\to\D_p)$.
We do so by induction on $q$.
In the case $q=0$, the desired map is evident, so suppose $q>0$ and let us reduce to the case $q-1$.
It suffices to define a map of diagrams from $(\C_0\to\cdots\to\C_{q-2}\to\langle\C_{q-1},\C_q\rangle_{\C_q(F_q(-),-)})$ to the result of applying the inductive definition of $\Groth$ to $\D_0\to\cdots\to\D_p$ to reduce $p$ to $i(q-1)$.
In other words, given a diagram
\begin{equation}
\begin{tikzcd}[column sep=large]
\C_{q-1}\ar[d]\ar[rr,"F_q"]&&\C_q\ar[d]\\
\D_{i(q-1)}\ar[r,"G_{i(q-1)+1}"]&\cdots\ar[r,"G_{i(q)}"]&\D_{i(q)}
\end{tikzcd}
\end{equation}
we should produce a functor $\langle\C_{q-1},\C_q\rangle_{\C_q(F_q(-),-)}\to\Groth(\D_{i(q-1)}\to\cdots\to\D_{i(q)})$.
There is a tautological such functor since pullback of bimodules is associative: $\B(F(G(-)),-)=\B((F\circ G)(-),-)$, so there is a full subcategory
\begin{equation}
\langle\D_{i(q-1)},\D_{i(q)}\rangle_{\D_{i(q)}((G_{i(q)}\circ\cdots\circ G_{i(q-1)+1})(-),-)}\subseteq\Groth(\D_{i(q-1)}\to\cdots\to\D_{i(q)}).
\end{equation}
This fixes the functoriality of the Grothendieck construction of diagrams of $\ainf$-categories indexed by finite totally ordered sets.

\begin{example}
Consider the Grothendieck construction of a diagram of identity functors $\Groth(\C\to\C\to\C\to\C)$.
It can be described alternatively as follows.
Consider the $\ainf$-category $4\C$ whose set of objects is $\{0,1,2,3\}\times\Ob\C$ and in which the morphisms and $\ainf$ operations are simply those from $\C$ (ignoring the $\{0,1,2,3\}$ factor).
The Grothendieck construction above is the subcategory of $4\C$ in which we allow nonzero morphisms from $(i,X)$ to $(j,Y)$ only when $i\leq j$.
Notice in particular that the $\mu^3$ operation for a chain of objects $((0,X_0),(1,X_1),(2,X_2),(3,X_3))$ is just the $\mu^3$ operation for $(X_0,X_1,X_2,X_3)$ in $\C$; in particular, it can be nontrivial (despite the fact that we began with a strict diagram of $\ainf$-categories and $\ainf$-functors).
\end{example}

Now for a diagram of $\ainf$-categories $\{\C_\sigma\}_{\sigma\in\Sigma}$ indexed by any poset $\Sigma$, we define $\Groth_{\sigma\in\Sigma}\C_\sigma$ to have objects $\bigsqcup_{\sigma\in\Sigma}\C_\sigma$, with morphism spaces from $\C_\sigma$ to $\C_{\sigma'}$ vanishing unless $\sigma\leq\sigma'$, and so that the full subcategory spanned by $\Ob\C_{\sigma_0}\sqcup\cdots\sqcup\Ob\C_{\sigma_p}$ for any $\sigma_0<\cdots<\sigma_p$ is given by $\Groth(\C_{\sigma_0}\to\cdots\to\C_{\sigma_p})$.
The Grothendieck construction is functorial under maps of posets: a weakly order preserving map $\Sigma\to T$ covered by a map of diagrams $\{\C_\sigma\}_{\sigma\in\Sigma}\to\{\D_\tau\}_{\tau\in T}$ induces a functor $\Groth_{\sigma\in\Sigma}\C_\sigma\to\Groth_{\tau\in T}\D_\tau$.

\begin{lemma}\label{grothff}
If each map $\C_\sigma\to\D_\sigma$ is fully faithful (resp.\ essentially surjective, generating, split-generating), then so is $\Groth_{\sigma\in\Sigma}\C_\sigma\to\Groth_{\sigma\in\Sigma}\D_\sigma$.\qed
\end{lemma}

\begin{definition}\label{hocolimdef}
The \emph{homotopy colimit} of a diagram of $\ainf$-categories $\{\C_\sigma\}_{\sigma\in\Sigma}$ is defined as the localization of the Grothendieck construction
\begin{equation}
\hocolim_{\sigma\in\Sigma}\C_\sigma:=\Bigl(\Groth_{\sigma\in\Sigma}\C_\sigma\Bigr)[A_\Sigma^{-1}]
\end{equation}
at the collection $A_\Sigma$ of `adjacent' morphisms $X_\sigma\to F_{\sigma'\sigma}X_\sigma$ corresponding to the identity map in $H^0\C_{\sigma'}(F_{\sigma'\sigma}X_\sigma,F_{\sigma'\sigma}X_\sigma)=H^0(\Groth_{\sigma\in\Sigma}\C_\sigma)(X_\sigma,F_{\sigma'\sigma}X_\sigma)$ for $\sigma\leq\sigma'$ and $X_\sigma\in\C_\sigma$ (compare \cite[Lemma A.3.5.13]{luriehttpub}).
\end{definition}

Homotopy colimits enjoy the same functoriality as the Grothendieck construction: a map $\Sigma\to T$ covered by a map of diagrams $\{\C_\sigma\}_{\sigma\in\Sigma}\to\{\D_\tau\}_{\tau\in T}$ induces a functor
\begin{equation}\label{hocolimfunctorial}
\hocolim_{\sigma\in\Sigma}\C_\sigma\to\hocolim_{\tau\in T}\D_\tau
\end{equation}
since $A_\Sigma$ is sent into $A_T$.
(Recall the precise definition \cite[Definition 3.17]{gpssectorsoc} of the localization $\C[W^{-1}]$ as the quotient of $\C$ by the set of all cones $[X\xrightarrow aY]$ where $X,Y\in\C$ and $a\in\C(X,Y)$ is a cycle representing an element of $W$; this definition is strictly functorial in the sense that for $F:\C\to\D$ with $F(W)\subseteq Z$, there is a canonically induced functor $\C[W^{-1}]\to\D[Z^{-1}]$.)

Recall that an $\ainf$-functor is called a quasi-equivalence (resp.\ pre-triangulated equivalence, Morita equivalence) iff it is fully faithful and essentially surjective (resp.\ generating, split-generating).

\begin{lemma}\label{hocolimofequivalences}
If each map $\C_\sigma\to\D_\sigma$ is a quasi-equivalence (resp.\ pre-triangulated equivalence, Morita equivalence), then so is $\hocolim_{\sigma\in\Sigma}\C_\sigma\to\hocolim_{\sigma\in\Sigma}\D_\sigma$.
\end{lemma}

\begin{proof}
Combine Lemma \ref{grothff} with \cite[Corollary 3.14]{gpssectorsoc}.
\end{proof}

As a special case of the functoriality of homotopy colimits under maps of diagrams, note that for any diagram $\{\C_{\sigma}\}_{\sigma \in \Sigma^\triangleright}$ (where $\Sigma^\triangleright = \Sigma \cup \{*\}$ denotes the poset formed from $\Sigma$ by adding a maximal element), there is an induced functor
\begin{equation}\label{univhocolimfunctor}
\hocolim_{\sigma\in\Sigma}\C_\sigma\to\C_*[\1^{-1}]\xleftarrow\sim\C_*
\end{equation}
where $\1$ denotes the class of identity morphisms (the homotopy colimit of the diagram $\C_*$ over the one-element poset $\{*\}$ is given by this somewhat silly localization $\C_*[\1^{-1}]$).
This map is functorial, in the sense that a map of posets $\Sigma\to T$ covered by a map of diagrams $\{\C_{\sigma}\}_{\sigma \in \Sigma^\triangleright}\to\{\D_\tau\}_{\tau \in T^\triangleright}$ induces a diagram
\begin{equation}
\begin{tikzcd}
\hocolim\limits_{\sigma\in\Sigma}\C_\sigma\ar[r]\ar[d]&\C_*[\1^{-1}]\ar[d]&\ar[l,swap,"\sim"]\C_*\ar[d]\\
\hocolim\limits_{\tau\in T}\D_\tau\ar[r]&\D_*[\1^{-1}]&\ar[l,swap,"\sim"]\D_*
\end{tikzcd}
\end{equation}

\begin{definition}
A diagram $\{\C_\sigma\}_{\sigma\in\Sigma^\triangleright}$ is called a \emph{homotopy colimit diagram} iff \eqref{univhocolimfunctor} is a pre-triangulated equivalence.
It is called an \emph{almost homotopy colimit diagram} iff \eqref{univhocolimfunctor} is fully faithful.
In the case of a square diagram (namely $\Sigma=(\bullet\leftarrow\bullet\to\bullet)$), we may write \emph{homotopy pushout (square)} (respectively \emph{almost homotopy pushout (square)}) in place of \emph{homotopy colimit (diagram)} (respectively \emph{almost homotopy colimit (diagram)}).
\end{definition}

\begin{remark}
Requiring \eqref{univhocolimfunctor} to be a quasi-equivalence, a pre-triangulated equivalence, or a Morita equivalence are all reasonable definitions of `homotopy colimit diagram'.
They do not differ much: they all entail full faithfulness of \eqref{univhocolimfunctor} and some sort of `full image' statement (essentially surjective, generating, or split-generating, respectively), of which full faithfulness is the most interesting and nontrivial condition.
They simply correspond to working in an `$\infty$-category of $\ainf$-categories' (compare Remark \ref{inftycathocolim} below) or its full subcategory of pre-triangulated or split-closed pre-triangulated $\ainf$-categories.
We consider pre-triangulated equivalence because it is the notion which appears in our main descent results Theorems \ref{mvthm} and \ref{weinsteindescent}.
Quasi-equivalence is obviously too strong a requirement for these results: a Lagrangian in $X_1\cup X_2$ has no reason to be isomorphic in the Fukaya category to one contained entirely in $X_1$ or $X_2$, and indeed in even the simplest of examples this is false.
Morita equivalence is unnecessarily weak.
\end{remark}

It is easy enough to check (using Lemma \ref{hocolimofequivalences}) that for a map of diagrams $\{\C_\sigma\}_{\sigma\in\Sigma^\triangleright}\to\{\D_\sigma\}_{\sigma\in\Sigma^\triangleright}$ where each $\C_\sigma\to\D_\sigma$ is a pre-triangulated equivalence, each diagram is a(n almost) homotopy colimit diagram iff the other is.

A functor $F:\C\to\D$ is said to quotient by a set of isomorphism classes $Q\subseteq\left|\Tw\C\right|$ iff $F(Q)$ are zero objects and the induced functor $\C/Q\to\D/F(Q)$ is a quasi-equivalence (this condition is independent of the choice of set of objects representing the isomorphism classes in $Q$ used to define $\C/Q\to\D/F(Q)$).

\begin{lemma}[Localization and homotopy colimits commute]\label{hocolimlocalcommute}
If each map $\C_\sigma\to\D_\sigma$ quotients by some set of isomorphism classes $Q_\sigma\subseteq\left|\Tw\C_\sigma\right|$, then $\hocolim_{\sigma\in\Sigma}\C_\sigma\to\hocolim_{\sigma\in\Sigma}\D_\sigma$ quotients by the union $\bigcup_\sigma Q_\sigma$.\qed
\end{lemma}

A basic property of colimits is that \emph{cofinal} diagrams of categories have the same colimit (see, e.g., \cite[\S 4.1.1]{luriehttpub} for rather general homotopy-invariant notions of cofinality and this property).
For our purposes, we need only the following very special case of this property:

\begin{lemma}\label{cofinalityspecialcase}
Suppose that an inclusion of posets $P \hookrightarrow \Sigma$ satisfies the following properties:  
\begin{enumerate}
\item \label{cofinal:minimal} For every $\sigma \in \Sigma$, the subset of elements of $P$ greater than or equal to $\sigma$ has a unique minimal element $p_{\sigma}$.
\item \label{cofinal:upwardsclosed} For all $p\in P$ and $p\leq q$, we have $q\in P$.
\end{enumerate}
Then, the canonical map $\hocolim_{\sigma \in P} \C_{\sigma} \to \hocolim_{\sigma \in \Sigma} \C_{\sigma}$ is a quasi-equivalence.
\end{lemma}

\begin{proof}
Let us denote by $\G_{\Sigma} := \Groth_{\sigma \in \Sigma} \C_{\sigma}$ the Grothendieck construction of the diagram indexed by $\Sigma$, and similarly $\G_{P}$ and $\G_{\Sigma - P}$.
We note first that condition \ref{cofinal:upwardsclosed} implies that there is a semi-orthogonal decomposition $\G_{\Sigma} = \langle \G_{\Sigma - P}, \G_P \rangle$.
In the formation of the homotopy colimit over $\Sigma$, the class of morphisms $A_{\Sigma}$ we need to invert can be divided into three classes $A_\Sigma=A_P\sqcup A_{\Sigma-P}\sqcup A_{\Sigma-P\to P}$, where $A_P$ and $A_{\Sigma-P}$ are the morphisms between objects of $\G_P$ and $\G_{\Sigma-P}$, respectively, and $A_{\Sigma-P\to P}$ are the morphisms from $\G_{\Sigma - P}$ to $\G_P$.

Let $A_{\Sigma - P \to P}^{\min} \subseteq A_{\Sigma - P \to P}$ be the subset of morphisms from an object of $\C_{\sigma}$ to an object of $\C_{p_{\sigma}}$, where $p_{\sigma}$ denotes the minimal $p \in P$ greater than $\sigma$ guaranteed by \ref{cofinal:minimal}. 
Observe that any morphism in $A_{\Sigma - P \to P}$ can be factored as a morphism in $A_{\Sigma - P \to P}^{\min}$ followed by a morphism in $A_P$;
it follows that the localization of $\G_{\Sigma}$ by $A_{\Sigma}$ coincides with the localization by $A_{\Sigma - P} \sqcup A_{\Sigma - P \to P}^{\min} \sqcup A_P$.
(Note that we may always localize in steps: for classes of morphisms $W,Z\subseteq H^0\C$, the natural functors $\C[W^{-1}][Z^{-1}]\to\C[(W\cup Z)^{-1}][Z^{-1}]\leftarrow\C[(W\cup Z)^{-1}]$ are both quasi-equivalences.)

Now observe that $\cones(A^{\min}_{\Sigma-P\to P})$ is left-orthogonal to $\G_P$.
To see this, consider any morphism $z \in A^{\min}_{\Sigma - P \to P}$, concretely $z: c_{\sigma} \to f_{p_{\sigma},\sigma} c_{\sigma}$ for some $\sigma\in\Sigma-P$, and note that for any object $d\in \C_p \subseteq \G_P$ for $p \in P$, the minimality condition \ref{cofinal:minimal} on $p_{\sigma}$ implies that $\sigma\leq p$ iff $p_\sigma\leq p$.
If $\sigma\leq p$ and $p_\sigma\leq p$ are both false, then there is nothing to show; if they both hold, then the map
\begin{equation}
\C_p(f_{p,\sigma}c_\sigma,d)=\G_\Sigma(f_{p_\sigma,\sigma}c_\sigma,d)\xrightarrow{\circ z}\G_\Sigma(c_\sigma,d)=\C_p(f_{p,\sigma}c_\sigma,d)
\end{equation}
is a quasi-isomorphism as desired.

Since $\cones(A_{\Sigma - P \to P}^{\min})$ is left-orthogonal to $\G_P$, it follows that $\G_P \hookrightarrow \G_\Sigma [(A_{\Sigma - P \to P}^{\min})^{-1}]$ is fully faithful \cite[Lemma 3.13]{gpssectorsoc}.
We may now further localize both sides by $A_P$ (Lemma \ref{quotientfullfaithful}) to obtain a fully faithful embedding
\begin{equation}
\G_P[A_P^{-1}]\hookrightarrow\G_\Sigma [(A_P\sqcup A_{\Sigma - P \to P}^{\min})^{-1}].
\end{equation}
Now, we claim that \ref{cofinal:minimal} implies that the morphisms in $A_{\Sigma - P}$ are already isomorphisms in this target $\G_\Sigma [(A_P\sqcup A_{\Sigma - P \to P}^{\min})^{-1}]$.
Indeed, since $\sigma \leq \sigma'$ implies $p_{\sigma} \leq p_{\sigma'}$, any adjacent morphism $c \to f_{\sigma', \sigma} c$ in $A_{\Sigma - P}$ differs from the corresponding adjacent morphism  $f_{p_{\sigma},\sigma} c \to f_{p_{\sigma'}, p_{\sigma}} f_{p_\sigma,\sigma} c$ (which is in $A_P$) by morphisms in $A_{\Sigma - P \to P}^{\min}$, and hence we conclude by \cite[Lemma 3.13]{gpssectorsoc} that
\begin{equation}
\G_P[A_P^{-1}]\hookrightarrow\G_\Sigma [(A_P\sqcup A_{\Sigma - P \to P}^{\min}\sqcup A_{\Sigma-P})^{-1}]
\end{equation}
is fully faithful (and we have already argued above that this functor is the desired map on homotopy colimits $\hocolim_{\sigma \in P} \C_{\sigma} \to \hocolim_{\sigma \in \Sigma} \C_{\sigma}$).
Finally, note that this functor is essentially surjective since every element of $\Sigma$ is $\leq$ some element of $P$ by \ref{cofinal:minimal}.
\end{proof}

\begin{corollary}\label{colimitmaximalelement}
If $\Sigma$ has a maximal element $\tau \in \Sigma$, then the natural functor $\C_\tau\xrightarrow\sim\hocolim_{\sigma\in\Sigma}\C_{\sigma}$ is a quasi-equivalence.
\qed
\end{corollary}

The quasi-equivalence in Corollary \ref{colimitmaximalelement} above is moreover functorial in $\Sigma$.
This is easier to see for its inverse, which may be realized by beginning with the natural map $\Groth_{\sigma\in\Sigma}\C_\sigma\to\C_\tau$ and localizing so that the domain becomes $\hocolim_{\sigma\in\Sigma}\C_\sigma$.
Now for any map $\Sigma\to\Sigma'$ of posets with maximal elements $\tau$ and $\tau'$ (where $\tau$ is not necessarily sent to $\tau'$), covered by a map between diagrams of $\ainf$-categories $\{\C_\sigma\}_{\sigma\in\Sigma}\to\{\C_{\sigma'}\}_{\sigma'\in\Sigma'}$, there is a tautologically commutative diagram
\begin{equation}
\begin{tikzcd}
\displaystyle\Groth_{\sigma\in\Sigma}\C_\sigma\ar{r}\ar{d}&\C_\tau\ar{d}\\
\displaystyle\Groth_{\sigma'\in\Sigma'}\C_{\sigma'}\ar{r}&\C_{\tau'}
\end{tikzcd}
\end{equation}
and localizing this diagram to obtain $\hocolim_{\sigma\in\Sigma}\C_\sigma\to\hocolim_{\sigma'\in\Sigma'}\C_{\sigma'}$ on the left gives the desired result.

\begin{proposition}\label{semiorthhocolim}
Suppose $\{\C_\sigma\}_{\sigma\in\Sigma^\triangleright}$ satisfies the following properties:
\begin{enumerate}
\item\label{hypff}All functors $\C_\sigma\hookrightarrow\C_{\sigma'}$ for $\sigma\leq\sigma'$ are fully faithful.
\item Every $\Tw^\pi\C_\sigma$ is generated by the images of $\C_{\sigma'}$ for $\sigma'<\sigma$ and objects which are left-orthogonal to all these images (call such $X\in\Tw^\pi\C_\sigma$ \emph{right-new}).
\item\label{neworth}If $X\in\C_\sigma$ and $Y\in\C_{\sigma'}$ are right-new and $\sigma$ and $\sigma'$ are incomparable, then $\Hom_{\sigma''}(X,Y)$ is acyclic for every $\sigma''\geq\sigma,\sigma'$.
\end{enumerate}
Then $\{\C_\sigma\}_{\sigma\in\Sigma^\triangleright}$ is an almost homotopy colimit diagram.
\end{proposition}

\begin{proof}
We consider $\G:=\Groth_{\sigma\in\Sigma}\C_\sigma$ and $\H:=\hocolim_{\sigma\in\Sigma}\C_\sigma$.
Since $\Tw^\pi\C_\Sigma$ is generated by right-new objects, it is equivalent to localize $\G$ at the adjacent morphisms $Y_\tau\to F_{\tau'\tau}Y_\tau$ for right-new objects $Y_\tau\in\Tw^\pi\C_\tau$.
We now claim that every right-new object $X_\sigma\in\Tw^\pi\C_\sigma$ is, when regarded as an object of $\Tw^\pi\G$, left-orthogonal to the cones on these adjacent morphisms.
Indeed, to verify that $\G(X_\sigma,\cone(Y_\tau\to F_{\tau'\tau}Y_\tau))$ is acyclic, there are a few cases to consider: when $\sigma$ and $\tau$ are incomparable, this follows from \ref{neworth}, when $\sigma\leq\tau$, this follows from full faithfulness of $\C_\tau\to\C_{\tau'}$, and when $\sigma>\tau$, the only possibility for a nonzero morphism is if $\tau'\geq\sigma$, and then it remains acyclic since $X_\sigma$ is right-new and $\C_\sigma\to\C_{\tau'}$ is fully faithful.
Thus the map $\G(X_\sigma,-)\xrightarrow\sim\H(X_\sigma,-)$ is a quasi-isomorphism for any right-new object $X_\sigma\in\Tw^\pi\C_\sigma$ \cite[Lemma 3.13]{gpssectorsoc}.

Now $\Tw^\pi\H$ is generated by images of right-new objects, so to check that $\H\to\C_*$ is a quasi-equivalence, it is enough to check that that $\G\to\C_*$ induces a quasi-isomorphism on morphisms $X_\sigma\to Y_{\sigma'}$ for right-new objects $X_\sigma\in\Tw^\pi\C_\sigma$ and $Y_{\sigma'}\in\Tw^\pi\C_{\sigma'}$.
If $\sigma$ and $\sigma'$ are incomparable, then $\G(X_\sigma,Y_{\sigma'})$ is acyclic by inspection, and $\C_*(F_{*\sigma}X_\sigma,F_{*\sigma'}Y_{\sigma'})$ is acyclic by \ref{neworth}.
If $\sigma'<\sigma$, then $\G(X_\sigma,Y_{\sigma'})$ is acyclic by inspection, and $\C_*(F_{*\sigma}X_\sigma,F_{*\sigma'}Y_{\sigma'})$ is acyclic since $X_\sigma$ is right-new and $\C_\sigma\to\C_*$ is fully faithful.
If $\sigma\leq\sigma'$, then $\G(X_\sigma,Y_{\sigma'})$ and $\C_*(F_{*\sigma}X_\sigma,F_{*\sigma'}Y_{\sigma'})$ are both quasi-isomorphic to $\C_{\sigma'}(F_{\sigma'\sigma}X_\sigma,X_{\sigma'})$ since $\C_{\sigma'}\to\C_*$ is fully faithful.
Thus $\{\C_\sigma\}_{\sigma\in\Sigma^\triangleright}$ is an almost homotopy colimit diagram.
\end{proof}

\begin{example}\label{semiorthsquare}
Let $\C$, $\D_1$, $\D_2$ be $\ainf$-categories, and let $\B_i$ be $(\C,\D_i)$-bimodules.
The square of $\ainf$-categories
\begin{equation}
\begin{tikzcd}
\C\ar{r}\ar{d}&\langle\C,\D_1\rangle_{\B_1}\ar{d}\\
\langle\C,\D_2\rangle_{\B_2}\ar{r}&\langle\C,\D_1\sqcup\D_2\rangle_{\B_1\sqcup\B_2}
\end{tikzcd}
\end{equation}
satisfies the hypotheses of Proposition \ref{semiorthhocolim} and is thus a homotopy pushout square.
\end{example}

\begin{example}\label{formalreverse}
Let $\{\C_\sigma\}_{\sigma\in\Sigma}$ be any diagram.
The diagram $\{\Groth_{\tau\leq\sigma}\C_\tau\}_{\sigma\in\Sigma^\triangleright}$ is a homotopy colimit diagram by Proposition \ref{semiorthhocolim}.
Localizing at $A_{\Sigma^{\leq\sigma}}$ and appealing to Lemma \ref{hocolimlocalcommute} and Corollary \ref{colimitmaximalelement}, we conclude that the diagram $\sigma\mapsto\C_\sigma$ and $\ast\mapsto\bigl(\Groth_{\sigma\in\Sigma}\C_\sigma\bigr)[A_\Sigma^{-1}]$ is a homotopy colimit diagram.
We have thus derived what was our original definition of homotopy colimits as a consequence of the formal properties Lemma \ref{hocolimlocalcommute} and Proposition \ref{semiorthhocolim} they satisfy (note that in this argument, Corollary \ref{colimitmaximalelement} was \emph{not} used as a formal property of homotopy colimits, rather it was applied as a statement about localizations of Grothendieck constructions).
In other words, given a functor which associates an $\ainf$-category to any diagram of $\ainf$-categories indexed by a finite poset, that functor is equivalent to the functor $\{\C_\sigma\}_{\sigma\in\Sigma}\mapsto\hocolim_{\sigma\in\Sigma}\C_\sigma$ defined above iff it commutes with localization (as in Lemma \ref{hocolimlocalcommute}) and satisfies Proposition \ref{semiorthhocolim}.
\end{example}

\begin{remark}
The hypotheses of Proposition \ref{semiorthhocolim} are not (obviously) preserved under passing to opposite $\ainf$-categories.
This asymmetry can be traced back to the asymmetry in our definition of homotopy colimits, which is also not obviously preserved under passing to opposites.
The converse is true as well: by Example \ref{formalreverse}, verifying the `opposite' of Proposition \ref{semiorthhocolim} for our definition of homotopy colimits would imply they are compatible with passing to opposites.
\end{remark}

\begin{remark}\label{inftycathocolim}
A folklore result (for example, there is the unpublished work of Cohn \cite{cohn}) states that pre-triangulated $\ainf$-categories over a field (or perhaps more generally a commutative ring, with cofibrancy assumptions as in \cite[\S 3.1]{gpssectorsoc}) $k$ are the same as $k$-linear stable $\infty$-categories.
In the latter $\infty$-category (of $k$-linear stable $\infty$-categories), we may consider (homotopy) colimits in the $\infty$-categorical sense (i.e.\ satisfying the relevant universal property).
We could also consider an appropriately defined $\infty$-category of $\ainf$-categories.
We expect (and closely related results appear in the literature \cite[Lemma A.3.5.13]{luriehttpub}) that these notions of homotopy colimit agree with what we have defined here, namely that a diagram of pre-triangulated $\ainf$-categories over $\Sigma^\triangleright$ is a homotopy colimit diagram in the sense we define here iff it is a colimit diagram in the $\infty$-categorical sense.
To prove this, it would suffice, by Example \ref{formalreverse}, to show that the $\infty$-categorical colimits commute with localization in the sense of Lemma \ref{hocolimlocalcommute} and satisfy Proposition \ref{semiorthhocolim} (even just in the special case of $\{\Groth_{\tau\leq\sigma}\C_\tau\}_{\sigma\in\Sigma^\triangleright}$).
\end{remark}

It is a well-known fact that, given a ``nice decomposition'' of a diagram $\Sigma$, the (homotopy) colimit can be itself decomposed into (a colimit of) smaller colimits (see e.g., \cite[\S 4.2.3]{luriehttpub} which establishes a general framework for such decomposition formulae).
For our purposes, the following special case of this property (discussed as motivation in \cite[beginning of \S 4.2]{luriehttpub}) will be sufficient:

\begin{lemma}\label{hocolimdecomposition}
Consider a poset $\Sigma$ with a map to the poset $\bullet\leftarrow\bullet\to\bullet$, and denote by $P$ and $Q$ the inverse images of $\bullet\leftarrow\bullet$ and $\bullet\to\bullet$, respectively, intersecting along $R:=P\cap Q$.
Equivalently, $P$ and $Q$ are sub-posets of $\Sigma$ with $P\cup Q=\Sigma$, such that no element of $P \backslash R$ is comparable to any element of $Q \backslash R$, and no element of $\Sigma\setminus R$ is less than any element of $R$.
Then the square
\begin{equation}\label{hocolimsquare}
\begin{tikzcd}
\displaystyle\hocolim_{\sigma \in R} \C_{\sigma} \ar{r}\ar{d}& \displaystyle\hocolim_{\sigma \in Q} \C_{\sigma}\ar{d}\\
\displaystyle\hocolim_{\sigma \in P} \C_{\sigma}  \ar{r}& \displaystyle\hocolim_{\sigma \in \Sigma} \C_{\sigma}
\end{tikzcd}
\end{equation}
is a homotopy pushout.
\end{lemma}

\begin{proof}
    Let $\G_\Sigma:=\Groth_{\sigma\in\Sigma}\C_\sigma$ and similarly denote by
    $\G_{P}$, $\G_Q$, and $\G_R$ the Grothendieck constructions of the
    restricted diagrams over $P$, $Q$, and $R$. By definition, there is a commutative square of fully faithful embeddings
    \begin{equation}\label{grothendieckdecomposition}
    \begin{tikzcd}
    \G_R \ar{r}\ar{d}& \G_Q\ar{d}\\
    \G_P  \ar{r}& \G_\Sigma.
    \end{tikzcd}
    \end{equation}
    In fact, we claim that \eqref{grothendieckdecomposition} is a homotopy
    pushout. To see this, note that the hypotheses on $P$, $Q$, and $R$ imply
    that $\G_P = \langle \G_R, \G_{P \backslash R} \rangle$, $\G_Q = \langle
    \G_R, \G_{Q \backslash R} \rangle$, and in $\G_{\Sigma}$ there are no
    morphisms between $\G_{P \backslash R}$ and $\G_{Q \backslash R}$ (in either direction); now
    apply Example \ref{semiorthsquare}.

To obtain \eqref{hocolimsquare} from \eqref{grothendieckdecomposition}, we localize at morphisms $A_P$, $A_Q$, $A_R$, and $A_\Sigma$.
Since $A_\Sigma=A_P\cup A_Q$, we may apply Lemma \ref{hocolimlocalcommute} to the natural map from \eqref{grothendieckdecomposition} to \eqref{hocolimsquare} to conclude that \eqref{hocolimsquare} is a homotopy pushout.
\end{proof}

\begin{lemma}\label{compositepushout}
For any diagram
\begin{equation}\label{compositediagram}
\begin{tikzcd}
\A\ar{r}\ar{d}&\B\ar{d}\ar{r}&\C\ar{d}\\
\D\ar{r}&\E\ar{r}&\F
\end{tikzcd}
\end{equation}
in which the leftmost square a homotopy pushout, the composite square is a homotopy pushout iff the rightmost square is.
\end{lemma}

\newlength\downarrowwidth
\settowidth\downarrowwidth{$\downarrow$}
\newlength\Acalwidth
\settowidth\Acalwidth{$\A$}
\newlength\Bcalwidth
\settowidth\Bcalwidth{$\B$}
\newlength\widthone
\setlength\widthone{0.5\Acalwidth}
\addtolength\widthone{-0.5\downarrowwidth}
\newlength\allwidth
\settowidth\allwidth{$\A\to\B$}
\newlength\widthtwo
\setlength\widthtwo{1\allwidth}
\addtolength\widthtwo{-0.5\Bcalwidth}
\addtolength\widthtwo{-0.5\downarrowwidth}
\addtolength\widthtwo{-\widthone}
\addtolength\widthtwo{-\downarrowwidth}

\begin{proof}
By Lemma \ref{hocolimdecomposition}, the following square is a homotopy pushout
\begin{equation}
\begin{tikzcd}
\hocolim\left(\begin{matrix}\A\to\B\hfill\\\hspace{\widthone}\downarrow\hfill\\\D\hfill\end{matrix}\right)\ar{r}\ar{d}&\hocolim\left(\begin{matrix}\A\to\B\to\C\hfill\\\hspace{\widthone}\downarrow\hfill\\\D\hfill\end{matrix}\right)\ar{d}\\
\hocolim\left(\begin{matrix}\A\to\B\hfill\\\hspace{\widthone}\downarrow\hspace{\widthtwo}\downarrow\hfill\\\D\to\E\hfill\end{matrix}\right)\ar{r}&\hocolim\left(\begin{matrix}\A\to\B\to\C\hfill\\\hspace{\widthone}\downarrow\hspace{\widthtwo}\downarrow\hfill\\\D\to\E\hfill\end{matrix}\right)
\end{tikzcd}
\end{equation}
Since the left square in \eqref{compositediagram} is a homotopy pushout, the left vertical arrow above is a pre-triangulated equivalence.
Thus the right vertical arrow is also a pre-triangulated equivalence.
Now consider the following diagram (induced by the natural maps of diagrams):
\begin{equation}
\begin{tikzcd}
\hocolim\left(\begin{matrix}\A\to\B\to\C\hfill\\\hspace{\widthone}\downarrow\hfill\\\D\hfill\end{matrix}\right)\ar{d}&\ar{l}\hocolim\left(\begin{matrix}\A\to\C\hfill\\\hspace{\widthone}\downarrow\hfill\\\D\hfill\end{matrix}\right)\ar{d}\\
\hocolim\left(\begin{matrix}\A\to\B\to\C\hfill\\\hspace{\widthone}\downarrow\hspace{\widthtwo}\downarrow\hfill\\\D\to\E\hfill\end{matrix}\right)\ar{r}&\hocolim\left(\begin{matrix}\B\to\C\hfill\\\hspace{\widthone}\downarrow\hfill\\\E\hfill\end{matrix}\right)
\end{tikzcd}
\end{equation}
We just saw above that the left vertical arrow is a pre-triangulated equivalence.
The bottom horizontal arrow is a pre-triangulated equivalence by applying Lemma \ref{cofinalityspecialcase} to the natural map in the opposite direction (which is a section), and the top horizontal arrow is a quasi-equivalence by inspection.
It follows that the right vertical arrow is a pre-triangulated equivalence, which gives the desired result.
\end{proof}

\bibliographystyle{amsplain}
\bibliography{descent}

\newcommand{\chsort}[1]{}
\providecommand{\bysame}{\leavevmode\hbox to3em{\hrulefill}\thinspace}
\providecommand{\MR}{\relax\ifhmode\unskip\space\fi MR }
\providecommand{\MRhref}[2]{%
  \href{http://www.ams.org/mathscinet-getitem?mr=#1}{#2}
}
\providecommand{\href}[2]{#2}
\begin{thebibliography}{10}

\bibitem{abbondandoloschwarz}
Alberto Abbondandolo and Matthias Schwarz, \emph{Floer homology of cotangent
  bundles and the loop product}, Geom. Topol. \textbf{14} (2010), no.~3,
  1569--1722. \MR{2679580 (2011k:53126)}

\bibitem{abouzaidhighergenus}
Mohammed Abouzaid, \emph{On the {F}ukaya categories of higher genus surfaces},
  Adv. Math. \textbf{217} (2008), no.~3, 1192--1235. \MR{2383898}

\bibitem{abouzaidcriterion}
\bysame, \emph{A geometric criterion for generating the {F}ukaya category},
  Publ. Math. Inst. Hautes \'Etudes Sci. (2010), no.~112, 191--240. \MR{2737980
  (2012h:53192)}

\bibitem{abouzaidcotangent}
\bysame, \emph{A cotangent fibre generates the {F}ukaya category}, Adv. Math.
  \textbf{228} (2011), no.~2, 894--939. \MR{2822213 (2012m:53192)}

\bibitem{abouzaidtwisted}
\bysame, \emph{On the wrapped {F}ukaya category and based loops}, J. Symplectic
  Geom. \textbf{10} (2012), no.~1, 27--79. \MR{2904032}

\bibitem{abouzaidseidelunpublished}
Mohammed Abouzaid and Paul Seidel, \emph{Lefschetz fibration techniques in
  wrapped {F}loer cohomology}, In preparation.

\bibitem{abouzaidseidel}
\bysame, \emph{An open string analogue of {V}iterbo functoriality}, Geom.
  Topol. \textbf{14} (2010), no.~2, 627--718. \MR{2602848 (2011g:53190)}

\bibitem{aurouxborderedfloericm}
Denis Auroux, \emph{Fukaya categories and bordered {H}eegaard-{F}loer
  homology}, Proceedings of the {I}nternational {C}ongress of {M}athematicians.
  {V}olume {II}, Hindustan Book Agency, New Delhi, 2010, pp.~917--941.
  \MR{2827825}

\bibitem{aurouxborderedfloergokova}
\bysame, \emph{Fukaya categories of symmetric products and bordered
  {H}eegaard-{F}loer homology}, J. G\"{o}kova Geom. Topol. GGT \textbf{4}
  (2010), 1--54. \MR{2755992}

\bibitem{avdek}
Russell Avdek, \emph{Contact surgery, open books, and symplectic cobordisms},
  ProQuest LLC, Ann Arbor, MI, 2013, Thesis (Ph.D.)--University of Southern
  California. \MR{3193067}

\bibitem{avdekjsg}
\bysame, \emph{Liouville hypersurfaces and connect sum cobordisms}, J.
  Symplectic Geom. \textbf{19} (2021), no.~4, 865--957. \MR{4371552}

\bibitem{birancorneacobordism1}
Paul Biran and Octav Cornea, \emph{Lagrangian cobordism. {I}}, J. Amer. Math.
  Soc. \textbf{26} (2013), no.~2, 295--340. \MR{3011416}

\bibitem{birancornealefschetz}
\bysame, \emph{Cone-decompositions of {L}agrangian cobordisms in {L}efschetz
  fibrations}, Selecta Math. (N.S.) \textbf{23} (2017), no.~4, 2635--2704.
  \MR{3703462}

\bibitem{bourgeoisekholmeliashberg}
Fr{\'e}d{\'e}ric Bourgeois, Tobias Ekholm, and Yasha Eliashberg, \emph{Effect
  of {L}egendrian surgery}, Geom. Topol. \textbf{16} (2012), no.~1, 301--389,
  With an appendix by Sheel Ganatra and Maksim Maydanskiy. \MR{2916289}

\bibitem{bourgeoissablofftraynor}
Fr\'{e}d\'{e}ric Bourgeois, Joshua~M. Sabloff, and Lisa Traynor,
  \emph{Lagrangian cobordisms via generating families: construction and
  geography}, Algebr. Geom. Topol. \textbf{15} (2015), no.~4, 2439--2477.
  \MR{3402346}

\bibitem{chantrainerelnonexact}
Baptiste Chantraine, \emph{A note on exact {L}agrangian cobordisms with
  disconnected {L}egendrian ends}, Proc. Amer. Math. Soc. \textbf{143} (2015),
  no.~3, 1325--1331. \MR{3293745}

\bibitem{cdrgggeneration}
Baptiste Chantraine, Georgios~Dimitroglou Rizell, Paolo Ghiggini, and Roman
  Golovko, \emph{Geometric generation of the wrapped {F}ukaya category of
  {W}einstein manifolds and sectors}, Ann. Sci. \'{E}c. Norm. Sup\'{e}r. (4)
  (to appear) \textbf{arXiv:1712.09126} (2017), 1--78.

\bibitem{chekanov}
Yuri Chekanov, \emph{Differential algebra of {L}egendrian links}, Invent. Math.
  \textbf{150} (2002), no.~3, 441--483. \MR{1946550}

\bibitem{cieliebakhandle}
Kai Cieliebak, \emph{Handle attaching in symplectic homology and the chord
  conjecture}, J. Eur. Math. Soc. (JEMS) \textbf{4} (2002), no.~2, 115--142.
  \MR{1911873}

\bibitem{cieliebakeliashberg}
Kai Cieliebak and Yakov Eliashberg, \emph{From {S}tein to {W}einstein and
  back}, American Mathematical Society Colloquium Publications, vol.~59,
  American Mathematical Society, Providence, RI, 2012, Symplectic geometry of
  affine complex manifolds. \MR{3012475}

\bibitem{cohn}
Lee Cohn, \emph{Differential graded categories are k-linear stable infinity
  categories}, Arxiv Preprint \textbf{arXiv:1308.2587} (2013), 1--35.

\bibitem{courtemassot}
Sylvain Courte and Patrick Massot, \emph{Contactomorphism groups and
  {L}egendrian flexibility}, Arxiv Preprint \textbf{arXiv:1803.07997} (2018),
  1--37.

\bibitem{rizellambientsurgery}
Georgios Dimitroglou~Rizell, \emph{Legendrian ambient surgery and {L}egendrian
  contact homology}, J. Symplectic Geom. \textbf{14} (2016), no.~3, 811--901.
  \MR{3548486}

\bibitem{Dyckerhoff-Kapranov}
Tobias Dyckerhoff and Mikhail Kapranov, \emph{Triangulated surfaces in
  triangulated categories}, J. Eur. Math. Soc. (JEMS) \textbf{20} (2018),
  no.~6, 1473--1524. \MR{3801819}

\bibitem{efimovks}
A.~I. Efimov, \emph{A proof of the {K}ontsevich-{S}oibelman conjecture}, Mat.
  Sb. \textbf{202} (2011), no.~4, 65--84. \MR{2830236}

\bibitem{ekholmetnyresullivan}
Tobias Ekholm, John Etnyre, and Michael Sullivan, \emph{Non-isotopic
  {L}egendrian submanifolds in {$\mathbb R^{2n+1}$}}, J. Differential Geom.
  \textbf{71} (2005), no.~1, 85--128. \MR{2191769}

\bibitem{ekholmetnyresabloff}
Tobias Ekholm, John~B. Etnyre, and Joshua~M. Sabloff, \emph{A duality exact
  sequence for {L}egendrian contact homology}, Duke Math. J. \textbf{150}
  (2009), no.~1, 1--75. \MR{2560107}

\bibitem{ekholmhondakalman}
Tobias Ekholm, Ko~Honda, and Tam\'{a}s K\'{a}lm\'{a}n, \emph{Legendrian knots
  and exact {L}agrangian cobordisms}, J. Eur. Math. Soc. (JEMS) \textbf{18}
  (2016), no.~11, 2627--2689. \MR{3562353}

\bibitem{ekholm-lekili}
Tobias Ekholm and Yanki Lekili, \emph{Duality between {L}agrangian and
  {L}egendrian invariants}, Geom. Topol. (to appear) \textbf{arXiv:1701.01284}
  (2017), 1--104.

\bibitem{ekholmng}
Tobias Ekholm and Lenhard Ng, \emph{Legendrian contact homology in the boundary
  of a subcritical {W}einstein 4-manifold}, J. Differential Geom. \textbf{101}
  (2015), no.~1, 67--157. \MR{3356070}

\bibitem{eliashbergstein}
Yakov Eliashberg, \emph{Topological characterization of {S}tein manifolds of
  dimension {$>2$}}, Internat. J. Math. \textbf{1} (1990), no.~1, 29--46.
  \MR{1044658 (91k:32012)}

\bibitem{eliashberg-icm}
\bysame, \emph{Invariants in contact topology}, Proceedings of the
  {I}nternational {C}ongress of {M}athematicians, {V}ol. {II} ({B}erlin, 1998),
  no. Extra Vol. II, 1998, pp.~327--338. \MR{1648083}

\bibitem{eliashbergweinsteinrevisited}
\bysame, \emph{Weinstein manifolds revisited}, Modern geometry: a celebration
  of the work of {S}imon {D}onaldson, Proc. Sympos. Pure Math., vol.~99, Amer.
  Math. Soc., Providence, RI, 2018, pp.~59--82. \MR{3838879}

\bibitem{eliashbergganatralazarev}
Yakov Eliashberg, Sheel Ganatra, and Oleg Lazarev, \emph{Flexible
  {L}agrangians}, Int. Math. Res. Not. IMRN (2020), no.~8, 2408--2435.
  \MR{4090744}

\bibitem{eliashberggromovconvexsymplectic}
Yakov Eliashberg and Mikhael Gromov, \emph{Convex symplectic manifolds},
  Several complex variables and complex geometry, {P}art 2 ({S}anta {C}ruz,
  {CA}, 1989), Proc. Sympos. Pure Math., vol.~52, Amer. Math. Soc., Providence,
  RI, 1991, pp.~135--162. \MR{1128541}

\bibitem{fooochapter10}
Kenji Fukaya, Yong-Geun Oh, Hiroshi Ohta, and Kaoru Ono, \emph{Chapter 10:
  Lagrangian {S}urgery and holomorphic discs \emph{from} {L}agrangian
  intersection {F}loer theory: anomaly and obstruction}, Unpublished chapter
  \textbf{\url{https://www.math.kyoto-u.ac.jp/~fukaya/Chapter10071117.pdf}}
  (2007), 1--182.

\bibitem{gabrielzisman}
P.~Gabriel and M.~Zisman, \emph{Calculus of fractions and homotopy theory},
  Ergebnisse der Mathematik und ihrer Grenzgebiete, Band 35, Springer-Verlag
  New York, Inc., New York, 1967. \MR{0210125}

\bibitem{ganatrawrapcy}
Sheel Ganatra, \emph{Symplectic cohomology and duality for the wrapped {F}ukaya
  category}, Arxiv Preprint \textbf{arXiv:1304.7312} (2013), 1--166, Thesis
  (Ph.D.)--Massachusetts Institute of Technology. \MR{3121862}

\bibitem{gpssectorsoc}
Sheel Ganatra, John Pardon, and Vivek Shende, \emph{Covariantly functorial
  wrapped {F}loer theory on {L}iouville sectors}, Publ. Math. Inst. Hautes
  \'{E}tudes Sci. \textbf{131} (\chsort{a}2020), 73--200. \MR{4106794}

\bibitem{gpswrappedconstructible}
Sheel Ganatra, John Pardon, and Vivek Shende, \emph{Microlocal {M}orse theory
  of wrapped {F}ukaya categories}, ArXiv Preprint \textbf{arXiv:1809.08807}
  (\chsort{c}2018), 1--84.

\bibitem{gaokunneth}
Yuan Gao, \emph{Wrapped {F}loer cohomology and {L}agrangian correspondences},
  Arxiv Preprint \textbf{arXiv:1703.04032} (\chsort{a}2017), 1--70.

\bibitem{gaokunneth2}
\bysame, \emph{Functors of wrapped {F}ukaya categories from {L}agrangian
  correspondences}, Arxiv Preprint \textbf{arXiv:1712.00225} (\chsort{b}2017),
  1--144.

\bibitem{girouxpardon}
Emmanuel Giroux and John Pardon, \emph{Existence of {L}efschetz fibrations on
  {S}tein and {W}einstein domains}, Geom. Topol. \textbf{21} (2017), no.~2,
  963--997. \MR{3626595}

\bibitem{golovkospinning}
R.~Golovko, \emph{A note on the front spinning construction}, Bull. Lond. Math.
  Soc. \textbf{46} (2014), no.~2, 258--268. \MR{3194745}

\bibitem{groman}
Yoel Groman, \emph{Floer theory and reduced cohomology on open manifolds},
  Geom. Topol. \textbf{27} (2023), no.~4, 1273--1390. \MR{4602417}

\bibitem{katzarkovkerr}
Ludmil Katzarkov and Gabriel Kerr, \emph{Partially wrapped {F}ukaya categories
  of simplicial skeleta}, Arxiv Preprint \textbf{arXiv:1708.06038} (2017),
  1--29.

\bibitem{keatingcusp}
Ailsa Keating, \emph{Homological mirror symmetry for hypersurface cusp
  singularities}, Selecta Math. (N.S.) \textbf{24} (2018), no.~2, 1411--1452.
  \MR{3782425}

\bibitem{kontsevichnotes}
Maxim Kontsevich, \emph{Symplectic geometry of homological algebra}, Preprint
  \textbf{\url{http://www.ihes.fr/~maxim/TEXTS/Symplectic_AT2009.pdf} and
  \url{http://www.ihes.fr/~maxim/TEXTS/picture.pdf}} (2009), 1--8.

\bibitem{kontsevich1998lectures}
\bysame, \emph{Lectures at {ENS} {P}aris},  (Spring 1998), set of notes taken
  by {J}. {B}ellaiche, {J.-F.} {D}at, {I.} {M}arin, {G.} {R}acinet and {H.}
  {R}andriambololona.

\bibitem{lekilipolishchuk}
Yanki Lekili and Alexander Polishchuk, \emph{Auslander orders over nodal stacky
  curves and partially wrapped {F}ukaya categories}, J. Topol. \textbf{11}
  (2018), no.~3, 615--644. \MR{3830878}

\bibitem{luriehttpub}
Jacob Lurie, \emph{Higher topos theory}, Annals of Mathematics Studies, vol.
  170, Princeton University Press, Princeton, NJ, 2009. \MR{2522659
  (2010j:18001)}

\bibitem{lurieha}
Jacob Lurie, \emph{{H}igher {A}lgebra}, Available for download (2017), 1--1553.

\bibitem{lyubashenkomultilinear}
Volodymyr Lyubashenko, \emph{{$A_\infty$}-morphisms with several entries},
  Theory Appl. Categ. \textbf{30} (2015), Paper No. 45, 1501--1551.
  \MR{3421458}

\bibitem{lyubashenkomanzyuk}
Volodymyr Lyubashenko and Oleksandr Manzyuk, \emph{Quotients of unital
  {$A_\infty$}-categories}, Theory Appl. Categ. \textbf{20} (2008), No. 13,
  405--496. \MR{2425553}

\bibitem{lyubashenkoovsienko}
Volodymyr Lyubashenko and Sergiy Ovsienko, \emph{A construction of quotient
  {$A_\infty$}-categories}, Homology, Homotopy Appl. \textbf{8} (2006), no.~2,
  157--203. \MR{2259271}

\bibitem{mau}
Sikimeti Ma'u, \emph{Quilted strips, graph associahedra, and {$A_\infty$}
  {$n$}-modules}, Algebr. Geom. Topol. \textbf{15} (2015), no.~2, 783--799.
  \MR{3342676}

\bibitem{maydanskiyseidel}
Maksim Maydanskiy and Paul Seidel, \emph{Lefschetz fibrations and exotic
  symplectic structures on cotangent bundles of spheres}, J. Topol. \textbf{3}
  (2010), no.~1, 157--180. \MR{2608480}

\bibitem{Nadler-Shende}
David Nadler and Vivek Shende, \emph{Sheaf quantization in {W}einstein
  symplectic manifolds}, Arxiv Preprint \textbf{arXiv:2007.10154} (2020).

\bibitem{Nadler-Tanaka}
David Nadler and Hiro~Lee Tanaka, \emph{A stable {$\infty$}-category of
  {L}agrangian cobordisms}, Adv. Math. \textbf{366} (2020), 107026, 97.
  \MR{4070298}

\bibitem{oanceakunneth}
Alexandru Oancea, \emph{The {K}\"unneth formula in {F}loer homology for
  manifolds with restricted contact type boundary}, Math. Ann. \textbf{334}
  (2006), no.~1, 65--89. \MR{2208949}

\bibitem{Pascaleff-Sibilla}
James Pascaleff and Nicol\`o Sibilla, \emph{Topological {F}ukaya category and
  mirror symmetry for punctured surfaces}, Compos. Math. \textbf{155} (2019),
  no.~3, 599--644. \MR{3923362}

\bibitem{seideldehntwist}
Paul Seidel, \emph{A long exact sequence for symplectic {F}loer cohomology},
  Topology \textbf{42} (2003), no.~5, 1003--1063. \MR{1978046}

\bibitem{seidelsubalgebras}
\bysame, \emph{{$A_\infty$}-subalgebras and natural transformations}, Homology
  Homotopy Appl. \textbf{10} (2008), no.~2, 83--114. \MR{2426130}

\bibitem{seidelbook}
\bysame, \emph{Fukaya categories and {P}icard-{L}efschetz theory}, Zurich
  Lectures in Advanced Mathematics, European Mathematical Society (EMS),
  Z\"urich, 2008. \MR{2441780 (2009f:53143)}

\bibitem{seidellefschetzvi}
Paul Seidel, \emph{Fukaya ${A_\infty}$-structures associated to {L}efschetz
  fibrations.\ {VI}}, Arxiv Preprint \textbf{arXiv:1810.07119v2} (2018), 1--83.

\bibitem{sheridanversality}
Nick Sheridan, \emph{Versality of the relative {F}ukaya category}, Geom. Topol.
  \textbf{24} (2020), no.~2, 747--884. \MR{4153652}

\bibitem{sikorav}
Jean-Claude Sikorav, \emph{Some properties of holomorphic curves in almost
  complex manifolds}, Holomorphic curves in symplectic geometry, Progr. Math.,
  vol. 117, Birkh\"{a}user, Basel, 1994, pp.~165--189. \MR{1274929}

\bibitem{sylvanthesis}
Zachary Sylvan, \emph{On partially wrapped {F}ukaya categories}, J. Topol.
  \textbf{12} (2019), no.~2, 372--441. \MR{3911570}

\bibitem{sylvantalk}
Zachary Sylvan, \emph{Talks at {MIT} workshop on {L}efschetz fibrations},
  (March 2015).

\bibitem{thomason}
R.~W. Thomason, \emph{Homotopy colimits in the category of small categories},
  Math. Proc. Cambridge Philos. Soc. \textbf{85} (1979), no.~1, 91--109.
  \MR{510404}

\bibitem{wehrheimwoodwardquilted}
Katrin Wehrheim and Chris~T. Woodward, \emph{Quilted {F}loer cohomology}, Geom.
  Topol. \textbf{14} (2010), no.~2, 833--902. \MR{2602853}

\bibitem{weinstein}
Alan Weinstein, \emph{Contact surgery and symplectic handlebodies}, Hokkaido
  Math. J. \textbf{20} (1991), no.~2, 241--251. \MR{1114405}

\end{thebibliography}
\addcontentsline{toc}{section}{References}

\end{document}